\documentclass[11 pt]{amsart}
\usepackage[utf8]{inputenc}
\usepackage[english]{babel}
\usepackage{amsmath}
\usepackage{amssymb}
\usepackage{amsfonts}
\usepackage{amsthm}
\usepackage{mathrsfs}
\usepackage{graphicx}
\usepackage{enumitem}
\setlist[itemize]{leftmargin=*}
\usepackage{xcolor}
\usepackage{url}
\usepackage{bm}
\usepackage[all]{xy}
\usepackage{xcolor}
\usepackage{hyperref}
\usepackage{esint}
\usepackage{xcolor}
\hypersetup{
      	colorlinks,
      	linkcolor={blue!60!black},
      	citecolor={green!60!black},
      	urlcolor={green!60!black}
}
\usepackage[capitalise,nameinlink]{cleveref}

\crefformat{equation}{#2(#1)#3}
\crefrangeformat{equation}{#3(#1)#4 to~#5(#2)#6}
\crefmultiformat{equation}{#2(#1)#3}{ and~#2(#1)#3}{, #2(#1)#3}{ and~#2(#1)#3}
\usepackage{mathtools}
\usepackage{etoolbox}

\newcommand{\N}{\mathbb{N}}
\newcommand{\Z}{\mathbb{Z}}

\newcommand{\R}{\mathbb{R}}
\newcommand{\C}{\mathbb{C}}

\newcommand{\de}{\partial}

\newcommand{\uno}{\bm{1}}
\newcommand{\nin}{\not\in}

\renewcommand{\bar}{\overline}

\DeclareMathOperator{\spt}{spt}

\newcommand{\ang}[1]{\langle #1\rangle}

\renewcommand{\epsilon}{\varepsilon}

\usepackage[top=2.5cm,bottom=2.5cm,left=2.5cm,right=2.5cm]{geometry}

\makeatletter
\renewcommand{\tocsection}[3]{%
  \indentlabel{\@ifnotempty{#2}{\bfseries\ignorespaces\makebox[35pt][l]{#1 #2\quad}}}\bfseries#3}
\renewcommand{\tocsubsection}[3]{%
  \indentlabel{\@ifnotempty{#2}{\ignorespaces\makebox[35pt][l]{#1 #2\quad}}}#3}
\newcommand\@dotsep{4}
\def\@tocline#1#2#3#4#5#6#7{\relax
  \ifnum #1>\c@tocdepth 
  \else
    \par \addpenalty\@secpenalty\addvspace{#2}%
    \begingroup \hyphenpenalty\@M
    \@ifempty{#4}{%
      \@tempdima\csname r@tocindent\number#1\endcsname\relax
    }{%
      \@tempdima#4\relax
    }%
    \parindent\z@ \leftskip#3\relax \advance\leftskip\@tempdima\relax
    \rightskip\@pnumwidth plus1em \parfillskip-\@pnumwidth
    #5\leavevmode\hskip-\@tempdima{#6}\nobreak
    \leaders\hbox{$\m@th\mkern \@dotsep mu\hbox{.}\mkern \@dotsep mu$}\hfill
    \nobreak
    \hbox to\@pnumwidth{\@tocpagenum{\ifnum#1=1\bfseries\fi#7}}\par
    \nobreak
    \endgroup
  \fi}
\AtBeginDocument{%
\expandafter\renewcommand\csname r@tocindent0\endcsname{0pt}
}
\def\l@section{\@tocline{1}{0pt}{1pc}{5pc}{}}
\def\l@subsection{\@tocline{2}{0pt}{2.5pc}{5pc}{}}
\makeatother      

\newtheorem{Th}{Theorem}
\newtheorem{Prop}[Th]{Proposition}
\newtheorem{Co}[Th]{Corollary}
\newtheorem{Lm}[Th]{Lemma}

\theoremstyle{definition}
\newtheorem{Dfi}[Th]{Definition}
\newtheorem{Rm}[Th]{Remark}

\numberwithin{equation}{section}
\numberwithin{Th}{section}
\def\thesection{\Roman{section}}

\newcommand{\be}{\begin{equation}}
\newcommand{\ee}{\end{equation}}
\newcommand{\bes}{\begin{equation*}}
\newcommand{\ees}{\end{equation*}}

\newcommand\res{\mathop{\hbox{\vrule height 7pt width .5pt depth 0pt
\vrule height .5pt width 6pt depth 0pt}}\nolimits}

\def\ti{\tilde}
\def\lf{\left}
\def\rg{\right}

\def\al{\alpha}
\def\la{\lambda}

\def\ep{\varepsilon}

\def\ds{\displaystyle}
\def\ov{\overline}
\def\Om{\Omega}
\def\om{\omega}
\def\p{\partial}

\def\r{\mathfrak r}

\def\res{\mathop{\hbox{\vrule height 7pt width .5pt 
depth 0pt\vrule height .5pt width 6pt depth 0pt}}\nolimits}

\def\Xint#1{\mathchoice
{\XXint\displaystyle\textstyle{#1}}
{\XXint\textstyle\scriptstyle{#1}}
{\XXint\scriptstyle\scriptscriptstyle{#1}}
{\XXint\scriptscriptstyle\scriptscriptstyle{#1}}
\!\int}
\def\XXint#1#2#3{{\setbox0=\hbox{$#1{#2#3}{\int}$ }
\vcenter{\hbox{$#2#3$ }}\kern-.6\wd0}}

\def\dashint{\Xint-}

\begin{document}
        \title[A variational theory for the area of Legendrian surfaces]{A variational theory for the area \\ of Legendrian surfaces}
        \author{Alessandro Pigati
        \ and\ Tristan Rivi\`ere
        }
       	\begin{abstract}
       		We study a new notion of critical point for the area of surfaces {under the Legendrian constraint}, first introduced in \cite{Riv3} and called \emph{parametrized Hamiltonian stationary Legendrian varifolds (PHSLVs)}.
            We establish several fundamental properties of these objects, including their sequential compactness 
            and an optimal regularity result, showing that they are smooth immersions away from a locally finite set of branch points and Schoen--Wolfson conical singularities.
            This generalizes in particular the regularity theory of Schoen--Wolfson for minimizers \cite{SW2} to general critical points.
            
            This theory can be used to show two new variational results: every min-max operation with the area of (closed, immersed) Legendrian surfaces  in a closed Sasakian 5-dimensional manifold
            is achieved by a Hamiltonian stationary map with this regularity;
            also, the minimal area in any given exact isotopy class of Legendrian immersions of $S^2$ is realized by such a map.
            
            Along the way, we prove an effective {monotonicity formula} for general two-dimensional stationary varifolds in the Legendrian setting, 
            as well as the closure of integral stationary varifolds among rectifiable ones, in spite of the lack of compactness of the latter.
       	\end{abstract}
    
        \maketitle
        \tableofcontents
        \frenchspacing
		\raggedbottom

\section{Introduction}

In the early 90's, Oh in \cite{Oh2} introduced the problem of studying critical points of the area among Lagrangian surfaces in an arbitrary symplectic Riemannian manifold. Such surfaces are called {\it Hamiltonian stationary} or \emph{H-minimal} surfaces. This variational problem  is motivated by natural questions such as the study of the Plateau problem in Lagrangian homology classes or the construction of calibrated minimal surfaces within given Hamiltonian isotopy classes in Calabi--Yau geometries (Thomas--Yau conjecture); more recently, the second-named author stressed the importance of the study of area variations among Lagrangian surfaces for the min-max construction of minimal surfaces in the sphere $S^n$, in relation to the Willmore conjecture, or other special ambients \cite{Riv-Will}.

\subsection{The parametric approach}

While considering variational problems for Lagrangian surfaces, it is natural at first to adopt the classical parametric approach of  Douglas and  Rad\'o, who gave at the time a successful framework for the  resolution of the {\it Plateau problem} in  a Euclidean space  or more generally in a Riemannian manifold (after the work of  Sacks and Uhlenbeck). The central objects in this parametric  approach would be, in the Lagrangian-constrained case, weakly conformal $W^{1,2}$ maps $v$ from a Riemann surface into the symplectic manifold $(M^m,g,\om)$ canceling the symplectic form: $v^\ast\om=0$. In their pioneering analytical work on the area variation under the pointwise Lagrangian constraint,  Schoen and  Wolfson \cite{SW2} adopted  this framework to prove the existence of  area minimizers in any Lagrangian homology class, enjoying also some partial regularity that we are going to make more precise below.

While the existence part can be obtained by using  relatively mild arguments, the regularity of a minimizer poses serious difficulties which eventually have been overcome by Schoen and Wolfson.   One of the new challenges posed by the Lagrangian constraint, compared  to the classical ``isotropic case'' (i.e., the unconstrained  case of Douglas, Rad\'o, and Sacks--Uhlenbeck), comes from the Euler--Lagrange equation. Even while considering the simplest  possible framework of maps  $v$ from a closed Riemann surface $\Sigma$ into ${\C}^2$,  equipped with the standard
symplectic form $\om:=dz_1\wedge dz_2+dz_3\wedge dz_4$, the constraint $v^\ast\om=0$ is generating a {\it Lagrange  multiplier} which happens to be  a map, the so-called {\it Lagrangian angle} ${\mathfrak g}:\Sigma\rightarrow S^1$, on  which absolutely nothing is known
from a function theoretic perspective. For any Lagrangian immersion $v : \Sigma\rightarrow {\C}^2$,  the Lagrangian angle is given by
\[
v^\ast(dw_1\wedge dw_2)= {\mathfrak g}\, d\operatorname{vol}_{v},
\]
where $w_1=z_1+iz_2$, $w_2=z_3+iz_4$, and $d\operatorname{vol}_{v}$ is the volume form induced by the immersion $v$. A classical result, due to Dazord (which extends to general K\"ahler--Einstein manifolds \cite{Daz}; see also \cite{Cas}), gives the expression of the mean curvature of the immersion into ${\C}^2$ in terms of the Lagrangian angle:
\[
\vec{H}_v=2^{-1}{\frak g}^{-1}\vec{\nabla} {\frak g},
\]
where $\vec{\nabla} {\frak g}$ is the intrinsic gradient of the complex-valued function ${\frak g}$ on the immersed surface (recall that, in a conformal chart $(x_1,x_2)$ for the immersion $v$, one has
${\frak g}^{-1}\vec{\nabla} {\frak g}=e^{-2\la} {\frak g}^{-1}{\nabla} {\frak g}\cdot\nabla v$, where $e^\lambda$ is the conformal factor).
Coming back to the variations of the area under Lagrangian constraint,  the corresponding Euler--Lagrange equation formally reads in isothermal coordinates as follows:
\begin{equation}
\label{0.1}
\begin{cases}
\operatorname{div}({\frak g}^{-1}\nabla v)=0\\
\operatorname{div}({\frak g}^{-1}\nabla {\frak g})=0.
\end{cases}
\end{equation}
If $v$ is a smooth conformal  immersion, this equation is  characterizing the {\it Hamiltonian stationary Lagrangian surfaces}.  From a purely variational perspective, however,   in absence of any information on ${\mathfrak g}$ and assuming only $v\in W^{1,2}(\Sigma,{\C}^2)$, one cannot even give  a meaning to the Euler--Lagrange equation (for instance, we would need at least $\mathfrak{g}\in H^{1/2}(D^2,S^1)$ in order  to give a distributional meaning to the second equation in \eqref{0.1}). In order to overcome this major obstacle and perform a PDE analysis for a  minimizer in a Lagrangian homology class,  Schoen and Wolfson restricted attention  to area variations given by infinitesimal deformations in the target  preserving the Lagrangian constraint $v^\ast\om=0$. Such compactly supported variations, called {\it Hamiltonian variations}, are of the form
\[
V_F:=J\nabla^{{\C}^2}F,
\]
where $F$ is an arbitrary compactly supported function in ${\C}^2$, called a {\it Hamiltonian} potential, and the criticality of  a conformal map $u$ can be formulated  as follows:
\begin{equation}
\label{0.2}
\int_{\Sigma}\nabla(J(\nabla^{{\C}^2}F)\circ  v)\cdot_{{\C}^2}\nabla v\, dx^2=0\quad\text{for all }F\in C^\infty_c({\C}^2).
\end{equation}
Obviously this condition by itself is not strong enough to help developing a regularity theory: even in the isotropic case (i.e., in absence of the Lagrangian constraint) a  condition of the form
$\int_{\Sigma}\nabla(X\circ  v)\cdot\nabla v\, dx^2=0$ for any compactly supported vector field $X$
does not imply that $v$ is harmonic and smooth.

\subsection{Lagrangian versus Legendrian}

As a first step, while performing  a compactness and regularity theory for a variational problem on the area allowing exclusively variations in the ambient, we look for  a \emph{monotonicity} property.   However, it has been  discovered by Minicozzi \cite[Section 3]{Min} and Schoen--Wolfson \cite{SW2} that Hamiltonian stationary Lagrangian surfaces, even in the simplest framework of ${\C}^2$ equipped with the standard symplectic form, do not enjoy a monotonicity property of the area.

Nevertheless,  Schoen and Wolfson also discovered that an monotonicity property exists for the {Legendrian lifts} (when they exist) of these surfaces, called  {\it Hamiltonian stationary Legendrian surfaces}. A Lagrangian map $v$ from a surface $\Sigma$ into ${\C}^2$ is said to \emph{admit a Legendrian lift} (or to be \emph{exact}) if there exists a global function $\varphi:\Sigma\to{\R}$ such that
\[
d\varphi_v=v_1\,dv_2-v_2\, dv_1+v_3\,dv_4-v_4\, dv_3.
\]
In other words, denoting
$$\mathbb{H}^2:=\C^2\times\R,$$
which has a natural group structure called the \emph{Heisenberg group},
the map  $u:=(v,\varphi_v):\Sigma\to\mathbb{H}^2$ is canceling the canonical contact form $\alpha$ on $\mathbb{H}^2$, i.e.,
\[
u^*\al=0,\quad\al := -d\varphi+z_1\, dz_2-z_2\,dz_1+z_3\,dz_4-z_4\, dz_3,
\]
where we use coordinates $(z,\varphi)$ on $\C^2\times\R$.
Observe moreover that the tangent map $\pi_\ast$ to the canonical projection $\pi:\mathbb{H}^2\to\C^2$ realizes an isometry from the horizontal planes $H:=\operatorname{ker}(\al)$ onto ${\C}^2$ for the metric
\[
g_{{\mathbb H}^2}=\pi^\ast g_{{\C}^2}+\al\otimes\al.
\] 
Hence, the area of any Legendrian lift coincides with the area of the Lagrangian map, and \emph{locally} the study of area variations for the former under Legendrian constraints corresponds to the study of area variations for the latter under Lagrangian constraints. 

At the Legendrian level, the  vector fields preserving infinitesimally the Legendrian constraint are called {\it Hamiltonian vector fields} and have the form
\[
2W_F:=J_H\nabla^HF-2F \p_\varphi,
\]
where $F$ is an arbitrary smooth, compactly supported function on ${\mathbb H}^2$ called a {\it Hamiltonian} potential, $\nabla^H$ is the orthogonal projection of $\nabla F$ onto $H$, and $J_H$ is the pullback by $\pi$ (on $H$) of the canonical complex structure $J$ on ${\C}^2$.
Hence, the stationarity condition at the Legendrian level reads
\begin{equation}
\label{0.2-a}
\int_{{\C}}\nabla(W_F\circ  u)\cdot_{{\mathbb H}^2}\nabla u\, dx^2=0
\quad\text{for all } F\in C_c^\infty({\mathbb H}^2).
\end{equation}
In \cite{SW2} it is proved that every $C^1$ Lagrangian  map admits locally a Legendrian lift, and that this property extends to $W^{1,2}$ area-minimizing Lagrangian maps \cite[Corollary 2.9]{SW2}.  The Hamiltonian vector fields $W_F$ at the Legendrian level are obviously more numerous than the ones at the Lagrangian level; in particular, one of the key insights of Schoen--Wolfson is that the infinitesimal action of dilations at the Lagrangian level (which is known to be the action generating monotonicity properties) is Hamiltonian at the Legendrian level but not at the Lagrangian one:
we have
\[
-2W_{\varphi}=\sum_{j=1}^4z_j\nabla^Hz_j+2\varphi \p_\varphi,
\]
while $\sum_{j=1}^4z_j\p_{z_j}$ is not equal to $J\nabla F$ for any $F:\C^2\to\R$.
This explains why a monotonicity property should hold ``upstairs'' for the Legendrian lifts while it does not hold ``downstairs'' at the Lagrangian level.

\subsection{Area minimization in Legendrian homology classes}

Due to the existence of a monotonicity formula for Legendrian stationary maps  it is then natural to pose the regularity question at the level of Legendrian maps into ${\mathbb H}^2$. Thanks to a result by Godlinski, Kopczynski, and Nurowski \cite{GKN}, the Heisenberg group ${\mathbb H}^2$ is the infinitesimal model for any Sasakian manifold; recall that a Riemannian manifold $(M^5,g,\al)$, where $\al$ is a contact form, is called a \emph{Sasakian manifold} if the cone $({\R}_+\times M^5, k:=dt^2+t^2 g)$ with the  non-degenerate symplectic form $2^{-1} d(t^2\al)$ is K\"ahler (a standard example is $S^5$). Letting ${J}$ be the compatible complex structure, the tangent vector field along $M^5$ given by
${R}:=J(t\p_t)$ has unit norm and is orthogonal to the horizontal hyperplanes given by $H=\operatorname{ker}(\al)$; such a vector field is called a {\it Reeb vector field} for the distribution $H$.

Moreover, this distribution of planes is invariant under the action of $J$ in the cone ${\R}_+\times M^5$. Thus, Sasakian manifolds are the ``odd-dimensional counterparts'' of K\"ahler manifolds (see a more detailed discussion in \cite[Section 6]{Riv3}). The main result of Schoen and Wolfson, formulated at the Legendrian level, is the following.

\begin{Th}
\label{th-SW2}\cite{SW2} Let $(M^5,g,\al)$ be a Sasakian manifold. Any Legendrian homology class in $M$ is realized by an area-minimizing Hamiltonian stationary Legendrian map from a closed Riemann surface to $M$. Moreover, this map is weakly conformal and is a smooth immersions away from finitely many branch points and isolated conical singularities.
\hfill $\Box$
\end{Th}

The possible existence of these conical singularities, called ``Schoen--Wolfson cones,''  is one of the main discoveries in  \cite{SW2}. It has been proved moreover in \cite{MiWo} that for some particular homology classes (in certain ambients) \emph{any} minimizer must have such singularities. These have been classified in \cite{SW2} and their blow-up is of the form
\[
u_{p,q}: {\C}\to \mathbb{C}^2,\quad (r,\theta)\mapsto \frac{r^{\sqrt{pq}}}{\sqrt{p+q}} \begin{pmatrix}\sqrt{q}e^{ip\theta}
\\ i\sqrt{p}e^{-iq\theta}\end{pmatrix},
\]
where $p,q\in \N^*$. The intersections with $S^3$ are the  $(p,q)$ torus knots. These singularities are of topological nature. The space of oriented Lagrangian planes in ${\C}^2$ is $\Lambda(2)\cong U(2)/SO(2)$. The determinant operation gives a well-defined map from $\Lambda(2)$ to $S^1$. The restriction of this map to the tangent bundle of the immersion $u_{p,q}$ on any positively oriented circle surrounding the conical singularity gives a map $S^1\to S^1$ of topological degree $p-q$, called the {\it Maslov class} at the point. 

Understanding the possible locations of the Schoen--Wolfson cones is still an open problem in general. Some progress in answering this challenging question has been made in \cite{GOR} and \cite{Gaia}.

The proof  of the regularity part of Theorem \ref{th-SW2} is based on a monotonicity formula (established for regular enough Hamiltonian stationary immersions), a  {\it small tilt-excess regularity} theorem, a classification of the tangent cones,  and the property  satisfied by minimizers in homology classes that weakly converging sequences of such maps are in fact strongly converging in $W^{1,2}$. This last property  is reminiscent of the one discovered in \cite{ScU}
by Schoen and Uhlenbeck and is established in \cite{SW2} using a comparison argument, which is not viable for general critical points.

%

\subsection{The notion of parametrized Hamiltonian stationary Legendrian varifolds (PHSLV)}
As  announced in the abstract, one of the main goals of the present paper is to extend the Schoen--Wolfson existence and regularity theory to \emph{non-minimizing} Hamiltonian stationary Legendrian surfaces. This ambition faces numerous new difficulties, among which we are listing the following two major ones:
\begin{itemize}[leftmargin=8mm]
\item[(i)] find a suitable class of Legendrian ``weak surfaces'' for which a min-max theory can be developed;
\item[(ii)] bypass the problem that the Euler--Lagrange equation \eqref{0.1} is not well-posed while studying the regularity of non-minimizing Hamiltonian stationary Legendrian ``weak surfaces,'' thus avoiding comparison arguments.
\end{itemize}

A more comprehensive discussion of new phenomena and new difficulties arising in this Legendrian setting is
provided later in the introduction.

As a chief example, we would like to stress that, in contrast with the now well-understood unconstrained case, also called {\it isotropic case} in this paper (i.e., critical points of the area without the Legendrian constraint: minimal surfaces), the possible existence of conical singularities, whose number is a priori totally uncontrolled, and the  possible formation of a ``continuum'' of those along a weakly converging sequence, is creating a completely new technical challenge, which in the minimizing case in \cite{SW2} was solved by the property that
\begin{equation}
W^{1,2}\text{ weak convergence}\quad\Longleftrightarrow\quad W^{1,2}\text{ strong convergence}. \tag{P}
\end{equation}
This last property is a priori not available anymore in the general case  and one of the main achievements of the present work is to obtain the same \emph{optimal} Schoen--Wolfson regularity result without property (P).

In \cite{Riv3} the second-named author proved that every nontrivial min-max operation on the area among Legendrian surfaces in a closed  Sasakian five-dimensional manifold $(M^5,g,\al)$ is always achieved by a {\it parametrized Hamiltonian stationary Legendrian varifold (PHSLV)}. A PHSLV is a triple $(\Sigma,u,N)$ where:

\begin{itemize}[leftmargin=8mm]
\item[(i)] $\Sigma$ is a possibly open Riemann surface without boundary;
\item[(ii)] $u\in W^{1,2}_{loc}(\Sigma, M^5)$ satisfies 
\[
u^\ast\al=0\quad\text{ and }\quad\p u\dot{\otimes}\p u=0,
\]
that is, $u$ is Legendrian and weakly conformal;
\item[(iii)] $N\in L^\infty_{loc}(\Sigma, \N^\ast)$;
\item[(iv)] for any $f\in C^\infty_c(\Sigma,{\R}_+)$ and for a.e. $t>0$, given a compactly supported Hamiltonian vector field $W_F$ on $M^5$ such that
\[
\operatorname{spt}(F)\subset\subset M\setminus u(f^{-1}(t)),
\]
there holds
$$\int_{f>t}N\nabla(W_F\circ u)\cdot\nabla u\, dx^2=0.$$
\end{itemize}

Observe that, compared to the parametric approach of Schoen and Wolfson, this class of objects differs by two main properties:
\begin{itemize}[leftmargin=8mm]
\item[(i)]
the definition of PHSLV allows for general {\it integer multiplicity} $N$,
which is needed to guarantee general compactness properties of this class, while in the minimizing case from \cite{SW2} one can restrict  to $N\equiv 1$,
since in compactness and blow-up arguments property (P) ensures that $N\equiv1$ still holds at the limit,
while this is a priori not true for general critical points;
\item[(ii)] the stationarity condition is much more general than \eqref{0.2-a}: it permits to ``localize'' the stationarity property  of the varifold associated with $(\Sigma,u,N)$ in terms of the \emph{domain} of $u$
(in \cite{SW2} the localization property is sometimes implicitly used, often in conjunction with the minimizing hypothesis,
whereas it is systematically introduced and then exploited at several crucial points in the present work).
\end{itemize}
The simpler notion of {\it parametrized stationary varifold (PSV)} has been introduced in \cite{Riv-IHES}, where the main result of \cite{Riv3} was established for the area in the isotropic case inside any closed Riemannian manifold. 
The existence of a parametrization, together with the corresponding localization property, was compensating for the absence of a PDE while considering the resulting stationary varifold, and opened the door to the optimal regularity result proved by the authors in \cite{PiRi1},
as well as a better understanding of parametrized varifolds arising variationally \cite{PiRi2,Pi}.
It has been proved by the two authors in \cite{PiRi1} that  the space of PSVs in a closed Riemannian manifold is sequentially compact in a suitable sense. One of the purposes of the first  part of the present work is to extend these facts to the Legendrian framework. 

Our ultimate motivation is to study the regularity of a PHSLV in an arbitrary {\it Sasakian} five-dimensional manifold, and in particular the realization by Hamiltonian stationary surfaces of an arbitrary nontrivial min-max value for the area among Legendrian surfaces in such a manifold. The present work is the second step after \cite{Riv3} in this program. Since we are interested in local properties of PHSLV, we will mostly work in the Heisenberg Group ${\mathbb H}^2$, which is the universal blow-up of such manifolds,
but throughout the paper we will point out the correct analogue of each important statement in a general ambient (which does not enjoy the dilation symmetry).

\subsection{Main results of the present work}
In the first part of this work, we begin with a general theory of \emph{Hamiltonian stationary Legendrian varifolds (HSLVs)},
which was missing in the literature,
providing a geometric measure theory toolbox which is going to be used heavily later on.
In particular, we show an effective monotonicity formula, generalizing \cite{Riv2} (which required having a smooth immersion)
to an arbitrary HSLV.
For simplicity, we state it in $\mathbb{H}^2\cong\C^2\times\R$, where we use coordinates $(z,\varphi)$ and let
$\r^4:=|z|^4+4\varphi^2$ and $\sigma:=2\varphi/|z|^2$ (whose arctangent is a smooth function on $\mathbb{H}^2\setminus\{0\}$).

\begin{Th}
Given a HSLV $\mathbf{v}$ on $\mathbb{H}^2$ and a suitable cut-off function $\chi:\R_+\to\R_+$, letting
\begin{align*}\Theta^\chi(q,a)&:=-\int_{G} \frac{|\nabla^{\mathcal P}\r_q|^2}{\r_q} a^{-1}\chi'(\r_q/a) \,d{\mathbf v}({\mathcal P},p)\\
&\phantom{:}\quad-\int_{G} \lf[\frac{2\varphi_q}{\r_q^3} a^{-1}\chi'(\r_q/a) \arctan\sigma_q\rg] \, d{{\mathbf v}}({\mathcal P},p)\\
&\phantom{:}\quad+\frac{1}{4}\int_{G} {\r_q^4}  \nabla^{\mathcal P}\arctan\sigma_q\cdot\nabla^{\mathcal P} [\r_q^{-3} a^{-1}\chi'(\r_q/a)  {\arctan\sigma_q} ]\, d{{\mathbf v}}({\mathcal P},p),\end{align*}
we have
$$0\le\Theta^\chi(q,a)\le\Theta^\chi(q,b)\quad\text{for all }0<a<b,$$
as well as
$$\theta^\chi(q)+\int_{0<\r_q<b}|\nabla^{\mathcal P}\arctan\sigma_q|^2\,d\mathbf{v}(\mathcal{P},p)
\le Cb^{-2}|\mathbf{v}|(B_{2b}^\r(q)\setminus \bar B_b^\r(q)),$$
where $\theta^\chi(q):=\lim_{\ep\to0}\Theta^\chi(q,\ep)\in\R_+$ exists and is called the \emph{density} at $q$. \hfill $\Box$
\end{Th}

Here $\r_q,\varphi_q,\sigma_q$ are defined by left translation, shifting the origin to $q$.
Compared to similar statements in \cite{SW2} and \cite{Riv2}, this is a true monotonicity at all scales
and works for general (Hamiltonian stationary) varifolds. In turn, this is used to show the following.
Note that monotonicity does \emph{not} hold for higher dimensional objects, as shown in \cite[Appendix B]{Orr}.

\begin{Th}
The class of rectifiable HSLVs $\mathbf{v}$ with $\frac{\theta^\chi}{2\pi}\in\N^*$ a.e. on $\operatorname{spt}|\mathbf{v}|$
is closed \emph{among rectifiable varifolds}. \hfill $\Box$
\end{Th}

We observe that compactness of such integral HSLVs fails, even in a closed ambient. The following counterexample
is inspired by a similar one in $\mathbb{H}^2$ by Orriols \cite{Orr}; the varifolds $\mathbf{v}_k$
are in fact smooth Hamiltonian stationary Legendrian embeddings.

\begin{Th}
In $(S^5,g,\al)$, with $g$ the round metric and $\al$ the canonical contact form,
there exists a sequence of rectifiable HSLVs $\mathbf{v}_k$ with $\frac{\theta^\chi(\mathbf{v}_k,q)}{2\pi}\in\N^*$ for all $q\in\operatorname{spt}|\mathbf{v}_k|$, such that $\mathbf{v}_k\rightharpoonup\mathbf{v}_\infty$ for a non-rectifiable limit
$\mathbf{v}_\infty$ supported on a Hopf fiber. \hfill $\Box$
\end{Th}

We should also mention that the previous results, namely monotonicity and closure of integral varifolds among rectifiable ones,
require some completely new ideas, as discussed more in detail later in the introduction.

After developing this general theory, we turn to the structure of \emph{parametrized} varifolds and their sequential compactness. The following theorem is one of the main achievements.

\begin{Th}
\label{th-0}
In a closed Sasakian ambient $M^5$, given $k\in\N$,
the set of varifolds induced by PHSLV$^*$s $(\Sigma,u,N)$ with  closed domain $\Sigma$
and $\operatorname{genus}(\Sigma)\le k$
is \emph{sequentially closed} under varifold convergence. \hfill $\Box$
\end{Th}

In this statement, $\Sigma$ is possibly disconnected and $\operatorname{genus}(\Sigma):=\sum_{S}\operatorname{genus}(S)$
as $S$ varies among the connected components of $\Sigma$. This result holds assuming the slightly stronger notion
of PHSLV$^*$, given in Definition \ref{strong.phslv}; for a smooth immersion in a K\"ahler--Einstein $M^5$, the latter is equivalent to the fact that the closed one-form
$\ast {\frak g}^{-1}\, d{\frak g}$ is exact.
We also have the following.

\begin{Th}
\label{th-0bis}
In a closed Sasakian ambient $M^5$, the set of varifolds induced by PHSLVs $(\Sigma,u,N)$ with a fixed closed domain $\Sigma$
and \emph{controlled conformal class}
is sequentially closed under varifold convergence. \hfill $\Box$
\end{Th}

However, we stress that, differently from the isotropic case,
the statement fails in the class of PHSLVs, as shown by the previous counterexample.

The question whether or not the weak sequential closure holds in the class of PHSLVs while assuming a control of the Legendrian Morse index is under investigation by the two authors (the latter is the dimension of the largest subspace of Hamiltonian variations where the second variation of area is negative definite).

In situations where degeneration of the conformal class do not occur,
this weak sequential closure holds in the class of PHSLVs. In the regularity theory we mostly deal
with domains such as $\mathbb{D}$ or $\C$ (and we need to consider sequences of maps with energy lower than possible bubbles,
such as blow-ups). This explains why our regularity theory assumes just the notion of PHSLV.
In particular, we can classify blow-ups of an arbitrary PHSLV.

\begin{Th}
Given a PHSLV $(\Sigma,u,N)$ in a closed Sasakian ambient $M^5$ and $x_0\in\Sigma$, there exists a notion of \emph{parametrized blow-up}
at $x_0$. The image of any such parametrized blow-up $(\C,u_{x_0},N_{x_0})$ is either a plane or a Schoen--Wolfson cone in $\C^2\subset\mathbb{H}^2$. \hfill $\Box$
\end{Th}


The sequential compactness of PHSLVs (in low energy regime, or equivalently when $M^5$ is rescaled to resemble $\mathbb{H}^2$)
is then exploited in successive steps to reach the following \emph{optimal regularity} result,
which is the first main outcome of this work.

\begin{Th}
\label{th-01} Every PHSLV $(\Sigma,u,N)$ in a closed Sasakian manifold $M^5$ is a smooth immersion away from isolated branch points and isolated conical singularities (whose blow-ups are Schoen--Wolfson cones), with $N$  constant  on each connected component of $\Sigma$.
\hfill $\Box$
\end{Th}

In the previous statements, we tacitly assume that $u$ is not constant on any connected component of $\Sigma$.
The classification of blow-ups is obtained in tandem with the regularity theorem, proceeding by induction on $\sup_\Sigma N$
(roughly speaking). Again, there are several new difficulties compared to the isotropic case \cite{PiRi1},
whose discussion is postponed.

Let us now come to variational applications.
Given $(M^5,g,\al)$ a Sasakian manifold, and $\Sigma$ a closed oriented surface, we introduce the set $\mathfrak{M}$ of Legendrian $W^{2,4}(\Sigma,M)$ immersions  from $\Sigma$ to $M$.
It is proved in \cite{Riv3} that ${\mathfrak M}$ has the structure of Banach manifold and possesses a compatible Finsler structure for which the associated {\it Palais distance} is complete.

A collection ${\mathcal A}$ of compact subsets of ${\mathfrak M}$
is said to be an \emph{admissible family} if it is invariant under the action of homeomorphisms of ${\mathfrak M}$ isotopic to the identity
(in fact, one can also require this just for deformations that agree with the identity except near the energy level $W$ defined below). The min-max value, also called the \emph{width} associated with $\mathcal{A}$, is
\[
W(\mathcal A):=\inf_{A\in{\mathcal A}} \max_{u\in A} \operatorname{area}(u).
\]
Our second main result, which is obtained by combining Theorem \ref{th-01} with the main result of \cite{Riv3}, is then the following.

\begin{Th}
\label{th-02}
Let $(M^5,g,\al)$ be a closed Sasakian five-dimensional manifold. Let ${\mathcal A}$ be an admissible family in the Banach Manifold ${\mathfrak M}$
whose width
\[
W(\mathcal{A})>0.
\]
Then $W(\mathcal{A})$ is the area of a smooth Hamiltonian stationary Legendrian immersion $u:\Sigma'\to M$, possibly with isolated branch points and Schoen--Wolfson conical singularities, whose domain $\Sigma'$ is a possibly disconnected closed oriented surface
with $\operatorname{genus}(\Sigma')\le \operatorname{genus}(\Sigma)$.
\hfill $\Box$
\end{Th}

In order to state another application, we recall a classical notion from Legendrian co-bordism theory originally introduced by Arnol'd. A regular isotopy $u_t:\Sigma\to(M^5,g,\al)$  of Legendrian immersions is called {\it exact} if there exists a family of Hamiltonian functions 
$f_t:\Sigma\to\R$ such that the variation of $u_t$ is the Hamiltonian vector field generated by $f_t$. In conformal coordinates for $u_t$,
calling $e^\la$ the conformal factor, this reads
\[
\frac{d u}{dt}=e^{-2\la} J_H \nabla f_t\cdot\nabla u_t-2 f_t \p_\varphi.
\]
If $u_t$ is an embedding for every $t$, this notion of exact regular isotopy coincides with the classical notion of {\it Hamiltonian isotopy}.

\begin{Th}
\label{th-03}
Let $M^5$ be a closed Sasakian manifold and let ${\mathfrak C}$ be an exact regular isotopy class of Legendrian immersions of a closed surface $S^2$ into $M^5$ such that
\[
 {\mathfrak A}:=\inf_{u\in {\mathfrak C}} \operatorname{area}(u)>0.
\]
Then $ {\mathfrak A}$ equals the area of a smooth Hamiltonian stationary Legendrian immersion $u:\Sigma'\to M$, possibly with isolated branch points and Schoen--Wolfson conical singularities, whose domain $\Sigma'$ is a union of spheres.
\hfill $\Box$
\end{Th}

Observe that, even though it is dealing with a minimization problem, this result cannot be deduced from the main result in \cite{SW2}. Indeed, the regularity results proved in \cite{SW2} are based on comparison arguments replacing $u$ with maps within the same homology class,
but a priori not in the same isotopy class. Theorem \ref{th-03} is particularly interesting since Hamiltonian isotopy classes are known to be immensely more numerous than Lagrangian homology classes.

%

\subsection{New phenomena compared to the isotropic setting}

We now highlight some of the chief novelties of the Legendrian setting creating some of the new difficulties that we face, compared to the simpler isotropic situation. Besides these, there are several additional technical difficulties: just to mention another one, it is sometimes hard to come up with efficient proofs
which in the isotropic case just come from an intuitive choice of a vector field, since in the definition of stationarity
we are restricted to Hamiltonian vector fields $W_F$, which involve two differentiations of $F$ intertwined with the rotation $J_H$.

Broadly speaking, there are three major phenomena, appearing at increasing levels of weakness of the notion of
Hamiltonian stationary Legendrian surface that we consider. We discuss them in $\mathbb{H}^2$, for simplicity.

\begin{itemize}[leftmargin=8mm]
\item[(i)] Assuming that we have a smooth conformal immersion $u$, for the projection $v:=\pi\circ u$ we have a PDE of the form
$$\Delta v+i\nabla \beta\cdot\nabla v=0.$$
This differs from the usual Laplace equation (that one would have in the isotropic case)
by a term involving a harmonic one-form $h:=d\beta$. In spite of its qualitative smoothness, we do not have any quantitative bound on $h$ a priori. Indeed, it appears as a sort of Lagrange multiplier associated to the pointwise Legendrian constraint.
This makes it difficult to derive useful elliptic estimates from the PDE.
\item[(ii)] Removing smoothness assumptions on $u$, we face the presence of Schoen--Wolfson conical singularities,
which naturally appear even for minimizers. Since $JH$ is parallel to the cross-section for such cones \cite[Section 7]{SW2},
assuming that these singularities $x_k$ are isolated we see that, across them,
$$dh=\sum_k c_k\delta_{x_k}$$
is a sum of Dirac masses. Further, the number and location of such singularities is uncontrolled as well,
and these could even be not isolated a priori, rendering the PDE practically useless.
\item[(iii)] Finally, at the varifold level (i.e., for a HSLV),
sequences of integer rectifiable HSLVs can converge to a non-rectifiable limit,
differently from the isotropic case where Allard's compactness holds.
This reflects the highly anisotropic nature of the Carnot--Carath\'eodory metric, or the Kor\'anyi metric on $\mathbb{H}^2$,
for which curves such as Reeb integral curves have Hausdorff dimension equal to $2$ instead of $1$.
\end{itemize}

Let us now discuss very briefly how each difficulty in the previous list is overcome in our work.
The first two appear in particular while showing the crucial fact that a parametrized blow-up arises as a \emph{strong} $W^{1,2}$ limit
(thus, we manage to show property (P) at least for blow-ups).

\begin{itemize}[leftmargin=8mm]
\item[(i)] Sequential compactness of PHSLVs $(\Sigma,u,N)$, which does not use the PDE, allows to derive local doubling bounds for the Dirichlet energy measure
$|\nabla u|^2\,dx^2$, just in terms of $\Sigma$ and the total mass.
Assuming smoothness (which can be done inductively on lower multiplicity regions in the regularity proof),
we can use these doubling bounds together with a Liouville-type argument to obtain local bounds on $h=d\beta$.
\item[(ii)] After showing that such Schoen--Wolfson singularities are isolated (in suitable pinched-density sets)
by a standard dimension reduction technique, we face the issue that they could become denser and denser along a sequence
of rescalings $u_k$ giving a blow-up. However, failure of a strong $W^{1,2}$ convergence
is detected by a jump in the multiplicity in the limit, which is ruled out by a careful continuity argument
(at small scales, Schoen--Wolfson singularities are well separated, and thus here we are close to a picture with no multiplicity jump).
\item[(iii)] The last issue is circumvented
by assuming a stronger definition of stationarity compared to the initial PHSLV definition.
In turn, this yields a point removability result for limits of such varifolds,
which rules out this phenomenon in applications. Moreover, in the context of bubbling, we show that
no energy dissipates in neck regions just assuming the PHSLV definition (while for collars, appearing when the conformal structure degenerates, this fails: see Theorem \ref{th-count-ex}).
\end{itemize}

Note that the third point is related to the second derivative of $F$ appearing in $W_F$,
due to which points are not always removable singularities for a two-dimensional HSLV
(while they are for two-dimensional stationary varifolds in the isotropic setting).

\begin{Rm}
The third point also leads to the speculation that failure of compactness of integer rectifiable HSLVs
is solely due to the possible appearance of stationary \emph{Reeb integral curves} in the limit $\mathbf{v}$, such as Hopf fibers in $S^5$;
indeed, the proof of Allard's rectifiability theorem in this setting
shows that, for $|\mathbf{v}|$-a.e. $p$, the blow-up at $p$ is either a Legendrian plane or
a varifold supported on the Reeb axis of $\mathbb{H}^2$. \hfill $\Box$
\end{Rm}

\subsection{Comparison with existing regularity results}

Let us now briefly compare our optimal regularity result (see Theorem \ref{reg.thm.chart} for a precise statement)
with other ones which already appeared in the literature.
The comparison with the work of Schoen--Wolfson \cite{SW2} is quite straightforward,
in that both \cite{SW2} and the present paper deal with arbitrary $W^{1,2}$-parametrized surfaces
with the only a priori bound of having finite area,
and while \cite{SW2} considers minimizers we are able to deal with general critical points equipped with $L^\infty$ integer multiplicities.

Roughly speaking, one of the core difficulties in the regularity theory (and also in the variational construction)
of Hamiltonian stationary Legendrian parametrizations
is that the Dirichlet energy involves the same order of differentiation as the Legendrian constraint.
One of the first works dealing with a similar situation is the one by Evans--Gariepy \cite{EG.reg},
who studied area-preserving maps on the plane. Here the authors manage to obtain a partial regularity result (by transforming the situation to a scalar problem by a clever change of variables), although at the expense of considering minimizers
and assuming artificially that the map is Lipschitz.

In other works, such as \cite{BCW} by Bhattacharya--Chen--Warren,
the full regularity is obtained for Hamiltonian stationary Lagrangian submanifolds,
but with the a priori assumption that they are $C^1$.
While this leads to a full regularity, as showed also by Schoen--Wolfson in \cite[Theorem 4.1]{SW2} by linearizing the PDE
to the biharmonic equation, this assumption automatically rules out Schoen--Wolfson singularities (which could appear among minimizers), and thus is again quite artificial in a geometric variational setting.

In the work \cite{BS} by Bhattacharya--Skorobogatova, Hamiltonian stationary Lagrangian graphs are considered,
with the a priori assumption that they are Lipschitz.
In this interesting work, viewing these as graphs of gradient maps (thus generated by a $W^{2,\infty}$ function),
the authors study the resulting fourth order nonlinear scalar equation, reaching a conditional regularity result.
However, again the graphicality assumption rules out Schoen--Wolfson singularities (which are \emph{never} graphs).

Given that our work assumes only finite area (the weakest possible assumption required by the study of variations of the area, i.e., having a $W^{1,2}$ weakly conformal parametrization),
it makes a significant leap in the regularity theory, at least for two-dimensional objects.
We expect that our techniques will shed new light in similar problems, where so far the understanding is restricted
to Lipschitz graphs.

Finally, let us mention that, in higher dimension, using the intrinsic approach of currents (thus avoiding parametrizations), 
Orriols recently developed an existence and partial regularity theory of area minimizers under the Legendrian constraint \cite{Orr}. From a technical perspective the use of the monotonicity formula, which is proven to fail in higher dimension for general Hamiltonian stationary Legendrian varifolds, is replaced by the combination of comparison arguments with an isoperimetric inequality for Legendrian currents. 

\subsection{Strategy of proof and organization of the paper}
The paper is structured as follows. In Section II, after some preliminaries on the Heisenberg group and the Legendrian constraint,
we give the precise definition of PHSLV (see Definition \ref{df-PHSLV}).

In Section III we define and study general Hamiltonian stationary Legendrian varifolds (HSLV).
In particular, we prove a monotonicity formula (see Theorem \ref{mono.cor} and Corollary \ref{mono.cor2}), by carefully refining and generalizing the one originally introduced for  immersions in \cite{Riv2} to a much weaker framework, exploiting the Hamiltonian $\arctan\sigma$ (which happens to be smooth and $0$-homogeneous on $\mathbb{H}^2\setminus\{0\}$) suitably cut-off with the Folland--Kor\'anyi gauge $\r$, and we derive a number of standard consequences,
such as upper semi-continuity of the density, in space and under varifold limits.

In Section IV we prove the best possible analogue of Allard's compactness of integral stationary varifolds \cite{All},
namely we show their closure among rectifiable ones (see Theorem \ref{cpt.int}). Although the scheme of proof is standard,
one particular step turns out to be very subtle in this Legendrian setting: namely, to show the fact that a HSLV with zero tilt-excess is
a union of parallel planes (Lemma \ref{zero.exc}), we have to perform an iterated blow-up, obtaining more and more algebraic constraints until we are able to close the loop.

Section V is dedicated to a point removability result for PHSLVs (Proposition \ref{pr-VI.1}), itself deduced from an analogous result for general HSLVs (Proposition \ref{remov.gen}),
assuming in both cases a slightly stronger notion of stationarity; note that, since the second derivatives of the Hamiltonian
$F$ appear in the associated vector field $W_F$, this does not simply follow by a capacity argument.

The goal of Section VI is to prove a number of structural properties of PHSLVs, such as a universal lower bound for the density (Proposition \ref{co-dens}), a quantitative continuity of the underlying map (Proposition \ref{pr-cont}), the rectifiability of the support (Proposition \ref{pr-rect-d-K}), and the upper semi-continuity of a better representative $\tilde N$ of the multiplicity function $N$ (Proposition \ref{pr-robust}).

The proof of Theorem \ref{th-0} is the main purpose of Section VII, where a more complete formulation of the result is also given (Theorem \ref{th-sequential-weak-closure}); its proof is based on an important energy quantization result
(Lemma \ref{lm-energy-quant}). Since this is the only part of the paper
where there are significant simplifications in $\mathbb{H}^2$ compared to a closed Sasakian ambient $M^5$
(due to the symmetry by dilations, reflected in the absence of bubbling in $\mathbb{H}^2$ and in the fact that here most statements are effective at all scales),
we will often comment on what are the relevant changes in a general closed $M^5$.
We also discuss how to deal with bubbling and degenerating conformal class (see Remark \ref{bubbling}, Remark \ref{neck}, Lemma \ref{bubbling.bis}, and Remark \ref{neck.bis}).

In Section VIII we start developing the regularity theory (see Theorem \ref{reg.thm.chart} for the precise statement)
and we explain the induction process governing the proof. We also prove a rigidity result
for blow-ups (Proposition \ref{rigidity}); in the classification of tangent cones, it allows to assume that the multiplicity $\tilde N$ has a strict maximum
at the origin, thus triggering the inductive assumption on the complement.
Inspired by \cite{PiRi1}, we also define \emph{admissible} and \emph{strongly admissible} points (see Definition \ref{adm.def}),
and we complete the base case of the induction, by showing that all points are admissible in this case
and appealing to a small tilt-excess regularity theorem of Schoen--Wolfson (see Proposition \ref{ep.reg}).

In Section IX we start attacking the inductive step, classifying tangent cones at admissible points by exploiting the inductive assumption
(see Proposition \ref{par.cone} and Corollary \ref{par.cone.imm}).
Moreover, we exploit this understanding of blow-ups, and in particular the fact that there the Dirichlet energy is a doubling measure, to deduce again that in fact all points are admissible (see Proposition \ref{all.admissible}).
We also show that Schoen--Wolfson singularities cannot accumulate among points of similar multiplicity $\tilde N$ (see Proposition \ref{sw.countable}).

We finish the inductive step of the proof of regularity in Section X, by looking at a point $\bar x$ of high multiplicity,
at the boundary of a region consisting of lower multiplicity points (an idea borrowed from \cite{PiRi1}).
Roughly speaking, we can reach the conclusion that such high multiplicity points are isolated
if we can prove that any blow-up at $\bar x$ satisfies property (P), i.e., if we can upgrade the a priori weak $W^{1,2}$ convergence
of the rescalings of $u$ to a strong one (see Proposition \ref{blow.up.strong}). This is carried out first
assuming that there are no Schoen--Wolfson conical singularities, and then including this possibility,
by two different arguments, as explained before in the introduction.

Finally, in the appendix, we give an explicit example (Theorem \ref{th-count-ex}) showing that integer rectifiable HSLVs, and even PHSLVs,
can converge to a non-rectifiable limit in a closed Sasakian manifold such as $S^5$,
a phenomenon ruled out in applications by requiring a stronger notion of stationarity (see Definition \ref{strong.phslv}).

\subsection*{Acknowledgements}
The authors are grateful to Filippo Gaia and Gerard Orriols for many useful conversations.
They also wish to thank Mario Micallef and Richard Schoen for their interest in this work.

\section{Preliminaries}
\subsection{Geometry of the Heisenberg group}
We give here some fundamental notions from the Heisenberg group geometry that we will use in this work. A thorough and way more complete presentation can be found in \cite{CDPT}. 
We denote by ${\mathbb H}^2$ the Heisenberg group over ${\C}^2$. The coordinates in ${\mathbb H}^2$ will be denoted $(z_1,\dots, z_{4},\varphi)$, where the last coordinate $\varphi$ is called the {\it Legendrian coordinate}.
The canonical projection from  ${\mathbb H}^2$ onto ${\C}^2$ which consists in ``forgetting'' the Legendrian coordinate $\varphi$ will be denoted $\pi$.
\[
\pi(z_1,\dots, z_{4},\varphi)=(z_1,\dots, z_{4}).
\]
The so-called {\it horizontal hyperplanes} $H$ are spanned at every point by the following four vectors:
\[
X_j:=\frac{\p}{\p z_{2j-1}}-z_{2j}\frac{\p}{\p\varphi},\quad Y_j:=\frac{\p}{\p z_{2j}}+z_{2j-1}\frac{\p}{\p\varphi},\quad\text{for }j=1,2.
\]
We define a Riemannian metric on ${\mathbb H}^2$ by requiring that $(X_1,Y_1,X_2,Y_2,\p_\varphi)$ realizes an orthonormal basis at every point.
Thus, the tangent map $\pi_\ast:T{\mathbb H}^{2}\to T{\mathbb C}^{2}$ to the canonical projection $\pi:{\mathbb H}^{2}\to{\mathbb C}^{2}$ given by
\[
\pi_\ast X_j=\frac{\p}{\p z_{2j-1}},\quad \pi_\ast Y_j=\frac{\p}{\p z_{2j}},\quad\pi_\ast\frac{\p}{\p\varphi}=0
\]
realizes at every point an isometry between $H$ and $T{\mathbb C}^{2}$. In particular we observe that
\begin{equation}
\label{meth2}
g_{{\mathbb H}^2}=\pi^\ast g_{{\C}^2}+\al\otimes \al, \quad \al:=-d\varphi +\sum_{j=1}^2 (z_{2j-1}\,dz_{2j}-z_{2j}\,dz_{2j-1}).
\end{equation}
Observe at this stage that the metric on ${\mathbb H}^2$ is equivalent to the Euclidean metric of ${\R}^5$ on any compact  set.
Also, for an $\mathbb{H}^2$-valued map, requiring it to be in $L^\infty_{loc}$ for this metric on ${\mathbb H}^2$  is equivalent to the same requirement for the Euclidean one, a fact that will be tacitly used later on.

On $H$ we define the following complex structure:
\[
J_HX_j:=Y_j.
\]
The Riemannian manifold $({\mathbb H}^2, g_{{\mathbb H}^2})$ becomes a Lie group with the operation
\[
(z,\varphi)\ast(z',\varphi'):=\lf(z+z', \varphi+\varphi'+\sum_{j=1}^2 (z_{2j-1}z'_{2j}-z_{2j}z'_{2j-1})\rg),
\]
where the neutral element is $(0,0)$ and the inverse to any element $(z,\varphi)$ is obviously given by
\[
(z,\varphi)^{-1}=(-z,-\varphi).
\]

\begin{Rm}
The vector fields $X_j,Y_j$ and the metric $g_{\mathbb{H}^2}$, as well as $\alpha$, the hyperplane distribution, and $J_H$, are preserved by left multiplication, i.e.,
by the diffeomorphism $\ell_p(x):=p\ast x$, for any given $p\in\mathbb{H}^2$. \hfill $\Box$
\end{Rm}

We denote (as in \cite{Riv2})
\[
\rho^2:=\sum_{j=1}^4z_j^2, \quad\sigma:=\frac{2\varphi}{\rho^2},\quad\r^4:=\rho^4+4\varphi^2.
\]
The function $\sigma$ will be called the {\it phase}, while $\r$ is the {\it Folland--Kor\'anyi gauge}.

For $t\in\R$,
the \emph{dilation} map $\delta_t:\mathbb{H}^2\to\mathbb{H}^2$ given by
\begin{equation}
\label{deltat}
\delta_t(z,\varphi):=(tz,t^2\varphi)
\end{equation}
is obviously a group homomorphism. Moreover, given $A\in U(2)$, we introduce the \emph{rotation} $R_A:\mathbb{H}^2\to\mathbb{H}^2$ given by
\begin{equation}
\label{rota}
R_A(z,\varphi):=(Az,\varphi),
\end{equation}
which is again a homomorphism since
\begin{equation}
\label{ra.id}
\sum_{j=1}^2 ((Uz)_{2j-1}(Uz')_{2j}-(Uz)_{2j}(Uz')_{2j-1})=\ang{iUz,Uz'}=\ang{U(iz),Uz'}=\ang{z,z'}.
\end{equation}

We introduce the map on ${\mathbb H}^2\times {\mathbb H}^2$ given by
\begin{equation}
\label{kord}
d_{K}((z,\varphi),(z',\varphi')):=\r((z,\varphi)^{-1}\ast(z',\varphi'))
\end{equation}
and, viewing $z,z'\in\C^2$,
we compute for any choice of pair of points $p:=(v,\phi)$ and $q:=(w,\psi)$ that
\begin{align*}
\r^4(p\ast q)&=|v+w|^4+4|\phi+\psi+ v_1 w_2-v_2w_1+v_3w_4-v_4w_3|^2\\
&=||v+w|^2+2i \lf(\phi+\psi+\langle iv,w\rangle\rg)|^2\\
&=||v|^2+2i \phi+|w|^2+2i \psi+2\langle v,w\rangle+2i\langle iv,w\rangle|^2\\
&\le[||v|^2+2i \phi|+||w|^2+2i\psi|+2|\langle v,w\rangle+i\langle iv,w\rangle|]^2.
\end{align*}
Observe that the first two terms inside the square are $\r^2(p)+\r^2(q)$, while
\[
|\langle v,w\rangle+i\langle iv,w\rangle|^2=\langle v,w\rangle^2+\langle iv,w\rangle^2\le |v|^2|w|^2,
\]
as $v\perp iv$.
Since $|v|\le\r(p)$ and $|w|\le\r(q)$, we then have
$$\r^4(p\ast q)\le |\r^2(p)+\r^2(q)+ 2\r(p)\r(q)|^2\le|\r(p)+\r(q)|^4.$$
This inequality, together with the definition of $d_K$, imply immediately the following classical lemma (see for instance \cite{CDPT}).
\begin{Lm}
\label{lm-kord}
The map $d_{K}:{\mathbb H}^2\times {\mathbb H}^2\to[0,\infty)$ defines a distance,  called \emph{Kor\'anyi distance}.\hfill $\Box$
\end{Lm}

\begin{Rm}\label{dk.right}
Clearly, $d_K$ is left-invariant, in the sense that $$d_K(a*p,a*q)=d_K(p,q).$$ Moreover,
a straightforward computation gives $a^{-1}*p*a=p*(0,0,0,0,2\sum_{j=1}^2 (p_{2j-1}a_{2j}-p_{2j}a_{2j-1}))$;
plugging $p^{-1}*q$ in place of $p$, we obtain the bound
$$d_K(p*a,q*a)=\r(a^{-1}*(p^{-1}*q)*a)\le\r(p^{-1}*q)+2\sqrt{\rho(p^{-1}*q)\rho(a)},$$
and hence
$$d_K(p*a,q*a)\le d_K(p,q)+2\sqrt{\rho(a)}\sqrt{d_K(p,q)}$$
for all $p,q,a\in\mathbb{H}^2$. \hfill $\Box$
\end{Rm}

We will denote by ${\mathcal H}^s_K$ the $s$-dimensional Hausdorff measure constructed out of this distance.

\begin{Rm}\label{isometry}
The maps $\delta_t$ and $R_A$ satisfy $\delta_t^*\alpha=t^2\alpha$ and $R_A^*\alpha=\alpha$, thanks to \eqref{ra.id}.
In particular, they are isomorphisms preserving $H$ (for $t\neq0$). They also preserve $J_H$, since this holds at the origin
and $J_H$ is left-invariant (note that, for an isomorphism $\psi$, we have $\psi\circ\ell_p=\ell_{\psi(p)}\circ\psi$).
Moreover, we have
$$d_K(\delta_t(p),\delta_t(q))=|t| \cdot d_K(p,q),$$
while $R_A$ is an isometry. \hfill $\Box$
\end{Rm}

\subsection{Hamiltonian deformations}
Observe that
\[
d\al=2dz_1\wedge dz_2+2dz_3\wedge dz_4=2\pi^\ast\om,
\]
where $\om=dz_1\wedge dz_2+dz_3\wedge dz_4$ is the standard symplectic form on ${\C}^2$. We now consider vector fields $W$ on ${\mathbb H}^2$ such that the associated flow $\Psi_t$ preserves the kernel of $\al$.
This is equivalent to the existence of a function $f_t$ on ${\mathbb H}^2\times \R$ such that
$$\Psi_t^\ast\al= f_t\al.$$
Taking the derivative with respect to $t$ at $t=0$ and using Cartan's formula, we obtain
$$\frac{\p f_t}{\p t}\Big|_{t=0}\al(Z)={\mathcal L}_W\al(Z)=d(\al(W))(Z)+d\al(W,Z).$$
We denote by $W=W^H+W^V$ the orthogonal decomposition of $W$ along $H$ and $\p_\varphi$. Let $F(z,\varphi):=\al(W)=\al(W^V)$, so that $W^V=-F\p_\varphi$. Plugging $Z^H$ in place of $Z$ gives
\begin{align*}
0&=dF(Z^H)+2\pi^\ast\om(W,Z^H)\\
&=\langle\nabla^H F, Z^H\rangle+2\om(\pi_\ast W^H,\pi_\ast Z^H)\\
&=\langle\nabla^H F, Z^H\rangle+2\langle i\pi_\ast W^H,\pi_\ast Z^H\rangle\\
&=\langle\nabla^H F+2J_HW^H, Z^H\rangle,
\end{align*}
where $\nabla^HF:=(\nabla F)^H$ is the orthogonal projection of $\nabla F$ onto $H$. Since this holds for any choice of $Z$,
we obtain $2W^H=J_H\nabla^HF$.
Hence we conclude that
\begin{equation}
\label{I-3}
2W=J_H\nabla^H F-2F\p_{\varphi}.
\end{equation}
Starting from an arbitrary function $F$, we can also reverse the argument and conclude that $W=:W_F$ given by \eqref{I-3} generates a flow preserving the kernel of $\al$.
Since
\[
\nabla^HF=\sum_{j=1}^2[\langle dF,X_j\rangle X_j+\langle dF,Y_j\rangle Y_j]=\sum_{j=1}^2[(\p_{z_{2j-1}}F-z_{2j}\p_\varphi F)X_j+(\p_{z_{2j}}F+z_{2j-1}\p_\varphi F)Y_j],
\]
we have the expansion
\begin{align}
\label{I-4}
\begin{aligned}
2W_F&=\sum_{j=1}^2[(\p_{z_{2j-1}}F-z_{2j}\p_\varphi F)Y_j-(\p_{z_{2j}}F+z_{2j-1}\p_\varphi F)X_j]-2F\p_{\varphi}\\
&=\sum_{j=1}^2[(\p_{z_{2j-1}}F)\p_{z_{2j}}-(\p_{z_{2j}}F)\p_{z_{2j-1}}]-\p_\varphi F\sum_{k=1}^4z_k\p_{z_k}+\lf[\sum_{k=1}^4z_k\p_{z_k}F - 2F\rg]\p_{\varphi}.
\end{aligned}
\end{align}

Before defining the main object studied in this work, we need a few more basic definitions.
A (locally bounded) smooth or Sobolev map $u$ of a surface $\Sigma$ into ${\mathbb H}^2$ is called {\it Legendrian}  if it is tangent to $H$ at every point. This is equivalent to the {\it contact condition} $u^\ast\al=0$.
Composing a Legendrian map $u$ with $\pi$ gives a \emph{Lagrangian} map $v:=\pi\circ u$ into ${\C}^{2}$, namely a map satisfying $v^\ast\omega=0$.

Let $u=(u_1,u_2,u_3,u_4,u_\varphi)\in W^{1,2}(\Sigma,{\mathbb H}^2)$. Assume $u$ is weakly conformal: namely, in any local conformal chart $(x_1,x_2)$ for $\Sigma$,
a.e. we have
\begin{equation}
\label{V.1}
\begin{cases}
|\p_{x_1}u|_{{\mathbb H}^2}^2=|\p_{x_2}u|^2_{{\mathbb H}^2}\\
\p_{x_1}u\cdot_{{\mathbb H}^2}\p_{x_2}u=0.
\end{cases}
\end{equation}
We also assume that $u$ is Legendrian.
Observe that, since the canonical projection $\pi_\ast$ realizes an isometry between $H$ and ${\C}^2$, we have 
\[
|\nabla u|^2_{{\mathbb H}^2}=|\nabla v|^2_{{\C}^2},
\]
where $v:=\pi\circ u$. We have also
\begin{equation}
\label{V.19-as}
|\nabla u|_{{\R}^5}=|\nabla  v|^2+\lf|\sum_{j=1}^2 [u_{2j-1}\nabla u_{2j}-u_{2j}\nabla u_{2j-1}]  \rg|^2\le [1+|v|^2]|\nabla v|_{{\C}^2}^2.
\end{equation}
Hence, if a Lagrangian map is assumed to be in $L^\infty\cap W^{1,2}(\Sigma,{\mathbb H}^2)$, it is automatically in $L^\infty\cap W^{1,2}(\Sigma,{\R}^5)$ and there holds
\begin{equation}
\label{V.20-as}
\int_\Sigma |\nabla u|_{{\mathbb H}^2}\,dx^2\le\int_\Sigma |\nabla u|_{{\R}^5}\,dx^2\le  [1+\|v\|_\infty^2]\int_\Sigma |\nabla u|_{{\mathbb H}^2}\, dx^2.
\end{equation}

We now introduce the main variational object studied here, which is a parametrized version of a constrained critical point for the area
(the constraint being the Legendrian condition).
\begin{Dfi}
\label{df-PHSLV}
Let $\Sigma$ be a  Riemann surface and let $u\in L^\infty_{loc}\cap W_{loc}^{1,2}(\Sigma,{\mathbb H}^2)$, as well as $N\in L^\infty(\Sigma,\N^*)$, where $\N^*:=\N\setminus\{0\}$. Assume that $u$ is Legendrian and weakly conformal.
The triple $$(\Sigma,u,N)$$ is a \emph{parametrized Hamiltonian stationary Legendrian varifold (PHSLV)} if given $f\in C^\infty_c(\Sigma,{\R}_+)$, for almost every $t>0$ and for any function $F\in C^\infty_c(\mathbb{H}^2\setminus u(f^{-1}(t)))$, it holds for the associated vector field $W_F$ given by \eqref{I-4} that
\begin{equation}
\label{V.2rep}
0=\int_{f>t} N\nabla(W_F\circ u)\cdot_{{\mathbb H}^2}\nabla u\,dx^2,
\end{equation}
where we use local conformal coordinates on $\Sigma$. \hfill $\Box$
\end{Dfi}

\begin{Rm}
In this definition, we implicitly restrict to those $t>0$ such that the level set $\{f=t\}$ is a disjoint union of embedded smooth loops and the restriction $u|_{\{f=t\}}$ has a continuous representative (which holds for a.e. $t>0$), so that $u(f^{-1}(t))\subset\mathbb{H}^2$ is a compact set.
In the sequel, when we say that a property holds \emph{for a.e. domain $\omega\subset\subset\Sigma$}
we mean that it holds for $\omega:=\{f>t\}$, for any choice of $f\in C^\infty_c(\Sigma,\R_+)$ and a.e. $t>0$. \hfill $\Box$
\end{Rm}

\begin{Rm}
More generally, we say that a triple $(\Sigma,u,N)$ is a \emph{PHSLV on an open set $U\subseteq\mathbb{H}^2$}
if the previous requirement holds for all $F\in C^\infty_c(U\setminus u(\p\omega))$, for a.e. $\omega\subset\subset\Sigma$. \hfill $\Box$
\end{Rm}

\begin{Rm}\label{h.minimal}
When $N=1$ and $u$ is a Legendrian lift of a smooth Lagrangian immersion $v:\Sigma\to\C^2$ (namely, we have $\pi\circ u=v$),
then this notion is equivalent to the H-minimality introduced by Oh \cite{Oh2},
since $\pi_*(2W_F)=\sum_{j=1}^2[(\p_{z_{2j-1}}F)\p_{2j}-(\p_{z_{2j}}F)\p_{2j-1}]$ whenever $F$ depends only on $z$.
\hfill $\Box$
\end{Rm}
Since $\pi_\ast$ is an isometry, writing $X_\ell$ and $Y_\ell$ in place of $X_\ell\circ u$ and $Y_\ell\circ u$ for simplicity, there holds
\begin{equation}
\label{V.3rep}
0=\int_{f>t}N\sum_{\ell=1}^2[\lf(\nabla(W_F\circ u)\cdot_{{\mathbb H}^2}X_\ell\rg)\cdot\nabla u_{2\ell-1}+\lf(\nabla(W_F\circ u)\cdot_{{\mathbb H}^2}Y_\ell\rg)\cdot\nabla u_{2\ell}]\, dx^2.
\end{equation}
Recalling \eqref{I-3}, we have for $\ell=1,2$ that
\begin{align}
\label{V.4rep}
\begin{aligned}
&2\nabla(W_F\circ u)\cdot_{{\mathbb H}^2}X_\ell\\
&=\nabla(J_H\nabla^HF\circ u)\cdot_{{\mathbb H}^2}X_\ell\\
&=\sum_{j=1}^2\nabla[\lf(\p_{z_{2j-1}}F\circ u-u_{2j}\p_\varphi F\circ u\rg)Y_j-\lf(\p_{z_{2j}}F\circ u+u_{2j-1}\p_\varphi F\circ u\rg)X_j]\cdot_{{\mathbb H}^2}X_\ell\\
&=-\nabla\lf(\p_{z_{2\ell}}F\circ u+u_{2\ell-1}\p_\varphi F\circ u\rg),
\end{aligned}
\end{align}
where we used the fact that the differentials $\nabla X_j$ and $\nabla Y_j$ have image in the span of $\p_\varphi$, and thus orthogonal to $H$.
Similarly, we have
$$2\nabla(W_F\circ u)\cdot_{{\mathbb H}^2}Y_\ell=\nabla\lf(\p_{z_{2\ell-1}}F\circ u-u_{2\ell}\p_\varphi F\circ u\rg).$$
Hence, combining these identities, the stationarity condition becomes
\begin{align}
\label{V.2}
\begin{aligned}
0&=\int_{f>t} N\sum_{j=1}^2\lf[\nabla u_{2j}\cdot\nabla\lf[\frac{\p F}{\p {z_{2j-1}}}\circ u\rg]-\nabla u_{2j-1}\cdot\nabla\lf[\frac{\p F}{\p {z_{2j}}}\circ u\rg]\rg]\,dx^2\\
&\quad-\int_{f>t} N\sum_{k=1}^4\nabla u_k\cdot\nabla\lf[ u_k\frac{\p F}{\p\varphi}\circ u\rg]\, dx^2.
\end{aligned}
\end{align}
\section{General monotonicity formula}
In this section we consider a more general class of varifolds, defined as follows.

\begin{Dfi}
Let $\Pi:G\to\mathbb{H}^2$ denote the Grassmannian bundle of Legendrian two-dimensional planes in $\mathbb{H}^2$;
we denote elements of $G$ as $\mathcal{P}$, or as pairs $(\mathcal{P},p)$ when we want to emphasize the underlying $p\in\mathbb{H}^2$.
Given an open set $U\subseteq\mathbb{H}^2$,
a \emph{Hamiltonian stationary Legendrian varifold (HSLV)} $\mathbf{v}$ on $U$ is a Radon measure on $\Pi^{-1}(U)$ such that
$$\int_{\Pi^{-1}(U)} \operatorname{div}_{\mathcal P}W_F\,d\mathbf{v}(\mathcal{P},p)=0$$
for all Hamiltonian vector fields $W_F$ as above. \hfill $\Box$
\end{Dfi}

With a little abuse of notation, we will often write a domain of integration in $\mathbb{H}^2$ to mean its preimage under $\Pi$.

\begin{Rm}\label{hslv.from.phslv}
Given a PHSLV $(\Sigma,u,N)$ and a.e. domain $\omega=\{f>t\}\subset\subset\Sigma$, we have an induced varifold
$\mathbf{v}_\omega$ given by
$$\mathbf{v}_\omega(B):=\int_{\omega\cap u^*B}N\frac{|\nabla u|^2}{2}\,dx^2,\quad\text{for }B\subseteq G,$$
where $u^*B=\{x\,:\,(\operatorname{img}\nabla u(x),u(x))\in B\}$. This varifold $\mathbf{v}_\omega$
restricts to a HSLV on $\mathbb{H}^2\setminus u(f^{-1}(t))$ (open for a.e. $t>0$). \hfill $\Box$
\end{Rm}

\subsection{A pointwise identity}
Let ${\mathcal P}$ be a Legendrian two-dimensional plane in ${\mathbb H}^2$. Let $(Z_1,Z_2)$ be an orthonormal basis of ${\mathcal P}$.
Note that $[X_j,X_j]=[X_j,\p_\varphi]=0$, and similarly for $Y_j$; by the Koszul formula, we then have
$$\nabla_{X_j}\p_\varphi\cdot_{\mathbb H^2}X_j=\nabla_{Y_j}\p_\varphi\cdot_{\mathbb H^2}Y_j=0.$$
Thus, we can discard the vertical part $W^V$ in the computation of $\operatorname{div}_{\mathcal P}W$.
We then have by definition
$$\operatorname{div}_{\mathcal P}W= \sum_{j=1}^2\nabla_{Z_j}W^H\cdot_{{\mathbb H}^2}Z_j= \sum_{j,\ell=1}^2\nabla_{Z_j}[(W\cdot_{{\mathbb H}^2}X_\ell)X_\ell]\cdot_{{\mathbb H}^2}Z_j+\sum_{j,\ell=1}^2\nabla_{Z_j}[(W\cdot_{{\mathbb H}^2}Y_\ell)Y_\ell]\cdot_{{\mathbb H}^2}Z_j.$$
Observe that, for any choice of $\{A,B,C\}\subset\{X_1,Y_1,X_2,Y_2\}$, the commutators $[A,B]$, $[B,C]$ and $[C,A]$
are orthogonal to $H$; hence, the Koszul formula
implies that $\nabla_A B\cdot_{{\mathbb H}^2}C=0$. This gives
\begin{align}
\label{V.8-6}
\begin{aligned}
\operatorname{div}_{\mathcal P}W
&= \sum_{j,\ell=1}^2 Z_j(W\cdot_{{\mathbb H}^2}X_\ell)(X_\ell\cdot_{{\mathbb H}^2}Z_j)+\sum_{j,\ell=1}^2 Z_j(W\cdot_{{\mathbb H}^2}Y_\ell)(Y_\ell\cdot_{{\mathbb H}^2}Z_j)\\
&= \sum_{\ell=1}^2 \nabla^{\mathcal P}(W\cdot_{{\mathbb H}^2}X_\ell)\cdot_{{\mathbb H}^2}X_\ell+\sum_{\ell=1}^2 \nabla^{\mathcal P}(W\cdot_{{\mathbb H}^2}Y_\ell)\cdot_{{\mathbb H}^2}Y_\ell.
\end{aligned}
\end{align}
For $2W_F=J_H\nabla^H F-2F\p_\varphi$ as above, since $\nabla^Hz_{2\ell-1}=X_\ell$ and $\nabla^Hz_{2\ell}=Y_\ell$, there holds
\begin{align}
\label{V.8-9}
\begin{aligned}
&2\operatorname{div}_{\mathcal P}W_F\\
&=-\sum_{\ell=1}^2\nabla^{\mathcal P}(\p_{z_{2\ell}}F+z_{2\ell-1}\p_{\varphi}F)\cdot \nabla^{\mathcal P}z_{2\ell-1}+\sum_{\ell=1}^2\nabla^{\mathcal P}(\p_{z_{2\ell-1}}F-z_{2\ell}\p_{\varphi}F)\cdot \nabla^{\mathcal P}z_{2\ell}\\
&=-\sum_{\ell=1}^2\nabla^{\mathcal P}(\p_{z_{2\ell}}F)\cdot \nabla^{\mathcal P}z_{2\ell-1}+\sum_{\ell=1}^2\nabla^{\mathcal P}(\p_{z_{2\ell-1}}F)\cdot \nabla^{\mathcal P}z_{2\ell}\\
&\quad-|\nabla^{\mathcal P}z|^2\p_\varphi F-2^{-1}\nabla^{\mathcal P}(\p_\varphi F)\cdot\nabla^{\mathcal P}\rho^2.
\end{aligned}
\end{align}
Assuming now that $F$ is a function of $(\rho^2,\varphi)$, there holds
$$dF=\frac{\p F}{\p\rho^2}\, d\rho^2+\frac{\p F}{\p\varphi}\, d\varphi=2z_k\frac{\p F}{\p\rho^2}\, dz_k+\frac{\p F}{\p\varphi}\, d\varphi,$$
and hence $\frac{\p F}{\p z_k}=2z_k\frac{\p F}{\p\rho^2}$.
Inserting these identities in \eqref{V.8-9} and noting that
$$\nabla^H\varphi=\sum_{\ell=1}^2(z_{2\ell-1}Y_\ell-z_{2\ell}X_\ell)
=\sum_{\ell=1}^2(z_{2\ell-1}\nabla^Hz_{2\ell}-z_{2\ell}\nabla^Hz_{2\ell-1}),$$
we obtain
\begin{align}
\label{V.8-12}
\begin{aligned}
2\operatorname{div}_{\mathcal P}W_F&=-2\sum_{\ell=1}^2\nabla^{\mathcal P}( z_{2\ell}\p_{\rho^2}F)\cdot \nabla^{\mathcal P}z_{2\ell-1}+2\sum_{\ell=1}^2\nabla^{\mathcal P}(z_{2\ell-1}\p_{\rho^2}F)\cdot \nabla^{\mathcal P}z_{2\ell}\\
&\quad-|\nabla^{\mathcal P}z|^2\p_\varphi F-2^{-1}\nabla^{\mathcal P}(\p_\varphi F)\cdot\nabla^{\mathcal P}\rho^2\\
&=2\nabla^{\mathcal P}\lf[\frac{\p F}{\p\rho^2}   \rg]\cdot\nabla^{\mathcal P}\varphi-|\nabla^{\mathcal P}z|^2\frac{\p F}{\p \varphi}-2^{-1}\nabla^{\mathcal P}\lf[\frac{\p F}{\p \varphi}\rg]\cdot\nabla^{\mathcal P}\rho^2.
\end{aligned}
\end{align}
It can be checked that $|\nabla^{\mathcal P}z|^2=2$, but for now we keep this factor so as to keep more homogeneity in the computations below.
Let ${\mathbf v}$ be a HSLV, so that for any smooth compactly supported function $F$ depending on $(\rho^2,\varphi)$ we have by definition
$$2^{-1}\int_{G}  \nabla^{\mathcal P}\rho^2\cdot\nabla^{\mathcal P}\lf[\frac{\p F}{\p\varphi}\rg]\, d{\mathbf v}({\mathcal P},p)+\int_{G} |\nabla^{\mathcal P} z|^2\frac{\p F}{\p\varphi}\,d{\mathbf v}({\mathcal P},p)
-2\int_{G}  \nabla^{\mathcal P}\varphi\cdot\nabla^{\mathcal P} \lf[\frac{\p F}{\p \rho^2}\rg]\, d{\mathbf v}({\mathcal P},p)=0.$$
We  now consider a smooth cut-off function $\chi$ on ${\R}_+$ such that $\chi'\le0$ and
$$
\chi(t)=\begin{cases}
1&\text{for }t\le1\\
0&\text{for }t\ge2.
\end{cases}
$$

Letting $0<\ep<1$, we consider
\[
F(\rho^2,\varphi):= (\chi(\r)-\chi(\r/\ep))\arctan\sigma,
\]
where $\arctan\sigma$ is extended by continuity to $\mathbb{H}^2\setminus\{0\}$ (so that it equals $\operatorname{sgn}(\varphi)\frac{\pi}{2}$
on the $\varphi$-axis $\{\rho=0\}$).
We compute, away from $\{\rho=0\}$, that
\begin{equation}
\frac{\p F}{\p \varphi}=\frac{\p\r}{\p\varphi}( \chi'(\r)-\ep^{-1}\chi'(\r/\ep) ) \arctan\sigma+( \chi(\r)-\chi(\r/\ep) ) \frac{\p\sigma}{\p\varphi}\frac{1}{1+\sigma^2}.
\end{equation}
Away from $\{\rho=0\}$ we have
$$
\r^3\p_\varphi\r=\p_{\varphi}\varphi^2=2\varphi,\quad \frac{\p\sigma}{\p\varphi}\frac{1}{1+\sigma^2}=\frac{2}{\rho^2}\frac{1}{1+\sigma^2}=\frac{2}{\rho^2+4\rho^{-2}\varphi^2}=\frac{2\rho^2}{\r^4}.
$$
We note in passing that the previous expression is smooth on the whole ${\mathbb H}^2\setminus\{0\}$.
Hence,
\begin{equation}
\label{V.12}
\frac{\p F}{\p \varphi}=2\frac{\varphi}{\r^3}( \chi'(\r)-\ep^{-1}\chi'(\r/\ep) ) \arctan\sigma+2( \chi(\r)-\chi(\r/\ep) )\frac{\rho^2}{\r^4}.
\end{equation}
Similarly, we compute
$$
\r^3\p_{\rho^2}\r=\frac{\rho^2}{2},\quad\frac{\p\sigma}{\p\rho^2}\frac{1}{1+\sigma^2}
=-2\frac{1}{1+\sigma^2}\frac{\varphi}{\rho^4}=-2\frac{\varphi}{\r^4},
$$
and so away from $\{\rho=0\}$ we have
\begin{equation}
\label{V.15-a}
\frac{\p F}{\p \rho^2}=\frac{\rho^2}{2\r^3}( \chi'(\r)-\ep^{-1}\chi'(\r/\ep) ) \arctan\sigma-2( \chi(\r)-\chi(\r/\ep) )\frac{\varphi}{\r^4}.
\end{equation}

\begin{Rm}
This computation shows that $\arctan\sigma$ is of class $C^\infty$ outside of the origin,
and hence $F$ is of class $C^\infty$ on the whole $\mathbb{H}^2$. \hfill $\Box$
\end{Rm}

\begin{Prop}\label{magic.id}
On the open set $\mathbb{H}^2\setminus\{\rho=0\}$ we have the identity
\begin{align*}
&2^{-1}\nabla^{\mathcal P}\rho^2\cdot\nabla^{\mathcal P}\lf[\frac{\p F}{\p\varphi}\rg]
+|\nabla^{\mathcal P} z|^2\frac{\p F}{\p\varphi}
-2 \nabla^{\mathcal P}\varphi\cdot\nabla^{\mathcal P} \lf[\frac{\p F}{\p \rho^2}\rg]\\
&=2\frac{|\nabla^{\mathcal P}\r|^2}{\r} ( \chi'(\r)-\ep^{-1}\chi'(\r/\ep) )
+|\nabla^{\mathcal P} z|^2 \lf[\frac{2\varphi}{\r^3} ( \chi'(\r)-\ep^{-1}\chi'(\r/\ep) ) \arctan\sigma\rg]\\
&\quad-\frac{\r^4}{2} \nabla^{\mathcal P}\arctan\sigma\cdot\nabla^{\mathcal P} [\r^{-3} ( \chi'(\r)-\ep^{-1}\chi'(\r/\ep) ) {\arctan\sigma}  ]\\
&\quad+2 |\nabla^{\mathcal P}\arctan\sigma|^2 ( \chi(\r)-\chi(\r/\ep) ),
\end{align*}
for the function $F$ introduced above. \hfill $\Box$
\end{Prop}

\begin{proof}
Away from $\{\rho=0\}$ we have
\begin{align*}
&2^{-1}\nabla^{\mathcal P}\rho^2\cdot\nabla^{\mathcal P}\lf[\frac{\p F}{\p\varphi}\rg]\\
&=2^{-1} \nabla^{\mathcal P}\rho^2\cdot\nabla^{\mathcal P}\lf[ \frac{\rho^2\,\sigma}{\r^3}( \chi'(\r)-\ep^{-1}\chi'(\r/\ep) ) \arctan\sigma \rg]
+ \nabla^{\mathcal P}\rho^2\cdot\nabla^{\mathcal P}\lf[ ( \chi(\r)-\chi(\r/\ep) )\frac{\rho^2}{\r^4} \rg]\\
&=2^{-1} |\nabla^{\mathcal P}\rho^2|^2\lf[ \frac{\sigma}{\r^3} ( \chi'(\r)-\ep^{-1}\chi'(\r/\ep) )\arctan\sigma \rg]\\
&\quad+2^{-1} \nabla^{\mathcal P}\rho^2\cdot\nabla^{\mathcal P}\sigma \lf[ \frac{\rho^2}{\r^3} ( \chi'(\r)-\ep^{-1}\chi'(\r/\ep) ) \lf(\arctan\sigma+ \frac{\sigma}{1+\sigma^2}\rg) \rg]\\
&\quad+2^{-1} \nabla^{\mathcal P}\rho^2\cdot\nabla^{\mathcal P}\r\lf[ -3\frac{\rho^2}{\r^4} ( \chi'(\r)-\ep^{-1}\chi'(\r/\ep) ) \sigma \arctan\sigma \rg]\\
&\quad+2^{-1} \nabla^{\mathcal P}\rho^2\cdot\nabla^{\mathcal P}\r\lf[ \frac{\rho^2}{\r^3} ( \chi''(\r)-\ep^{-2}\chi''(\r/\ep) ) \sigma \arctan\sigma \rg]\\
&\quad+\frac{\rho^2}{\r^4} \nabla^{\mathcal P}\rho^2\cdot\nabla^{\mathcal P}\r ( \chi'(\r)-\ep^{-1}\chi'(\r/\ep) )
+\nabla^{\mathcal P}\rho^2\cdot\nabla^{\mathcal P}\lf[ \frac{\rho^2}{\r^4} \rg] ( \chi(\r)-\chi(\r/\ep) ).
\end{align*}

We also have
\begin{align*}
\ds |\nabla^{\mathcal P} z|^2\frac{\p F}{\p\varphi}
=2|\nabla^{\mathcal P} z|^2 \lf[\frac{\varphi}{\r^3} ( \chi'(\r)-\ep^{-1}\chi'(\r/\ep) ) \arctan\sigma+(\chi(\r)-\chi(\r/\ep))\frac{\rho^2}{\r^4}\rg],
\end{align*}
and finally, using the fact that $\rho^4=\frac{\r^4}{1+\sigma^2}$, we compute that
\begin{align*}
&2\nabla^{\mathcal P}\varphi\cdot\nabla^{\mathcal P} \lf[\frac{\p F}{\p \rho^2}\rg]\\
&=\nabla^{\mathcal P}(\sigma\rho^2)\cdot\nabla^{\mathcal P} \lf[  \frac{\rho^2}{2\r^3} ( \chi'(\r)-\ep^{-1}\chi'(\r/\ep) ) \arctan\sigma  \rg]
-4\nabla^{\mathcal P}\varphi\cdot\nabla^{\mathcal P} \lf[ ( \chi(\r)-\chi(\r/\ep) )\frac{\varphi}{\r^4}  \rg]\\
&=|\nabla^{\mathcal P}\rho^2|^2\lf[  \frac{1}{2\r^3} ( \chi'(\r)-\ep^{-1}\chi'(\r/\ep) ) \sigma \arctan\sigma  \rg]
+\frac{|\nabla^{\mathcal P}\sigma|^2}{(1+\sigma^2)^2} \lf[  \frac{\r}{2} ( \chi'(\r)-\ep^{-1}\chi'(\r/\ep) ) \rg]\\
&\quad+\nabla^{\mathcal P}\sigma\cdot\nabla^{\mathcal P}\rho^2 \lf[  \frac{\rho^2}{2\r^3} ( \chi'(\r)-\ep^{-1}\chi'(\r/\ep) ) \lf(\arctan\sigma +\frac{\sigma}{1+\sigma^2}\rg)  \rg]\\
&\quad+\nabla^{\mathcal P}\rho^2\cdot\nabla^{\mathcal P}\r \lf[ -\frac{3\rho^2}{2\r^4} ( \chi'(\r)-\ep^{-1} \chi'(\r/\ep) ) \sigma \arctan\sigma  \rg]\\
&\quad+\nabla^{\mathcal P}\sigma\cdot\nabla^{\mathcal P}\r \lf[ -\frac{3}{2} ( \chi'(\r)-\ep^{-1} \chi'(\r/\ep) ) \frac{\arctan\sigma}{1+\sigma^2}  \rg]\\
&\quad+\nabla^{\mathcal P}\rho^2\cdot\nabla^{\mathcal P}\r \lf[ \frac{\rho^2}{2\r^3} ( \chi''(\r)-\ep^{-2}\chi''(\r/\ep) ) \sigma \arctan\sigma  \rg]\\
&\quad+\nabla^{\mathcal P}\sigma\cdot\nabla^{\mathcal P}\r \lf[ \frac{\r}{2} ( \chi''(\r)-\ep^{-2} \chi''(\r/\ep) ) \frac{\arctan\sigma}{1+\sigma^2}  \rg]\\
&\quad-4\frac{\varphi}{\r^4} \nabla^{\mathcal P}\varphi\cdot\nabla^{\mathcal P}\r ( \chi'(\r)-\ep^{-1} \chi'(\r/\ep) )
-4\nabla^{\mathcal P}\varphi\cdot\nabla^{\mathcal P}\lf[\frac{\varphi}{\r^4}  \rg] ( \chi(\r)-\chi(\r/\ep) ).
\end{align*}

We now claim that
\begin{equation}
\label{V.29}
2^{-1}\rho^2 |\nabla^{\mathcal P}z|^2=\rho^2|\nabla^{\mathcal P}\rho|^2+|\nabla^{\mathcal P}\varphi|^2.
\end{equation}
Indeed, as above, given a Legendrian plane ${\mathcal P}\in G$, let $(Z_1,Z_2)$ be an orthonormal basis,
and let $Z_\ell':=\pi_\ast Z_\ell\in\C^2$. Since $\pi_*$ is an isometry, $(Z_1',Z_2')$ is still an orthonormal pair of vectors in ${\C}^2$, spanning a Lagrangian plane $\mathcal{P}'$.
Hence, $(Z_1',Z_2',iZ_1',iZ_2')$ is an orthonormal basis of ${\C}^2$, giving
$$\rho^2=|z|^2=|z\cdot Z_1|^2+|z\cdot Z_2|^2+|z\cdot iZ_1|^2+|z\cdot iZ_2|^2.$$
We also have
$$|\nabla^{\mathcal P}z|^2=\sum_{k=1}^4\sum_{\ell=1}^2|Z_\ell\cdot\nabla^H z_k|^2=\sum_{k=1}^4\sum_{\ell=1}^2|Z_\ell\cdot e_k|^2=2,$$
where $(e_1,e_2,e_3,e_4)$ is the canonical basis of $\C^2$, while the identity
$$2z\cdot Z_\ell'=\sum_{k=1}^4 2z_k(Z_\ell\cdot \nabla^H z_k)=Z_\ell\cdot\nabla^H\rho^2$$
gives
$$\sum_{\ell=1}^2 |z\cdot Z_\ell'|^2=\frac{|\nabla^{\mathcal P}\rho^2|^2}{4}=\rho^2|\nabla^{\mathcal P}\rho|^2.$$
Similarly, since $iZ_\ell'=\sum_{k=1}^4 (Z_\ell'\cdot e_k) ie_k$ and $\nabla^H\varphi=z_1\nabla^Hz_2-z_2\nabla^Hz_1+z_3\nabla^Hz_4-z_4\nabla^Hz_3$, we have
$$z\cdot iZ_\ell'=(Z_\ell'\cdot e_1) z_2-(Z_\ell'\cdot e_2) z_1+(Z_\ell'\cdot e_3) z_4-(Z_\ell'\cdot e_4) z_3=-Z_\ell\cdot\nabla^H\varphi,$$
and hence
$$\sum_{\ell=1}^2 |z\cdot iZ_\ell'|^2=|\nabla^{\mathcal P}\varphi|^2;$$
the claim follows by combining the previous identities.

We apply the previous claim to the following sum:
\begin{align*}
&\nabla^{\mathcal P}\rho^2\cdot\nabla^{\mathcal P}\lf(\frac{\rho^2}{\r^4}\rg) 
+2 |\nabla^{\mathcal P} z|^2 \frac{\rho^2}{\r^4}
+4 \nabla^{\mathcal P}\varphi\cdot\nabla^{\mathcal P}\lf(\frac{\varphi}{\r^4}\rg)\\
&=\frac{|\nabla^{\mathcal P}\rho^2|^2}{\r^4}
-\rho^2\nabla^{\mathcal P}\rho^2\cdot\frac{\nabla^{\mathcal P}\r^4}{\r^8}
+4\rho^2\frac{|\nabla^{\mathcal P} \rho|^2}{\r^4}
+8\frac{|\nabla^{\mathcal P}\varphi|^2}{\r^4}
-4\varphi\nabla^{\mathcal P}\varphi\cdot\frac{\nabla^{\mathcal P}\r^4}{\r^8}\\
&=\frac{1}{\r^4}\lf[2|\nabla^{\mathcal P}\rho^2|^2+8|\nabla^{\mathcal P}\varphi|^2-\rho^2\nabla^{\mathcal P}\rho^2\cdot\frac{\nabla^{\mathcal P}[\rho^4+4\varphi^2]}{\r^4}-32\varphi^2\frac{|\nabla^{\mathcal P}\varphi|^2}{\r^4}-2\nabla^{\mathcal P}\varphi^2\cdot\frac{\nabla^{\mathcal P}\rho^4}{\r^4}\rg]\\
&=\frac{1}{\r^4}\lf[ 2\lf(1-\frac{\rho^4}{\r^4}\rg) |\nabla^{\mathcal P}\rho^2|^2+8\lf(1-4\frac{\varphi^2}{\r^4}\rg)|\nabla^{\mathcal P}\varphi|^2-8 \frac{\rho^2}{\r^4}\nabla^{\mathcal P}\rho^2\cdot\nabla^{\mathcal P}\varphi^2 \rg].
\end{align*}
Recalling that $\r^4=\rho^4(1+\sigma^2)$, the previous expression equals
\begin{align*}
&\frac{1}{\r^4}\lf[ \frac{2\sigma^2}{1+\sigma^2}|\nabla^{\mathcal P}\rho^2|^2+\frac{8\rho^4}{\r^4}|\nabla^{\mathcal P}\varphi|^2-\frac{8}{1+\sigma^2}\rho^{-2}\nabla^{\mathcal P}\rho^2\cdot\nabla^{\mathcal P}\varphi^2 \rg]\\
&=\frac{8}{\r^4(1+\sigma^2)}\lf[ |{\varphi}\rho^{-2}\nabla^{\mathcal P}\rho^2|^2+ |\nabla^{\mathcal P}\varphi|^2- 2\varphi \rho^{-2}\nabla^{\mathcal P}\rho^2\cdot\nabla^{\mathcal P}\varphi \rg]\\
&=\frac{8}{\r^4(1+\sigma^2)} \lf|\nabla^{\mathcal P}\varphi-{\varphi}\rho^{-2}\nabla^{\mathcal P}\rho^2\rg|^2\\
&=\frac{8}{(1+\sigma^2)^2} \lf|\rho^{-2}\nabla^{\mathcal P}\varphi+{\varphi}\nabla^{\mathcal P}\rho^{-2}\rg|^2.
\end{align*}
Hence we have established the following identity:
\begin{equation}
\label{V.33}
\nabla^{\mathcal P}\rho^2\cdot\nabla^{\mathcal P}\lf(\frac{\rho^2}{\r^4}\rg) 
+2|\nabla^{\mathcal P} z|^2\frac{\rho^2}{\r^4}
+4\nabla^{\mathcal P}\varphi\cdot\nabla^{\mathcal P}\lf(\frac{\varphi}{\r^4}\rg)
=2|\nabla^{\mathcal P}\arctan\sigma|^2.
\end{equation}
Combining the previous computations, we see that the terms multiplying $\chi$ and $\chi''$ match in the two sides of the statement.
As for $\chi'$, the same holds once we use the simple identity
$$ \rho^2\nabla^{\mathcal P}\rho^2\cdot\nabla^{\mathcal P}\r+4\varphi\nabla^{\mathcal P}\varphi\cdot\nabla^{\mathcal P}\r=2|\nabla^{\mathcal P}\r|^2 \r^3, $$
after a number of cancellations.
\end{proof}

\subsection{Monotonicity formula and its consequences}
We start with a direct consequence of the previous proposition.

\begin{Prop}
\label{mon-for-wea} Let ${\mathbf v}$ be a Hamiltonian stationary Legendrian varifold (HSLV). Then, under the previous notation, the following identity holds:
\begin{align}
\label{V.56}
\begin{aligned}
0&=2\int_{G} \frac{|\nabla^{\mathcal P}\r|^2}{\r} ( \chi'(\r)-\ep^{-1}\chi'(\r/\ep) )\,d{\mathbf v}({\mathcal P},p)\\
&\quad+\int_{G} |\nabla^{\mathcal P} z|^2 \lf[\frac{2\varphi}{\r^3} ( \chi'(\r)-\ep^{-1}\chi'(\r/\ep) ) \arctan\sigma\rg] \, d{{\mathbf v}}({\mathcal P},p)\\
&\quad-\frac{1}{2}\int_{G} {\r^4}  \nabla^{\mathcal P}\arctan\sigma\cdot\nabla^{\mathcal P} [\r^{-3} ( \chi'(\r)-\ep^{-1}\chi'(\r/\ep) ) {\arctan\sigma} ]\, d{{\mathbf v}}({\mathcal P},p)\\
&\quad+2\int_{G}  |\nabla^{\mathcal P}\arctan\sigma|^2 ( \chi(\r)-\chi(\r/\ep) )\, d{\mathbf v}({\mathcal P},p),
\end{aligned}
\end{align}
where $G$ denotes the Grassmannian bundle of Legendrian two-planes over $\mathbb{H}^2$. \hfill $\Box$
\end{Prop}

\begin{proof}
We already observed that the function $F$ considered above is smooth and compactly supported on $\mathbb{H}^2$.
Hence, for the associated Hamiltonian vector field $W_F$, we have
$$\int_G\operatorname{div}_{\mathcal P}W_F\,d\mathbf{v}(\mathcal P,p)=0.$$
Now the left-hand side of the identity stated in Proposition \ref{magic.id} is equal to $-2\operatorname{div}_{\mathcal P}W_F$.
While this identity was obtained only on $\mathbb{H}^2\setminus\{\rho=0\}$,
it is valid everywhere (once we replace the left-hand side with $-\operatorname{div}_{\mathcal P}W_F$), since both $\operatorname{div}_{\mathcal P}W_F$ and its right-hand side extend continuously to all of $G$:
indeed, recall that $\arctan\sigma$ is smooth outside of the origin. Hence, the claim follows directly from Proposition \ref{magic.id}.
\end{proof}

In the sequel, we slightly restrict the class of cut-off functions $\chi$. Namely, besides the condition that $\chi=1$ on $[0,1]$
and $\chi=0$ on $[2,\infty)$, we also require that
$$-\chi'=\eta^2,\quad\text{for some }\eta\in C^\infty_c((1,2)).$$
Given a HSLV $\mathbf{v}$ on an open set $U\subseteq\mathbb{H}^2$ and a ball $B_{2a}^\r(q)\subseteq U$, we now consider the quantity
\begin{align}\label{capital.theta.def}
\begin{aligned}
\Theta^\chi(q,a)&:=-\int_{G} \frac{|\nabla^{\mathcal P}\r_q|^2}{\r_q} a^{-1}\chi'(\r_q/a) \,d{\mathbf v}({\mathcal P},p)\\
&\phantom{:}\quad-\int_{G} \lf[\frac{2\varphi_q}{\r_q^3} a^{-1}\chi'(\r_q/a) \arctan\sigma_q\rg] \, d{{\mathbf v}}({\mathcal P},p)\\
&\phantom{:}\quad+\frac{1}{4}\int_{G} {\r_q^4}  \nabla^{\mathcal P}\arctan\sigma_q\cdot\nabla^{\mathcal P} [\r_q^{-3} a^{-1}\chi'(\r_q/a)  {\arctan\sigma_q} ]\, d{{\mathbf v}}({\mathcal P},p),
\end{aligned}
\end{align}
where we let $\r_q(x):=\r(q^{-1}*x)$, and similarly we define $\varphi_q$ and $\arctan\sigma_q$.
We will often drop the superscript $\chi$ in the sequel, writing $\Theta^\chi$ in place of $\Theta$.

The following statement is a \emph{monotonicity formula} for the area in this Legendrian setting. It constitutes one of the fundamental tools in the present work and it improves on a weaker monotonicity statement obtained by Schoen--Wolfson \cite{SW2}, both in terms of effectiveness and simplicity of proof (we just mention that in \cite{SW2} the proof involved solving a certain wave-type equation and relied on certain properties of special functions). A more similar version of it was obtained for smooth immersions in \cite{Riv2}.

\begin{Th}\label{mono.cor}
Assume that $\mathbf{v}$ is a HSLV on an open set $U\subseteq\mathbb{H}^2$. Then we have
$$0\le\Theta^\chi(q,a)\le\Theta^\chi(q,b)\quad\text{for all }0<a<b\le d_K(q,\mathbb{H}^2\setminus U)/2.$$
Moreover, the \emph{density}
$$\theta^\chi(q):=\lim_{\ep\rightarrow 0}
-\frac{1}{\ep}\int_{G}\chi'(\r_q/\ep)\lf[\frac{|\nabla^{\mathcal P}\r_q|^2}{\r_q}
+\frac{2\varphi_q}{\r_q^3}\arctan\sigma_q\rg]\,d\mathbf{v}(\mathcal P,p)$$
exists in $\R_+=[0,\infty)$ and we have
$$\theta^\chi(q)=\lim_{\ep\to0}\Theta^\chi(q,\ep),$$
as well as
$$\theta^\chi(q)+\int_{0<\r_q<b}|\nabla^{\mathcal P}\arctan\sigma_q|^2\,d\mathbf{v}(\mathcal{P},p)
\le Cb^{-2}|\mathbf{v}|(B_{2b}^\r(q)\setminus \bar B_b^\r(q)),$$
for a constant $C>0$ depending only on $\chi$. \hfill $\Box$
\end{Th}

\begin{Rm}
Up to harmless error terms,
a similar monotonicity formula holds in arbitrary closed Sasakian manifolds, whose infinitesimal model is $\mathbb{H}^2$ (see \cite[Section VI]{Riv3} and the references therein),
with a similar proof. Of course, in this case monotonicity is only effective at small scales.
However, in all of the following arguments, we can always work at small enough scales;
in particular, blow-ups are again varifolds on $\mathbb{H}^2$. \hfill $\Box$
\end{Rm}

\begin{Rm}
Although the integrand defining $\Theta^\chi$ is not guaranteed to be nonnegative,
we observe that the one in the definition of $\theta^\chi$ is always nonnegative, since $\chi'\le0$ and (away from $\{\rho_q=0\}$) $\varphi_q$ has the same sign as $\sigma_q$,
so that $\varphi_q\arctan\sigma_q\ge0$. \hfill $\Box$
\end{Rm}

\begin{Rm}\label{blow}
Any blow-up of $\mathbf{v}$ at $q$ has $\nabla^{\mathcal P}\arctan\sigma=0$ on its support (away from the origin):
indeed, the rescaled varifold $\mathbf{v}_{q,a}:=(\delta_{1/a}\circ\ell_{q^{-1}})_*\mathbf{v}$ satisfies
$$\int_{0<\r<R}|\nabla^{\mathcal P}\arctan\sigma|^2\,d\mathbf{v}_{q,a}(\mathcal{P},p)
=\int_{0<\r_q<Ra}|\nabla^{\mathcal P}\arctan\sigma_q|^2\,d\mathbf{v}(\mathcal{P},p)$$
for all $a,R>0$ with $2Ra<d_K(q,\mathbb{H}^2\setminus U)$, so that the left-hand side converges to zero for fixed $R>0$.
\hfill $\Box$
\end{Rm}

\begin{proof}
First, note that we have
$$|\nabla^{\mathcal P}z|^2=2$$
for any $\mathcal P\in G$: indeed, letting $(Z_1,Z_2)$ be an orthonormal basis of $\mathcal P$, we have
$$|\nabla^{\mathcal P}z|^2=\sum_{k=1}^2\sum_{\ell=1}^4|\nabla^H z_\ell\cdot Z_k|^2.$$
The claim now follows from the fact that $\nabla^Hz_{2j-1}=X_j$ and $\nabla^Hz_{2j}=Y_j$.

In the sequel, we can assume that $q=0$ and $b=1$, up to a left translation and a dilation. Assume momentarily that
\begin{equation}\label{ass.del.cazzo}
\liminf_{\ep\to0}\int_{\epsilon<\r<2\epsilon}|\nabla^{\mathcal P}\arctan\sigma|^2\,d\mathbf{v}(\mathcal P,p)<\infty.
\end{equation}

We rearrange \eqref{V.56} as
\begin{align}\label{mono.rearr}
\begin{aligned}
&2\int_{G}  |\nabla^{\mathcal P}\arctan\sigma|^2 ( \chi(\r)-\chi(\r/\ep) )\, d{\mathbf v}({\mathcal P},p)\\
&\quad-2\int_{G} \frac{|\nabla^{\mathcal P}\r|^2}{\r} \ep^{-1}\chi'(\r/\ep)\,d{\mathbf v}({\mathcal P},p)
-\int_{G} \frac{4\varphi}{\r^3} \ep^{-1}\chi'(\r/\ep) \arctan\sigma \, d{{\mathbf v}}({\mathcal P},p)\\
&\quad+\frac{1}{2}\int_{G} {\r^4}  \nabla^{\mathcal P}\arctan\sigma\cdot\nabla^{\mathcal P} [\r^{-3} \ep^{-1}\chi'(\r/\ep) {\arctan\sigma}  ]\, d{{\mathbf v}}({\mathcal P},p)\\
&=-2\int_{G} \frac{|\nabla^{\mathcal P}\r|^2}{\r} \chi'(\r)\,d{\mathbf v}({\mathcal P},p)
-\int_{G} \frac{4\varphi}{\r^3} \chi'(\r) \arctan\sigma \, d{{\mathbf v}}({\mathcal P},p)\\
&\quad+\frac{1}{2}\int_{G} {\r^4}  \nabla^{\mathcal P}\arctan\sigma\cdot\nabla^{\mathcal P} [\r^{-3} \chi'(\r) {\arctan\sigma} ]\, d{{\mathbf v}}({\mathcal P},p),
\end{aligned}
\end{align}
which proves that $\Theta^\chi(q,\ep)\le\Theta^\chi(q,1)$ for $0<\ep<1$, and thus the inequality $\Theta^\chi(q,a)\le\Theta^\chi(q,b)$ in the statement.

Since $\chi'(\r)$ vanishes outside the set $\{1<\r<2\}$, the right-hand side above is bounded by
$$C\int_{1<\r<2}1\,d\mathbf{v}(\mathcal{P},p).$$
Moreover, the first three terms in the left-hand side are nonnegative. To conclude,
we need to control the last term in the left-hand side of \eqref{mono.rearr}. A simple expansion (using also the fact that $\r\in[\ep,2\ep]$ on the support of $\chi'(\r/\ep)$) shows that it is bounded by
the integral of
$$C|\nabla^{\mathcal P}\arctan\sigma|^2|\chi'(\r/\ep)|+C\ep^{-1}|\nabla^{\mathcal P}\arctan\sigma||\nabla^{\mathcal P}\r|(|\chi'(\r/\ep)|+|\chi''(\r/\ep)|).$$
Since $-\chi'=\eta^2$, we have $|\chi''|=|(\eta^2)'|=2\eta|\eta'|\le C\eta\uno_{(1,2)}$, so that
$$C\ep^{-1}|\nabla^{\mathcal P}\arctan\sigma||\nabla^{\mathcal P}\r||\chi''(\r/\ep)|
\le -\frac{C\delta}{\ep^2}\chi'(\r/\ep)|\nabla^{\mathcal P}\r|^2+\frac{C}{\delta}\uno_{\ep<\r<2\ep}|\nabla^{\mathcal P}\arctan\sigma|^2$$
for an arbitrary $\delta>0$.
Similarly, the term $C\ep^{-1}|\nabla^{\mathcal P}\arctan\sigma||\nabla^{\mathcal P}\r||\chi'(\r/\ep)|$ obeys the same bound.
Thus,
\begin{align*}
&\int_{G} {\r^4}  |\nabla^{\mathcal P}\arctan\sigma|\cdot|\nabla^{\mathcal P}[\r^{-3} \ep^{-1}\chi'(\r/\ep) {\arctan\sigma} ]|\, d{{\mathbf v}}({\mathcal P},p)\\
&\le C\delta\int_G\frac{|\nabla^{\mathcal P}\r|^2}{\r} \ep^{-1}\chi'(\r/\ep)\,d{\mathbf v}({\mathcal P},p)
+\frac{C}{\delta}\int_{\ep<\r<2\ep}|\nabla^{\mathcal P}\arctan\sigma|^2\,d{\mathbf v}({\mathcal P},p).
\end{align*}
Choosing $\delta>0$ small enough, we can absorb the integral of $-\frac{|\nabla^{\mathcal P}\r|^2}{\r}\ep^{-1}\chi'(\r/\ep)$,
which appeared in the left-hand side of \eqref{mono.rearr}.
Once we let $\ep\to0$, we deduce that
\begin{align*}
&2\int_{0<\r<1}  |\nabla^{\mathcal P}\arctan\sigma|^2 \, d{\mathbf v}({\mathcal P},p)\\
&\le C\int_{1<\r<2}1\,d\mathbf{v}(\mathcal{P},p)+C\liminf_{\ep\to0}\int_{\epsilon<\r<2\epsilon}\ep^{-2}|\nabla^{\mathcal P}\arctan\sigma|^2\,d\mathbf{v}(\mathcal P,p),
\end{align*}
which is finite by our assumption \eqref{ass.del.cazzo}.
In particular, the integral of $|\nabla^{\mathcal P}\arctan\sigma|^2$ on $\{0<\r<1\}$ is finite, so
we can upgrade \eqref{ass.del.cazzo} to
$$\limsup_{\ep\to0}\int_{\epsilon<\r<2\epsilon}|\nabla^{\mathcal P}\arctan\sigma|^2\,d\mathbf{v}(\mathcal P,p)=0.$$
Reinserting this information in the previous computation, we deduce
\begin{align*} &\int_{0<\r<1}2|\nabla^{\mathcal P}\arctan\sigma|^2 ( \chi(\r)-\chi(\r/\ep) ) \, d{{\mathbf v}}({\mathcal P},p)\\
&\quad+\limsup_{\ep\to0}\int_{G} \lf[-2(1-C\delta) \frac{|\nabla^{\mathcal P}\r|^2}{\r} \ep^{-1}\chi'(\r/\ep)
- \frac{4\varphi}{\r^3} \ep^{-1}\chi'(\r/\ep) \arctan\sigma\rg] \, d{{\mathbf v}}({\mathcal P},p)\\
&\le C\int_{1<\r_p<2}1\,d\mathbf{v}(\mathcal{P},p) \end{align*}
for an arbitrarily small $\delta>0$. Thus, in \eqref{mono.rearr} we see that the last term in the left-hand side goes to zero as $\ep\to0$; it follows that the limit defining $\theta^\chi$ exists, and moreover we have
$$\theta^\chi(q)=\lim_{\ep\to0}\Theta^\chi(q,\ep),$$
since the third term in \eqref{capital.theta.def} goes to zero as $\ep\to0$, as we showed above.

We now remove the technical assumption \eqref{ass.del.cazzo}. Note that, by homogeneity of $\arctan\sigma$ under dilations,
we have $|\nabla^{\mathcal P}\arctan\sigma|^2\le C\r^{-2}$. In particular, the statement holds at all points
$p\in U\setminus S$ (for radii $0<2r<d_K(p,\mathbb{H}^2\setminus U)$), where the exceptional set $S$ is given by
$$S:=\lf\{p\in U\,:\,\limsup_{\ep\to0}\ep^{-2}|\mathbf{v}|(B_\ep^\r(p))=\infty\rg\}.$$
A simple application of Vitali's covering lemma shows that $\mathcal{H}^2_K(S)=0$. In particular, $U\setminus S$ is dense in $U$.
We can then take a sequence $p_j\to 0$, with $p_j\in U\setminus S$, and repeat the previous proof to obtain
$$\int_{0<\r_{p_j}<1/2}|\nabla^{\mathcal P}\arctan\sigma_{p_j}|^2\,d\mathbf{v}(\mathcal P,p)\le C\int_{1/2<\r_{p_j}<1}1\,d\mathbf{v}(\mathcal P,p),$$
which is bounded by a constant independent of $j$. By Fatou's lemma, we obtain
$$\int_{0<\r<1/2}|\nabla^{\mathcal P}\arctan\sigma_{0}|^2\,d\mathbf{v}(\mathcal P,p)<\infty,$$
showing that \eqref{ass.del.cazzo} holds also at $p=0$.
\end{proof}

\begin{Co}\label{dens.usc}
Assume that $\mathbf{v}$ is a HSLV on an open set $U\subseteq\mathbb{H}^2$.
Then the density $\theta^\chi$ of $\mathbf{v}$ is upper semi-continuous on $U$. \hfill $\Box$
\end{Co}

\begin{proof}
Fix $q\in U$ and consider a sequence $q_k\to q$ of points in $U$. Given any $\lambda>0$, we need to show that eventually
$$\theta^\chi(q_k)\le\theta^\chi(q)+\lambda.$$
We let $b>0$ such that $B_{4b}^\r(q_k)\subseteq U$ for all $k$. Clearly, we have
$$\theta^\chi(q_k)=\lim_{a\to0}\Theta^\chi(q_k,a)\le\Theta^\chi(q_k,b),$$
as well as $\Theta^\chi(q_k,b)\to\Theta^\chi(q,b)$, since the integrand defining $\Theta^\chi$ is cut-off near the center. Thus, eventually we have
$$\theta^\chi(q_k)\le\Theta^\chi(q,b)+\frac{\lambda}{2}.$$
Finally, since $\Theta^\chi(q,b)\to\theta^\chi(q)$ as $b\to0$, we can choose $b>0$ so small that
$$\Theta^\chi(q,b)\le\theta^\chi(q)+\frac{\lambda}{2},$$
obtaining the claim.
\end{proof}

The same argument shows the following more general fact.

\begin{Co}\label{dens.limits}
Assume that $\mathbf{v}_k$ is a sequence of HSLVs on an open set $U\subseteq\mathbb{H}^2$ converging to
$\mathbf{v}$. Then $\mathbf{v}$ is a HSLV and, given any sequence $q_k\to q\in U$ (with $q_k\in U$), we have
$$\theta^\chi(\mathbf{v},q)\ge\limsup_{k\to\infty}\theta^\chi(\mathbf{v}_k,q_k).$$
In particular, if $\theta^\chi(\mathbf{v}_k,p)\ge\nu>0$ at $|\mathbf{v}_k|$-a.e. $p$ then this actually holds for all
$p\in\operatorname{spt}|\mathbf{v}_k|$ and $$\theta^\chi(\mathbf{v},p)\ge\nu>0\quad\text{for all } p\in\operatorname{spt}|\mathbf{v}|,$$ as well as
$$\operatorname{spt}|\mathbf{v}_k|\to\operatorname{spt}|\mathbf{v}|$$
in the local Hausdorff topology. \hfill $\Box$
\end{Co}

The following statement gives upper density bounds on the support, at macroscopic scales.
As a consequence, a $\mathcal{H}^2_K$-negligible set is also $|\mathbf{v}|$-negligible.

\begin{Co}\label{mono.cor2}
There exists a universal constant $C>0$ such that the following holds.
Assuming that $\mathbf{v}$ is a HSLV on $U\supseteq B_{2s}^\r(q)$, we have
$$\frac{|\mathbf{v}|(B_r^\r(q))}{r^2}\le C\frac{|\mathbf{v}|(B_{2s}^\r(q)\setminus B_{s}^\r(q))}{s^2}$$
for all $0<r\le s/2$, where $|\mathbf{v}|:=\Pi_*\mathbf{v}$ denotes the weight of $\mathbf{v}$. \hfill $\Box$
\end{Co}

\begin{proof}
By a left translation and dilation, we can assume that $q=0$ and $s=1$.
By Theorem \ref{mono.cor} we have
$$\int_{0<\r<1}|\nabla^{\mathcal{P}}\arctan\sigma|^2\,d\mathbf{v}(\mathcal P,p)\le C|\mathbf{v}|(B_2^\r(0)\setminus B_1^\r(0)).$$
Thus, employing the same absorption used in the proof of Theorem \ref{mono.cor}, we see that
$$I(\ep):=-2\int_{G} \frac{|\nabla^{\mathcal P}\r|^2}{\r} \ep^{-1}\chi'(\r/\ep)\,d{\mathbf v}({\mathcal P},p)
-\int_{G} \frac{4\varphi}{\r^3} \ep^{-1}\chi'(\r/\ep) \arctan\sigma \, d{{\mathbf v}}({\mathcal P},p)$$
is bounded by the same quantity (for a possibly different $C>0$), for all $0<\ep\le 1/2$.

Now we compute that
$$J_H\nabla^H\frac{\rho^2}{2}=J_H\sum_{j=1}^4z_j\nabla^Hz_j=J_H(z_1X_1+z_2Y_1+z_3X_2+z_4Y_2)=z_1Y_1-z_2X_1+z_3Y_2-z_4X_2
=\nabla^H\varphi,$$
and hence
\begin{equation}\label{r.sigma.j}
\r^3 J_H(\nabla^H\r)=J_H\nabla^H\frac{\rho^4}{4}+J_H\nabla^H\varphi^2
=\rho^2\nabla^H\varphi-2\varphi\nabla^H\frac{\rho^2}{2}=\frac{\rho^4}{2}\nabla^H\sigma,
\end{equation}
an identity that we will use later on.
Since $\r^3\nabla^H\r=\rho^2\nabla^H\frac{\rho^2}{2}+2\varphi\nabla^H\varphi$ and $\nabla^H\frac{\rho^2}{2}\perp\nabla^H\varphi$
are two vectors with the same norm $\rho$,
we have
\begin{equation}\label{nabla.r.norm}
|\nabla^H\r|^2=\frac{\rho^4}{\r^6}\lf|\nabla^H\frac{\rho^2}{2}\rg|^2+\frac{4\varphi^2}{\r^6}|\nabla^H\varphi|^2
=\frac{\rho^4+4\varphi^2}{\r^6}\cdot\rho^2=\frac{\rho^2}{\r^2}=\frac{1}{\sqrt{1+\sigma^2}}.
\end{equation}
Moreover, we have $\frac{2\varphi}{\r^2}=\sigma\frac{\rho^2}{\r^2}=\frac{\sigma}{\sqrt{1+\sigma^2}}$ off the $\varphi$-axis $\{\rho=0\}$,
and actually (discarding the intermediate equalities) this holds also on $\{\rho=0\}\setminus\{0\}$, provided we interpret $\frac{\sigma}{\sqrt{1+\sigma^2}}=\operatorname{sgn}(\varphi)$ here; indeed, recall that $\sigma=\operatorname{sgn}(\varphi)\cdot(+\infty)$ on this set.
Thus,
we obtain
$$I(\ep)=-\frac2\ep\int_G\chi'(\r/\ep)\r^{-1}\lf[\frac{1+\sigma\arctan\sigma}{\sqrt{1+\sigma^2}}-|\nabla^{\mathcal P^\perp}\r|^2\rg]\,d\mathbf{v}(\mathcal P,p).$$

Now an elementary computation shows that
$$1\le\frac{1+\sigma\arctan\sigma}{\sqrt{1+\sigma^2}}\le\frac\pi2,$$
while thanks to \eqref{r.sigma.j} we have
$$|\nabla^{\mathcal P^\perp}\r|^2=\frac{\rho^8}{4\r^6}|\nabla^{\mathcal P}\sigma|^2=\r^2\frac{|\nabla^{\mathcal P}\sigma|^2}{4(1+\sigma^2)^2}
=\frac{\r^2}{4}|\nabla^{\mathcal P}\arctan\sigma|^2,$$
again an identity which extends to all of $\mathbb{H}^2\setminus\{0\}$. Hence,
$$\frac{2}{\ep}\int_G|\chi'(\r/\ep)|\cdot\r^{-1}|\nabla^{\mathcal P^\perp}\r|^2\,d\mathbf{v}(\mathcal P,p)
\le C(\chi)\int_{\ep<\r<2\ep}|\nabla^{\mathcal P}\arctan\sigma|^2\,d\mathbf{v}(\mathcal P,p).$$
Since this term obeys the desired bound, we obtain
$$-\frac2\ep\int_G\chi'(\r/\ep)\r^{-1}\le C|\mathbf{v}|(B_2^\r(0)\setminus B_1^\r(0))$$
for all $\ep\in(0,1/2]$. Taking $\chi$ such that $\chi'<0$ on $[4/3,5/3]$, we deduce
$$|\mathbf{v}|(B_{5\ep/3}^\r(0)\setminus B_{4\ep/3}^\r(0))\le C\ep^2|\mathbf{v}|(B_2^\r(0)\setminus B_1^\r(0)),$$
from which the conclusion easily follows.
\end{proof}

\section{Closure of integral varifolds among rectifiable ones}
Using the monotonicity formula from the previous section, we now prove a suitable version of Allard's compactness of integral stationary varifolds in this setting,
a statement which is perhaps interesting on its own.
The proof is essentially a very careful adaptation of the original argument by Allard \cite{All}, although we follow
more closely the presentation from \cite[Section 6]{DPP}
and some steps are much subtler in the present setting (see, e.g., Lemma \ref{zero.exc} below).

\begin{Th}\label{cpt.int}
Assume that $\mathbf{v}_k$ is a sequence of rectifiable HSLVs on an open set $U\subseteq\mathbb{H}^2$,
converging to a rectifiable varifold $\mathbf{v}_\infty$ here. If $\theta^\chi(\mathbf{v}_k,p)\in2\pi\N^*$
for $|\mathbf{v}_k|$-a.e. $p$, then the same holds for the limit varifold. \hfill $\Box$
\end{Th}

\begin{Rm}\label{orr}
Differently from the isotropic situation, the rectifiability of the limit has to be assumed and does \emph{not} come for free.
A counterexample was found in \cite[Appendix B]{Orr}; we also refer to Theorem \ref{th-count-ex},
where we give a counterexample in a closed ambient, namely in the sphere $S^5$.
Rephrasing slightly the example from \cite{Orr}, we consider
$$\Sigma:=\C/2\pi\Z(1-i)=\R^2/2\pi\Z(1,-1)$$
and $u:\Sigma\to\mathbb{H}^2$ given by
$$u(x)=u(x_1,x_2):=(\cos(x_1),\sin(x_1),\cos(x_2),\sin(x_2),x_1+x_2),$$
which has
$$\p_{x_j}u(x)=-\sin(x_j)X_j+\cos(x_j)Y_j.$$
Thus, the differential is a linear isometry at each point and $u$ is a Legendrian lift of the map $\pi\circ u$, which parametrizes the Clifford torus.
It is easy to check that $u$ is a proper embedding inducing a HSLV on $\mathbb{H}^2$ (cf. \cite[Theorem 2.7]{Oh2} and Remark \ref{h.minimal}), whose blow-down is
$$\mathbf{v}(\mathcal P,p)=2\pi\cdot\mu(\mathcal P)\otimes
(\mathcal{H}^1\res L)(p),$$
where $L:=\{z=0\}$ and $\mu$ is the uniform measure on the torus of Legendrian planes
$$(e^{ia},e^{ib})\mapsto\mathcal{P}^{(a,b)}:=\operatorname{span}\{\cos(a)X_1+\sin(a)Y_1,\cos(b)X_2+\sin(b)Y_2\}.$$
Moreover, observing that $|\nabla^H\r|=0$ (by \eqref{nabla.r.norm}) and $\r=\sqrt{2|\varphi|}$ on $\operatorname{spt}|\mathbf{v}|$, we see that
\begin{align*}
\theta^\chi(\mathbf{v},0)&=\lim_{\ep\to0}-\frac1\ep\int_{G}\chi'(\sqrt{2|\varphi|}/\ep)\frac{\pi|\varphi|}{\sqrt{8|\varphi|^3}}\,d|\mathbf{v}|\\
&=\lim_{\ep\to0}-\frac2\ep\int_0^\infty\chi'(\sqrt{2\varphi}/\ep)\frac{\pi}{\sqrt{8\varphi}}\cdot2\pi\,d\varphi\\
&=-2\pi^2\int_0^\infty\chi'(t)\,dt\\
&=2\pi^2,
\end{align*}
so that $\mathbf{v}$ is not rectifiable and on $\operatorname{spt}|\mathbf{v}|$ we have $\frac{\theta^\chi}{2\pi}=\pi\nin\N$.
\hfill $\Box$
\end{Rm}

The following is a useful observation showing that two rectifiability conditions based on the Euclidean and Heisenberg geometries agree.

\begin{Lm}\label{equiv.rect}
Given an open set $U\subseteq\mathbb{H}^2$ and a HSLV $\mathbf{v}$ on $U$ with $\theta^\chi(p)\ge\nu>0$ for $|\mathbf{v}|$-a.e. $p$, the varifold $\mathbf{v}$ is rectifiable
if and only if, for $|\mathbf{v}|$-a.e. $p$, any anisotropic blow-up of the form
$$\lim_{r\to0}(\delta_{1/r}\circ\ell_{p^{-1}})_*\mathbf{v}$$
along a sequence of radii $r\to0$ equals $\frac{\theta^\chi(p)}{2\pi}$ times a Legendrian plane, depending only on $p$
(recall that $\ell_{p^{-1}}(x):=p^{-1}*x$).
If either holds, then at $|\mathbf{v}|$-a.e. $p$ the isotropic blow-up (in terms of $g_{\mathbb{H}^2}$) agrees with the anisotropic blow-up. \hfill $\Box$
\end{Lm}

\begin{proof}
Note that $\theta^\chi\ge\nu$ on $\operatorname{spt}|\mathbf{v}|$, by Corollary \ref{dens.usc}.
We disintegrate $\mathbf{v}(\mathcal P,p)=\mu_p(\mathcal P)\otimes|\mathbf{v}|(p)$, where $\mu_p$
is a probability measure on $\Pi^{-1}(p)$. By left-invariance of $G$, we can identify $\mu_p$ with a probability measure on $G_0:=\Pi^{-1}(0)$.
We now consider the set $A$ of approximate continuity points for $p\mapsto\mu_p$ (with respect to the weak-$^*$ topology on probabilities), in terms of $d_K$-balls: namely, $p\in\operatorname{spt}|\mathbf{v}|$ belongs to $A$ if, for any continuous $f:G_0\to\R$ (or equivalently for a countable dense collection of such functions), we have
$$\int_{B_\ep^\r(p)}|f(\mu_q)-f(\mu_p)|\,d|\mathbf{v}|(q)=o(|\mathbf{v}|(B_\ep^\r(p)))\quad\text{as }\ep\to0,$$
or equivalently if this integral is $o(\ep^2)$, in view of Corollary \ref{mono.cor2}.
We have $p\in A$ for $|\mathbf{v}|$-a.e. $p$: indeed, note that $|\mathbf{v}|$ is locally a doubling measure on the metric space $(\operatorname{spt}|\mathbf{v}|,d_K)$.

If $\mathbf{v}$ is rectifiable, then for $|\mathbf{v}|$-a.e. $p$ we have $\mu_p=\delta_{\mathcal Q(p)}$ for some Legendrian plane $\mathcal Q(p)\in G_0$. Assuming also $p\in A$, any blow-up $\mathbf{w}$ satisfies the assumptions of Lemma \ref{zero.exc} below (up to a rotation), which tells us
that $\mathbf{w}$ is a constant multiple of $\mathcal Q(p)$. Since $\theta^\chi(\mathbf{w},0)=\theta^\chi(\mathbf{v},p)$, this
proves one implication.

To see the reverse implication, we consider a point $p\in A$ where any blow-up is $\frac{\theta^\chi(\mathbf{v},p)}{2\pi}$ times a plane $\mathcal{Q}(p)$.
In particular, we have $\mu_p=\delta_{\mathcal Q(p)}$.
Up to a translation and a rotation, we can assume that $p=0$ and $\mathcal{Q}(0)=\operatorname{span}\{X_1(0),X_2(0)\}$.
By a straightforward compactness argument (using Corollary \ref{dens.limits}), we see that $\varphi(q)=o(\rho^2(q))$ as $q\to0$ in $\operatorname{spt}|\mathbf{v}|$.
Hence, we easily deduce that
$$\lim_{\ep\to0}\frac{|\mathbf{v}|(B_\ep(0))}{\ep^2}=\lim_{\ep\to0}\frac{|\mathbf{v}|(B^\r_\ep(0))}{\ep^2}=\theta^\chi(\mathbf{v},0),$$
as well as
$$\lim_{\ep\to0}\ep^{-2}(\tilde\delta_{1/\ep})_*|\mathbf{v}|=\lim_{\ep\to0}\ep^{-2}(\delta_{1/\ep})_*|\mathbf{v}|=\frac{\theta^\chi(\mathbf{v},0)}{2\pi}\mathcal{H}^2\res\mathcal Q(0),$$
where $\tilde\delta_{1/\ep}(q):=\ep^{-1}q$ is the Euclidean dilation. Thus, $|\mathbf{v}|$ is a rectifiable measure
and $\mathbf{v}(\mathcal P,p)=\delta_{\mathcal Q(p)}(\mathcal P)\otimes|\mathbf{v}|(p)$.
\end{proof}

The proof of Theorem \ref{cpt.int} is based on the following lemma, whose analogue in the Euclidean space is a simple exercise
but which turns out to be quite subtle for HSLVs in the Heisenberg group.

\begin{Lm}\label{zero.exc}
Given a HSLV $\mathbf{v}$ on an open set $U\subseteq\mathbb{H}^2$, assume that its density
$\theta^\chi(p)\ge \nu>0$ for all $p\in\operatorname{spt}|\mathbf{v}|$ and that
\begin{equation}\label{tensor.v}\mathbf{v}(\mathcal P,p)=\delta_{\mathcal P_p}(\mathcal P)\otimes|\mathbf{v}|(p),\end{equation}
where we set $\mathcal P_p:=\operatorname{span}\{X_1(p),X_2(p)\}$. Then $\mathbf{v}$
is locally a finite union of left translates of the plane $\mathcal{P}_0\subset\mathbb{H}^2$, with constant multiplicity.
Moreover, if $\mathbf{v}$ is an (anisotropic) blow-up then it is a constant multiple of $\mathcal{P}_0$.
 \hfill $\Box$
\end{Lm}

Note that we can equivalently require that $\theta^\chi(p)\ge \nu>0$ for $|\mathbf{v}|$-a.e. $p$,
by Corollary \ref{dens.usc}.
We will implicitly use the fact that this density assumption is stable under varifold limits, and in particular under blow-ups,
by Corollary \ref{dens.limits}.
We also observe that, under the assumption \eqref{tensor.v}, stationarity can be conveniently rewritten as
\begin{equation}\label{tensor.v.bis}\int_{U}[X_1(Y_1(F))+X_2(Y_2(F))]\,d|\mathbf{v}|=0\quad\text{for all }F\in C^\infty_c(U).\end{equation}

\begin{proof}
First we show the statement for blow-ups, namely we assume that $U=\mathbb{H}^2$ and $\nabla^{\mathcal P}\sigma$ vanishes
on $\Pi^{-1}(\mathbb{H}^2\setminus\{z=0\})\cap\operatorname{spt}(\mathbf{v})$, as seen in Remark \ref{blow}.
In particular, by \eqref{tensor.v}, for $p\in\operatorname{spt}|\mathbf{v}|$ we have
$$\nabla^H\sigma(p)\cdot X_j(p)=0\quad\text{for }j=1,2.$$
Computing $\nabla^H\sigma=2\rho^{-2}(\nabla^H\varphi-\sigma\nabla^H\frac{\rho^2}{2})$ and recalling that
$$\nabla^H\varphi=z_1Y_1-z_2X_1+z_3Y_2-z_4X_2,\quad \nabla^H\frac{\rho^2}{2}=z_1X_1+z_2Y_1+z_3X_2+z_4Y_2,$$
we deduce that $z_2=-\sigma z_1$ and $z_4=-\sigma z_3$ on $\operatorname{spt}|\mathbf{v}|\setminus\{z=0\}$.

We now claim that $\sigma=0$ on $\operatorname{spt}|\mathbf{v}|\setminus\{z=0\}$, so that
$$\operatorname{spt}|v|\subseteq \{\varphi=0\}\cup\{z=0\}.$$
Then, taking any Hamiltonian of the form $F(z,\varphi):=z_1z_2\psi(\varphi)$, with $\psi\in C^\infty_c(\R\setminus\{0\})$,
and recalling \eqref{tensor.v.bis}, it is easy to deduce that actually $\operatorname{spt}|\mathbf{v}|\subseteq\{\varphi=0\}$.
Taking into account that $z_{2j}=-\sigma z_{2j-1}=0$ off the $\varphi$-axis, we obtain $\operatorname{spt}|\mathbf{v}|\subseteq\mathcal P_0$.
Finally, taking $$F(z,\varphi):=a(z_1,z_3)z_2+b(z_1,z_3)z_4$$ for $a,b\in C^\infty_c(\R^2)$, we deduce that $\mathbf{v}$ is a stationary varifold in the usual sense, and hence by the constancy theorem it has constant multiplicity, as desired.

To check the previous claim, assume by contradiction that $\bar p\in\operatorname{spt}|\mathbf{v}|\setminus(\mathcal P_0\cup\{z=0\})$
and consider a blow-up $\mathbf{w}$ at $\bar p$, i.e., a limit of rescalings $(\delta_{1/r}\circ\ell_{\bar p^{-1}})_*\mathbf{v}$ along a sequence $r\to0$.
Writing $\bar p=(\bar z,\bar\varphi)$ and taking $p'=(z',\varphi')\in\operatorname{spt}|\mathbf{v}|$, we get
$$z_{2j}'-\bar z_{2j}=-(\sigma(p')z_{2j-1}'-\sigma(\bar p)\bar z_{2j-1})=-\sigma(\bar p)(z_{2j-1}'-\bar z_{2j-1})
-(\sigma(p')-\sigma(\bar p))z_{2j-1}',$$
and hence for any point $p=(z,\varphi)\in\operatorname{spt}|w|$ we have the linearized equation
$$z_{2j}=-\sigma(\bar p)z_{2j-1}-\bar z_{2j-1}[Y_1(\sigma)(\bar p)z_2+Y_2(\sigma)(\bar p)z_4]$$
(recall that $X_j(\sigma)$ vanishes at $\bar p$).
Abbreviating $\bar\rho:=\rho(\bar p)$ and $\bar\sigma:=\sigma(\bar p)$, we also have
$$Y_\ell(\sigma)(\bar p)=\frac{2}{\bar\rho^2}(\bar z_{2\ell-1}-\bar\sigma\bar z_{2\ell})=\frac{2(1+\bar\sigma^2)}{\bar\rho^2}\bar z_{2\ell-1},$$
which gives
$$
-\bar\sigma z_1=\lf(1+\frac{2(1+\bar\sigma^2)\bar z_1^2}{\bar\rho^2}\rg) z_2+\frac{2(1+\bar\sigma^2)\bar z_1\bar z_3}{\bar\rho^2} z_4,
$$
and a similar equation interchanging the indices $(1,2)$ with $(3,4)$. Since $\bar\sigma=\frac{2\bar\varphi}{\bar\rho^2}\neq0$, we obtain
\begin{equation}\label{lineariz}
z_1=-\frac{\bar\rho^2+2\bar z_1^2+2\bar z_2^2}{2\bar\varphi} z_2-\frac{\bar z_1\bar z_3+\bar z_2\bar z_4}{\bar\varphi} z_4,
\end{equation}
and a similar equation expressing $ z_3$ as a constant linear combination of $z_2,z_4$, valid on $\operatorname{spt}|\mathbf{w}|$.
We now get a contradiction by using Lemma \ref{too.many.bis} below (applied to $\mathbf{w}$).

Let us turn to the general case.
Up to shrinking and translating $U$, we can assume that $U=\mathcal{C}_a=V_a*D_a$ (see the definition \eqref{cyl.def} below).
We claim that
$$\operatorname{spt}|\mathbf{v}|=S*D_a$$
for some closed subset $S\subseteq V_a$. Note that, once this is shown, then $S$ is automatically a locally finite set,
since otherwise we could find a converging sequence $p_k\to\bar p$ in $V_a\cap\operatorname{spt}|\mathbf{v}|$,
giving a contradiction after a blow-up at $\bar p$ (using the previously established fact that blow-ups are multiples of $\mathcal{P}_0$
and the consequential convergence $\delta_{1/r}\circ\ell_{\bar p^{-1}}(\operatorname{spt}|\mathbf{v}|)\to\mathcal P_0$,
by Corollary \ref{dens.limits}).
The conclusion then follows by the previous constancy argument.

To prove the previous claim, we consider the map $f:\mathcal{C}_a\to V_a$ given by
$$f(z,\varphi):=(0,z_2,0,z_4,\varphi+z_1z_2+z_3z_4),$$
which associates with every $p\in \mathcal{C}_a$ the unique point $q\in V_a$ such that $p\in q*D_a$.
We observe that $X_j(f)=0$, as expected. Given $x\in D_a$, we let
$$A_x:=\operatorname{spt}|\mathbf{v}|\cap\pi_{\mathcal P_0}^{-1}(x)),$$
where $\pi_{\mathcal P_0}(z,\varphi):=(z_1,0,z_3,0,0)$.
Assume that
for some $q\in V_a$ and $x\in D_a$ we have
$$q*x\in A_x.$$
In order to show the claim, it suffices to show that,
given another point $x'\in D_a$, we have $q*x'\in A_{x'}$.
We now construct a Lipschitz function
$$h:[0,1]\to\operatorname{spt}|\mathbf{v}|$$
such that $h(0)=q*x$ and $\pi_{\mathcal P_0}\circ h(t)=(1-t)x+tx'$.
Once this is done, we then see that $$h'(t)\in\mathcal{P}_{h(t)}$$ for all $t\in[0,1]$ where the derivative exists, by a straightforward
blow-up analysis. Since $X_j(f)=0$, we deduce that $f\circ h$ is constant. Since $f\circ h(0)=q$, we have $f\circ h(1)=q$ as well,
and hence $h(1)\in(q*D_a)\cap A_{x'}$ (as $\pi_{\mathcal P_0}(h(1))=x'$). Since $(q*D_a)\cap\pi_{\mathcal P_0}^{-1}(x')=\{q*x'\}$,
we deduce that $h(1)=q*x'$, and thus $q*x'\in A_{x'}$, as desired.

In order to construct $h$, we make the following observation, which is again a simple consequence of the classification of blow-ups:
given $\bar p\in U$ and $\ep\in(0,1)$, there exists a radius $r_0(\bar p)\in(0,1)$ such that for $0<r<r_0(\bar p)$ we have
$$d_K(q,\operatorname{spt}|\mathbf{v}|\cap B_r^\r(\bar p))<\ep r\quad\text{for all }q\in(\bar p*\mathcal P_0)\cap B_r^\r(\bar p).$$
In particular, given any $v\in\mathcal P_0$ with $|v|<r$, we can find $p\in \operatorname{spt}|\mathbf{v}|\cap B_r^\r(\bar p)$
such that
$$\r(\bar p*v)^{-1}*p)<\ep r,$$
which easily implies that
$$|p-(\bar p*v)|<C\ep r$$
and in turn that
\begin{equation}\label{cone.claim}
|z-(\bar z+\pi(v))|<\ep r,\quad|\varphi-\bar\varphi|<Cr
\end{equation}
for a possibly different $C=C(\bar p)>0$, locally uniform in $\bar p$.
With this in hand, for fixed $\ep>0$ and $C'>0$, we consider the maximum $\tau\in[0,1]$ such that
there exist times $$t_0=0<t_1<\dots<t_k=\tau\le 1$$ and points $p_0=q*x,p_1,\dots,p_k\in\operatorname{spt}|\mathbf{v}|$ with
$$t_{i+1}-t_i\le\ep,\quad |z_{i+1}-z_i-(t_{i+1}-t_i)(x'-x)|\le\ep (t_{i+1}-t_i)|x'-x|,
\quad|\varphi_{i+1}-\varphi_i|\le C'(t_{i+1}-t_i),$$
where we identify $x,x'\in D_a$ with points in $\C^2$.
We observe that (up to removing some intermediate times) we can replace any such collection with another one in which $t_{i+2}-t_i>\ep$, so that $k\le 2\ep^{-1}+1$ is bounded and, by compactness, the maximum $\tau$ does indeed exist. By applying \eqref{cone.claim},
we see that we must have $\tau=1$, since otherwise starting with a collection with $t_k=\tau$ we could add an additional pair $(t_{k+1},p_{k+1})$ with $t_{k+1}>\tau$, provided that $C'$ is taken large enough (depending only on $x,x'\in D_a$ and $q\in V_a$).

Let us fix a collection $\Gamma_\ep:=\{(t_i,p_i)\mid i=0,\dots,k\}$ as above, with $t_k=1$.
As $\ep\to0$, we can extract a limit $\Gamma=\lim_{\ep\to0}\Gamma_\ep$ in the Hausdorff topology, up to a subsequence, and it is immediate to check that
$\Gamma$ is the graph of a Lipschitz function $h:[0,1]\to U$ such that
$$h(0)=q*x,\quad\pi\circ h(t)=\pi(q)+x+t(x'-x).$$
In particular, we also have
$$\pi_{\mathcal P_0}\circ h(t)=x+t(x'-x),$$
as desired.
\end{proof}

In the following lemmas, we tacitly assume that the varifold $\mathbf{v}$ in the statement has
density $\theta^\chi\ge\nu>0$ on its support and that \eqref{tensor.v} holds
(as already observed, both conditions are stable under limits and thus under blow-ups).

\begin{Lm}\label{const.sigma}
Assume that $\mathbf{v}$ is a blow-up such that, on $\mathbb{H}^2\setminus\{z=0\}$, we have $\sigma=c$
for some constant $c\in\R$. Then $c=0$ and $\mathbf{v}$ is a constant multiple of $\mathcal{P}_0$. \hfill $\Box$
\end{Lm}

\begin{proof}
We assume by contradiction that $c\neq0$, so that
$$\operatorname{spt}|\mathbf{v}|\subseteq\{z=0\}\cup\{2\varphi=c\rho^2\}.$$
Since $\{2\varphi\neq c\rho^2\}\cap\{z=0\}=\{z=0,\ \varphi\neq0\}=:L$,
the restriction $\mathbf{v}\res\Pi^{-1}(L)$ gives a HSLV on $\mathbb{H}^2\setminus\{0\}$. Taking a Hamiltonian of the form
$z_1z_2f(\varphi)$, with $f\in C^\infty_c(\R\setminus\{0\})$, it is immediate to conclude that this restriction vanishes, so that
$$\operatorname{spt}|\mathbf{v}|\subseteq\{2\varphi=c\rho^2\}.$$
We take any $0<\ep<1<R$ and consider a concave smooth function
$\psi:\R_+\to[0,1]$ such that $\psi(t)=t$ for $t\in[0,\ep]$, $\psi(t)=1$ for $t\in[R,\infty)$, and $\psi''(t)<0$ for $t\in(\ep,R)$.
Then the support of the Hamiltonian $\psi(\rho^2)-1$ intersects $\operatorname{spt}|\mathbf{v}|$ in a compact set.
As a consequence, we can use it in \eqref{tensor.v.bis}, obtaining
$$\int_{\mathbb{H}^2}\psi''(\rho^2)(z_1z_2+z_3z_4)\,d|\mathbf{v}|=0.$$
Recalling that $z_2=-\sigma z_1=-cz_1$ and similarly $z_4=-cz_3$, we obtain
$$z_1z_2+z_3z_4=-c(z_1^2+z_3^2)=-\frac{c}{1+c^2}\rho^2,$$
and hence
$$\int_{\mathbb{H}^2}\psi''(\rho^2)\rho^2\,d|\mathbf{v}|=0.$$
Since $\psi''<0$ on $(\ep,R)$ and $\psi''=0$ elsewhere, we deduce that $\operatorname{spt}|\mathbf{v}|\subseteq\{z=0\}$,
and thus $\operatorname{spt}|\mathbf{v}|=\{0\}$.
However, using the previous Hamiltonian $z_1z_2f(\varphi)$ with $f\in C^\infty_c(\R)$ and $f(0)=1$, we reach a contradiction.
We then have $c=0$ and hence $z_{2j}=-cz_{2j-1}=0$ on $\operatorname{spt}|\mathbf{v}|\setminus\{z=0\}$.
The conclusion follows as at the beginning of the previous proof.
\end{proof}

\begin{Lm}\label{too.many.constraints}
For a blow-up $\mathbf{v}$ there cannot exist three constants $\alpha,\beta,\gamma\in\R$
such that $$z_1=\alpha z_4,\ z_2=\beta z_4,\ z_3=\gamma z_4\quad\text{on }\operatorname{spt}|\mathbf{v}|.$$
Similarly, it cannot happen that each of $z_1,z_2,z_4$ is a constant multiple of $z_3$. \hfill $\Box$
\end{Lm}

\begin{proof}
By assumption $\{z_4=0\}\cap\operatorname{spt}|\mathbf{v}|=\{z=0\}\cap\operatorname{spt}|\mathbf{v}|$.
Moreover, as above we have $z_2=-\sigma z_1$ and $z_4=-\sigma z_3$ on $\operatorname{spt}|\mathbf{v}|\setminus\{z=0\}$, and thus
$$\beta z_4=z_2=-\sigma z_1=-\alpha \sigma z_4.$$
If $\alpha\neq0$, it follows that off the $\varphi$-axis $\{z=0\}$ we have $\sigma=-\frac{\beta}{\alpha}$.
Hence, by Lemma \ref{const.sigma}, $\mathbf{v}$ is a multiple of $\mathcal P_0$.
In particular, $z_4=0$ on its support, and hence also $z=0$, a contradiction.
\end{proof}

\begin{Lm}\label{too.many.bis}
For a blow-up $\mathbf{v}$ we cannot have
\begin{equation}\label{lin.constr} z_1=az_2+bz_4,\ z_3=cz_2+dz_4\quad\text{on }\operatorname{spt}|\mathbf{v}|\end{equation}
for constant numbers $a,b,c,d\in\R$. \hfill $\Box$
\end{Lm}

\begin{proof}
Assume moreover that $z_1=\eta z_3$ on $\operatorname{spt}|\mathbf{v}|$ for some constant $\eta$.
Then, since $z_2=-\sigma z_1$ and $z_4=-\sigma z_3$ off the $\varphi$-axis, we also have $z_2=\eta z_4$ (on the full support of $|\mathbf{v}|$),
and we reach a contradiction by the previous lemma. Similarly, $z_3$ cannot be a constant multiple of $z_1$.
Now let
$$U:=\mathbb{H}^2\setminus(\{z=0\}\cup\{\varphi=0\}).$$
We claim that we must have
\begin{equation}\label{nonlin.constr}a=-\frac{\rho^2+2(z_1^2+z_2^2)}{2\varphi},\quad d=-\frac{\rho^2+2(z_3^2+z_4^2)}{2\varphi}\end{equation}
on $U\cap\operatorname{spt}|\mathbf{v}|$. Given $p$ in this set, obviously any blow-up $\mathbf{w}$ at $p$ keeps satisfying \eqref{lin.constr}. However, if equations \eqref{nonlin.constr} fail at $p$,
then $\mathbf{w}$ also satisfies \eqref{lin.constr} for a different set of coefficients $(a',b',c',d')$ (as seen while deriving \eqref{lineariz}).
Assume for instance that $(a,b)\neq(a',b')$:
since for $(\tilde z,\tilde\varphi)\in\operatorname{spt}|\mathbf{w}|$ we have
$$a\tilde z_2+b\tilde z_4=\tilde z_1=a'\tilde z_2+b'\tilde z_4,$$
we obtain that one between $\tilde z_2$ and $\tilde z_4$ is a constant multiple of the other, and we obtain a contradiction from the previous lemma. The case where $(c,d)\neq(c',d')$ is completely analogous.

Having established \eqref{nonlin.constr}, assume that $U\cap\operatorname{spt}|\mathbf{v}|\neq\emptyset$. Observing that necessarily $a,d\neq0$, we obtain
$$\rho^2+2(z_1^2+z_2^2)=\frac{a}{d}[\rho^2+2(z_3^2+z_4^2)].$$
Dividing by $1+\sigma^2$ and recalling that $z_1^2+z_2^2=(1+\sigma^2)z_1^2$, we arrive at
$$3z_1^2+z_3^2=\frac{a}{d}[z_1^2+3z_3^2].$$
Thus, either $z_1^2$ is a constant multiple of $z_3^2$ (on $U\cap\operatorname{spt}|\mathbf{v}|$) or the reverse holds.
In the first case, we note that $z_3\neq0$ (everywhere on $U\cap\operatorname{spt}|\mathbf{v}|$), 
since otherwise we find a point where $z_1=z_3=0$ and thus $z_{2j}=-\sigma z_{2j-1}=0$, impossible since $z\neq0$ on $U$. Hence, locally $z_1$ is a constant multiple of $z_3$.
However, any blow-up at a point $p\in U\cap\operatorname{spt}|\mathbf{v}|$ would give a contradiction, by the first part of the proof. The second case is analogous.

We then conclude that $U\cap\operatorname{spt}|\mathbf{v}|=\emptyset$, which as usual implies that $\mathbf{v}$
is a multiple of $\mathcal{P}_0$, contradicting the assumptions.
\end{proof}

In order to prove Theorem \ref{cpt.int},
we consider a point $p_0\in\operatorname{spt}|\mathbf{v}_\infty|$ where a tangent plane exists, as in Remark \ref{equiv.rect}.
We claim that $\theta_0:=\theta^\chi(\mathbf{v}_\infty,p_0)\in 2\pi\N$. Without loss of generality, we can assume that $p_0=0$. Since $(\delta_{1/r})_*\mathbf{v}_\infty$ converges to a Legendrian plane $\mathcal{P}_0$ with constant multiplicity $\theta_0$, by means of a simple diagonal argument we can find suitable rescalings $(\delta_{1/r_k})_*\mathbf{v}_k$ converging to this plane. Thus, we can assume that $\mathbf{v}_\infty$
in fact coincides with the plane $\mathcal{P}_0$, with multiplicity $\theta_0$. Up to a rotation, we can also assume that $\mathcal{P}_0=\operatorname{span}\{\p_{z_1},\p_{z_3}\}$.
By Corollary \ref{dens.usc}, we have
$$\theta^\chi(\mathbf{v}_k,\cdot)\ge 2\pi\quad\text{on }\operatorname{spt}|\mathbf{v}_k|,$$
and by Corollary \ref{dens.limits} we deduce that the same holds for $\mathbf{v}_\infty$, as well as
$$\operatorname{spt}|\mathbf{v}_k|\to\operatorname{spt}|\mathbf{v}_\infty|=\mathcal{P}_0$$
in the local Hausdorff topology on $\mathbb{H}^2$ (after our rescaling operation, the varifolds $\mathbf{v}_k$ are defined and HSLV on open sets $U_k$ increasing to $\mathbb{H}^2$).

We have to show that $\theta_0\in\N$.
Before starting the actual proof, we need to introduce some notation.
In the sequel, we will use the map
$$\pi_{\mathcal P_0}:\mathbb H^2\to\mathcal P_0,\quad \pi_{\mathcal P_0}(z,\varphi):=(z_1,0,z_3,0,0)$$
and, given a radius $a>0$, we will consider the sets
$$D_a:=B_a^\r(0)\cap\mathcal{P}_0,\quad V_a:=B_a^\r(0)\cap V,\quad V:=\{z_1=z_3=0\},$$
as well as the cylinders
\begin{equation}\label{cyl.def}\mathcal{C}_{a,b}:=V_b*D_a=\{p''*p'\mid p'\in D_a,\ p''\in V_b\},\quad \mathcal{C}_a:=\mathcal{C}_{a,a}.
\end{equation}
Given $p\in\mathbb{H}^2$, we will also denote
$$\mathcal{C}_{a,b}(p):=p*\mathcal{C}_{a,b},\quad\mathcal{C}_a(p):=p*\mathcal{C}_{a}.$$
Note that $\pi_{\mathcal P_0}(p)$ is the unique point $p'\in\mathcal P_0$ such that $p\in V*p'$.

Given $p\in\mathbb{H}^2$, we let $\mathcal{P}_p:=\operatorname{span}\{X_1(p),X_2(p)\}$
and, given a cylinder $\mathcal{C}_a(q)$, we define the excess-like quantity
$$E_k(q,a):=a^{-2}\int_{\mathcal{C}_a(q)}\|\mathcal{P}-\mathcal{P}_{p}\|^2\,d\mathbf{v}_k(\mathcal P,p),$$
where we identify $\mathcal{P}$ and $\mathcal{P}_{p}$ with the orthogonal (in $T_p\mathbb{H}^2$) projection matrices and use the Hilbert--Schmidt norm of their difference.

\begin{Lm}\label{exc.max}
There exist sets $S_k\subseteq \mathcal{C}_{100}$ such that $|\mathbf{v}_k|(S_k)\le C\eta_k$ and
$$E_k(q,a)<\eta_k\quad\text{for all }\mathcal{C}_a(q)\subseteq\mathcal{C}_{100}\text{ with }q\nin S_k,$$
for a sequence $\eta_k\to0$. \hfill $\Box$
\end{Lm}

\begin{proof}
We take any vanishing sequence $\eta_k>0$ such that
$$\int_{\mathcal{C}_{100}}\|\mathcal P-\mathcal{P}_p\|^2\,d\mathbf{v}_k(\mathcal P,p)\le\eta_k^2.$$
This can be done since the left-hand side converges to zero, thanks to the varifold convergence of $\mathbf{v}_k$
to a multiple of $\mathcal{P}_0$.
Calling $S_k$ the set of points where the statement fails, by Vitali's covering lemma we can cover it with balls
$B_{10a_j}^\r(q_j)$ (depending also on $k$) such that $q_j\in S_k$ and the smaller balls $B_{2a_j}^\r(q_j)$ are disjoint, with
$$\int_{\mathcal{C}_{a_j}(q_j)}\|\mathcal P-\mathcal{P}_p\|^2\,d\mathbf{v}_k(\mathcal P,p)\ge\eta_k a_j^2.$$
In particular, the cylinders $\mathcal{C}_{a_j}(q_j)$ are disjoint. Summing over $j$, we obtain
$$\sum_j a_j^2\le\eta_k.$$
The claim now follows from Corollary \ref{mono.cor2}, which gives the bound $|\mathbf{v}_k|(B_{10a_j}^\r(q_j))\le Ca_j^2$
uniformly in $k$, so that $|\mathbf{v}_k|(S_k)\le\sum_j|\mathbf{v}_k|(B_{10a_j}^\r(q_j))\le C\eta_k$.
\end{proof}

We let $\Theta_k$ denote the quantity $\Theta$ defined in \eqref{capital.theta.def} for the varifold $\mathbf{v}_k$.
Recall that, by Theorem \ref{mono.cor}, we have
$$\Theta_k(q,a)\le\Theta_k(q,b)\quad\text{for }0<a\le b<100$$
and
$$\theta^\chi(\mathbf{v}_k,q)=\lim_{a\to0}\Theta_k(q,a).$$

\begin{proof}[Proof of Proposition \ref{cpt.int}]
The varifold convergence $\mathbf{v}_k\rightharpoonup\theta_0\cdot\mathcal{P}_0$ gives
$$(\pi_{\mathcal P_0})_*(\mathbf{v}_k\res\Pi^{-1}(\mathcal{C}_{100}))\rightharpoonup \theta_0\mathcal{H}^2\res D_{100}.$$
Moreover, as we saw in Lemma \ref{equiv.rect}, we have
$$d|\mathbf{v}_k|=\frac{\theta^\chi(\mathbf{v}_k,\cdot)}{2\pi}\,d(\mathcal{H}^2\res\operatorname{spt}|\mathbf{v}_k|)$$
and $\theta^\chi(\mathbf{v}_k,q)\in2\pi\N^*$ for $\mathcal{H}^2$-a.e. $q\in\operatorname{spt}|\mathbf{v}_k|$
(recall that, by Theorem \ref{mono.cor} and standard properties of Hausdorff measures, a $|\mathbf{v}_k|$-negligible subset of $\operatorname{spt}|\mathbf{v}_k|$ is also $\mathcal{H}^2_K$-negligible, and thus
$\mathcal{H}^2$-negligible).

Letting $E_k:=(\mathcal{C}_{100}\cap\operatorname{spt}|\mathbf{v}_k|)\setminus S_k$, by the area formula we then have
$$\int_{D_{a}}\sum_{q\in E_k\cap\pi_{\mathcal{P}_0}^{-1}(x)}
\frac{\theta^\chi(\mathbf{v}_k,q)}{2\pi}\,d\mathcal{H}^2(x)\to\theta_0\mathcal{H}^2(D_{a})$$
for any $a>0$. Thus, by a diagonal argument, we can select a sequence of points $x_k\to0$ such that
$$\sum_{q\in E_k\cap\pi_{\mathcal{P}_0}^{-1}(x_k)}\theta^\chi(\mathbf{v}_k,q)\ge 2\pi\theta_0-\eta_k,$$
up to modifying the sequence $\eta_k\to0$, and such that each term in the sum belongs to $2\pi\N^*$.
Assuming by contradiction that $\theta_0\nin\N$, we can then find a finite subset
$$F_k\subseteq E_k\cap\pi_{\mathcal{P}_0}^{-1}(x_k),$$
consisting of at most $[\theta_0]+1$ points, such that
$$\liminf_{k\to\infty}\sum_{q\in F_k}\theta^\chi(\mathbf{v}_k,q)\ge 2\pi([\theta_0]+1).$$
Up to translating each $\mathbf{v}_k$, we can assume that $x_k=0$.
We now let $$\zeta_a(p):=-\frac{\chi'(|\pi_{\mathcal P_0}(p)|/a)}{|\pi_{\mathcal P_0}(p)|/a}$$ and we claim that
\begin{equation}\label{int.pre.claim}
\int_{\mathcal{C}_{4a,4}}\zeta_a\,d|\mathbf{v}_k|\ge a^2\sum_{q\in F_k}\theta^\chi(\mathbf{v}_k,q)-\ep_ka^2\quad\text{for all }a\in(0,1],
\end{equation}
for another sequence $\ep_k\to0$. In particular, for $a=1$, in the limit this gives
$$2\pi\theta_0=\lim_{k\to\infty}\int_{\mathcal{C}_{4}}\zeta_1\,d|\mathbf{v}_k|\ge\limsup_{k\to\infty}\sum_{q\in F_k}\theta^\chi(\mathbf{v}_k,q),$$
which gives the desired contradiction.

To prove this key claim, we fix $\ep>0$ and let $\lambda>0$ be given by Lemma \ref{theta.tilde.theta} below.
In the sequel, we drop the subscript $k$ to simplify notation, even if the next constructions depend on $k$.
We call $a\in(0,1]$ a \emph{good} radius if $F=F_k$ can be partitioned as
$$F=F_{(1)}\sqcup\dots\sqcup F_{(\ell)}$$
with $\operatorname{diam}_K(F_{(j)})\le a$, $d_K(F_{(j)},F_{(j')})\ge 16a$ for $j\neq j'$,
and \eqref{almost.conical.int} for all $q\in F$.
Clearly, by Theorem \ref{mono.cor}, the set of good radii includes a collection of intervals
$$(0,s_0]\cup[r_1,s_1]\cup\dots\cup[r_m,s_m]\subset(0,1],$$
with $s_i<r_{i+1}$ and $m\le|F|$, as well as $r_{i+1}\le C(|F|,\lambda)s_i$ and $s_m\ge c(|F|,\lambda)>0$.
We can also require that the previous partition is constant on each interval $I_i$ and that
the partition for $I_i$ is a refinement of the partition for $I_{i+1}$ (roughly speaking, at a larger scale, some clusters might merge into a single one). For each $I_i$ we choose a collection $R_i\subset F$ of representative points, one in each set $F_{(1)},\dots,F_{(\ell)}$.

The partition for $I_0$ is just given by the singletons $\{q\}$ for each $q\in F$.
By Lemma \ref{theta.tilde.theta} below, for $a\in(0,s_0]$ we have
$$a^{-2}\int_{\mathcal{C}_{4a}(q)}\zeta_a\,d|\mathbf{v}|\ge\Theta(q,a)-\ep\ge\theta^\chi(q)-\ep.$$
Assume now that we have a bound of the form
\begin{equation}\label{int.claim}
\sum_{q\in R_i} s_i^{-2}\int_{\mathcal{C}_{4s_i}(q)}\zeta_{s_i}\,d|\mathbf{v}|\ge\sum_{q\in F}\theta^\chi(q)-C\ep
\end{equation}
for some $i=0,\dots,m-1$ (we just proved this for $i=0$), and let us show that a similar bound holds for $i+1$, with a larger $C$.
Indeed, given a set $F''\subseteq F$ in the partition for the interval of good radii $I_{i+1}$, we can write
$$F''=F_{(1)}'\sqcup\dots\sqcup F_{(n)}'$$
for suitable sets $F_{(1)}',\dots,F_{(n)}'$ in the partition for $I_i$. By construction, taking representatives
$$\{q_1'\}=F_{(1)}'\cap R_i,\quad\dots,\quad\{q_n'\}=F_{(n)}'\cap R_i,\quad\{q''\}=F''\cap R_{i+1},$$
the cylinders $\mathcal{C}_{4s_i}(q')$ are disjoint as $q'$ varies in $\{q_1',\dots,q_n'\}$. Moreover, we have
$$\bigcup_{q'}\mathcal{C}_{4s_i}(q')\subseteq\mathcal{C}_{4s_i,4s_i+r_{i+1}}(q''),$$
since $\operatorname{diam}_K(\{q_1',\dots,q_n',q''\})\le\operatorname{diam}(F'')\le r_{i+1}$.
We apply Lemma \ref{expansion} below to deduce
$$r_{i+1}^{-2}\int_{\mathcal{C}_{4r_{i+1},4r_{i+1}}(q'')}\zeta_{r_{i+1}}\,d|\mathbf{v}|\ge s_i^{-2}\int_{\mathcal{C}_{4s_i,4s_i+ r_{i+1}}(q'')}\zeta_{s_i}\,d|\mathbf{v}|-\ep$$
(clearly, we can assume that $4s_i\le r_{i+1}$).
Moreover, by Lemma \ref{theta.tilde.theta} again, we obtain
$$\Theta(q'',r_{i+1})\ge r_{i+1}^{-2}\int_{\mathcal{C}_{4r_{i+1},4 r_{i+1}}(q'')}\zeta_{r_{i+1}}\,d|\mathbf{v}|-\ep.$$
By monotonicity of $\Theta$, we also have
$$\Theta(q'',s_{i+1})\ge\Theta(q'',r_{i+1}),$$
and as before it holds that
$$s_{i+1}^{-2}\int_{\mathcal{C}_{4s_{i+1}}(q'')}\zeta_{s_{i+1}}\,d|\mathbf{v}|\ge\Theta(q'',s_{i+1})-\ep.$$
Combining these inequalities and summing over $q''\in R_{i+1}$, we get
$$\sum_{q''\in R_{i+1}} s_{i+1}^{-2}\int_{\mathcal{C}_{4s_{i+1}}(q'')}\zeta_{s_{i+1}}\,d|\mathbf{v}|\ge\sum_{q\in F}\theta^\chi(q)-C\ep,$$
obtaining \eqref{int.claim} for $i+1$ in place of $i$.
Finally, it is clear that an analogous argument proves our initial claim \eqref{int.pre.claim}.
\end{proof}

\begin{Lm}\label{expansion}
Given $\Lambda>1$ and $\ep>0$, for $k$ large enough we have
$$\frac{1}{b^2}\int_{\mathcal{C}_{4b,4b}(q)}\zeta_b\,d|\mathbf{v}_k|
>\frac{1}{a^2}\int_{\mathcal{C}_{4a,2b}(q)}\zeta_a\,d|\mathbf{v}_k|-\ep,$$
for any $q\in F_k$ and any two radii $0<a\le b\le1$ such that $b\le \Lambda a$.
\hfill $\Box$
\end{Lm}

\begin{proof}
Assume by contradiction that the claim fails for some $q_k\in F_k$ and two radii $a_k,b_k$, along a subsequence.
After a left translation by $q_k^{-1}$ and a dilation by a factor $1/b_k$, we obtain new HSLVs
$\mathbf{v}_k'$ on $\mathcal{C}_8(0)$ such that
$$\int_{\mathcal{C}_{8}(0)}\|\mathcal{P}-\mathcal{P}_{p}\|^2\,d\mathbf{v}_k'(\mathcal P,p)\le C\eta_k\to0,$$
thanks to Lemma \ref{exc.max}. Up to a subsequence, we then get a limit $\mathbf{v}=\lim_{k\to\infty}\mathbf{v}_k'$ (on $\mathcal{C}_8(0)$) satisfying
the assumptions of Lemma \ref{zero.exc}; this varifold is then a finite union of disks of the form $q*D_8$, with $q\in V_8$, each with constant multiplicity.

Letting $\alpha_k:=\frac{a_k}{b_k}\in[\Lambda^{-1},1]$, after rescaling we obtain
$$\int_{\mathcal{C}_{4,4}}\zeta_{1}\,d|\mathbf{v}_k'|\le\frac{1}{\alpha_k^2}\int_{\mathcal{C}_{4\alpha_k,2}}\zeta_{\alpha_k}\,d|\mathbf{v}_k'|-\ep.$$
Calling $\alpha:=\lim_{k\to\infty}\alpha_k$ (up to a subsequence), we deduce that
$$\int_{\mathcal{C}_{4,4}}\zeta_{1}\,d|\mathbf{v}|\le\frac{1}{\alpha^2}\int_{\bar{\mathcal{C}}_{4\alpha,2}}\zeta_{\alpha}\,d|\mathbf{v}|-\ep.$$
However, by the structure of $\mathbf{v}$ and the definition of $\zeta_a$, we clearly have
$$\frac{1}{\alpha^2}\int_{\bar{\mathcal{C}}_{4\alpha,2}}\zeta_{\alpha}\,d|\mathbf{v}|
=\frac{1}{\alpha^2}\int_{\bar{V}_{2}* D_{4\alpha}}\zeta_{\alpha}\,d|\mathbf{v}|
=\int_{\bar{V}_{2}* D_{4}}\zeta_{1}\,d|\mathbf{v}|,$$
yielding a contradiction since $\bar{V}_{2}* D_{4}\subseteq V_{4}*D_4=\mathcal{C}_{4,4}$.
\end{proof}

\begin{Lm}\label{theta.tilde.theta}
Given $\ep>0$, there exists $\lambda\in(0,1)$ such that the following holds for $k$ large enough:
for any $q\in F_k$ and any radius $0<a\le\lambda$, if
\begin{equation}\label{almost.conical.int}\int_{\lambda a<\r_q<a/\lambda}|\nabla^{\mathcal P}\arctan\sigma_q|^2\,d\mathbf{v}_k(\mathcal P,p)\le\lambda\end{equation}
then
$$\lf|\Theta_k(q,a)-\frac{1}{a^2}\int_{\mathcal{C}_{4a}(q)}\zeta_a\,d|\mathbf{v}_k|\rg|<\ep$$
holds true. \hfill $\Box$
\end{Lm}

\begin{proof}
Let us fix $\ep>0$ and, by contradiction, using a diagonal argument, assume that the claim fails along a subsequence with centers $q_k\in F_k$ and radii $0<a_k\le\lambda_k\to0$.
After a left translation by $q_k^{-1}$ and a dilation by a factor $a_k^{-1}$,
we obtain varifolds $\mathbf{v}_k'$ which, up to a subsequence,
converge to a varifold $\mathbf{v}$ on $\mathbb{H}^2$ satisfying
the assumptions of Lemma \ref{zero.exc} and such that $\nabla^{\mathcal P}\arctan\sigma=0$ on $\operatorname{spt}(\mathbf{v})\cap\{\r>0\}$.
Thus, as shown by Lemma \ref{zero.exc} (and its proof), $\mathbf{v}$ is a constant multiple of $\mathcal P_0$, giving
$$\Theta(\mathbf{v},0,1)=\int_{\mathcal{C}_{4}(0)}\zeta_1\,d|\mathbf{v}|.$$
Thus, for $k$ large enough the statement was true, a contradiction.
\end{proof}

\section{A point removability result for PHSLVs}
In this section we show that if we have a PHSLV on $\mathbb{H}^2$,
defined on a punctured Riemann surface $\Sigma\setminus S$ for a locally finite set $S$, then it extends to a PHSLV
defined on $\Sigma$, provided that some technical assumptions are satisfied.
Among them, we assume a slightly stronger notion of stationarity, as follows.

\begin{Dfi}\label{strong.phslv}
We say that $(\Sigma,u,N)$ is a \emph{PHSLV$^*$} if,
for a.e. $\omega\subset\subset\Sigma\setminus S$, we can test
stationarity with all Hamiltonian vector fields $W_F$ associated with an
$F\in C^\infty_c(\mathbb{H}^2)$ which is \emph{locally constant} near $u(\p\omega)$.
\hfill $\Box$
\end{Dfi}

\begin{Rm}
Note that $(W_F)^H$, which appears in $\operatorname{div}_{\mathcal P}W_F=\operatorname{div}_{\mathcal P}(W_F)^H$ in the definition of stationarity, is still compactly supported in $\mathbb{H}^2\setminus u(\p\omega)$.
Also, since $(W_F)^H$ does not change if we add a constant to $F$, we can equivalently consider
all functions $F\in C^\infty(\mathbb{H}^2)$ locally constant near $u(\p\omega)$ and constant near infinity (i.e., outside of a compact set).
 \hfill $\Box$
\end{Rm}

\begin{Rm}
This stronger assumption is quite natural, for the following reason. Denoting by $L$ the positive $\varphi$-axis and taking $\mathcal{Q}_X(p):=\operatorname{span}\{X_1(p),X_2(p)\}$
and $\mathcal{Q}_Y(p):=\operatorname{span}\{Y_1(p),Y_2(p)\}$, it can be checked that the varifold
$$\mathbf{v}(\mathcal{P},p):=\frac{\delta_{\mathcal{Q}_X(p)}(\mathcal P)+\delta_{\mathcal{Q}_Y(p)}(\mathcal P)}{2}\otimes(\mathcal{H}^1\res L)(p)$$
is a HSLV on $\mathbb{H}^2\setminus\{0\}$. However, it is not a HSLV on $\mathbb{H}^2$, and indeed it does not satisfy the stronger stationarity
condition obtained by taking any $F\in C^\infty_c(\mathbb{H}^2)$ constant near $0$. \hfill $\Box$
\end{Rm}

We start with a simple observation, exploiting some tools from the next section.

\begin{Prop}
Let $u\in W^{1,2}_{loc}(\Sigma)$ and $N\in L^\infty_{loc}(\Sigma,\N^*)$.
Assume that $(\Sigma\setminus S,u,N)$ is a PHSLV on $\mathbb{H}^2$. Then $u$ has a continuous representative on $\Sigma$.
\hfill $\Box$
\end{Prop}

\begin{proof}
First of all, continuity holds away from $x_0$, by Proposition \ref{pr-cont} below.
We now assume without loss of generality that $\Sigma$ is an open set in $\C$ and $S=\{x_0\}$.
As in the proof of Proposition \ref{pr-cont}, we can find a decreasing sequence of radii $r_k\to0$ such that
$$\operatorname{diam}_K u(\p B_{r_k}(x_0))\to0.$$
Up to a subsequence, we can assume that the sets $u(\p B_{r_k}(x_0))$ either converge to a point $p$
or go off to infinity. In the first case, we must have $u(x)\to p$ as $x\to x_0$: if not, up to a further subsequence,
we could find points $x_k\in A_k:=B_{r_k}(x_0)\setminus B_{r_{k+1}}(x_0)$ such that
$\liminf_{k\to\infty}d_K(u(x_k),p)>0$. However, this contradicts Proposition \ref{co-dens} below, together with Theorem \ref{mono.cor},
which would imply that the induced varifold $\mathbf{v}_{A_k}$ has a lower bound on the mass, while clearly
$$\limsup_{k\to\infty}\int_{A_k}N|\nabla u|^2\,dx^2=0.$$
Thus, in this case we are done.
In the second case, an analogous argument gives $\r\circ u(x)\to\infty$ as $x\to x_0$.
Thus, for any $F\in C^\infty_c(\mathbb{H}^2)$, the composition $F\circ u$ vanishes in a neighborhood of $x_0$.
It is then clear that $(\Sigma,u,N)$ satisfies the definition of PHSLV, since if $F$ vanishes near $u(\p\omega)$ then it also vanishes
near $u(\p(\omega\setminus\bar B_r(x)))$ for $r>0$ small enough. However, the discontinuity of $u$ at $x_0$ contradicts Proposition \ref{pr-cont}.
\end{proof}

\begin{Prop}
\label{pr-VI.1}
Let $(\Sigma,u,N)$ be as in the previous statement. Assume that
$(\Sigma\setminus S,u,N)$ is a PHSLV$^*$ on $\mathbb{H}^2$,
with
\begin{equation}\label{sup.s}\liminf_{\ep\to0}\ep^{-2} \int_{\omega\cap\{\r_{u(x_0)}\circ u<\ep\}}N |\nabla u|^2\,dx^2<\infty\quad\text{for all }x_0\in S,\end{equation}
for any neighborhood $x_0\in\omega\subset\subset\Sigma$.
Then 
$({\Sigma},u,N)$ is a PHSLV on $\mathbb{H}^2$. \hfill $\Box$
\end{Prop}

\begin{Rm}
Recall that $\r_{u(x_0)}(q)=\r(u(x_0)^{-1}*q)=d_K(u(x_0),q)$
and that $u$ is continuous, so that $u(x_0)$ is defined.
Note that the last assumption is simply requiring that the mass of $\mathbf{v}_\omega$ in the ball of center $u(x_0)$ and radius $\ep$,
with respect to the distance $d_K$, is bounded by $O(\ep^2)$ for a sequence $\ep\to0$.
In a closed ambient, this assumption holds automatically for \emph{bubbles} defined on $\C=\hat \C\setminus\{\infty\}=S^2\setminus\{x_0\}$, in the context of the bubbling phenomenon, as a consequence of monotonicity. \hfill $\Box$
\end{Rm}

\begin{proof}
We fix a conformal metric $h$ on $\Sigma$.
Since $u$ is continuous, given $\om\subset\subset\Sigma$ it is easy to see that the stronger definition of stationarity is satisfied for the domain
$\om_r:=\om\setminus\bigcup_{x\in S\cap\omega}\bar B_r(x)$, for any $r>0$ small enough.
We define the finite set
$$S':=u(S\cap\omega)\setminus u(\p\omega)$$
and, for each $q\in S'$, we consider a smooth cut-off function $\chi_q:\mathbb{H}^2\to\R_+$ equal to $1$ near $q$
and supported in a bounded open set $U_q\subset\subset\mathbb{H}^2\setminus u(\p\omega)$,
such that $U_q\cap U_{q'}=\emptyset$ for $q,q'\in S'$ distinct.

Given any $F\in C^\infty_c(\mathbb{H}^2\setminus u(\p\omega))$, we need to show that the induced varifold $\mathbf{v}_\om$ satisfies
$$\int_G \operatorname{div}_{\mathcal P}W_F\,d\mathbf{v}_\omega(\mathcal P,p)=0.$$
We let
$$F_q:=\chi_q F,\quad \tilde F:=F-\sum_{q\in S'}F_q.$$
Since $\tilde F$ vanishes near $S'\cup u(\p\om)$, for $r$ small we have
$$\int_G \operatorname{div}_{\mathcal P}W_{\tilde F}\,d\mathbf{v}_\om(\mathcal P,p)=\int_G \operatorname{div}_{\mathcal P}W_{\tilde F}\,d\mathbf{v}_{\om_r}(\mathcal P,p)=0.$$
Finally, we can use Proposition \ref{remov.gen} below, applied to the varifold $\mathbf{v}:=\mathbf{v}_\om$ on $U_q$:
indeed, given another function $\hat F\in C^\infty_c(U_q)$ constant near $q$, we see that
$\hat F$ is locally constant on $u(\p\omega_r)$
for all $r>0$ small (as $\hat F$ vanishes near $u(\p\omega)\cup (S'\setminus\{q\})$), so that
$$\int_G \operatorname{div}_{\mathcal P}W_{\hat F}\,d\mathbf{v}_\om(\mathcal P,p)
=\int_G \operatorname{div}_{\mathcal P}W_{\hat F}\,d\mathbf{v}_{\om_r}(\mathcal P,p)=0.$$
Thus, by Proposition \ref{remov.gen}, we can conclude that
$$\int_G \operatorname{div}_{\mathcal P}W_{F_q}\,d\mathbf{v}_\om(\mathcal P,p)=0$$
holds as well.
\end{proof}


We used the following singularity removability for general varifolds, which
will also be useful to rule out energy dissipation in neck regions.

\begin{Prop}\label{remov.gen}
Assume that $\mathbf{v}$ is a varifold on an open set $U\subseteq{\mathbb H}^2$, restricting to a HSLV on $U\setminus\{q\}$ for some $q\in U$.
Assume also that we can test its stationarity with any Hamiltonian vector field $W_F$ generated by a function
$F\in C^\infty_c(U)$ constant near $q$, as well as
$$\liminf_{\ep\to0}\ep^{-2}|\mathbf{v}|(B_\ep^\r(q))<\infty.$$
Then $\mathbf{v}$ is a HSLV on $U$. \hfill $\Box$
\end{Prop}

\begin{proof}
We assume without loss of generality that $q=0$. Given $F\in C^\infty_c(U)$, we have to show that
\begin{equation}\label{remov.claim.bis}
\int_G\operatorname{div}_{\mathcal P}W_F\,d\mathbf{v}(\mathcal P,p)=0.
\end{equation}
Let $\chi:\R_+\to\R$ be a smooth decreasing function with $\chi=1$ on $[0,1/2]$ and $\chi=0$ on $[3/4,1]$, and denote $\chi_\ep(t):=\chi(t/\ep)$.
Letting
$$\tilde F:=(1-\chi_\ep\circ\r)F+(\chi_\ep\circ\r) F(0),$$
we decompose $W_F=W_{\tilde F}+W'$, where
$$W':=W_{(\chi_\ep\circ\r)(F-F(0))}.$$
Moreover, since $\tilde F$ is as in the statement, \eqref{remov.claim.bis} holds with $W_{\tilde F}$ in place of $W_F$.
Hence, to prove the claim it suffices to show that, along a suitable sequence $\ep\to0$, we have
$$\int_G\operatorname{div}_{\mathcal P}W'\,d\mathbf{v}(\mathcal P,p)=O(\ep).$$

Given a Legendrian plane $\mathcal{P}$ spanned by an orthonormal basis $(Z_1,Z_2)$, recall that
$$\operatorname{div}_{\mathcal P}W'=\sum_{j,\ell=1}^2 [(\nabla_{X_\ell}W'\cdot X_\ell) (X_\ell\cdot Z_j)+(\nabla_{Y_\ell}W'\cdot Y_\ell)(Y_\ell\cdot Z_j)].$$
Thus, thanks to the assumption $|\mathbf{v}|(B_\ep^\r(0))=O(\ep^2)$ along a sequence $\ep\to0$,
it suffices to show that
$$|\nabla_{X_\ell}W'\cdot X_\ell|+|\nabla_{Y_\ell}W'\cdot Y_\ell|\le C\ep^{-1}$$
for $\ep$ small. We check this only for the first term, since for the second one the computation is analogous.
Since
$$2(W')^H=\ep^{-1}\chi'(\r/\ep)(F-F(0))\sum_{\ell=1}^2[X_\ell(\r)Y_\ell-Y_\ell(\r)X_\ell]
+\chi(\r/\ep)\sum_{\ell=1}^2[X_\ell(F)Y_\ell-Y_\ell(F)X_\ell],$$
we can compute
\begin{align*}
-2\nabla_{X_\ell} W'\cdot X_\ell
&=\ep^{-2}\chi''(\r/\ep)(F-F(0))X_\ell(\r)Y_\ell(\r)\\
&\quad+\ep^{-1}\chi'(\r/\ep)(F-F(0))X_\ell(Y_\ell(\r))\\
&\quad+\ep^{-1}\chi'(\r/\ep)[X_\ell(F)Y_\ell(\r)+Y_\ell(F)X_\ell(\r)]\\
&\quad+\chi(\r/\ep)X_\ell(Y_\ell(F)).
\end{align*}
By homogeneity of $\r$, we have $X_\ell(\r)\le C$ (in fact, we have the more precise bound \eqref{nabla.r.norm})
and $|X_\ell(Y_\ell(\r))|\le C\r^{-1}\le C\ep^{-1}$ on the support of $\chi'(\r/\ep)$.
Since $|F-F(0)|\le C\ep$ here (by smoothness of $F$), we see that all the terms in the expansion are bounded by $C\ep^{-1}$,
as desired.
\end{proof}

\section{Basic properties of PHSLVs}


\subsection{A universal lower bound for the density} 

Let $(\Sigma,u, N)$ be a PHSLV and fix a decreasing cut-off function $\chi:\R_+\to\R$ with $\chi=1$ on $[0,1]$ and $\chi=0$ on $[2,\infty)$.
We let
$${\mathcal G}_u:=\{x\in\Sigma\,:\,x\text{ is a Lebesgue point for $u$ and $\nabla u$}\}$$
and 
$${\mathcal G}^f_u:=\{x\in{\mathcal G}_u\,:\,|\nabla u|(x)\ne 0\}.$$
Note that, for $x\in\mathcal{G}_u^f$, the differential $\nabla u(x)$ is an injective, linear conformal map, with values in a Legendrian two-plane
$\mathcal{P}\subset H_{u(x)}$.
For a.e. $\omega\subset\subset\Sigma$, we consider the induced varifold $\mathbf{v}_\omega$.
We now establish the following lemma for the density of $\mathbf{v}_\omega$.
\begin{Prop}\label{dens.lb}
\label{lm-IV.3} Given $p\in u(\omega\cap{\mathcal G}_u^f)\setminus u(\p\omega)$, we have
\begin{equation}
\label{IV.3}
\theta^\chi(p)=\lim_{\ep\rightarrow 0}-\frac{1}{\ep}\int_{\omega} \chi'\lf(\frac{\r_p}{\ep}\rg) \lf[N\frac{|\nabla\r_p|^2}{\r_p}+N\frac{\varphi_p}{\r_p^3}\arctan\sigma_p |\nabla u|^2\rg]\, dx^2\ge 2\pi,
\end{equation}
where we write $\r_p$ in place of $\r_p\circ u$ (and similarly for $\varphi_p$ and $\arctan\sigma_p$), as well as
\begin{equation}
\label{IV.3-a}
\liminf_{\ep\rightarrow 0}\frac{1}{\ep}\int_{\omega\cap\{\ep\le\r_p<2\ep\}} N\frac{|\nabla\r_p|^2}{\r_p}\, dx^2\ge 2\pi,
\end{equation}
for any cut-off function $\chi$ as above. \hfill $\Box$
\end{Prop}
\begin{proof}
The formula for $\theta^\chi(p)$ follows directly from the definition of induced varifold.
In the sequel, up to a left translation, we assume that $p=0$.
We can find a local conformal chart where $0\in{\mathcal G}^f_u$ and
$u(0)=0$.
Modulo a rotation (see Remark \ref{isometry}), we can further assume that
$$\p_{x_1}u(0)=e^{\la_0} X_1(0)=(e^{\la_0},0,0,0,0),\quad\p_{x_2}u(0)=e^{\la_0} X_2(0)=(0,0,e^{\la_0},0,0).$$
We will show only \eqref{IV.3-a}. The same proof will show that
$$\lim_{\ep\rightarrow 0}-\frac{1}{\ep}\int_{\omega} \chi'\lf(\frac{\r_p}{\ep}\rg) N\frac{|\nabla\r_p|^2}{\r_p}\,dx^2\ge2\pi,$$
proving also \eqref{IV.3}.

Since $0$ is a Lebesgue point for $\nabla u$, we have
$$\lim_{r\rightarrow 0} \dashint_{B_r(0)}\sum_{j=1}^2 |\p_{x_j} u(x)-e^{\la_0} X_j(0)|^2\, dx^2=0.$$
This implies in particular
$$\lim_{r\rightarrow 0} \dashint_{B_r(0)}[|\nabla u_1- e^{\la_0} \p_{x_1}|^2+|\nabla u_2|^2+|\nabla u_3- e^{\la_0} \p_{x_2}|^2+|\nabla u_4|^2]\, dx^2=0.$$
Thanks to \cite[Theorem 6.1]{EG}, the map $u$ is approximately differentiable at $0$ in the sense that
$$\lim_{r\rightarrow 0} r^{-2} \dashint_{B_r(0)}|u(x)-e^{\la_0} X_1(0) x_1-e^{\la_0} X_2(0) x_2|^2\, dx^2=0.$$
This implies in particular
\begin{equation}
\label{IV.6-0a}
\lim_{r\rightarrow 0} r^{-2} \dashint_{B_r(0)}[|u_1(x)-e^{\la_0}x_1|^2+ |u_2(x)|^2+|u_3(x)-e^{\la_0}x_2|^2+|u_4(x)|^2]=0.
\end{equation}
Recalling that $\nabla(\varphi\circ u)= u_1\nabla u_2-u_2\nabla u_1+u_3\nabla u_4-u_4\nabla u_3$ and using the previous bounds, together with Cauchy--Schwarz,
we get
$$\dashint_{B_r(0)}|\nabla(\varphi\circ u)|\, dx^2=o(r).$$
We claim that, in fact,
\begin{equation}
\label{varphi.r.4}
\dashint_{B_r(0)}|\varphi\circ u|^2\, dx^2=o(r^4).
\end{equation}
Indeed, the
Sobolev--Poincar\'e inequality gives
$$\sqrt{\dashint_{B_r(0)} \lf|\varphi\circ u-\dashint_{B_r(0)} (\varphi\circ u)\, dx^2\rg|^2\, dx^2}\le Cr^{-1}\int_{B_r(0)}|\nabla(\varphi\circ u)|\, dx^2=o(r^2),$$
thanks to the previous integral bound on $\nabla(\varphi\circ u)$. It remains to bound the average of $\varphi\circ u$ on $B_r(0)$.
Given $f\in W^{1,1}(\mathbb{D})$, we have
$$\frac{d}{dr}\dashint_{B_r(0)} f\,dx^2=\frac{d}{dr}\dashint_{B_1(0)}f(ry)\, dy^2=\dashint_{B_1(0)}\nabla f(ry)\cdot y\,dy^2=\frac{1}{r}\dashint_{B_r(0)}\nabla f(x)\cdot x\, dx^2.$$
Applying this to $f:=\varphi\circ u$, we obtain
$$\dashint_{B_r(0)} (\varphi\circ u)\, dx^2-\dashint_{B_s(0)} (\varphi\circ u)\, dx^2=\int_{s}^r \frac{dt}{t}\dashint_{B_t(0)}\nabla(\varphi\circ u)(x)\cdot x\, dx^2=o(r^2)$$
for all $0<s<r$. Letting $s\to0$ we deduce that
$$\dashint_{B_r(0)} (\varphi\circ u)\, dx^2=o(r^2),$$
and hence the claimed bound \eqref{varphi.r.4}.

This implies that
$$\dashint_{B_r(0)}\lf|\r^4\circ u(x)-\sum_{j=1}^4u_j^4\rg|^{1/2}\, dx^2=2 \dashint_{B_r(0)}|\varphi\circ u|\, dx^2=o(r^2),$$
and hence
\begin{align*}
&\dashint_{B_r(0)}\lf|\r^4\circ u(x)- e^{4\la_0} |x|^4\rg|^{1/2}\, dx^2\\
&\le\dashint_{B_r(0)}\lf|\sum_{j=1}^4u_j^4- e^{4\la_0}|x|^4\rg|^{1/2}\, dx^2+o(r^2)\\
&\le\dashint_{B_r(0)}[|u_1(x)^4-e^{4\la_0}x_1^4|^{1/2}+u_2^2+|u_3(x)^4-e^{4\la_0}x_2^4|^{1/2}+u_4^2]\, dx^2+o(r^2)\\
&\le\sum_{j=1}^2\lf(\dashint_{B_r(0)}|u_{2j-1}(x)^2-e^{2\la_0}x_j^2|\,dx^2\rg)^{1/2}\lf(\dashint_{B_r(0)}|u_{2j-1}(x)^2+e^{2\la_0}x_j^2|\,dx^2\rg)^{1/2}+o(r^2)\\
&=o(r^2).
\end{align*}
This implies that
\begin{equation}\label{r.is.x.in.meas}
\lim_{r\rightarrow 0} r^{-2} |\{x\in B_r(0)\,:\,|\r^4\circ u(x)- e^{4\la_0} |x|^4|>\ep r^4\}|=0\quad\text{for all }\ep>0,
\end{equation}
so that
\begin{equation}\label{diff.ann}
\int_{B_r(0)\setminus B_{r/2}(0)} |1-{\mathbf 1}_{e^{\la_0} r/2\le \r\circ u< e^{\la_0} r}|\, dx^2=o(r^2).
\end{equation}

We now introduce the sets
$$A^\r_{e^{\la_0}r}(0):=B^\r_{e^{\la_0}r}(0)\setminus B^\r_{e^{\la_0}r/2}(0),\quad A_r(0):=B_r(0)\setminus B_{r/2}(0),$$
where $B^\r_s(p):=\{q\in\mathbb{H}^2\,:\,\r_p(q)<s\}=\{q\in\mathbb{H}^2\,:\,d_K(p,q)<s\}$ is the ball of center $p$ and radius $s$ with respect
to the distance $d_K$.
Letting $\tilde A_r:=u^{-1}(A^\r_{e^{\la_0}r}(0))\cap A_r(0)$, we clearly have
\begin{align*}
\int_{\tilde A_r}\frac{1}{\sqrt{1+\sigma^2}}\frac{|\nabla u|^2}{\r}\, dx^2
&=2e^{2\la_0}\int_{\tilde A_r}\frac{\r^{-1}}{\sqrt{1+\sigma^2}}\,dx^2
+\int_{\tilde A_r}\frac{\r^{-1}}{\sqrt{1+\sigma^2}}|\nabla u(x)-\nabla u(0)|^2\, dx^2\\
&\quad+2 \int_{\tilde A_r}\frac{\r^{-1}}{\sqrt{1+\sigma^2}}\nabla u(0)\cdot (\nabla u(x)-\nabla u(0))\, dx^2
\end{align*}
(where we write $\r$ and $\sigma$ in place of $\r\circ u$ and $\sigma\circ u$, and we let $\sigma\circ u:=\operatorname{sgn}(\varphi\circ u)(+\infty)$ when $\rho\circ u=0$; note that on $\tilde A_r$ we never have $\rho\circ u=\varphi\circ u=0$).
Since on $\tilde A_r$ we have $\r\circ u\in(e^{\la_0}r/2,e^{\la_0}r)$, the second term on the right-hand side is $o(r)$;
the same holds for the third one, by Cauchy--Schwarz. Hence,
$$\int_{\tilde A_r}\frac{1}{\sqrt{1+\sigma^2}}\frac{|\nabla u|^2}{\r}\, dx^2
=2e^{2\la_0}\int_{\tilde A_r}\frac{\r^{-1}}{\sqrt{1+\sigma^2}}\,dx^2+o(r).$$
Because of \eqref{r.is.x.in.meas}, for any $\ep>0$ we have
$$\int_{\tilde A_r}\frac{\r^{-1}}{\sqrt{1+\sigma^2}}{\mathbf 1}_{|\r(x)-e^{\la_0}|x||>\ep r}\, dx^2
\le C\frac{|\tilde A_r\cap\{|\r(x)-e^{\la_0}|x||>\ep r\}|}{r}\, dx^2=o(r),$$
as well as
\begin{align*}
&2 e^{2\la_0}\int_{\tilde A_r} \frac{\r^{-1}}{\sqrt{1+\sigma^2}}{\mathbf 1}_{|\r(x)-e^{\la_0}|x||\le\ep r}\, dx^2\\
&=2 e^{\la_0}\int_{\tilde A_r} \frac{|x|^{-1}}{\sqrt{1+\sigma^2}}{\mathbf 1}_{|\r(x)-e^{\la_0}|x||\le\ep r}\, dx^2+O(\ep r)\\
&=2 e^{\la_0}\int_{\tilde A_r} \frac{|x|^{-1}}{\sqrt{1+\sigma^2}}\, dx^2+O(\ep r)+o(r).
\end{align*}
Combining the previous bounds with a simple diagonal argument, we finally obtain
\begin{align}
\label{IV.9-g}
\begin{aligned}
&\int_{\tilde A_r}\frac{1}{\sqrt{1+\sigma^2}}\frac{|\nabla u|^2}{\r}\, dx^2\\
&=2e^{\la_0}\int_{\tilde A_r}\frac{1}{\sqrt{1+\sigma^2}}|x|^{-1}\,dx^2+o(r).
\end{aligned}
\end{align}

Moreover, the integral bound \eqref{varphi.r.4} gives
$$|\{x\in\tilde A_r\,:\,|\varphi\circ u|(x)>\ep \r^2\circ u(x)\}|=o(r^2)$$
for any given $\ep>0$. Recalling that $\frac{1}{\sqrt{1+\sigma^2}}=\frac{\rho^2}{\r^2}=\sqrt{1-4\r^{-4}\varphi^2}$,
we deduce that
$$2e^{\la_0}\int_{\tilde A_r} \lf|1-\frac{1}{\sqrt{1+\sigma^2}}\rg||x|^{-1}\,dx^2=O(\ep r)+o(r).$$
Thus, using again a diagonal argument and recalling also \eqref{diff.ann}, the previous bounds give
\begin{align}
\label{IV.9-k}
\begin{aligned}
\int_{\tilde A_r}\frac{1}{\sqrt{1+\sigma^2}}\frac{|\nabla u|^2}{\r}\, dx^2
&=2e^{\la_0}\int_{\tilde A_r}|x|^{-1}\, dx^2+o(r)\\
&=2e^{\la_0}\int_{A_r(0)} |x|^{-1}\, dx^2+o(r)\\
&=4\pi e^{\la_0} r+o(r).
\end{aligned}
\end{align}

For $x\in \mathcal{G}_u^f$ such that $\r\circ u(x)\neq0$, at the point $u(x)$ we decompose
\[
\nabla^H\r=(\nabla^H\r)^T+(\nabla^H\r)^\perp,
\]
where $(\nabla^H\r)^T$ denotes the orthogonal projection of $\nabla^H \r$ onto the Lagrangian plane $\operatorname{span}\{\p_{x_1}u,\p_{x_2}u\}$ (this depends not only on $u(x)$ but also on $x$). Letting $2e^{2\lambda}:=|\nabla u(x)|^2$, by conformality of $\nabla u(x)$ we have
$$(\nabla^H\r)^T\circ u=e^{-2\la} \sum_{k=1}^2 \p_{x_k}(\r\circ u)\p_{x_k}u.$$
Away from $\{\rho\circ u=0\}$ we similarly decompose $\nabla^H\sigma=(\nabla^H\sigma)^T+(\nabla^H\sigma)^\perp$.
Recalling \eqref{r.sigma.j} and the fact that $J_H$ realizes an isometry from $\operatorname{span}\{\p_{x_1}u,\p_{x_2}u\}$ to its orthogonal inside $H$, we obtain
$$\r^3J_H[(\nabla^H\r)^T]=\frac{\rho^4}{2}(\nabla^H\sigma)^\perp,\quad\r^3J_H[(\nabla^H\r)^\perp]=\frac{\rho^4}{2}(\nabla^H\sigma)^T.$$
Hence, for $x\in\mathcal{G}_u^f$ such that $\rho\circ u(x)>0$, we have
$$\frac{\rho^4}{2}e^{-2\la}\sum_{k=1}^2\p_{x_k}(\sigma\circ u)\p_{x_k}u
=\r^3e^{-2\la}\sum_{k=1}^2\langle(\nabla^H\r)^\perp,J_H\p_{x_k}u\rangle \p_{x_k}u.$$
Since $e^{-\la}(J_H\p_{x_1}u,J_H\p_{x_2}u)$ is an orthonormal basis of the normal plane, we deduce that
$$\rho^8|\nabla(\sigma\circ u)|^2=2\r^8\frac{|(\nabla^H\r)^\perp|^2}{\r^2} |\nabla u|^2.$$
In particular, writing again $\sigma$ in place of $\sigma\circ u$, we have
$$|\nabla\arctan\sigma|^2=\frac{|\nabla\sigma|^2}{(1+\sigma^2)^2}=\frac{\rho^8}{\r^8}|\nabla\sigma|^2=2\frac{|(\nabla^H\r)^\perp|^2}{\r^2}|\nabla u|^2.$$
Recall that $\arctan\sigma$ extends to a smooth function on $\mathbb{H}^2\setminus\{0\}$
and note that this identity (discarding intermediate equalities) is valid also at any $x\in\mathcal{G}_u^f$ such that $u(x)\in\{\rho=0\}\setminus\{0\}$,
since both sides vanish (as $\nabla^H\arctan\sigma=0$ and $\r^3\nabla^H\r=\nabla^H\varphi^2=\varphi J_H\nabla^H\rho^2=0$ here).
By Theorem \ref{mono.cor}, we then have
$$\lim_{r\rightarrow 0}\frac1r\int_{\tilde A_r}N\frac{|(\nabla^H\r)^\perp|^2}{\r}|\nabla u|^2\, dx^2=0.$$
Since $|\nabla\r|=|(\nabla^H\r)^T|e^\lambda$, the left-hand side of \eqref{IV.3-a} is
\begin{align*}
\lim_{r\rightarrow 0}\frac{1}{e^{\lambda_0}r}\int_{\omega\cap\{e^{\lambda_0}r\le\r<2e^{\lambda_0}r\}}N\frac{|\nabla\r|^2}{\r}\,dx^2
&\ge e^{-\lambda_0}\lim_{r\rightarrow 0}\frac1r\int_{\tilde A_r}N\frac{|\nabla\r|^2}{\r}\,dx^2\\
&=e^{-\lambda_0}\lim_{r\rightarrow 0}\frac1r\int_{\tilde A_r}N\frac{|\nabla^H\r|^2}{\r}\frac{|\nabla u|^2}{2}\,dx^2.
\end{align*}
Recalling \eqref{nabla.r.norm}, the conclusion follows from \eqref{IV.9-k}.
\end{proof}

\subsection{Continuity of the underlying map}
We now show that the map $u$ is in fact continuous, in a quantitative way.

\begin{Prop}
\label{pr-cont}
Let $(\Sigma,u, N)$ be a PHSLV. Then $u$ admits a continuous representative and there exists a universal constant $C_1>0$ such that, in any conformal parametrization $\phi:B_1(0)\rightarrow \Sigma$,
$$
\operatorname{diam}_K^2 u\circ\phi(B_{1/2}(0))\le C_1 \int_{\phi(B_1(0))}N|\nabla u|^2\, dx^2.
$$
Here $\operatorname{diam}_K$ denotes the diameter with respect to $d_K$. \hfill $\Box$
\end{Prop}
\begin{Rm}
\label{rm-IV.1} Since the Carnot--Carath\'eodory distance is equivalent to the Kor\'anyi distance (by left-invariance and homogeneity with respect to the dilations $\delta_t$), we can replace $\operatorname{diam}_K$ with $\operatorname{diam}_{CC}$. The Carnot--Carath\'eodory distance is obviously larger than the distance $d_{{\mathbb H}^2}$ induced by $g_{{\mathbb H}^2}$. Hence the stated inequality holds as well (with a possibly different constant $C_1$) if we measure the diameter with respect to $d_{\mathbb{H}^2}$. \hfill $\Box$
\end{Rm}
\begin{Rm}
\label{rm-IV.2} 
The control of the modulus of continuity will be important to pass to the limit the fact that $u\in L^\infty_{loc}$ while considering sequences of PHSLVs with uniformly bounded masses.\hfill $\Box$
\end{Rm}

\begin{proof}
With a slight abuse of notation, we write $u$ in place of $u\circ\phi$.
Using Fubini and the mean value theorem we obtain an $s\in(1/2,1)$ such that $u\in W^{1,2}(\p B_s(0))\hookrightarrow C^0(\p B_s(0))$ and
$$\lf(\int_{\p B_s(0)}|\nabla u|\,d\mathcal{H}^1\rg)^2\le 2\pi s\cdot \int_{\p B_s(0)}|\nabla u|^2\,d\mathcal{H}^1
\le 4\pi s\cdot \int_{B_1(0)}|\nabla u|^2\,dx^2,$$
as well as
$$\mathcal{H}^1(\p B_s(0)\setminus\mathcal{G}_u)=0.$$
This gives
$$\operatorname{diam}^2_{{\mathbb H}^2}u(\p B_s(0))\le 4\pi \int_{B_1(0)}|\nabla u|^2\,dx^2,$$
where the diameter is taken with respect to the metric given by $g_{{\mathbb H}^2}$. Observe that, since the rectifiable curve $u(\p B_s(0))$ is horizontal, we have as well
$$\operatorname{diam}^2_{CC}u(\p B_s(0))\le 4\pi \int_{B_1(0)}|\nabla u|^2\,dx^2.$$
Since the Carnot--Carath\'eodory  distance $d_{CC}$ is comparable with $d_K$, there exists a universal constant $C_1'>0$ such that 
$$\operatorname{diam}_{K}u(\p B_s(0))\le C_1' \sqrt{\int_{B_1(0)}|\nabla u|^2\,dx^2}.$$

We now fix $q\in u(\mathcal{G}_u\cap\p B_s(0))$. We claim that, for any $x\in B_s(0)\cap\mathcal{G}_u^f$, we have
$$\r_q(u(x))\le C_1\sqrt{\int_{\phi(B_1(0))}N|\nabla u|^2\, dx^2}$$
for another universal constant $C_1>0$. Once this claim is proved, we will have
\begin{equation}\label{cont.claim.bis}
\|\r_q^4\circ u\|_{L^\infty(B_s(0))}\le C_1^4\lf[\int_{\phi(B_1(0))}N|\nabla u|^2\, dx^2\rg]^2=:\bar C.
\end{equation}
Indeed, $\r_q^4$ is smooth and we have $\r_q^4\circ u\le\bar C$ on $B_s(0)\cap\mathcal{G}_u^f$,
while $\nabla u=0$ a.e. on $B_s(0)\setminus\mathcal{G}_u^f$. Thus, given any $\psi\in C^\infty_c((\bar C,\infty))$,
we see that $\nabla(\psi\circ\r_q^4\circ u)=0$ on $B_s(0)$, obtaining $\psi\circ\r_q^4\circ u=0$ here (as $q$ is the image of a point in $\mathcal{G}_u\cap\p B_s(0)$) and thus \eqref{cont.claim.bis}.
In turn, this implies the statement with $C_1$ replaced by $(2C_1)^2$.

In order to prove the previous claim, let $p:=u(x)$. If $p\in u(\p B_s(0))$ then the claim follows from the bound for $\operatorname{diam}_K u(\p B_s(0))$
(as usual, $u(\p B_s(0))$ denotes the image of the continuous representative of $u|_{\p B_s(0)}$).
Assuming then $p\nin u(\p B_s(0))$, we let $\omega:=B_s(0)$ and consider the induced varifold $\mathbf{v}_\omega$,
which restricts to a HSLV on $\mathbb{H}^2\setminus u(\p B_s(0))$. Letting
$$2r:=d_K(p, u(\p B_s(0))),$$
by Theorem \ref{mono.cor} we have
$$\theta^\chi(p)\le Cr^{-2}\int_{\omega\cap\{r<\r_p\circ u<2r\}}N\frac{|\nabla u|^2}{2}\,dx^2\le Cr^{-2}\int_\omega N\frac{|\nabla u|^2}{2}\,dx^2,$$
and from Proposition \ref{dens.lb} it follows that
$$r^2\le\frac{C}{4\pi}\int_\omega N |\nabla u|^2\,dx^2,$$
proving the claim.
\end{proof}

From now on, we always replace $u$ with its continuous representative.

\begin{Rm}
It is immediate to check that the requirement in the definition of PHSLV now holds for \emph{every} open set $\omega\subset\subset\Sigma$
(rather than for a.e. domain $\omega\subset\subset\Sigma$). \hfill $\Box$
\end{Rm}

\subsection{Properties of the density}
We consider an arbitrary smooth cut-off function $\chi:\R_+\to\R$ satisfying the previous assumptions, namely
$\chi=1$ on $[0,1]$, $\chi=0$ on $[2,\infty)$, $\chi'\le0$, and also $\sqrt{-\chi'}\in C^\infty_c((1,2))$.
The present subsection is devoted to the proof of the following proposition, itself a consequence of the upper semi-continuity of $\theta^\chi$ for general HSLVs.

\begin{Prop}
\label{co-dens}
Assume that $(\Sigma,u,N)$ is a PHSLV on $\mathbb{H}^2$ and let $\omega\subset\subset\Sigma$.
Then the induced varifold $\mathbf{v}_\omega$ satisfies
$$\theta^\chi(p)\ge\limsup_{k\to\infty}\theta^\chi(p_k)\quad\text{whenever }p_k\to p\nin u(\p\omega).$$
Moreover, we have
$$\theta^\chi(p)\ge2\pi$$
for all $p\in u(\omega)\setminus u(\de\omega)$.
\hfill $\Box$
\end{Prop}

\begin{proof}
The first assertion follows immediately from Corollary \ref{dens.usc}.
Let us now fix a point $p\in u(\omega)\setminus u(\de\omega)$.
Given any $\psi\in C^\infty_c(\mathbb{H}^2\setminus u(\p\omega))$ with $\psi(p)=1$, the function
$\psi\circ u$ cannot be constant on $\omega$, since it vanishes at the boundary and equals $1$ on $\omega\cap u^{-1}(p)\neq\emptyset$. Hence,
$$\int_\omega|\nabla(\psi\circ u)|^2\,dx^2>0.$$
By the chain rule, we can then find $x\in\omega\cap\mathcal{G}_u^f$ such that $u(x)\in\operatorname{spt}(\psi)$.
This shows that $p$ belongs to the closure of $u(\omega\cap\mathcal{G}_u^f)\setminus u(\p\omega)$, and the second assertion follows from Proposition \ref{dens.lb}.
\end{proof}

\subsection{Rectifiability of the image}
By classical results on Sobolev functions (see, e.g., the appendix in \cite{PiRi1}), the image $u(\mathcal{G}_u)$ is rectifiable and ${\mathcal H}^2$-measurable, with respect to the Euclidean distance or equivalently the distance
induced by $g_{{\mathbb H}^2}$. 
We now show the stronger result that this holds also with respect to the finer distance $d_K$ (see also Remark \ref{cpb} below).

\begin{Prop}
\label{pr-rect-d-K}
Assume that $(\Sigma,u,N)$ is a PHSLV on $\mathbb{H}^2$. Then we have
$$\mathcal{H}^2_K(u(\Sigma\setminus\mathcal{G}_u^f))=0$$
and the image $u(\mathcal{G}_u)$ is a countable union of Lipschitz images, namely
we can cover $\mathcal{G}_u\subseteq\bigcup_jF_j$ with Borel sets such that
$$u|_{F_j}:F_j\to\mathbb{H}^2$$
is Lipschitz (endowing $\Sigma$ with a reference conformal metric $h$ and $\mathbb{H}^2$ with $d_K$).
\hfill $\Box$
\end{Prop}
\begin{proof}
Assume without loss of generality that $\Sigma=\mathbb{D}$. We introduce the sets
$$\Om_j:=\{x\in \mathbb{D}\,:\,\operatorname{dist}(x,\p\mathbb{D})>2/j\}$$
and
$$F_j':=\lf\{ x\in \Om_j\,:\,\sup_{0<r<1/j}\dashint_{B_{2r}(x)}|\nabla u|^2<j \rg\}.$$
Let $x,x'\in F_j'$ such that $r:=|x-x'|\le1/j$. Using Proposition \ref{pr-cont} (after rescaling $B_{2r}(x)$ to the unit ball) we have
$$d_{K}(u(x),u(x'))\le C\sqrt{\int_{B_{2r}(x)}|du|^2\,dx^2}\le C\sqrt{j}r=C\sqrt{j}|x-x'|.$$
Hence, $u_{S}$ is Lipschitz for any subset $S\subseteq F_j'$ of diameter at most $1/j$. This proves the second assertion.

Thanks to the area formula for Lipschitz maps from subsets of $\R^2$ into $({\mathbb H}^2,d_K)$ (see \cite{Kir} or more specifically \cite[Theorem 6.8]{CDPT}, keeping in mind that ${\R}^2$ can be viewed as a Carnot group of step $1$), for any measurable subset $A\subseteq {\mathcal G}_u$ we have
$$\int_{A} \frac{|\nabla u(x)|^2}{2}\,dx^2=\int_{{\mathbb H}^2} N(u,A,p)\, d{\mathcal  H}^2_K(p),$$
where $N(u,A,p)$ denotes the cardinality of $\{ x\in A\,:\,u(x)=p \}$.
Applying the formula to $A:={\mathcal G}_u\setminus {\mathcal G}_u^f$, on which $\nabla u=0$, we obtain
$${\mathcal  H}^2_K({\mathcal G}_u\setminus {\mathcal G}_u^f)=0.$$
To conclude, let  $W\subseteq\mathbb{D}$ be an arbitrary open set including $\mathbb{D}\setminus {\mathcal G}_u$. We claim that for any $x\in W$ there exists
an arbitrarily small radius $r>0$ such that
$$\int_{B_{2r}(x)}|\nabla u|^2\, dx^2\le 8\int_{B_{r}(x)}|\nabla u|^2\,dx^2+ (2r)^2.$$
If not, then we could find $j_0\in\N$ large such that for any $j\ge j_0$ we have in particular
$$\int_{B_{2^{-j}}(x)}|\nabla u|^2\,dx^2\ge 2^{-2j},\quad\int_{B_{2^{-j}}(x)}|\nabla u|^2\,dx^2\ge 8\int_{B_{2^{-j-1}}(x)}|\nabla u|^2\,dx^2,$$
giving the contradiction
$$2^{-2j}\le\int_{B_{2^{-j}}(x)}|\nabla u|^2\,dx^2=O(2^{-3j}).$$
Given any $\delta>0$, for any $x\in W$ we choose a radius $r_x>0$ such that $B_{2r_x}(x)\subseteq W$, the previous claim holds, and
$\int_{B_{2r_x}(x)}|\nabla u|^2\,dx^2\le \delta^2$.
From the cover $\{B_{r_x}(x)\mid x\in W\}$, we can extract a Besicovitch subcover, denoted $\{B_{r_j}(x_j)\mid j\in I\}$, such that 
$${\mathbf 1}_{W}\le\sum_{j\in I}{\mathbf 1}_{B_{r_j}(x_j)}\le C_0{\mathbf 1}_{W},$$
where $C_0>0$ is a universal constant. Using again Proposition \ref{pr-cont} we deduce
$$\operatorname{diam}^2_K u( B_{r_j}(x_j))\le C_1\delta$$
and
\begin{align*}
\sum_{j\in I} \operatorname{diam}^2_K u( B_{r_j}(x_j))&\le C_1\sum_{j\in I}\int_{B_{2r_j}(x_j)}|\nabla u|^2\,dx^2\\
&\le 8C_1\sum_{j\in I}\int_{B_{r_j}(x_j)}|\nabla u|^2\,dx^2+\sum_{j\in I}(2r_j)^2\\
&\le 8C_1C_0 \int_W |\nabla u|^2\,dx^2+\frac4\pi\sum_{j\in I}|B_{r_j}(x_j)|\\
&\le 8C_1C_0 \int_W |\nabla u|^2\,dx^2+\frac4\pi C_0|W|.
\end{align*}
This holds for any open set $W$ containing $\mathbb{D}\setminus {\mathcal G}_u$. Since
$\mathbb{D}\setminus {\mathcal G}_u$ is negligible, we can make the last right-hand side as small as we want,
and the first assertion follows.
\end{proof}

\subsection{Structure of fibers}
In order to understand how the multiplicity $\frac{\theta^\chi}{2\pi}$ in the target
is related to the domain multiplicity $N$, and more specifically to produce a more appropriate domain counterpart (which will be denoted by $\tilde N$), it is important to study the structure of fibers. This is the content of the next statement.
\begin{Prop}
\label{pr-fiber}
Assume that $(\Sigma,u,N)$ is a PHSLV on $\mathbb{H}^2$.
Given $\omega\subset\subset\Sigma$ and
$p\nin u(\de\omega)$, the number of connected components of $\omega\cap u^{-1}(p)$ is finite. Moreover, for any $x\in {\mathcal G}_u^f$, the connected component of $u^{-1}(u(x))$ containing $x$ is just to $\{x\}$. 
\hfill $\Box$
\end{Prop}
\begin{proof}Let $p\nin u(\p\omega)$ and, denoting by $\theta^\chi$ the density function for the varifold $\mathbf{v}_\omega$, let
$$M:=\frac{\theta^\chi(p)}{2\pi}.$$
We claim that the number of connected components of $u^{-1}(p)$ is not larger than $M$. If this were not the case, we could find $[M]+1$ disjoint compact sets $K_1,\dots,K_{[M]+1}$ whose union is $\omega\cap u^{-1}(p)$, where
$[M]$ is the integer part of $M$.
Let $\om_1,\dots,\om_{[M]+1}\subset\subset\omega$ be disjoint open sets such that $K_j\subseteq\omega_j$;
note in particular that $p\nin u(\p\omega_j)$. The induced varifold $\mathbf{v}_{\om_j}$ then restricts to a HSLV on $\mathbb{H}^2\setminus u(\p\omega_j)$ and, thanks to Proposition \ref{co-dens}, we have
$$\lim_{\ep\rightarrow 0}\int_{\omega_j}-\frac{\chi'(\r_p/\ep)}{\ep}N\lf[\frac{|\nabla\r_p|^2}{\r_p}+N\frac{\varphi_p}{2\r_p^3}\arctan\sigma_p|\nabla u|^2\rg]\,dx^2\ge 2\pi.$$
Summing over $j=1,\dots,[M]+1$ gives
$$\theta^\chi(p)=\lim_{\ep\rightarrow 0}\int_{\omega}-\frac{\chi'(\r_p/\ep)}{\ep}N\lf[\frac{|\nabla\r_p|^2}{\r_p}+N\frac{\varphi_p}{2\r_p^3}\arctan\sigma_p|\nabla u|^2\rg]\,dx^2\ge 2\pi([M]+1),$$
contradicting the definition of $M$. Hence, the number of connected components of $u^{-1}(p)$ is not larger than $[M]$.

Consider now a point $x\in {\mathcal G}_u^f$ and denote by $K_x$ the connected component of $u^{-1}(u(x))$ containing $x$. We claim that $K_x=\{x\}$.  We choose a local conformal chart $\phi$ centered at $x$ and, by abuse of notation, we write $u$ in place of $u\circ\phi^{-1}$.
We can assume that $u(0)=0$ and $\frac{|\nabla u(0)|^2}{2}=1$, up to a translation and a dilation in $\mathbb{H}^2$.
Since $u$ is conformal and Legendrian, $(\p_{x_1}u(0),\p_{x_2}u(0))$ defines an orthonormal basis of a Legendrian plane at the origin. Modulo a rotation (see also Remark \ref{isometry}), we can assume that
$$\p_{x_1}u(0)=(1,0,0,0,0),\quad \p_{x_2}u(0)=(0,0,1,0,0).$$
As in the proof of Proposition \ref{dens.lb}, we have
$$\int_{B_r(0)}\sum_{j=1}^2 |\p_{x_j} u(x)-X_j(0)|^2\, dx^2=o(r^2),$$
as well as 
$$\dashint_{B_r(0)}[|u_1(x)-x_1|^2+ |u_2(x)|^2+|u_3(x)-x_2|^2+|u_4(x)|^2]\, dx^2=o(r^2).$$
This is saying that $v_r(x):=r^{-1}\pi\circ u(rx)$ converges to $v_0(x):=(x_1,0,x_2,0)$ in $W^{1,2}(\mathbb{D})$ as $r\to0$,
where $\pi:\mathbb{H}^2\to\C^2$ is the canonical projection. By Fatou's lemma, we have
$$\int_0^1\liminf_{r\to0}\int_{\p B_s(0)} [|v_r-v_0|^2+|\nabla v_r-\nabla v_0|^2]\,d\mathcal{H}^1\,ds
\le\lim_{r\to0}\int_{\mathbb{D}}[|v_r-v_0|^2+|\nabla v_r-\nabla v_0|^2]\,dx^2=0.$$
Hence, we can select an $s\in(0,1)$ such that each $v_r$ restricts to a function in $W^{1,2}(\p B_s(0))$,
converging weakly to $v_0$ in this space, along a subsequence. Since this space compactly embeds in $C^0(\p B_s(0))$,
eventually we have $0\nin v_r(\de B_s(0))$. It follows that $K_0\subset B_{rs}(0)$, and thus $K_0=\{0\}$.
\end{proof}

\subsection{\texorpdfstring{The integer nature of the density $\bm{{\mathcal H}^2_K}$-a.e.}{The integer nature of the density ${\mathcal H}^2_K$-a.e.}}
We now establish the fact that $\frac{\theta^\chi(p)}{2\pi}\in\N$ for $\mathcal{H}^2_K$-a.e. $p\in\mathbb{H}^2$.

\begin{Prop}
\label{pr-int-dens}
Given a PHSLV $(\Sigma,u,N)$ on $\mathbb{H}^2$ and $\omega\subset\subset\Sigma$, the following holds for the induced varifold
$\mathbf{v}_\omega$.
For  ${\mathcal H}^2_K$-a.e. $p\nin u(\p\omega)$ there holds
\begin{equation}
\label{V.88}
\theta^\chi(p)\in 2\pi\N.
\end{equation}
Moreover, assuming also that $u^{-1}(p)\subseteq\mathcal{G}_u^f$ and that $u^{-1}(p)$ consists exclusively of Lebesgue points for $N$, we have
\begin{equation}
\label{V.88-a}
\theta^\chi(p)=2\pi \sum_{x\in u^{-1}(p)}N(x)
\end{equation}
(recall that the sum is finite by the previous proposition).
\hfill $\Box$
\end{Prop}

\begin{proof}
Clearly, we have $\theta^\chi(p)=0$ for any $p\nin u(\bar\omega)$.
Let $S$ be the set of points which are not Lebesgue for either $\nabla u$ or $N$ (recall that $u$ is continuous).
The same proof used in Proposition \ref{pr-rect-d-K} shows that
$$\mathcal{H}^2_K(u(S))=0.$$
In the sequel, we can then consider $p\in u({\mathcal G}_u^f)\setminus u(\p\omega\cup S)$
and we are left to show that \eqref{V.88-a} holds.

Thanks to Proposition \ref{pr-fiber}, we have $u^{-1}(p)=\{x_1,\dots, x_n\}\subseteq\omega\cap S$.
We fix disjoint neighborhoods $\om_1,\dots,\om_n$ conformally equivalent to $\mathbb{D}$.
Since $\omega\cap\{\r_p\circ u<2\ep\}\subseteq\bigcup_j\omega_j$ for $\ep>0$ small, it suffices to show that
for each $j$ the varifold $\mathbf{v}_{\om_j}$ has density
$$\theta^\chi(p)=2\pi N(x_j).$$
By replacing $\om$ with $\om_j$, we can assume in the sequel that $n=1$ and write $\mathbb{D}$ in place of $\omega$.
As in the previous proofs, we can also assume that our reference point is $0\in\mathbb{D}$, that $u(0)=p=0$, and that
$$\p_{x_1}u(0)=(1,0,0,0,0),\quad \p_{x_2}u(0)=(0,0,1,0,0).$$

We now claim that
\begin{equation}\label{rho.detach}\lim_{x\to0}\frac{\rho\circ u(x)}{|x|}=1.\end{equation}
Indeed, defining $v_r$ as in the previous proof, given any $y\in S^1$ and $t\in(0,1/2)$
we can select a small radius $s\in(t,2t)$ such that $v_r\to v_0$ uniformly on $\p B_s(y)$, along a subsequence (again, this can be done by finding $s$ such that $\liminf_{r\to0}\|v_r-v_0\|_{W^{1,2}(\p B_s(y))}=0$).
Moreover, since $\pi:\mathbb{H}^2\to\C^2$ is $1$-Lipschitz, applying Proposition \ref{pr-cont} to the maps $\delta_{1/r}\circ u(r\cdot)$ we obtain
$$\operatorname{diam}^2 v_r(B_s(y))\le C_1\int_{B_{2s}(y)}N|\nabla v_r|^2\,dx^2\to C_1\int_{B_{2s}(y)}N|\nabla v_0|^2\,dx^2\le Cs^2.$$
In particular, since $v_r\to v_0$ on $\p B_s(y)$, we obtain
$$\limsup_{r\to0}\||v_r|-1\|_{L^\infty(B_t(y))}\le Cs\le Ct$$
along the subsequence.
Since this can be done for any initially chosen subsequence $r_k\to0$, we deduce that the last inequality holds for $r\to0$, and the claim follows.

Since $\r\ge\rho$, the previous claim implies that
\begin{align*}
\theta^\chi(0)&=\lim_{\ep\to0}\int_{\mathbb{D}}-\frac{1}{\ep}\chi'\lf(\frac{\r}{\ep}\rg) \lf[N\frac{|\nabla\r|^2}{\r}+N\frac{\varphi}{\r^3}\arctan\sigma |\nabla u|^2\rg]\,dx^2\\
&=\lim_{\ep\to0}\int_{B_{2\ep}(0)}-\frac{1}{\ep}\chi'\lf(\frac{\r}{\ep}\rg) \lf[N\frac{|\nabla\r|^2}{\r}+N\frac{\varphi}{\r^3}\arctan\sigma |\nabla u|^2\rg]\,dx^2.
\end{align*}
Since $S_\ep:=\{x\in B_{2}(0)\,:\,N(\ep x)\neq N(0)\}$ has $|S_\ep|\to0$ and $\nabla v_\ep\to\nabla v_0$,
we have
$$\ep^{-2}\int_{B_{2\ep}(0)\cap\{N\neq N(0)\}}|\nabla u|^2\,dx^2=\int_{S_\ep}|\nabla v_\ep|^2\,dx^2\to0$$
as $\ep\to0$. Since the integrands above are bounded by
$$C\ep^{-1}\r^{-1}|\nabla u|^2\le C\ep^{-2}|\nabla u|^2$$
(as $\r$ is comparable with $\ep$ on the support of $\chi'(\r/\ep)$),
it follows that in the previous formula we can replace $N$ with $N(0)$.

Moreover, fixing $\delta\in(0,1)$, we let
$$\tilde A_\ep^\delta:=\{x\in B_{2\ep}(0)\setminus B_{\delta\ep}(0)\,:\,\chi'(\r\circ u(x)/\ep)\neq0\}.$$
As we saw along the proof of Proposition \ref{dens.lb}, we have
$$|\{x\in \tilde A_\ep^\delta\,:\, |\varphi\circ u(x)|>\lambda\r^2\circ u(x)\}|=o(\ep^2)$$
for any fixed $\delta,\lambda>0$, as $\ep\to0$. Since $\frac{\sigma^2}{1+\sigma^2}=\frac{4\varphi^2}{\r^4}$ outside the origin
(the left-hand side is understood to be $1$ on the $\varphi$-axis minus the origin), we obtain
$$|\{x\in \tilde A_\ep^\delta\,:\, |\sigma\circ u(x)|>\lambda\}|=o(\ep^2)$$
for any fixed $\delta,\lambda>0$. Since $|\varphi|\le\r^2$, it follows that
$$\lim_{\ep\to0}-\frac{1}{\ep}\int_{\tilde A_\ep^\delta}\chi'\lf(\frac{\r}{\ep}\rg) \frac{\varphi}{\r^3}\arctan\sigma |\nabla u|^2\,dx^2
=0.$$
On the other hand, exactly as in the proof of Proposition \ref{dens.lb}, we have
\begin{align*}
\lim_{\ep\to0}-\frac{1}{\ep}\int_{\tilde A_\ep^\delta}\chi'\lf(\frac{\r}{\ep}\rg) \frac{|\nabla\r|^2}{\r}\,dx^2
&=\lim_{\ep\to0}-\frac{1}{\ep}\int_{\tilde A_\ep^\delta}\chi'\lf(\frac{\r}{\ep}\rg)\frac{|\nabla u|^2}{2\r\sqrt{1+\sigma^2}}\,dx^2\\
&=\lim_{\ep\to0}-\frac{1}{\ep}\int_{B_{2\ep}(0)\setminus B_{\delta\ep}(0)}\frac{\chi'(|x|/\ep)}{|x|}\,dx^2\\
&=2\pi.
\end{align*}

Finally, we have
$$-\frac1\ep\int_{B_{\delta\ep}(0)}\chi'\lf(\frac{\r}{\ep}\rg) \lf[\frac{|\nabla\r|^2}{\r}+\frac{\varphi}{\r^3}\arctan\sigma |\nabla u|^2\rg]\,dx^2
\le C\ep^{-2}\int_{B_{\delta\ep}(0)}|\nabla u|^2\,dx^2
= C\delta^2\int_{B_1(0)}|\nabla_{\delta\ep}|^2\,dx^2,$$
which converges to $C\delta^2\cdot 2\pi$. By a diagonal argument, it follows that
$$\theta^\chi(0)=N(0)\lim_{\ep\to0}-\frac{1}{\ep}\int_{B_{2\ep}(0)}\chi'\lf(\frac{\r}{\ep}\rg)\,dx^2 \lf[\frac{|\nabla\r|^2}{\r}+\frac{\varphi}{\r^3}\arctan\sigma |\nabla u|^2\rg]\,dx^2
=2\pi N(0),$$
as desired.
\end{proof}

\subsection{\texorpdfstring{A robust representative of $\bm{N}$ and its upper semi-continuity}{A robust representative of $N$ and its upper semi-continuity}}

Following \cite{PiRi1}, we now introduce the following $L^\infty_{loc}$ function on a subset of $\Sigma$, which is a sort of domain counterpart of $\frac{\theta^\chi}{2\pi}$.

\begin{Dfi}\label{robust.def}
We let $\tilde\Sigma\subseteq\Sigma$ denote the open set of points $x$ such that, for some open $\omega\subset\subset\Sigma$, we have $x\in\omega$
and $u(x)\nin u(\p\omega)$. We define $\tilde N:\tilde\Sigma\to[0,\infty)$ as follows:
$$\tilde{N}(x):=\inf_\om \lim_{\ep\rightarrow 0}-\frac{1}{2\pi\ep}\int_{\omega}\chi'\lf(\frac{\r_{u(x)}}{\ep}\rg) \lf[N\frac{|\nabla\r_{u(x)}|^2}{\r_{u(x)}}+\frac{\varphi_{u(x)}^2}{\r_{u(x)}^3}\arctan\sigma_{u(x)}|\nabla u|^2\rg]\,dx^2,$$
where $\omega$ ranges among open subsets $\omega\subset\subset\Sigma$ such that $x\in\omega$ and $u(x)\nin u(\p\omega)$. \hfill $\Box$
\end{Dfi}

Note that a priori this function might not be integer-valued.
\begin{Rm}
As shown in the proof of Proposition \ref{pr-fiber},
we have $\mathcal{G}_u^f\subseteq\tilde\Sigma$. \hfill $\Box$
\end{Rm}

\begin{Rm}
\label{rm-rob-rep}
Observe that the infimum in the definition of $\tilde{N}(x)$ is in fact a minimum. Indeed,
given $\omega$ as above, the compact set $u^{-1}(u(x))\cap\omega$ has finitely many connected components.
Calling $K_x$ the one containing $x$, the minimum is achieved (for instance) for any open set
$K_x\subseteq\omega'\subseteq\omega$ disjoint from the remainder $[u^{-1}(u(x))\cap\omega]\setminus K_x$. \hfill$\Box$
\end{Rm}

\begin{Prop}
\label{pr-robust}
The function $\ti{N}$ is upper semi-continuous and 
\begin{equation}
\label{N-bound}
1\le \ti{N}(x)\le C_0\ell^{-2}\int_{u^{-1}(B^\r_\ell(u(x))\cap\omega} N|du|^2\,dx^2,\quad\ell:=\frac{d_K(u(x),u(\p\omega))}{2}
\end{equation}
for a universal constant $C_0>0$, for any $\omega$ as above. Moreover, $\ti{N}=N$ a.e. on ${\mathcal G}^f_u$.
\hfill $\Box$
\end{Prop}
\begin{proof}
Let $x\in\tilde\Sigma$ and fix $\omega$ realizing the infimum in the definition of $\tilde N$.
We consider the induced varifold $\mathbf{v}_\omega$.
Thanks to Proposition \ref{co-dens} we have the lower bound $\ti{N}(x)=\frac{\theta^\chi(u(x))}{2\pi}\ge1$,
while the upper bound follows from Theorem \ref{mono.cor}.

Moreover, given a sequence of points $x_k\to x$, eventually we have $x_k\in\omega$ and the points $p_k:=u(x_k)$ converge to $p:=u(x)\nin u(\p\omega)$, by continuity of $u$.
Hence, eventually we have $x_k\in\tilde\Sigma$ and $p_k\nin u(\p\omega)$. Hence, by Proposition \ref{co-dens} again, we have
$$\limsup_{k\to\infty}2\pi\tilde N(x_k)\le\limsup_{k\to\infty}\theta^\chi(p_k)\le\theta^\chi(p)=2\pi\tilde N(p).$$

The fact that $\ti{N}=N$ at a.e. $x\in{\mathcal G}^f_u$ (specifically, at any $x\in\mathcal{G}_u^f$ which is also a Lebesgue point for $N$) is a direct consequence of \eqref{V.88-a}.
\end{proof}

\section{Sequential compactness of PHSLVs}

We now show a fundamental compactness property of the class of parametrized varifolds studied in this work.
This property will be one of the key tools in the regularity theory developed in the second part of the paper.

\begin{Th}
\label{th-sequential-weak-closure}
Let $(\Sigma, [h_k])$ be a sequence of Riemann surfaces (for a fixed connected $\Sigma$) and assume that the sequence of metrics $h_k$ is pre-compact in $C^\infty_{loc}$.
Assume that $(\Sigma,u_k,N_k)$ is a sequence of PHSLVs in $\mathbb{H}^2$ such that $u_k(\bar x)$ stays bounded, for a fixed reference point $\bar x\in\Sigma$,
and that
$$\limsup_{k\to\infty}\int_\omega N_k|\nabla u_k|_{h_k}^2\,d\operatorname{vol}_{h_k}<\infty\quad\text{for all }\omega\subset\subset\Sigma,$$
as well as the existence of $C(\omega)>0$ such that
\begin{equation}
\label{teofrasto}
\limsup_{k\to\infty}\int_{\omega\cap\{\r_p\circ u_k<R\}} N_k|\nabla u_k|_{h_k}^2\,d\operatorname{vol}_{h_k}
\le C(\omega)R^2
\end{equation}
for all $\omega\subset\subset\Sigma$, $p\in\mathbb{H}^2$, and $R>0$.
Then, along a subsequence, the limit
$$u_\infty:=\lim_{k\to\infty}u_k\quad\text{exists in }C^0_{loc}\text{ and weakly in }W^{1,2}_{loc}$$
and, for a suitable new conformal class $[\hat h_\infty]$, there exist
a limit PHSLV
$$(\Sigma,\hat u_\infty,\hat N_\infty)$$
and a (locally) quasiconformal homeomorphism
$\psi:(\Sigma,[h_\infty])\to(\Sigma,[\hat h_\infty])$
such that
$$u_\infty=\hat u_\infty\circ\psi.$$
Moreover, we have the limit of Radon measures
\begin{equation}\label{rad.limit}\lim_{k\to\infty}N_k\frac{|\nabla u_k|_{h_k}^2}{2}\,d\operatorname{vol}_{h_k}=(\psi^{-1})_*
\lf[\hat N_\infty\frac{|\nabla \hat u_\infty|_{\hat h_\infty}^2}{2}\,d\operatorname{vol}_{\hat h_\infty}\rg]\end{equation}
and the induced varifolds $\mathbf{v}_{k,\omega}$ (see Remark \ref{hslv.from.phslv}) satisfy
$$\mathbf{v}_{\infty,\psi(\omega)}=\lim_{k\to\infty}\mathbf{v}_{k,\omega}\quad\text{on }\mathbb{H}^2\setminus u_\infty(\p\om)$$
for any $\omega\subset\subset\Sigma$. \hfill $\Box$
\end{Th}
\begin{Rm}
A similar result holds on a closed Sasakian manifold $M^5$, except that the convergence
in $C^0_{loc}$ is guaranteed only away from a locally finite set, and the last equality could become an inequality (if $\bar\omega$ intersects this set), due to possible bubbling (cf. Remark \ref{no.c.star} below).
Nonetheless, assuming the slightly stronger condition given in Definition \ref{strong.phslv},
we can recover an equality by a standard bubble-tree analysis, as explained in Remark \ref{bubbling},
and when $\Sigma$ is closed we can even remove the assumption of having a controlled conformal class (see Remark \ref{neck}; in both cases, $\Sigma$ becomes a possibly disconnected Riemann surface in the limit).
However, as shown by the counterexample of Theorem \ref{th-count-ex}, this cannot be done with
the initial definition of PHSLV.
\hfill $\Box$
\end{Rm}

\begin{Rm}\label{teofrasto.mild}
The upper bound \eqref{teofrasto} is a mild assumption which is typically satisfied in practice,
e.g., when $\Sigma$ and the ambient are closed, or while studying parametrized blow-ups, as a consequence of monotonicity.
\hfill $\Box$
\end{Rm}

For simplicity, we assume that the conformal class is independent of $k$.
In the general case, given any smooth $\omega\subset\subset\Sigma$,
we can find diffeomorphisms $\psi_k:\bar\omega\to\psi(\bar\omega)$ converging smoothly to the identity
such that $\psi_k^*[h_k]=[h_\infty]$, thus reducing to this situation up to using heavier notation.

We then endow $\Sigma$ with a metric $h$ inducing the fixed conformal class $[h]=[h_k]=[h_\infty]$.
We can assume that $u_k(\bar x)\to0$, up to a subsequence.
By covering each $\omega\subset\subset\Sigma$ with a connected union of conformal disks,
Proposition \ref{pr-cont} gives
$$\limsup_{k\to\infty}\|\r\circ u_k\|_{L^\infty(\omega)}<\infty.$$
Thus, in view of the local equivalence between the $\R^5$ and $\mathbb{H}^2$ Riemannian metrics,
we obtain that $(u_k)$ is bounded in $W^{1,2}(\omega,\R^5)$. Up to a subsequence, we can then extract a limiting map
$u_\infty:\Sigma\to\mathbb{H}^2$ such that
$$u_k\rightharpoonup u_\infty\quad\text{weakly in }W^{1,2}_{loc}.$$
We observe that the map $u_\infty$ is still $L^\infty_{loc}$ and Legendrian, since the condition $u_k^*\alpha=0$ is stable under weak limits
in $W^{1,2}_{loc}(\Sigma,\R^5)$.
%
%
We introduce the Radon measures $$d\nu_k:=N_k\frac{|\nabla u_k|^2}{2}\,dx^2$$ on $\Sigma$.
By assumption, we can extract a limit Radon measure $\nu_\infty$, up to a subsequence.
Note that, for any $\om\subset\subset\Sigma$, the pushforward of $\nu_k\res\omega$ via $u_k$ is simply the weight of the induced varifold $\mathbf{v}_{k,\om}$, namely we have
$$ (u_k)_*(\nu_k\res\om)=|\mathbf{v}_{k,\om}|. $$
Moreover, as seen in Remark \ref{hslv.from.phslv}, the varifold $\mathbf{v}_{k,\om}$ restricts
to a HSLV on $\mathbb{H}^2\setminus u_k(\p\om)$.
We postpone the actual proof of Theorem \ref{th-sequential-weak-closure}, since we first need another key result.

\subsection{Energy quantization}
Before continuing the proof, we will establish the following lemma, which is an \emph{energy quantization} result.
In its statement, we fix a conformal reference metric $h$ on $\Sigma$, used to define balls on the domain.
\begin{Lm}
\label{lm-energy-quant}
There exist two universal constants $c_\ast,C_2>0$ such that the following holds.
Given $\omega\subset\subset\Sigma$ open, assume that (along a subsequence) the maps $u_k\to u_\infty$ uniformly on $\de \omega$.
Taking any $a>2\ell>0$ such that
$$a\ge\lim_{k\to\infty}\operatorname{diam}_K(u_k(\omega )),$$
$$\ell\ge\lim_{k\to\infty}\operatorname{diam}_K(u_k(\p \omega ))=\operatorname{diam}_K(u_\infty(\p \omega )),$$
we have either
\begin{equation}
\label{concl-1}
\nu_\infty(\omega )\ge c_\ast a^2
\end{equation}
or
\begin{equation}
\label{concl-2}
\operatorname{diam}_K(u_k(\omega ))\le C_2\ell+C_2 \ell^{-1}\max_{p'\in u_k(\p \omega )}|\mathbf{v}_{k,\omega }|(B^\r_{4\ell}(p'))
\le C\ell
\end{equation}
for $k$ large enough, where $C$ depends only on the sequence and $\omega$.
\hfill $\Box$
\end{Lm}
\begin{Rm}\label{no.c.star}
In the Heisenberg group, this statement can be immediately improved.
Indeed, by taking $a>0$ large enough, we can obviously falsify \eqref{concl-1}, so that \eqref{concl-2} always holds true.
However, in a closed Sasakian manifold $M^5$, an analogous statement holds only for bounded scales $a\le a_0(M)$;
as a consequence of possible bubbling, \eqref{concl-2} might fail in general, even for small $\ell$
(but in this case \eqref{concl-1} holds with a right-hand side $c(M)=c_*a_0(M)^2>0$). \hfill $\Box$
\end{Rm}

\begin{proof}
We select $p_k\in u_k(\omega )$ maximizing the $d_K$-distance from $u_k(\de \omega )$
and fix $p_k'\in u_k(\p \omega )$ such that
$$d_K(p_k,p_k')=d_K(p_k,u_k(\de \omega )).$$
We observe that eventually
$$\operatorname{diam}_K(u_k(\omega ))\le2\max_{p\in u_k(\omega )}d_K(p,u_k(\de \omega ))+2\ell
=2d_K(p_k,p_k')+2\ell;$$
hence, we can assume without loss of generality that
$$d_K(p_k,p_k')>8\ell.$$
Recall that $\chi=1$ on $[0,1]$ and $\chi=0$ on $[2,\infty)$ and let
$$\zeta(t):=1-\chi\lf(\frac{t}{2\ell}\rg),$$
so that $\zeta=0$ on $[0,2\ell]$ and $\zeta=1$ on $[4\ell,\infty)$.
In the sequel, we will often write $p$ and $p'$ in place of $p_k$ and $p_k'$, for simplicity.
 
We test stationarity of $\mathbf{v}_{k,\omega}$ with $W_{(\zeta\circ\r_{p'})F_p}$,
where
$$F_p:=(\chi(\r_p/a)-\chi(\r_p/\ep))\arctan\sigma_p$$
(note that eventually $(\zeta\circ\r_{p'})F_p$ vanishes on $u_k(\p \omega )\subseteq B_{2\ell}^\r(p')$).
Using \eqref{V.3rep}--\eqref{V.2} and recalling from Proposition \ref{magic.id} that
\begin{align*}
-2\operatorname{div}_{\mathcal P}W_{F_p}&=2\frac{|\nabla^{\mathcal P}\r_p|^2}{\r_p} \lf( \frac{\chi'(\r_p/a)}{a}-\frac{\chi'(\r_p/\ep)}{\ep} \rg)
+|\nabla^{\mathcal P} z|^2 \lf[\frac{2\varphi_p}{\r_p^3} \lf( \frac{\chi'(\r_p/a)}{a}-\frac{\chi'(\r_p/\ep)}{\ep} \rg) \arctan\sigma_p\rg]\\
&\quad-\frac{\r_p^4}{2} \nabla^{\mathcal P}\arctan\sigma_p\cdot\nabla^{\mathcal P} \lf[\r_p^{-3} \lf( \frac{\chi'(\r_p/a)}{a}-\frac{\chi'(\r_p/\ep)}{\ep} \rg) {\arctan\sigma_p} \rg],
\end{align*}
we obtain
\begin{align*}
&2\int_{\omega } N_k (\zeta\circ\r_{p'})(\chi(\r_p/a)-\chi(\r_p/\ep)) |\nabla\arctan\sigma_p|^2\,dx^2\\
&\quad-\frac{2}{\ep}\int_{\omega } N_k (\zeta\circ\r_{p'})\chi'(\r_{p}/\ep)\lf[\frac{|\nabla\r_{p}|^2}{\r_{p}}+\frac{\varphi_p}{2\r_{p}^3} |\nabla u_k|^2\arctan\sigma_{p}\rg]\,dx^2\\
&=-\frac{2}{a} \int_{\omega }N_k (\zeta\circ\r_{p'}) \chi'(\r_{p}/a)\lf[\frac{|\nabla\r_{p}|^2}{\r_{p}}+\frac{\varphi_p}{2\r_{p}^3} |\nabla u_k|^2\arctan\sigma_{p}\rg]\,dx^2\\
&\quad+\frac{1}{2a}\int_{\omega } N_k (\zeta\circ\r_{p'}) \r_{p}^4 \nabla\lf[ \arctan\sigma_{p} \frac{\chi'(\r_p/a)}{\r_{p}^3} \rg] \cdot \nabla\arctan\sigma_p\,dx^2\\
&\quad-\frac{1}{2\ep}\int_{\omega } N_k (\zeta\circ\r_{p'}) \r_{p}^4 \nabla\lf[ \arctan\sigma_{p} \frac{\chi'(\r_{p}/\ep)}{\r_{p}^3} \rg]\cdot \nabla\arctan\sigma_p\,dx^2\\
&\quad+A+B
\end{align*}
(as usual, some compositions with $u_k$ are omitted), where
\begin{align*}
A&:=\int_{{\omega }} N_k \sum_{j=1}^2 \nabla (u_k)_{2j}\cdot \nabla\lf[\lf(\nabla^H(\zeta\circ\r_{p'})\cdot X_j\rg) F_p\rg]\,dx^2\\
&\phantom{:}\quad-\int_{{\omega }} N_k \sum_{j=1}^2 \nabla (u_k)_{2j-1}\cdot \nabla\lf[\lf(\nabla^H(\zeta\circ\r_{p'})\cdot Y_j\rg) F_p\rg]\,dx^2
\end{align*}
and
\begin{align*}
B&:=\int_{{\omega }} N_k \sum_{j=1}^2 \nabla (u_k)_{2j}\cdot \nabla(\zeta\circ\r_{p'}) \lf(\nabla^H F_p\cdot X_j\rg)\,dx^2\\
&\phantom{:}\quad-\int_{{\omega }} N_k \sum_{j=1}^2 \nabla (u_k)_{2j-1}\cdot \nabla(\zeta\circ\r_{p'}) \lf(\nabla^H F_p\cdot Y_j\rg)\,dx^2.
\end{align*}

We first bound $A$. We rewrite it as
\begin{align*}
A&=\int_{{\omega }} N_k \sum_{j=1}^2 \nabla (u_k)_{2j}\cdot \nabla\lf[\lf(\nabla^H(\zeta\circ\r_{p'})\cdot X_j\rg) (F_p-F_p(p'))\rg]\,dx^2\\
&-\int_{{\omega }} N_k \sum_{j=1}^2 \nabla (u_k)_{2j-1}\cdot \nabla\lf[\lf(\nabla^H(\zeta\circ\r_{p'})\cdot Y_j\rg) (F_p-F_p(p'))\rg]\,dx^2,
\end{align*}
where we tested stationarity with $W_{\zeta\circ\r_{p'}}$ in order to subtract the constant $F_p(p')$ from $F_p$.
Expanding
\begin{align*}
&\nabla^H\lf(\nabla^H(\zeta\circ \r_{p'})\cdot X_j\rg)\\
&=(\zeta''\circ\r_{p'}) (\nabla^H \r_{p'}\cdot X_j) \nabla^H\r_{p'}
+(\zeta'\circ\r_{p'}) \nabla^H(\nabla^H \r_{p'}\cdot X_j)
\end{align*}
we get
$$|\nabla^H(\nabla^H(\zeta\circ \r_{p'})\cdot X_j)|\le \frac{C}{\ell^2}{\mathbf 1}_{B^\r_{4\ell}(p')\setminus B_{2\ell}^\r(p')},$$
and similarly for $Y_j$.
Since $d_K(p,p')>8\ell$, using the equivalence between the Carnot--Carath\'eodory distance and the Kor\'anyi distance $d_K$, we have
$$\lf\|F_p-F_p(p')\rg\|_{L^\infty(B^\r_{4\ell}(p'))}\le C \|\nabla ^H F_p\|_{L^\infty(B^\r_{4\ell}(p'))}\cdot 4\ell\le \frac{C\ell}{a}$$
(note that on $B_{4\ell}^\r(p')$ we have $\chi(\r_p/\ep)=0$ for $\epsilon$ small, namely for $4\ell+2\ep<8\ell<d_K(p',p)$, so that eventually the balls $B_{4\ell}^\r(p')$ and $B_{2\ep}^\r(p)$ are disjoint). Combining the previous bounds gives
$$|A|\le \frac{C}{\ell a}\int_{\omega \cap\{\r_{p'}<4\ell\}} N_k |\nabla u_k|^2\,dx^2.$$
As for $B$, a similar (simpler) argument implies exactly the same bound.

Thus, denoting by $\theta^\chi_k$ the density of $\mathbf{v}_{k,\omega}$ and letting $\ep\to0$, we obtain
\begin{align*}
&\theta^\chi_k(p)- \frac{C}{\ell a}\int_{\omega \cap\{r_{p'}<4\ell\}} N_k |\nabla u_k|^2\,dx^2\\
&\le-\frac{2}{a} \int_{\omega } N_k (\zeta\circ\r_{p'}) \chi'(\r_{p}/a) \lf[\frac{|\nabla\r_{p}|^2}{\r_{p}}
+\frac{\varphi_p}{2\r_{p}^3} |\nabla u_k|^2\rg]\,dx^2\\
&\quad+\frac{1}{2a}\int_{\omega } N_k (\zeta\circ\r_{p'}) \r_{p}^4\nabla\lf[ \arctan\sigma_{p} \frac{\chi'(\r_{p}/a)}{\r_{p}^3} \rg] \cdot \nabla\arctan\sigma_p\,dx^2\\
&\le\frac{C}{a^2}\int_{\omega }N_k |\nabla u_k|^2\,dx^2.
\end{align*}
Using Proposition \ref{co-dens} to lower bound $\theta^\chi_k(p)\ge2\pi$, we obtain
$$2\pi- \frac{C}{\ell a} \int_{\omega \cap\{\r_{p'}<4\ell\}} N_k |\nabla u_k|^2\,dx^2\le \frac{C\nu_k(\omega )}{a^2},$$
where $C>0$ is universal. Calling $C'$ this constant, we deduce that either
$$\int_{\omega \cap\{\r_{p'}<4\ell\}} N_k |\nabla u_k|^2\,dx^2\ge\frac{\pi}{C'}\ell a$$
or
$$\nu_k(\omega )\ge\frac{\pi}{C'}a^2.$$
Since eventually $\operatorname{diam}_K(u_k(\omega ))<2a$ and $p'$ stays bounded in $\mathbb{H}^2$, we can use \eqref{teofrasto} to bound the last integral, and the statement follows.
\end{proof}

\subsection{Conclusion of the proof}
The previous quantization result is exploited in the proof of the following lemma (cf. \cite[Lemma III.5]{Riv-IHES} and \cite[Lemma 4.2]{PiRi1}).

\begin{Lm}
\label{lm-abs-cont}
We have the uniform convergence $$u_k\to u_\infty\quad\text{in }C^0_{loc}(\Sigma).$$
Moreover, the limiting measure $\nu_\infty$ is absolutely continuous with respect to $\operatorname{vol}_h$ and the density vanishes a.e. on $\{\nabla u_\infty=0\}$.
\hfill $\Box$
\end{Lm}

\begin{Rm}
In a closed Sasakian ambient $M^5$, the correct analogue is that $\nu_\infty$ decomposes as an absolutely continuous part,
plus a locally finite sum $\sum_{j\in J}c_j\delta_{x_j}$ of atoms, each with $c_j\ge c(M)>0$ (an energy concentration reflecting bubbling).
Moreover, the uniform convergence holds locally on $\Sigma\setminus\{x_j\mid j\in J\}$. \hfill $\Box$
\end{Rm}

\begin{proof}
Given any $x_0\in\Sigma$ and $\ell>0$, as in the proof of Proposition \ref{pr-cont} we can select $\rho>0$ such that
$u_\infty$ is in $W^{1,1}(\de B_\rho(x_0))$, with image (of the continuous representative)
having diameter less than $\ell$. As in the proof of Proposition \ref{pr-fiber}, we can also assume that
$u_k$ converges uniformly to $u_\infty$ on $\de B_\rho(x_0)$. Thus, by the previous result and Remark \ref{no.c.star},
we have
$$\limsup_{k\to\infty}\operatorname{diam}_K u_k(B_\rho(x_0))\le C\ell$$
along a subsequence. Since this could be applied to any initially chosen subsequence, we immediately deduce uniform convergence.

To see that $\nu_\infty$ is absolutely continuous with respect to $\operatorname{vol}_h$,
take $x_0\in\Sigma\cap\operatorname{spt}(\nabla u_\infty)$ and
note that the previous argument, in conjunction with \eqref{teofrasto}, gives that for any $r>0$ small we can find $\rho\in(r,2r)$ and $\ell>0$ such that
$$\nu_\infty(B_\rho(x_0))\le C\ell^2,\quad \ell^2\le C\int_{B_{2r}(x_0)}|\nabla u_\infty|^2\,dx^2,$$
where $C=C(\om)$ (for any fixed $\om\subset\subset\Sigma$ containing $x_0$).
The claim now follows from a standard covering argument (similar to the one used in Proposition \ref{pr-rect-d-K}) and the absolute continuity of the measure $|\nabla u_\infty|^2\,d\operatorname{vol}_h$.
\end{proof}

In the sequel, we replace $u_\infty$ with its continuous representative.
We now obtain a much more precise structure for $\nu_\infty$.

 \begin{Lm}
\label{lm-integer-mult}
There exists a function
$N_\infty\in L^\infty_{loc}(\Sigma,\N^*)$ such that
\[
d\nu_\infty=N_\infty|\p_{x_1}u_\infty\wedge\p_{x_2}u_\infty|\,dx^2
\]
in any local conformal chart on $\Sigma$.
\hfill $\Box$
\end{Lm}

\begin{proof}
We tacitly work in a local chart, itself included in a fixed domain $\om\subset\subset\Sigma$. Writing $d\nu_\infty=f\,dx^2$, we can assume that $0$ is a Lebesgue point for $f$
and that $0\in\mathcal{G}_{u_\infty}$, as well as $\nabla u_\infty(0)\neq0$. Up to a left translation, we can also assume that $u_\infty(0)=0$.
We first observe that, as $r\to0$, the rescaled maps
$$u_\infty^{(r)}(x):=\delta_{1/r}\circ u_\infty(rx)$$
converge weakly in $W^{1,2}_{loc}(\C,\R^5)$ to the linear map
$$L(x)=x_1\p_{x_1}u_\infty(0)+x_2\p_{x_2}u_\infty(0).$$
Indeed, for any $R>0$, we have
\begin{align*}
\lim_{r\to0}\lim_{k\to\infty}\int_{B_{2R}(0)}N_k|\nabla u_k^{(r)}|^2\,dx^2
&=\lim_{r\to0}\lim_{k\to\infty}\int_{B_{2R}(0)}N_k|\nabla (\pi\circ u_k^{(r)})|^2\,dx^2\\
&=\lim_{r\to0}\lim_{k\to\infty}r^{-2}\int_{B_{2Rr}(0)}N_k|\nabla (\pi\circ u_k)|^2\,dx^2\\
&=(2R)^2\lim_{r\to0}\frac{2\nu_\infty(B_{2Rr}(0))}{(2Rr)^2}\\
&=8\pi R^2 f(0)
\end{align*}
(the maps $u_k^{(r)}$ are defined on $B_{2R}(0)$ for $r$ small enough).
Thus, taking into account Proposition \ref{pr-cont}, we see that
$$\limsup_{r\to0}\limsup_{k\to\infty}\operatorname{diam}_K^2 u_k^{(r)}(B_{R}(0))\le CR^2,$$
and hence
$$\limsup_{r\to0}\operatorname{diam}_K^2 u_\infty^{(r)}(B_{R}(0))\le CR^2.$$
Since $u_\infty^{(r)}(0)=0$, this gives
\begin{equation}\label{l.infty.bd.u.infty}\limsup_{r\to0}\|\r\circ u_\infty^{(r)}\|_{L^\infty(B_R(0))}\le CR\end{equation}
for all $R>0$. Moreover,
$$\int_{B_R(0)}|\nabla u_\infty^{(r)}|^2\,dx^2=\int_{B_R(0)}|\nabla(\pi\circ u_\infty^{(r)})|^2\,dx^2
=r^{-2}\int_{B_{Rr}(0)}|\nabla(\pi\circ u_\infty)|^2\,dx^2\to \pi R^2|\nabla u_\infty|^2(0).$$
By the local equivalence between the $\mathbb{H}^2$ and $\mathbb{R}^5$ metrics, we deduce that
$u_\infty^{(r)}$ converges weakly in $W^{1,2}_{loc}(\C,\R^5)$ to a limit map $u_\infty^{(0)}$, up to a subsequence.
Moreover, we clearly have $\pi\circ u_\infty^{(r)}\to\pi\circ L$ in $W^{1,2}_{loc}$, as well as $u_\infty^{(0)}(0)=0=L(0)$ by \eqref{l.infty.bd.u.infty}.
Further, $u_\infty^{(0)}$ is Legendrian, and it is easy to check that $L$ is Legendrian as well; since $\pi\circ u_\infty^{(0)}=\pi\circ L$, this forces $u_\infty^{(0)}=L$, as claimed.

Moreover, calling $\nu_k^{(r)}$ the domain measure associated with $u_k^{(r)}$, we have
$$\lim_{r\to0}\lim_{k\to\infty}\nu_k^{(r)}=\lim_{r\to0}\nu_\infty^{(r)}=f(0)\mathcal{L}^2.$$
By a diagonal argument, replacing each $u_k$ with a suitable rescaling $\delta_{1/r_k}\circ u_k(r_k\cdot)$,
we can then assume that $u_\infty=L$ and the claim becomes that
\begin{equation}\label{int.new.claim}
\nu_\infty(B_1(0))=\pi|L(e_1)\wedge L(e_2)|\cdot n,
\end{equation}
for an integer $n\ge1$ bounded solely in terms of $\omega$.

If $L$ has rank $1$, then $L(\bar B_1(0))$ is included in a line, itself included in $\{\varphi=0\}$.
Covering this set with $O(s^{-1})$ balls $B_j$ of radius $s$, we have
$$\limsup_{k\to\infty}\int_{B_1(0)\cap u_k^{-1}(B_j)}N_k|\nabla u_k|^2\,dx^2\le Cs^2$$
by \eqref{teofrasto} (which is preserved by our rescaling operation). Summing over $j$ and using the fact that
eventually $u_j(\bar B_1(0))\subseteq \bigcup_jB_j$ by the $C^0_{loc}$ convergence established in the previous result,
we obtain
$$\limsup_{k\to\infty}\int_{B_1(0)}N_k|\nabla u_k|^2\,dx^2=O(s^2)\cdot O(s^{-1})=O(s).$$
Since $s$ was arbitrary, the claim follows in this case.

Assume now that $|L(e_1)\wedge L(e_2)|\ne 0$ and denote by
\[
{\mathcal E}_\infty:=L(\bar B_1(0))
\]
the (filled) ellipse obtained as the image of $B_1(0)$ through $L$ (which might not be conformal a priori).
Moreover, let
$$\mathcal{P}_\infty:=L(\C)$$
be the Legendrian two-plane spanned by $L$. Up to a rotation, we can assume that
$$\mathcal{P}_\infty=\operatorname{span}\{X_1(0),X_2(0)\}=\{(\alpha,0,\beta,0,0)\mid\alpha,\beta\in\R\}.$$

We now claim that
\begin{equation}\label{vfd.conv.pre}
\int_{B_1(0)}[|\nabla (u_k)_2|^2+|\nabla (u_k)_4|^2+|\nabla(u_k)_\varphi|^2]\,dx^2\to0.
\end{equation}
Since $\nabla(u_k)_\varphi=(u_k)_1\nabla(u_k)_2-(u_k)_2\nabla(u_k)_1+(u_k)_3\nabla(u_k)_4-(u_k)_4\nabla(u_k)_3$,
it suffices to show that
$$\int_{B_1(0)}[|\nabla (u_k)_2|^2+|\nabla (u_k)_4|^2]\,dx^2\to0.$$
We consider an arbitrary cut-off $\xi:\mathcal{P}_\infty\to\R$ supported in the interior of $\mathcal{E}_\infty$, as well as
$$F(z,\varphi):=\xi(z_1,z_3) [z_1z_2+z_3z_4-\varphi].$$
Since $u_k\to L$ in $C^0(\bar B_1(0))$, eventually ${u}_k(\p B_1(0))\cap \operatorname{spt}(\xi)=\emptyset$. The stationarity condition  \eqref{V.2} then gives
\begin{align*}
0&=\int_{B_1(0)}\sum_{j=1}^2 N_k \nabla (u_k)_{2j}\cdot\nabla[(\xi\circ u_k) (u_k)_{2j}
+(\p_{z_{2j-1}}\xi\circ u_k) ((u_k)_1(u_k)_2+(u_k)_3(u_k)_4-(u_k)_\varphi)]\,dx^2\\
&\quad-\int_{B_1(0)}\sum_{j=1}^2 N_k \nabla (u_k)_{2j-1}\cdot\nabla[(\xi\circ u_k) (u_k)_{2j-1}]\,dx^2
+\int_{B_1(0)}N_k \sum_{\ell=1}^4\nabla (u_k)_\ell\cdot\nabla[(\xi\circ u_k) (u_k)_\ell]\,dx^2.
\end{align*}
Since $(u_k)_{2j}\to0$, the contributions from the terms containing $(u_k)_2$, $(u_k)_4$, or $(u_k)_\varphi$ (not differentiated) go to zero in the limit. Hence, the term on the first line equals
\begin{align*}&\sum_{j=1}^2\int_{B_1(0)}N_k(\xi\circ u_k)|\nabla (u_k)_{2j}|^2\,dx^2+\sum_{j,\ell=1}^2\int_{B_1(0)}N_k(\p_{z_{2j-1}}\xi\circ u_k)
(u_k)_{2\ell-1}\nabla(u_k)_{2j}\cdot\nabla(u_k)_{2\ell}\,dx^2\\
&\quad-\sum_{j=1}^2\int_{B_1(0)}N_k(\p_{z_{2j-1}}\xi\circ u_k)
\nabla(u_k)_{2j}\cdot\nabla(u_k)_\varphi\,dx^2+o(1)\\
&=\sum_{j=1}^2\int_{B_1(0)}N_k(\xi\circ u_k)|\nabla (u_k)_{2j}|^2\,dx^2\\
&\quad+\sum_{j,\ell=1}^2\int_{B_1(0)}N_k(\p_{z_{2j-1}}\xi\circ u_k)
(u_k)_{2\ell}\nabla(u_k)_{2j}\cdot\nabla(u_k)_{2\ell-1}\,dx^2+o(1)\\
&=\sum_{j=1}^2\int_{B_1(0)}N_k(\xi\circ u_k)|\nabla (u_k)_{2j}|^2\,dx^2+o(1),\end{align*}
thanks to the Legendrian condition $\nabla(u_k)_\varphi=\sum_{\ell=1}^2[(u_k)_{2\ell-1}\nabla(u_k)_{2\ell}-(u_k)_{2\ell}\nabla(u_k)_{2\ell-1}]$,
while the second line above equals
$$\sum_{j=1}^2\int_{B_1(0)}N_k(\xi\circ u_k)|\nabla (u_k)_{2j}|^2\,dx^2+o(1).$$
This proves \eqref{vfd.conv.pre}.

Along a subsequence, the induced varifolds $\mathbf{v}_{k,B_1(0)}$ converge to a varifold $\mathbf{v}_\infty$ (whose weight is) supported on $\mathcal{E}_\infty$.
Moreover, $\mathbf{v}_\infty$ restricts to a HSLV on $\mathbb{H}^2\setminus\p\mathcal{E}_\infty$, where $\p\mathcal{E}_\infty:=L(\p B_1(0))$.
Thanks to \eqref{vfd.conv.pre}, we also have
$$\mathbf{v}_\infty(\mathcal P,p)=\delta_{\mathcal P_\infty}(\mathcal P)\otimes|\mathbf{v}_\infty|(p).$$
For arbitrary $a,b\in C^\infty_c(\mathcal{P}_\infty)$ supported in the interior of $\mathcal{E}_\infty$, we take
$$F(z,\varphi):= -a(z_1,z_3)z_2-b(z_1,z_3)z_4,$$
for which the associated Hamiltonian vector field is
$$2W_F=J_H \nabla^H F-2 F\p_{\varphi}=a(z_1,z_3) \p_{z_1}+b(z_1,z_3) \p_{z_3}\quad\text{on }\mathcal P_\infty.$$
Hence,
${\mathbf v}_\infty$ is stationary in the classical isotropic sense, away from $\p\mathcal{E}_\infty$.
By the constancy theorem, there exists a constant $\theta_0>0$ such that
$$d|{\mathbf v}_\infty|=\theta_0\, d{\mathcal H}^2\res {\mathcal E}_\infty.$$
In fact, we could also have deduced this from the fact that
$$|\mathbf{v}_\infty|=\lim_{k\to\infty}|\mathbf{v}_{k,B_1(0)}|=\lim_{k\to\infty}(u_k)_*(\uno_{B_1(0)}\nu_k)
=L_*(\uno_{B_1(0)}\nu_\infty)$$
and the fact that $\nu_\infty$ is a constant multiple of $\mathcal{L}^2$.
Since $\mathcal{E}_\infty$ has area $\pi|L(e_1)\wedge L(e_2)|$,
we obtain
$$\nu_\infty(B_1(0))=|{\mathbf v}_\infty|(\mathcal{E}_\infty)=\theta_0\cdot\pi|L(e_1)\wedge L(e_2)|.$$
Using \eqref{teofrasto}, we see that $\theta_0$ is bounded by a constant $C(\omega)$.
Finally, from Theorem \ref{cpt.int} we deduce that $\theta_0\in\N$ (see also Lemma \ref{equiv.rect} and Remark \ref{cpb}).
\end{proof}

\begin{proof}[Proof of Theorem \ref{th-sequential-weak-closure}]
Recall that we already established the following:
we have the $C^0_{loc}$ convergence $u_k\to u_\infty$ and the measures $d\nu_k=N_k\frac{|\nabla u_k|^2}{2}\,dx^2$
converge to a limit of the form $d\nu_\infty=N_\infty|\p_{x_1}u_\infty\wedge\p_{x_2}u_\infty|\,dx^2$ (in any conformal chart).
We now fix $\om\subset\subset\Sigma$ and consider the induced varifolds $\mathbf{v}_{k,\om}$.

We claim that, on the complement of $u_\infty(\p\om)$, any subsequential limit $\mathbf{v}_\infty$ coincides with the varifold induced by $u_\infty$ with the multiplicity $N_\infty$, so that in particular the latter restricts to a HSLV on $\mathbb{H}^2\setminus u_\infty(\p\om)$. Indeed, it is straightforward to deduce that
$$|\mathbf{v}_\infty|=\lim_{k\to\infty}|\mathbf{v}_{k,\om}|=\lim_{k\to\infty}(u_k)_*(\nu_k\res\om)=(u_\infty)_*(\nu_\infty\res\om)$$
on $\mathbb{H}^2\setminus u_\infty(\p\om)$. To obtain the claim, we just have to show that $\mathbf{v}_\infty$ is rectifiable (as a varifold in $\R^5$),
so that it is uniquely determined by its own weight $|\mathbf{v}_\infty|$. Let $T$ be the $\mathcal{H}^2$-negligible set of points
such that $|\mathbf{v}_\infty|$ has a tangent plane (with respect to Euclidean dilations) at any $p\in\operatorname{spt}|\mathbf{v}_\infty|\setminus (T\cup u_\infty(\p\om))$. Further, let $T'$ be the $|\mathbf{v}_\infty|$-negligible
set of points such that any $p\in\operatorname{spt}|\mathbf{v}_\infty|\setminus (T'\cup u_\infty(\p\om))$ is an approximate continuity point of the Grassmannian part (in the disintegration of $\mathbf{v}_\infty$ with respect to $\Pi:G\to\mathbb{H}^2$), both in terms of Euclidean and anisotropic balls, as discussed while proving Proposition \ref{equiv.rect}.

As in the proof of Proposition \ref{pr-rect-d-K} (see also the proof of Lemma \ref{lm-abs-cont}), we see that $u_\infty$ carries negligible sets to $\mathcal{H}^2_K$-negligible sets and
$$\mathcal{H}^2_K(u_\infty(S))=0,$$
where $S$ is the set of points which are not in $\mathcal{G}_{u_\infty}^f$ or which are not Lebesgue points for $N_\infty$.
By \eqref{teofrasto}, we conclude that
$$|\mathbf{v}_\infty|(u_\infty(S))=0,$$
as well. For $|\mathbf{v}_\infty|$-a.e. $p\in\operatorname{spt}|\mathbf{v}_\infty|$ (with $p\nin u_\infty(\p\om)$), we then have $p\nin u_\infty(S)$. Moreover, it is a classical fact (similar to the proof of Proposition \ref{pr-rect-d-K})
that we can write $\Sigma\setminus S$ as a disjoint union of sets $E_0,E_1,E_2,\dots$,
such that $\operatorname{vol}_h(E_0)=0$ and, for $j\ge1$, $u_\infty$ maps $(E_j,d_h)$ in a bi-Lipschitz way to a subset of $(M_j,d_{\R^5})$, for a $C^1$ embedded surface $M_j\subset\R^5$, in such a way that the image of $\nabla u_\infty(x)$ is precisely $T_xM_j$ for all $x\in E_j$. Since $T_pM_j=T_pM_{j'}$ for $\mathcal{H}^2$-a.e. $p\in M_j\cap M_{j'}$, we can then find $S'\supseteq S$ such that
$\mathcal{L}^2(S'\setminus S)=0$ and $$\operatorname{img} \nabla u_\infty(x)=\operatorname{img} \nabla u_\infty(x')\quad\text{whenever }x,x'\nin S'\text{ and }u_\infty(x)=u_\infty(x'),$$
as well as
$$u_\infty(x)\nin T\cup T'\quad\text{for all }x\nin S'.$$
For $q\in u_\infty(\om)\setminus u_\infty(S'\cup\p\om)$, we may then call $\mathcal{P}_q$ the image of $\nabla u_\infty(x)$ for any $x\in u_\infty^{-1}(q)$.
Note that we still have $\mathcal{H}^2_K(u_\infty(S'))=0$, and thus
$$|\mathbf{v}_\infty|(u_\infty(S'))=0.$$

For $|\mathbf{v}_\infty|$-a.e. $p\nin u_\infty(\p\om)$, which we now fix, we then have $p\in\operatorname{spt}|\mathbf{v}_\infty|\setminus  u_\infty(S\cup S')$. We assume $p=0$,
up to a left translation. 
We deduce that any anisotropic blow-up $\mathbf{w}$, i.e., any limit of rescalings $(\delta_{1/r})_*\mathbf{v}_\infty$
along a sequence $r\to0$, has the form
$$\mathbf{w}(\mathcal P,p)=\mu(\mathcal P)\otimes|\mathbf{w}|(p).$$
To reach the claim, it suffices to show that $\mu=\delta_{\mathcal P_0}$: indeed, once this is done,
by definition of $T$ and $T'$ the Euclidean blow-up of $\mathbf{v}_\infty$ at $0$ is a constant multiple of $\mathcal{P}_0$, as desired.

As in the proof of Proposition \ref{pr-fiber}, the fiber
$$\om\cap u_\infty^{-1}(0)=\{x_1,\dots,x_n\}$$
is finite and is made of points in $\om\setminus S\subseteq\mathcal{G}_{u_\infty}^f$. By construction,
the image of $\nabla u_\infty$ at each of these points is the same Legendrian plane $\mathcal{P}_0$.
Arguing as in the proof of \eqref{rho.detach} (see also the proof of Lemma \ref{lm-abs-cont}),
we see that
$$\omega\cap u_\infty^{-1}(B_a^\r(0))\subseteq\bigcup_{j=1}^n B_{Ca}(x_j),$$
for small $a>0$. Thus, on each ball $B_R^{\r}(0)$,
the varifold $\mathbf{w}$ is a limit of suitable rescalings
$$\sum_{j=1}^n (\delta_{1/r_k})_*\mathbf{v}_{k,B_{2CRr_k}(x_j)}.$$
As in the first part of the proof of Lemma \ref{lm-integer-mult}, we can also arrange that each rescaled map $\delta_{1/r_k}\circ u_k\circ\phi_j^{-1}(r_k\cdot)$ converges in $W^{1,2}_{loc}\cap C^0_{loc}$
to a linear map with image $\mathcal{P}_0$, where each $\phi_j$ is a conformal chart centered at $x_j$.
Thus, as we saw along that proof, $\mathbf{w}$ does indeed coincide with a positive multiple of $\mathcal{P}_0$ on $B_R^\r(0)$.
Since $R>0$ was arbitrary, this establishes the claim that $\mu=\delta_{\mathcal P_0}$.

Finally, by lower semi-continuity of the $L^2$-norm, the convergence of $\nu_k$ to $\nu_\infty$ and Lemma \ref{lm-integer-mult} give
$$\frac{|\nabla u_\infty|^2}{2}\le N_\infty |\p_{x_1}u_\infty\wedge\p_{x_2}u_\infty|$$
a.e. in any conformal chart. On each conformal disk $D\subset\subset\Sigma$,
recall from the proof of Lemma \ref{lm-integer-mult} that we have
\begin{equation}\label{dist.bd}
\pi N_\infty\le C(D),
\end{equation}
for the constant $C(D)$ (i.e., $C(\omega)$ with $\omega:=D$) from \eqref{teofrasto}. Hence, arguing exactly as in \cite[pp. 2013--2014]{PiRi1},
we can construct a $\frac{C(D)^2}{\pi^2}$-quasiconformal homeomorphism $\psi:D\to\mathbb{D}$ such that $u_\infty\circ \psi^{-1}$ is weakly conformal.
Moreover, the chain rule for such maps (see, e.g., \cite[Lemma III.6.4]{LV}) shows that $\psi'\circ\psi^{-1}$ is conformal for any two of them. Hence, they give an atlas for a new smooth and conformal structure on the topological surface $\Sigma$. We now invoke the classical fact that
these two smooth structures are diffeomorphic to each other to conclude.
\end{proof}

\begin{Rm}\label{cpb}
The last proof shows also the following fact: given a PHSLV $(\Sigma,u,N)$ and an open set $\om\subset\subset\Sigma$ such that $|\mathbf{v}_\om|(B_R^\r(p))\le CR^2$ (for any ball $B_R^\r(p)\subset\mathbb{H}^2$),
for $|\mathbf{v}_\om|$-a.e. $p\in\operatorname{spt}|\mathbf{v}_\om|\setminus u(\p\om)$ the dilations around $p$, namely
$$(\delta_{1/r}\circ\ell_{p^{-1}})_*\mathbf{v}_\om,$$
converge as varifolds to a Legendrian plane with constant density (recall that $\ell_{p^{-1}}(x):=p^{-1}*x$).
It also shows that, for a PHSLV $(\Sigma,u,N)$, a set $F\subseteq u(\Sigma)$ is $\mathcal{H}^2$-negligible if and only if it is $\mathcal{H}^2_K$-negligible. This last fact is false for subsets of the support of $|\mathbf{v}|$,
for a general HSLV $\mathbf{v}$ (see the example in Remark \ref{orr} and note that $\mathcal{H}^2_K\res\{z=0\}$ is a nontrivial Radon measure). \hfill $\Box$
\end{Rm}

\begin{Rm}\label{bubbling}
In a closed ambient $M^5$, in case of bubbling, we can extract limit bubbles in the standard way, by using Proposition \ref{pr-VI.1}
to represent them as PHSLVs defined on $\hat{\mathbb{C}}=S^2$, provided the slightly stronger stationarity condition given in Definition \ref{strong.phslv} holds.
We also get a (possibly constant) limit PHSLV defined on $\Sigma$ in the same way.
We can rule out energy dissipation in \emph{neck regions} as follows: assuming (by restriction) that we have a sequence of PHSLVs
$$(S^1\times(0,R_k),u_k,N_k)$$
satisfying \eqref{teofrasto} with $\om:=S^1\times(0,R_k)$,
with the usual neck region assumption that $R_k\to\infty$ and
\begin{equation}\label{neck.ass}\sup_{a\in(0,R_k-1)}\int_{S^1\times(a,a+1)}N_k|\nabla u_k|^2\,dx^2\to0,\end{equation}
we claim that in fact
\begin{equation}\label{neck.claim}\int_{S^1\times(0,R_k)}N_k|\nabla u_k|^2\,dx^2\to0.\end{equation}
Indeed, given any two sequences $a_k,b_k\in(0,R_k)$ with $a_k<b_k$, thanks to \eqref{neck.ass} we can select $a_k'<b_k'$ such that
$|a_k'-a_k|+|b_k'-b_k|\le1$, and such that both
$u_k(S^1\times\{a_k'\})$ and $u_k(S^1\times\{b_k'\})$ have vanishing diameter. Thus, the two sets converge to two points $q_a,q_b$, respectively. Hence, up to another subsequence,
the limit
$$\mathbf{v}:=\lim_{k\to\infty}\mathbf{v}_{S^1\times(a_k,b_k)}=\lim_{k\to\infty}\mathbf{v}_{S^1\times(a_k',b_k')}$$
exists and is a HSLV on $M\setminus\{q_a,q_b\}$. Further, we can also test its stationarity
with Hamiltonian vector fields $W_F$ generated by an $F\in C^\infty(M)$ constant near $q_a$ and near $q_b$.
Hence, by Proposition \ref{remov.gen}, $\mathbf{v}$ is a HSLV on $M$.
If \eqref{neck.claim} does not hold, we can select $a_k$ and $b_k$ such that the limit
$\mathbf{v}$ has a positive, arbitrarily small mass; however, $\mathbf{v}$ has density $\theta^\chi\ge2\pi$
on its support by Proposition \ref{co-dens} and Corollary \ref{dens.limits},
and hence it obeys a universal lower bound $|\mathbf{v}|(M)\ge c(M)>0$ on the mass by Theorem \ref{mono.cor},
a contradiction.
\hfill $\Box$
\end{Rm}

\begin{Rm}\label{neck}
When $\Sigma$ is closed,
the previous argument also holds for
neck regions called ``collars'' appearing due to a \emph{degenerating conformal structure} (see \cite{Hum} for a concise treatment
of the Deligne--Mumford compactification of the space of closed Riemann surfaces).
If we do not assume the stronger definition of PHSLV$^*$,
then compactness of PHSLVs plainly fails, as shown in Theorem \ref{th-count-ex}. \hfill $\Box$
\end{Rm}

We finally complete the proof of Theorem \ref{th-0bis}. In its statement, we just have the PHSLV assumption.
Since the conformal class is controlled, the theorem follows from a standard bubble-tree analysis
and the following lemma (used inductively along the tree).

\begin{Lm}\label{bubbling.bis}
Assume that $(\Sigma_k,u_k,N_k)$ is a sequence of PHSLVs in a closed Sasakian ambient $M^5$,
satisfying \eqref{teofrasto} (with $\Sigma_k$ in place of $\omega$) and such that,
up to a conformal equivalence, we can write
$$\Sigma_k=\Sigma_k'\cup(S^1\times[0,R_k])\cup\Sigma_k'',\quad R_k\to\infty,$$
for two compact Riemann surfaces with boundary $\de\Sigma_k'=S^1\times\{0\}$ and $\de\Sigma_k''=S^1\times\{R_k\}$,
such that \eqref{neck.ass} holds.
Then, for any open $\omega_k\subseteq\Sigma_k''$ including the boundary $\p\Sigma_k''$,
any subsequential limit
$$\mathbf{v}:=\lim_{k\to\infty}\mathbf{v}_{k,\omega_k}$$
is a HSLV away from the subsequential limit $\lim_{k\to\infty} u_k(\p\omega_k)$
(where $\p\omega_k$ denotes the topological boundary, i.e., does not include $\p\Sigma_k''$).
Moreover, no energy is dissipated in the neck region, meaning that \eqref{neck.claim} holds.
\hfill $\Box$
\end{Lm}

\begin{proof}
The proof is similar to the one of Lemma \ref{lm-energy-quant}.
For simplicity, we assume $\omega_k=\Sigma_k''$; the general case is a trivial modification of the following argument.

We pick $p_k'\in u_k(S^1\times\{a_k\})$, with $a_k\in[0,R_k]$ chosen such that $u_k(S^1\times\{a_k\})$ has vanishing diameter, and let $$\tilde\omega_k:=[S^1\times(a_k,R_k)]\cup\Sigma_k''.$$
Given $F\in C^\infty(M)$, considering the associated Hamiltonian vector field $W_F$,
we claim that
$$\int_{G}\operatorname{div}_{\mathcal P}W_F\,d\mathbf{v}_{k,\tilde\omega_k}\to0$$
as $k\to\infty$, where $\mathbf{v}_{k,\tilde\omega_k}$ denotes the varifold induced by $(u_k,\Sigma_k,N_k)$ and the domain $\tilde\omega_k\subset\Sigma_k$.
Once this is done, we can conclude that $\mathbf{v}_{k,\Sigma_k''}$
converges to a HSLV, up to a subsequence. Also, by subtraction, as in Remark \ref{bubbling}
we see that any subsequential limit of $\mathbf{v}_{k,S^1\times(a_k,b_k)}$ is a HSLV with density $\theta^\chi\ge2\pi$ on its support, and thus \eqref{neck.claim} must hold.

To check this claim, we consider a vanishing sequence $\ell_k\to0$ such that
$$\operatorname{diam}_K(u_k(S^1\times\{a_k\}))<\ell_k,$$
as well as
\begin{equation}\label{teofrasto.ennesimo}
\int_{\r_{p_k'}\circ u_k<2\ell_k}N_k|\nabla u_k|^2\,dx^2\le C\ell_k^2,
\end{equation}
and we let $\chi_k(t):=\chi(t/\ell_k)$.
Since $(1-\chi_k\circ\r_{p_k'})F$ vanishes near $u_k(\p\tilde\omega_k)=u_k(S^1\times\{a_k\})$,
we obviously have
$$\int_{G}\operatorname{div}_{\mathcal P}W_{(1-\chi_k\circ\r_{p_k'})F}\,d\mathbf{v}_{k,\tilde\omega_k}=0.$$
Hence, in order to conclude, we just have to prove that
$$\int_{G}2\operatorname{div}_{\mathcal P}W_{(\chi_k\circ\r_{p_k'})F}\,d\mathbf{v}_{k,\tilde\omega_k}=A_k+B_k+C_k$$
goes to zero, where
\begin{align*}
A_k&:=\int_{{\tilde\omega_k }} N_k \sum_{j=1}^2 \nabla (u_k)_{2j}\cdot \nabla\lf[\lf(\nabla^H(\chi_k\circ\r_{p_k'})\cdot X_j\rg) F\rg]\,dx^2\\
&\phantom{:}\quad-\int_{{\tilde\omega_k }} N_k \sum_{j=1}^2 \nabla (u_k)_{2j-1}\cdot \nabla\lf[\lf(\nabla^H(\chi_k\circ\r_{p_k'})\cdot Y_j\rg) F\rg]\,dx^2
\end{align*}
(we omit composition with $u_k$) and
\begin{align*}
B_k&:=\int_{{\tilde\omega_k }} N_k \sum_{j=1}^2 \nabla (u_k)_{2j}\cdot \nabla(\chi_k\circ\r_{p_k'}) \lf(\nabla^H F\cdot X_j\rg)\,dx^2\\
&\phantom{:}\quad-\int_{{\tilde\omega_k }} N_k \sum_{j=1}^2 \nabla (u_k)_{2j-1}\cdot \nabla(\chi_k\circ\r_{p_k'}) \lf(\nabla^H F\cdot Y_j\rg)\,dx^2,
\end{align*}
as well as $$C_k:=\int_{G}(\chi_k\circ\r_{p_k'})\operatorname{div}_{\mathcal P}W_{F}\,d\mathbf{v}_{k,\tilde\omega_k}.$$

We first bound $A_k$. We rewrite it as
\begin{align*}
A&=\int_{{\tilde\omega_k }} N_k \sum_{j=1}^2 \nabla (u_k)_{2j}\cdot \nabla\lf[\lf(\nabla^H(\chi_k\circ\r_{p_k'})\cdot X_j\rg) (F-F(p_k'))\rg]\,dx^2\\
&-\int_{{\tilde\omega_k }} N_k \sum_{j=1}^2 \nabla (u_k)_{2j-1}\cdot \nabla\lf[\lf(\nabla^H(\chi_k\circ\r_{p_k'})\cdot Y_j\rg) (F-F(p_k'))\rg]\,dx^2,
\end{align*}
where we tested stationarity with $W_{\chi_k\circ\r_{p_k'}}=-W_{(1-\chi_k)\circ\r_{p_k'}}$ in order to subtract the constant $F(p_k')$ from $F$.
We have the bound
$$|\nabla^H(\nabla^H(\zeta\circ \r_{p_k'})\cdot X_j)|\le \frac{C}{\ell_k^2}{\mathbf 1}_{B^\r_{2\ell_k}(p_k')\setminus B_{\ell_k}^\r(p_k')},$$
and similarly for $Y_j$.
Using the equivalence between the Carnot--Carath\'eodory distance and the Kor\'anyi distance $d_K$, we also have
$$\lf\|F-F(p_k')\rg\|_{L^\infty(B^\r_{2\ell_k}(p_k'))}\le C\ell_k \|\nabla^H F\|_{L^\infty(B^\r_{2\ell_k}(p_k'))}\le C\ell_k.$$
Combining the previous bounds gives
$$|A_k|\le \frac{C}{\ell_k}\int_{\{\r_{p_k'}\circ u_k<2\ell_k\}} N_k |\nabla u_k|^2\,dx^2.$$
As for $B_k$, a similar (simpler) argument implies exactly the same bound, while obviously
$$|C_k|\le C\int_{\{\r_{p_k'}\circ u_k<2\ell_k\}} N_k |\nabla u_k|^2\,dx^2.$$
Hence, by \eqref{teofrasto.ennesimo} we have
$$|A_k|+|B_k|+|C_k|\le C\ell_k\to0,$$
as desired.
\end{proof}

\begin{Rm}\label{neck.bis}
It is important to point out why the previous proof does \emph{not} work for collars,
in the context of degenerating conformal class.
The fundamental difference is that a collar does \emph{not} disconnect $\Sigma_k$. Hence, we would need to cut it at two
different places (rather than just one), corresponding to two points $p_k',p_k''\in M$. However, the previous proof breaks down
when we subtract a constant from $F$ in the term $A_k$ (since the two constants
$F(p_k')$ and $F(p_k'')$ might be far from each other). And indeed the lemma is false for collars,
as shown in Theorem \ref{th-count-ex}.
\hfill $\Box$
\end{Rm}
\subsection{Tangent cones to PHSLVs}

In this part we consider a PHSLV $(\Sigma,u,N)$ and we fix a point $x_0\in\Sigma$.
We now show that a notion of \emph{parametrized blow-up} exists at $x_0$, under suitable assumptions.
The first one is the technical assumption that $x_0\in\tilde\Sigma$.
Recall that this means that we have $u(x_0)\nin u(\p\om)$ for a suitable $\om\subset\subset\Sigma$ containing $x_0$.
This assumption is typically satisfied in practice, as the next criterion shows (cf. Remark \ref{teofrasto.mild}).

\begin{Prop}\label{in.tilde.omega}
If $u$ is not constant in any neighborhood of $x_0$ and for some $\ep>0$ we have
\begin{equation}\label{teofrasto.bis}
\int_{B_\ep(x_0)\cap\{\r_{u(x_0)}\circ u<R\}}N|\nabla u|^2\,dx^2\le CR^2
\end{equation}
for all radii $R>0$, then $x_0\in\tilde\Sigma$.
 \hfill $\Box$
\end{Prop}

\begin{proof}
Since the claim is local, we can assume that $\Sigma=\mathbb{D}$, $x_0=0$, and $\ep=1$.
Let $\Lambda\ge2$ large, to be found later. Up to a translation, we can assume $u(0)=0$. We fix a radius $r\in(0,1/\Lambda)$
and let
$$s^2:=\int_{B_{r}(0)}|\nabla u|^2\,dx^2>0.$$
Since by \eqref{teofrasto.bis}, for $r<\ep$, we have
$$\int_{B_r(0)\cap\{\mathfrak{r}\circ u<s\}}|\nabla u|^2\,dx^2\le Cs^2,$$
we can find $\rho\in(r,\Lambda r)$ such that $u\in W^{1,2}(\de B_\rho(0))$ and, for a possibly different $C$, we have
$$\int_{\{\mathfrak{r}\circ u<s\}\cap\partial B_\rho(0)}|\nabla u|^2\,d\mathcal{H}^1\le \frac{Cs^2}{\rho\log\Lambda}.$$
Then Cauchy--Schwarz gives
$$\int_{\{\mathfrak{r}\circ u<s\}\cap\partial B_\rho(0)}|\nabla u|\,d\mathcal{H}^1\le \frac{Cs}{\sqrt{\log\Lambda}}=:\ell.$$
By \eqref{nabla.r.norm} we have $|\nabla(\mathfrak{r}\circ u)|\le|\nabla u|$ on $\{u\neq0\}$. Hence, if $\{\mathfrak{r}\circ u<\ell\}\cap\partial B_\rho(0)\neq\emptyset$ and $2\ell\le s$, then the previous inequality gives
$$\partial B_\rho(0)\subseteq\{\mathfrak{r}\circ u<2\ell\}$$
(since the image of $\mathfrak{r}\circ u$ on this circle
cannot include the full interval $(\ell,2\ell)$).
By Lemma \ref{lm-energy-quant} (applied to a constant sequence) and Remark \ref{no.c.star}, this gives
$$\operatorname{diam}_K u(B_\rho(0))\le C\ell+C\ell^{-1}|\mathbf{v}_{B_\rho(0)}|(B_{6\ell}^\r(0))\le C\ell,$$
where we used \eqref{teofrasto.bis}. In turn, this implies
$$s^2\le \int_{B_\rho(0)}|\nabla u|^2\,dx^2\le C\ell^2,$$
again by \eqref{teofrasto.bis}. Both $s<2\ell$ and the last inequality
are impossible for $\Lambda$ large enough.
Hence, for this radius $\rho\in(r,\Lambda r)$ we have $\mathfrak{r}\circ u\ge\ell$ on $\partial B_\rho(0)$,
and in particular $u(0)=0\nin u(\partial B_\rho(0))$.
\end{proof}
\begin{Prop}
\label{pr-tangent-cone}
Let $(\Sigma, u,N)$ be a PHSLV and let $x_0\in\tilde\Sigma$ such that $u$ is not constant in any neighborhood of $x_0$. Moreover, in a conformal chart centered at $x_0$, assume that for a sequence of radii $r_k\to0$ we have
\begin{equation}\label{adm.implicit}
\limsup_{k\to\infty}\frac{\int_{B_{Rr_k}(0)} |\nabla u|^2\, dx^2}{\int_{B_{r_k}(0)} |\nabla u|^2\, dx^2}\le C(R)
\end{equation}
for any given $R>1$, where as usual we use the metric $g_{\mathbb{H}^2}$ to measure the Dirichlet energy. Letting
$$u_k(y):= \delta_{1/s_k} \circ \ell_{u(0)^{-1}} \circ u(r_ky),\quad s_k^2:=\int_{B_{r_k}(0)} |\nabla u|^2\, dx^2,$$
and $N_k(y):=N(r_ky)$, up to extracting a subsequence,
we have $u_k\to u_\infty$ in $C^0_{loc}$ for a suitable map $u_\infty:\C\to\mathbb{H}^2$
and a limit PHSLV $(\C,\hat u_\infty,\hat N_\infty)$, such that
$$u_\infty=\hat u_\infty\circ\psi$$
for a quasiconformal homeomorphism $\psi:\C\to\C$ with $\psi(0)=0$, with the same conclusions as Theorem \ref{th-sequential-weak-closure}.
Moreover, the maps $u_\infty$ and $\hat u_\infty$ are proper, with
$$u_\infty^{-1}(0)=\hat u_\infty^{-1}(0)=\{0\},$$
and we have
$$\varphi\circ u_\infty=\varphi\circ \hat u_\infty=0,$$
as well as
$$\tilde N_\infty\le\tilde N_\infty(0)=\tilde N(x_0),$$
where $\tilde N_\infty$ denotes the function built in \ref{robust.def} for the limit.
\hfill $\Box$
\end{Prop}

\begin{Rm}
In fact, a posteriori \eqref{adm.implicit} is always satisfied, provided that a bound of the form \eqref{upper.bd.assumption} holds at least locally, as a consequence of the regularity theory developed in the second part of the paper. \hfill $\Box$.
\end{Rm}

\begin{Rm}
Note that this result holds without changes in the case of a closed ambient.
Indeed, once we magnify it at smaller and smaller scales, Remark \ref{no.c.star} applies;
in particular, we cannot have bubbling in this blow-up setting. \hfill $\Box$
\end{Rm}
\begin{proof} Without loss of generality, we can assume that $u(x_0)=0$.
Since
$$(\delta_t)_*X_j=tX_j,\quad (\delta_t)_*Y_j=tY_j$$
and $\nabla u$ takes values in horizontal planes, we see that for any $R'>1$ we have
$$\int_{B_{R'}(0)}|\nabla u_k|^2\,dx^2=s_k^{-2}\int_{B_{R'r_k}(0)}|\nabla u|^2\,dx^2\le C(R')$$
in the conformal chart.
Also, by assumption there exists $\omega_0\subset\subset\Sigma$ such that $0=u(x_0)\nin u(\p\omega_0)$.
Hence, for any $R,R'>0$ we have
\begin{align}\label{teofrasto.tris}
\begin{aligned}
\limsup_{k\to\infty}\int_{B_{R'}(0)\cap\{\r\circ u_k<R\}}N_k|\nabla u_k|^2\,dx^2
&=\limsup_{k\to\infty}s_k^{-2}\int_{B_{R'r_k}(0)\cap\{\r\circ u<Rs_k\}}N|\nabla u|^2\,dx^2\\
&\le\limsup_{k\to\infty}s_k^{-2}|\mathbf{v}_{\omega_0}|(B_{Rs_k}^\r(0))\\
&\le CR^2,
\end{aligned}
\end{align}
by Corollary \ref{mono.cor2} and the fact that $\mathbf{v}_{\omega_0}$ restricts to a HSLV on $\mathbb{H}^2\setminus u(\p\omega_0)$,
which is an open set containing the origin. Recalling that $N\in L^\infty_{loc}$, we see that all the assumptions
of Theorem \ref{th-sequential-weak-closure} are satisfied, except for the fact that we apply it with domains increasing to $\C$, which makes no difference in the proof.

We then obtain a limit map $u_\infty$ and a limit PHSLV $(\C,\hat u_\infty,\hat N_\infty)$ up to a subsequence.
Crucially, $u_\infty$ and $\hat u_\infty$ are not constant thanks to \eqref{rad.limit}.
This convergence of measures, together with \eqref{teofrasto.tris}, also gives
\begin{equation}\label{teofrasto.quadris}\int_{\r\circ\hat u_\infty<R}\hat N_\infty|\nabla\hat u_\infty|^2\,dx^2\le CR^2.\end{equation}
We now show that $\hat u_\infty$ is proper. By arguing exactly as in the previous proof,
we can find an increasing sequence of radii $\tau_j\to\infty$ such that
$$\inf_{\de B_{\tau_j}(0)}\r\circ\hat u_\infty\to\infty\quad\text{as }j\to\infty.$$
If $\hat u_\infty$ is not proper, then we can find points $x_j\in A_j:= B_{\tau_{j+1}}(0)\setminus\bar B_{\tau_j}(0)$
such that $\hat u_\infty(x_j)$ stays bounded, up to a subsequence. Applying Theorem \ref{mono.cor} to the varifold $\mathbf{v}_{\infty,A_j}$ (induced by the limit PHSLV and the domain $A_j$),
as well as Proposition \ref{co-dens}, we would get
$$|\mathbf{v}_{\infty,A_j}|(B_1^\r(\hat u_\infty(x_j)))\ge c>0$$
for $j$ large enough, since $B_1^\r(\hat u_\infty(x_j))$ is disjoint from $\hat u_\infty(\p A_j)$ eventually. Taking
$$R:=1+\sup_{j\to\infty}\r\circ \hat u_\infty(x_j),$$
we then get
$$\int_{\r\circ u_\infty<R}|\nabla\hat u_\infty|^2\,dx^2\ge\sum_j|\mathbf{v}_{\infty,A_j}|(B_1^\r(\hat u_\infty(x_j)))=\infty,$$
contradicting \eqref{teofrasto.quadris}.

In particular, since $\hat u_\infty$ (and thus also $u_\infty$) is proper, the function $\tilde N_\infty$ is defined on all of $\C$.
Also, a proof similar to the one used in Proposition \ref{pr-robust}, relying on Corollary \ref{dens.limits},
shows that
$$\tilde N_\infty(0)\ge\limsup_{k\to\infty}\tilde N_k(0)=\tilde N(x_0).$$
We can assume without loss of generality that $2\pi\tilde N(x_0)=\theta^\chi(\mathbf{v}_{\omega_0},x_0)$.
Moreover, since $\mathbf{v}_{\om_0}$ is a HSLV near the origin, Corollary \ref{mono.cor2} gives the bound
$$\limsup_{k\to\infty}s_k^{-2}|\mathbf{v}_{\om_0}|(B_{Rs_k}^\r(0))\le CR^2,$$
so that up to a subsequence the limit HSLV
$$\mathbf{w}:=\lim_{k\to\infty}(\delta_{1/s_k})_*\mathbf{v}_{\om_0}$$
exists and satisfies
\begin{equation}\label{w.mass.bd}
|\mathbf{w}|(B_R^\r(0))\le CR^2.
\end{equation}
The same derivation of \eqref{teofrasto.tris} shows that
\begin{equation}\label{grass.ub}
\mathbf{v}_{\infty,\psi(B_{R'}(0))}\le\mathbf{w}
\end{equation}
for any given $R'>0$. Since $\theta^\chi(\mathbf{w},0)=2\pi\tilde N(x_0)$, we conclude that also the reverse inequality
$\tilde N_\infty(0)\le\tilde N(x_0)$ holds. Moreover, the same argument used in the proof of Proposition \ref{pr-fiber},
together with \eqref{grass.ub} and the equality $\theta^\chi(\mathbf{w},0)=2\pi\tilde N_\infty(0)$, proves that
$\hat u_\infty^{-1}(0)$ is a compact connected set containing $0$. However, taking $x$ on its topological boundary, the argument used in the previous proof
gives arbitrarily small radii $r>0$ such that $\r\circ \hat u_\infty>0$ on $\p B_r(x)$, proving that
$$\hat u_\infty^{-1}(0)=\{0\}.$$

Let us now show that $2\pi\tilde N_\infty(x)\le\theta^\chi(\mathbf{w},0)=2\pi\tilde N_\infty(0)$ for all $x\in\C$. By \eqref{grass.ub} and the definition of $\tilde N_\infty$,
it suffices to show that $\theta^\chi(\mathbf{w},p)\le\theta^\chi(\mathbf{w},0)$ for any given $p\in\mathbb{H}^2$.
Using the notation \eqref{capital.theta.def} and calling $f_{q,a}$ the integrand in \eqref{capital.theta.def},
we clearly have
$$|f_{0,1}-f_{q,1}|\le C|q|\le Cd_K(0,q)\quad\text{for all }q\in B^\r_{1/2}(0).$$
By a simple scaling and translation-invariance, we obtain
$$|f_{q',a}-f_{q'',a}|\le C\frac{d_K(q',q'')}{a^3}\quad\text{for all }q',q''\in \mathbb{H}^2\text{ and }a>2d_K(q',q'').$$
Hence, letting $R_0:=d_K(0,p)$, by monotonicity and the last bound we have
\begin{equation}\label{rig.pre}
\Theta(\mathbf{w},p,\ep)\le\Theta(\mathbf{w},p,R)\le\Theta(\mathbf{w},0,R)+C\frac{R_0}{R^3}|\mathbf{w}|(B_{2R+2R_0}^\r(0))
\end{equation}
for all $0<\ep<R$ with $R>2R_0$. Hence, letting $\ep\to0$ and using \eqref{w.mass.bd}, we obtain
$$\theta^\chi(\mathbf{w},p)\le\Theta(\mathbf{w},0,R)+O(R^{-1}).$$
Since $\nabla^{\mathcal P}\arctan\sigma=0$ on $\spt(\mathbf{w})\setminus\Pi^{-1}(0)$ (see Remark \ref{blow}), monotonicity
implies that $R\mapsto\Theta(\mathbf{w},0,R)$ is constant and thus equal to $\theta^\chi(\mathbf{w},0)$, and the claim follows
once we let $R\to\infty$.

It remains to show that $\varphi\circ\hat u_\infty=0$. By \eqref{grass.ub} 
we have
$$\nabla^{\mathcal P}\arctan\sigma=0\quad\text{for all }(\mathcal P,p)\in\operatorname{spt}(\mathbf{v}_{\infty,\C})\setminus\Pi^{-1}(0).$$
Since $\hat u_\infty(x)\neq 0$ for $x\neq0$, we deduce that
$$\nabla(\arctan\sigma\circ\hat u_\infty)=0\quad\text{a.e. on }\C\setminus\{0\},$$
and hence $\arctan\sigma\circ\hat u_\infty$ is constant on $\C\setminus\{0\}$.
This constant cannot be $\pm\frac{\pi}{2}$, since in this case $\hat u_\infty$
would take values in $\{z=0\}$ and hence we would have $\nabla \hat u_\infty=0$, a contradiction since $\hat u_\infty$
is proper (and thus not constant).

Thus, for $x\neq0$ we have $\rho\circ \hat u_\infty(x)>0$ and $\sigma\circ \hat u_\infty$ is constant.
Since $\hat u_\infty$ vanishes only at $0$,
taking an arbitrary $r>0$ and $0<R<\min_{\p B_r(0)}\rho\circ\hat u_\infty$,
we see that the induced varifold $\mathbf{v}_{\infty,B_r(0)}$ is a HSLV on $\{\rho<R\}$,
and by applying the same proof of Lemma \ref{const.sigma} we conclude that the constant is zero.
\end{proof}

\section{Regularity: inductive setup and base case}

\subsection{Inductive setup}
From now on, up to working in a local chart,
we assume that we have a PHSLV varifold $(\Omega,u,N)$, with $\Omega\subseteq\C$ a connected open set and $u$ nonconstant (i.e., $\nabla u\not\equiv0$), such that for some constant $C>0$ we have
\begin{equation}\label{upper.bd.assumption}
\int_{u^{-1}(B_R(p))}N|\nabla u|^2\,dx^2\le C_0R^2
\end{equation}
for all $p\in\mathbb{H}^2$ and all radii $R>0$. Our long-term goal is the following regularity theorem.

\begin{Th}\label{reg.thm.chart}
There exist two disjoint, locally finite sets of points $\mathcal{S}_{SW},\mathcal{S}_{branch}\subset\Omega$ such that:
\begin{itemize}[leftmargin=8mm]
\item[(i)] for any $x_0\in \mathcal{S}_{SW}$ and any sequence $r_j\to0$, the maps
$$x\mapsto \delta_{1/s_j}\circ\ell_{u(x_0)^{-1}}\circ u(x_0+r_jx),\quad s_j^2:=\int_{B_{r_j}(x_0)}|\nabla u|^2\,dx^2$$
converge in $C^0_{loc}\cap W^{1,2}_{loc}$ to a map $\C\to\{\varphi=0\}$ whose image is a non-flat Schoen--Wolfson cone, up to a subsequence;
\item[(ii)] on $\Omega\setminus\mathcal{S}_{SW}$, the map $u$ is smooth (in fact, a branched immersion with branch points at $\mathcal{S}_{branch}$);
\item[(iii)] on $\Omega\setminus(\mathcal{S}_{SW}\cup\mathcal{S}_{branch})$, the map $u$ is a smooth immersion;
\item[(iv)] $N$ is a.e. constant. \hfill $\Box$
\end{itemize}
\end{Th}

Recall that $\tilde\Omega\subseteq\Omega$ is the (open) set of points $x\in\Omega$
such that $u(x)\not\in u(\partial\omega)$ for some open set $x\in\omega\subset\subset\Omega$.
We let $\Omega_{nc}\subseteq\Omega$ be the (relatively closed) distributional support of $\nabla u$;
in other words, $\Omega\setminus\Omega_{nc}$ is the largest open subset of $\Omega$ where $u$ is locally constant.

\begin{Prop}\label{nc.is.omega}
We have $\Omega_{nc}=\tilde\Omega=\Omega$. \hfill $\Box$
\end{Prop}

\begin{proof}
Proposition \ref{in.tilde.omega} shows the inclusion $\Omega_{nc}\subseteq\tilde\Omega$.
Assume now that $\Omega_{nc}\neq\Omega$ and let $\Omega'$ be a connected component of the open set $\Omega\setminus\Omega_{nc}$.
Then $u$ is constant on $\Omega'$; up to a translation, we can assume that $u|_{\Omega'}=0$.

Since $\Omega$ is connected and the closed set $\Omega_{nc}\neq\emptyset$ by assumption,
the relative boundary $\Omega\cap\partial\Omega'$ is not empty. Taking an arbitrary $x$ on this boundary, and thus in $\Omega_{nc}$,
the proof of Proposition \ref{in.tilde.omega} gives arbitrarily small radii $\rho>0$ such that $0\nin u(\partial B_\rho(x))$.
Thus, the circle $\partial B_\rho(x)$ is disjoint from $\Omega'$. Since $\Omega'$ is connected and $x$ belongs to its closure, this implies that $\Omega'\subseteq B_\rho(x)$,
and hence $\Omega'=\emptyset$ (as $\rho$ was arbitrarily small and $x\nin\Omega'$), a contradiction.
\end{proof}

In particular, $\tilde N:\Omega\to[1,\infty)$ is defined at every point of $\Omega$.
From the definition of $\tilde N$ and \eqref{upper.bd.assumption}, we see that
$$\sup_{\Omega}\tilde N<\infty.$$
We let $\nu\in\N^*$ be such that
\begin{equation}\label{def.nu}\sup_{\Omega}\tilde N\in(\nu-1,\nu].\end{equation}
We will prove the regularity theorem by induction on $\nu$.
As in \cite{PiRi1}, we now define \emph{admissible} points.

\begin{Dfi}\label{adm.def}
We say that $x_0\in\Omega$ is \emph{strongly admissible} if the bound
$$\limsup_{r\to0}\frac{\int_{B_{2r}(x_0)}|\nabla u|^2\,dx^2}{\int_{B_r(x_0)}|\nabla u|^2\,dx^2}<\infty$$
holds true. We say that $x_0$ is \emph{admissible} if there exist $\Lambda>0$ and a sequence $r_k\to0$ such that
$$\int_{B_{2^jr_k}(x_0)}|\nabla u|^2\,dx^2\le 2^{2\Lambda j}\int_{B_{r_k}(x_0)}|\nabla u|^2\,dx^2\quad\text{for }j=1,\dots,k$$
holds true. \hfill $\Box$
\end{Dfi}

It is important that, while $\Lambda$ depends on $x_0$, it is independent of $k$.
This will allow us to have a blow-up defined on the full complex plane $\C$.
Indeed, thanks to Proposition \ref{nc.is.omega} we have $x_0\in\tilde\Omega$
for all $x_0\in\Omega$. Moreover, thanks to the previous proposition,
we can always form the blow-up at an admissible point, along a \emph{suitable} sequence of radii $r_k\to0$
(as opposed to \emph{any} sequence for a strongly admissible one),
in the sense that Proposition \ref{pr-tangent-cone} always applies.

\begin{Prop}\label{zero.dim}
The image through $u$ of the set of non-admissible points has $d_K$-Hausdorff dimension zero. \hfill $\Box$
\end{Prop}

\begin{proof}
Let us fix $\Lambda>0$ and let
$a_\ell:=\log_2[\int_{B_{2^{-\ell}}(x_0)}|\nabla u|^2\,dx^2]+2\Lambda\ell$, which is defined for $\ell$ large enough.
We assume that eventually
$$a_\ell\ge \Lambda\ell$$
and claim that $x_0$ is admissible.
If not, then we can find $k,\ell_0\in\N^*$ such that, for any $\ell\ge\ell_0$, there exists $j(\ell)\in\{1,\dots,k\}$ such that
$$a_{\ell-j}>a_\ell.$$
This implies that
$$b_\ell:=\max\{a_{\ell},a_{\ell+1},\dots,a_{\ell+k-1}\}=\max\{a_{\ell},a_{\ell+1},\dots,a_{\ell+k-1},a_{\ell+k}\}\ge b_{\ell+1}$$
forms a decreasing sequence for $\ell\ge\ell_0$. However, this contradicts the fact that
$b_\ell\ge \Lambda\ell$.

If instead, for any $\Lambda>0$, we have $a_\ell<\Lambda\ell$ for infinitely many indices $\ell$, then for these we have
$$\int_{B_{2^{-\ell}}(x_0)}|\nabla u|^2\,dx^2<2^{-\Lambda\ell}.$$
The image of such points has Hausdorff dimension zero, by Proposition \ref{pr-cont} (cf. \cite[Lemma 5.3]{PiRi1}).
\end{proof}

\begin{Rm}
In \cite{PiRi1} we showed that, in the isotropic setting, non-admissible points are removable singularities.
The analogous result here seems challenging, due to the lack (even a posteriori) of an elliptic PDE. Rather, we will show that, in fact,
non-admissible points do not exist; the drawback is that the latter has to be established \emph{in tandem} with the induction used to show regularity.
In fact, admissibility of all points is also stated in \cite[Proposition 4.2]{SW2},
with a different proof. Here we prefer an argument which looks much more natural, exploiting the principle
that the doubling bounds required in Definition \ref{adm.def} are satisfied in a blow-up.
\hfill $\Box$
\end{Rm}

The following tools will be used in the sequel.
We start from a fact which follows easily from the work carried out previously by the two authors in the isotropic setting \cite{PiRi1}.

\begin{Prop}\label{planar}
For a PHSLV $(\Omega,u,N)$ taking values in a Lagrangian plane $\mathcal P\subset\C^2\times\{0\}$, the map $u$ is holomorphic up to a suitable linear isometric identification
$\mathcal P\cong\C$, and moreover $N$ is constant a.e. \hfill $\Box$
\end{Prop}

Note that a Lagrangian plane $\mathcal P\subset\C^2\times\{0\}$ is the same as a Legendrian plane
(namely, a two-dimensional linear subspace $\mathcal{P}\subset\R^5$ such that $T_0\mathcal{P}\subset T_0\mathbb{H}^2$
is horizontal, i.e., belongs to $G$; it is automatic that the same holds at all points of $\mathcal{P}$).

\begin{proof}
Up to a rotation, we can assume that $\mathcal P=\operatorname{span}\{\partial_{z_1},\partial_{z_3}\}$.
Given any open set $\omega\subset\subset\Omega$ and any two smooth maps
$$a,b\in C^\infty_c(\mathcal P\setminus u(\partial\omega)),$$
we can take $F(z,\varphi):=-a(z_1,z_3)z_2-b(z_1,z_3)z_4$, which vanishes near $u(\p\omega)$ and has
$$2\operatorname{div}_{\mathcal P}W_F=\partial_{z_1}a+\partial_{z_3}b.$$
Thus, the fact that $(\Omega,u,N)$ is a PHSLV implies that it is also a parametrized stationary varifold, as defined in \cite{PiRi1}. The fact that $u$ is holomorphic then follows from \cite[Theorem 3.7]{PiRi1}
(see also \cite[Theorem 3.3]{PiRi1} for the simpler case of a proper map).
Working away from the locally finite set $\{\nabla u=0\}$, we can consider an open set $\omega\subset\subset\Omega$ such that $u|_\omega$ is a diffeomorphism with its image.
The induced varifold $\mathbf{v}_\omega$ has constant density $\theta_0\in\N^*$ by the constancy theorem.
From the definition of $\tilde N$, it is straightforward to conclude that $\tilde N=\theta_0$ on $\omega$.
\end{proof}

The following is an important observation which essentially follows from the previous two results.

\begin{Prop}\label{rigidity}
Assume that the PHSLV $(\C,\hat u_\infty,\hat N_\infty)$ is a blow-up at an admissible point $x_0$,
as in Proposition \ref{pr-tangent-cone},
and recall that in this case $\tilde N_\infty(x)\le\tilde N_\infty(0)$ for all $x\in\C$.
If equality holds at some point $x\neq0$, then $\hat u_\infty$ is a holomorphic map with values in a Lagrangian plane
$\mathcal P\subset\C^2\times\{0\}$, and $\hat N_\infty$ is constant a.e. \hfill $\Box$
\end{Prop}


\begin{proof}
We use the same notation of the proof as in Proposition \ref{pr-tangent-cone}.
Assuming that $\tilde N_\infty(x)=\tilde N_\infty(0)$ for some $x\neq0$, then letting $p:=\hat u_\infty(x)\neq0$ (as $\hat u_\infty^{-1}(0)=\{0\}$) and recalling that
$$\Theta(\mathbf{w},0,R)=\theta^\chi(\mathbf{w},0)=2\pi\tilde N_\infty(0)\quad\text{for all }R>0,$$
as in the derivation of \eqref{rig.pre} by monotonicity we get
\begin{align*}
&2\pi\tilde N_\infty(x)+\int_{G\setminus\Pi^{-1}(p)}\chi(\r_p/R)|\nabla^{\mathcal P}\arctan\sigma_p|^2\,d\mathbf{w}(\mathcal P,p) \\
&\le\theta^\chi(\mathbf{w},p)+\int_{G\setminus\Pi^{-1}(p)}\chi(\r_p/R)|\nabla^{\mathcal P}\arctan\sigma_p|^2\,d\mathbf{w}(\mathcal P,p) \\
&\le\Theta(\mathbf{w},p,R) \\
&\le\Theta(\mathbf{w},0,R)+O(R^{-1}) \\
&=2\pi\tilde N_\infty(0)+O(R^{-1}).
\end{align*}
Letting $R\to\infty$, we deduce that
$$2\pi\tilde N_\infty(x)+\int_{G\setminus\Pi^{-1}(p)}|\nabla^{\mathcal P}\arctan\sigma_p|^2\,d\mathbf{w}(\mathcal P,p)
\le2\pi\tilde N_\infty(0),$$
and hence, thanks to the assumption that $\tilde N_\infty(x)=\tilde N_\infty(0)$,
we deduce that
$$\int_{\hat u_\infty\neq p}|\nabla\arctan(\sigma_p\circ\hat u_\infty)|^2\,dx^2=0.$$
Using the properness of $\hat u_\infty$ (and the fact that $\hat u_\infty$ cannot take values into $\{z_p=0\}=\{z=z(p)\}$),
we conclude exactly as in the proof of Proposition \ref{pr-tangent-cone} that
$$\varphi_p\circ \hat u_\infty=0.$$

In the sequel, we let $v:=\pi\circ\hat u_\infty$. Recalling that $\varphi\circ\hat u_\infty=0$, we then have $\hat u_\infty=(v,0)$.
Writing
$$p=\hat u_\infty(x)=(\hat z,0)\neq 0,$$
the fact that $\varphi_p\circ\hat u_\infty=0$ says that
$$0=\varphi\circ(p^{-1}*v)=-\hat z_1 v_2+\hat z_2 v_1-\hat z_3 v_4+\hat z_4 v_3,$$
or equivalently $J\hat z\perp v$ at every point, where $J(z):=iz$ in $\C^2$. Hence, assuming up to a rotation that $\hat z=e_3$, we obtain that $\hat u_\infty$ takes values in the three-dimensional space $\{z_4=\varphi=0\}$.

Recall that, by the Legendrian condition, the fact that $\varphi\circ\hat u_\infty=0$ says also that the vector
$Jv(x)=(-v_2(x),v_1(x),0,v_3(x))$ is perpendicular to the image of $\nabla v(x)$ at a.e. $x\in\C$. In other words, $(-v_2,v_1)$
is a.e. perpendicular to the image of $\nabla(v_1,v_2)$. Hence, the map
$$\hat{v}:=\frac{(v_1,v_2)}{|(v_1,v_2)|},$$
defined on the open set $\C\setminus\{v_1=v_2=0\}$, is locally constant.
Thus, on each connected component $\Omega'$, calling $a(\Omega')\in\R^2\subset\R^5$ the constant value, the map $v$ takes values in the Lagrangian plane
$\mathcal{P}(\Omega'):=\operatorname{span}\{a(\Omega'),e_3\}$.

By Proposition \ref{planar}, we can find a linear isometry $L(\Omega'):P(\Omega')\to\C$ such that $te_3\mapsto it$ for all $t\in\R$ and
such that $L(\Omega')\circ v$ is holomorphic. It is immediate to check that the map $h:\C\to\C$
given by this composition on each $\Omega'$ and by $0$ on $\{v_1=v_2=0\}$
is still $W^{1,2}$ and satisfies the Cauchy--Riemann equations.
Hence, $h$ is holomorphic. It follows that the set $\{v_1=v_2=0\}=\{h=0\}$ is discrete
and that, in fact, there is only one connected component $\Omega'$. Thus, $\hat u_\infty$ takes values in a Lagrangian plane and,
by Proposition \ref{planar}, $\hat N_\infty$ is constant a.e.
\end{proof}

We conclude with a sort of $\ep$-regularity result, which follows from an analogous statement in the work of Schoen--Wolfson \cite{SW2},
in which it was one of the fundamental tools for the regularity theory of minimizers.

\begin{Prop}\label{ep.reg}
There exists a universal constant $\ep_0>0$ with the following property.
Assume that $L:\R^2\to\mathcal{P}$ is a linear isometry taking values in a Legendrian plane $\mathcal{P}$ and assume that on a ball $B_r(x_0)\subset\Omega$ we have
$$\sup_{x\in B_r(x_0)}d_K(\ell_{u(x_0)^{-1}}\circ u(x),L(x))^2<\ep_0 r^2,$$
as well as
$$\int_{B_r(x_0)}[|\p_{x_1}u(x)-Z_1(u(x))|^2+|\p_{x_2}u(x)-Z_2(u(x))|^2]\,dx^2<\ep_0 r^2,$$
where $Z_1,Z_2$ are orthonormal, left-invariant vector fields such that $\mathcal{P}=\operatorname{span}\{Z_1(0),Z_2(0)\}$.
Moreover, assume that $N$ is a.e. constant on $\mathcal{G}_u^f$.
Then $u$ is a smooth embedding on $B_{r/2}(x_0)$. \hfill $\Box$
\end{Prop}

\begin{proof}
Recall that, in the definition of stationarity for a PHSLV, the value of $N$ matters only on $\mathcal{G}_u^f$.
Since stationarity is not affected if we multiply $N$ by a constant, we can then assume without loss of generality that
$N=1$ on all of $\Omega$. The result now follows from the proof of \cite[Theorem 4.1]{SW2}
(note carefully that, although the regularity theory in \cite{SW2} deals with minimizers, the minimality assumption is not used in the proof of this result).
\end{proof}

\subsection{\texorpdfstring{Base case of the induction: $\bm{\nu=1}$}{Base case of the induction}}

We now deal with the base case of the inductive argument where $\nu=1$. Recalling \eqref{def.nu} and the fact that $\tilde N(x)\ge1$ for all $x\in\tilde\Omega=\Omega$,
this means that
$$\tilde N(x)=1\quad\text{for all }x\in\Omega.$$
Recalling that $u$ is not constant and $\Omega$ is connected, we claim that the map $u$ is a smooth immersion.
Obviously, to prove this, it is enough to show that $u$ is a smooth embedding near any point $x\in\Omega$. We begin with a simpler case.

\begin{Prop}\label{adm.smooth}
If $x_0\in\Omega$ is admissible, then $u$ is a smooth embedding near $x_0$. \hfill $\Box$
\end{Prop}

\begin{proof}
Up to a left translation in $\mathbb{H}^2$, we can assume that $u(x_0)=0$.
We consider a parametrized blow-up at $x_0$, provided by Proposition \ref{pr-tangent-cone},
taken along a suitable sequence $r_k\to0$. Let $(\C,\hat u_\infty,\hat N_\infty)$ be the resulting PHSLV
and recall that
$$u_\infty=\hat u_\infty\circ\psi,$$
where $\psi:\C\to\C$ is a quasiconformal homeomorphism and $u_\infty$ is a $C^0_{loc}$ limit of rescalings $u_k$ of the map $u$.

Recalling that $\tilde N_\infty\ge1$ by Proposition \ref{pr-robust}, we also know from Proposition \ref{pr-tangent-cone} that
$$1\le \tilde N_\infty\le \tilde N_\infty(0)=\tilde N(x_0)=1.$$
Hence, $\tilde N_\infty=1$ is constant. By Proposition \ref{rigidity}, it follows that $\hat N_\infty$ is a.e. constant and
$\hat u_\infty$ is a proper holomorphic map taking values in a Lagrangian plane $\mathcal{P}\subset\C^2\times\{0\}$.
Since $\hat N_\infty=\tilde N_\infty=1$ a.e. on $\mathcal{G}_{\hat u_\infty}^f$ by Proposition \ref{pr-robust},
we have $\hat N_\infty=1$ a.e.

Moreover, the map $\hat u_\infty$ is injective, since if $p:=\hat u_\infty(x')=\hat u_\infty(x'')\neq0$
for two points $x'\neq x''$ then, using the same notation as in the proof of Proposition \ref{pr-tangent-cone}
(arguing as in the proof of Proposition \ref{pr-fiber} and noting that the fibers of $\hat u_\infty$ are finite sets), we see that
$$2\pi\tilde N_\infty(x')+2\pi\tilde N_\infty(x'')\le\theta^\chi(\mathbf{w},p)\le\theta^\chi(\mathbf{w},0)=2\pi\tilde N_\infty(0),$$
contradicting the fact that $\tilde N_\infty(x')=\tilde N_\infty(x'')=\tilde N_\infty(0)=1$.
Since $\hat u_\infty$ is an injective holomorphic map with $\hat u_\infty(0)=0$, it is linear.

Finally, recall from the proof of Theorem \ref{th-sequential-weak-closure} that we have the convergence of Radon measures
$$N_k\frac{|\nabla u_k|^2}{2}\,dx^2\rightharpoonup N_\infty|\partial_{x_1}u_\infty\wedge\partial_{x_2}u_\infty|\,dx^2,$$
where $N_\infty=\hat N_\infty\circ\psi=1$ a.e. (as quasiconformal homeomorphisms preserve negligibility of sets).
Since $N_k=\tilde N_k=1$ a.e. on $\mathcal{G}_{u_k}^f$, we conclude that
$$\frac{|\nabla u_k|^2}{2}\,dx^2\rightharpoonup |\partial_{x_1}u_\infty\wedge\partial_{x_2}u_\infty|\,dx^2
\le\frac{|\nabla u_\infty|^2}{2}\,dx^2.$$
By lower semi-continuity of the Dirichlet energy, this implies that the last inequality is an equality
and that the weak convergence $u_k\rightharpoonup u_\infty$ in $W^{1,2}_{loc}$ is in fact a strong one.

Hence, $u_\infty$ is already weakly conformal, in which case $\psi$ can be taken to be the identity (cf. the proof of Theorem \ref{pr-tangent-cone}). Thus, we have
$$u_\infty=\hat u_\infty.$$
Since $u_k\to u_\infty$ in $C^0_{loc}\cap W^{1,2}_{loc}(\C,\mathbb{H}^2)$, and since $u_k$ is a linear conformal map,
we conclude that (a rescaling of) $u_k$ eventually satisfies the assumptions of
Proposition \ref{ep.reg}. Thus, $u_k$ is a smooth embedding near the origin, for $k$ large enough,
which implies that $u$ is a smooth embedding near the point $x_0$.
\end{proof}

We now turn to the general case. The following argument will be useful also in the inductive step,
revealing that, in fact, all points in $\Omega_{nc}$ are strongly admissible.

\begin{Prop}\label{nc.full.stop}
Any point $x_0\in\Omega$ is strongly admissible, and thus
the map $u$ is a smooth immersion on all of $\Omega$. \hfill $\Box$
\end{Prop}

\begin{proof}
Fix $x_0\in\Omega$ and assume without loss of generality that $x_0=0$ and $u(x_0)=0$.
Let us assume by contradiction that $0\in\Omega$ is \emph{not} a strongly admissible point.
This means that we can find a sequence of radii
$r_j\to0$ and numbers $\ep_j\to0$ such that
\begin{equation}\label{too.fast.decay}\int_{B_{r_j/2}(0)}|\nabla u|^2\,dx^2=\ep_j^4s_j^2,\quad s_j^2:=\int_{B_{r_j}(0)}|\nabla u|^2\,dx^2.\end{equation}

Given any sequence $\tau_j\in[\ep_js_j,2\ep_js_j]$, let $D_j'=D_j'(\tau_j)$ be the connected component of $\{\mathfrak{r}\circ u<\tau_j\}$ containing $0$.
We claim that
\begin{equation}\label{incl.claim}B_{r_j/4}(0)\subseteq D_j'\subseteq B_{Cr_j}(0)\end{equation}
eventually,
for a constant $C>1$ independent of $j$.
The first inclusion follows directly from Proposition \ref{pr-cont} and our assumption, which gives
$$\operatorname{diam}_K u(B_{r_j/4}(0))\le C\ep_j^2s_j=o(\ep_js_j),$$
while the second one is obtained as in the proof of Proposition \ref{in.tilde.omega}:
for a large enough $\Lambda\ge2$, we can find $\rho_j\in (r_j,\Lambda r_j)$ such that
$$\r\circ u\ge\frac{s_j}{\sqrt{\log\Lambda}}\quad\text{on }\p B_{\rho_j}(0),$$
and the latter is greater than $2\ep_js_j\ge\tau_j$ eventually, giving
$$D_j'\subseteq B_{\rho_j}(0)\subseteq B_{\Lambda r_j}(0).$$

We now select a smooth domain $D_j$ in the following way:
for $j$ large enough we have $D_j'(2\ep_js_j)\subset\subset\Omega$
and, considering the compact set
$$K_j:=D_j'(2\ep_js_j)\cap\{\r\circ u\le\ep_js_j\},$$
we take a smooth domain $K_j\subset\omega_j\subset\subset D_j'(2\ep_js_j)$
and set $\Gamma_j':=\de\omega_j$.
Note that $\Gamma_j'$ a priori consists of one or more loops; we let $\Gamma_j$ denote the outer loop
and let $D_j\supseteq \omega_j\supseteq D_j'(\ep_js_j)$ be the domain enclosed by $\Gamma_j$, for which clearly we still have the inclusions
$$B_{r_j/4}(0)\subseteq D_j\subseteq B_{Cr_j}(0).$$

%

Moreover, by Lemma \ref{lm-energy-quant} (applied to a constant sequence) and Remark \ref{no.c.star}, since $\mathfrak{r}\circ u<2\ep_js_j$ on $\Gamma_j$,
we have $\mathfrak{r}\circ u\le C\ep_js_j$ on $D_j$, which by \eqref{upper.bd.assumption} gives
\begin{equation}\label{dj.ub}\int_{D_j}N|\nabla u|^2\,dx^2\le C(\ep_js_j)^2.\end{equation}

We now consider a conformal diffeomorphism
$$\phi_j:B_1(0)\to r_j^{-1}D_j\quad\text{with }\phi_j(0)=0.$$
By the classical Carathéodory theorem, $\phi_j$ extends to a homeomorphism $\partial B_1(0)\to\Gamma_j$.
We claim that $\phi_j$ converges (in $C^\infty_{loc}$) to a limit diffeomorphism $\phi_\infty:B_1(0)\to \phi_\infty(B_1(0))$.

Indeed, a limit in $C^\infty_{loc}$ exists since the maps $\phi_j$ are holomorphic and equi-bounded.
In order to conclude that $\phi_\infty$ is a diffeomorphism, it is enough to show that it is nonconstant
(see Lemma \ref{rouche}).
Assuming by contradiction that $\phi_\infty=x_0$ is constant, we observe that $\phi_j$ is bounded in $W^{1,2}$ since it is a sequence of injective conformal maps whose image has bounded area. Hence, they converge weakly to $\phi_\infty$ in $W^{1,2}$. Since the traces converge weakly in $H^{1/2}$,
they converge strongly in $L^2$ to $x_0$. Thus, the classical representation formula with the Poisson kernel shows that
$$\phi_j\to x_0\quad\text{in }C^\infty_{loc}(B_1(0)).$$
Since $\phi_j(0)=0$, we must have $x_0=0$. However, this contradicts the strong $L^2$ convergence of traces and the fact that $|\phi_j|\ge 1/4$ at the boundary $\partial B_1(0)$.

In the same way, we now consider another smooth domain
$0\in\Delta_j\subset\subset D_j$ diffeomorphic to the unit disk, such that
$$\mathfrak{r}\circ u\in(\ep_js_j/4,\ep_js_j/2)\quad\text{on }\partial \Delta_j.$$
Crucially, the next lemma gives also the inclusion
$$\phi_j^{-1}(r_j^{-1}\Delta_j)\subseteq B_{\alpha}(0)$$
eventually, for some $\alpha\in(0,1)$. Moreover, by monotonicity applied to the induced varifold $\mathbf{v}_{\Delta_j}$, we see that
$$\int_{\Delta_j}N|\nabla u|^2\,dx^2\ge|\mathbf{v}_{\Delta_j}|(B_{\ep_js_j/4}^\r(0))\ge c(\ep_js_j)^2.$$
As a consequence of this and \eqref{dj.ub},
there exist two other constants $0<c<C$ such that
$$c (\ep_js_j)^2\le\int_{B_{\alpha}(0)}[N|\nabla u|^2]\circ(r_j\phi_j)\frac{|r_j\nabla\phi_j|^2}{2}\,dx^2\le\int_{B_1(0)}[N|\nabla u|^2]\circ(r_j\phi_j)\frac{|r_j\nabla\phi_j|^2}{2}\,dx^2\le C(\ep_js_j)^2.$$
Setting
$$u_j:=\delta_{1/(\ep_js_j)}\circ u\circ(r_j\phi_j),\quad N_j:=N\circ(r_j\phi_j),$$
we then have
$$c\le\int_{B_\alpha(0)}N_j|\nabla u_j|^2\,dx^2\le\int_{B_1(0)}N_j|\nabla u_j|^2\,dx^2\le C.$$
We can then apply Theorem \ref{th-sequential-weak-closure}
and take a (subsequential) limit PHSLV $(B_1(0),\hat u_\infty,\hat N_\infty)$.
Note carefully that the limit map $\hat u_\infty$ is not constant, since \eqref{rad.limit}
implies that
$$\int_{\psi(B_\alpha(0))}|\nabla\hat u_\infty|^2\,dx^2>0$$
for a suitable quasiconformal homeomorphism $\psi$.

Recalling that $\tilde \Omega=\Omega$ and taking $0\in\omega_0\subset\subset\Omega$ such that $0\nin u(\p\omega_0)$,
as in the proof of Proposition \ref{pr-tangent-cone} we see that the induced limit varifold
$\mathbf{v}_\infty$ has support included in a blow-up of $\mathbf{v}_{\omega_0}$.
Moreover, for the (subsequential) Hausdorff limit
$$T:=\lim_{j\to\infty}\phi_j^{-1}(r_j^{-1}\p\Delta_j)\subseteq\bar B_\alpha,$$
we have $0\nin T$ since $u_\infty|_T$ takes values in $[1/4,1/2]$,
and moreover the connected component of $B_1(0)\setminus T$ containing $0$ is compactly included in $B_1(0)$,
since the same holds along the sequence (with $\phi_j^{-1}(r_j^{-1}\p\Delta_j)$ in place of $T$).
This shows that $0$ belongs to the domain of $\tilde N_\infty$.

Hence, arguing exactly as in the proof of Proposition \ref{pr-tangent-cone},
we obtain that $\hat u_\infty$ is not constant in any neighborhood of $0$.
Recalling \eqref{rad.limit},
this means that, for any fixed $\rho\in(0,1)$, there exists a constant $c(\rho)>0$ such that
$$\int_{B_\rho(0)}|\nabla u_j|^2\,dx^2>c(\rho)$$
eventually. Since $\phi_j$ converges to a limit diffeomorphism $\phi_\infty$, we also have
$\phi_j(B_\rho(0))\subseteq B_{C_0\rho}(0)$ eventually, for a constant $C_0>0$ independent of $\rho$ and $j$.
This immediately gives
$$\int_{B_{C_0\rho r_j}(0)}|\nabla u|^2\,dx^2\ge c(\rho)(\ep_js_j)^2.$$
Taking $\rho>0$ so small that $C_0\rho\le1/2$, we reach a contradiction with \eqref{too.fast.decay}.
\end{proof}

\begin{Rm}
At the end of the previous proof, we could have also concluded as in the proof of Proposition \ref{adm.smooth}:
since $\tilde N_\infty=1$ near $0$, we have $u_j\to u_\infty$ strongly in $W^{1,2}$ here and thus $\hat u_\infty=u_\infty$ is an injective holomorphic map near $0$,
which is then close to a linear conformal map at a small scale. Hence, Proposition \ref{ep.reg} applies to $u_j$, for $j$ large enough, showing that $u$ is a smooth embedding near $0$, a contradiction.
However, we preferred an argument which better generalizes to the situation of the inductive step. \hfill $\Box$
\end{Rm}
%

\begin{Lm}
\label{lm-sep-lev}
In the situation of the previous proof, there exists a constant $\alpha\in(0,1)$ such that
$$\phi_j^{-1}(r_j^{-1}\Delta_j)\subseteq B_\alpha(0)$$
for $j$ large enough. \hfill $\Box$
\end{Lm}

\begin{proof}
Let $A_j:=\bar B_1(0)\setminus {\phi_j^{-1}(r_j^{-1}\Delta_j)}$, which is diffeomorphic to a closed annulus.
It is well known (see, e.g., \cite[Theorem A.1]{DLP}) that there exist $\xi_j\in(0,1)$ and a conformal diffeomorphism
$$f_j:\bar B_1(0)\setminus B_{\xi_j}(0)\to A_j$$
which maps $\p B_{\xi_j}(0)\to \phi_j^{-1}(r_j^{-1}\p\Delta_j)$ and $\p B_1(0)\to\p B_1(0)$.

We claim that
$$\beta:=\limsup_{j\to\infty}\xi_j<1.$$
Indeed, assuming by contradiction that $\xi_j\to1$ along a subsequence,
by \eqref{dj.ub} we have
$$\int_{B_1(0)}|\nabla\tilde u_j|^2\,dx^2\le C(\ep_js_j)^2,$$
where $\tilde u_j:=u\circ(r_j\phi_j)\circ f_j$, and thus by Cauchy--Schwarz we deduce that
$$\int_{B_1(0)\setminus B_{\xi_j}(0)}|\nabla\tilde u_j|\,dx^2=o(\ep_js_j).$$
In particular, we can select $\theta_j\in\R/2\pi$ such that $\tilde u_j$
restricts to a $W^{1,1}$ function along the segment $S_j:=[\xi_je^{i\theta_j},e^{i\theta_j}]$,
with
$$\int_{S_j}|\tilde u_j'|\,d\mathcal{H}^1=o(\ep_js_j).$$
Recalling \eqref{nabla.r.norm}, this implies that the oscillation of
$\r\circ\tilde u_j|_{S_j}$ is $o(\ep_js_j)$. This contradicts the fact that the oscillation is
at least $\r\circ\tilde u_j(e^{i\theta_j})-\r\circ\tilde u_j(\xi_je^{i\theta_j})>\ep_js_j/2$
(recall that $\r\circ u>\ep_js_j$ on $\de D_j$ and $\r\circ u<\ep_js_j/2$ on $\de\Delta_j$).

Now, taking any $\beta'\in(\beta,1)$, we claim that
\begin{equation}\label{no.thinner}\limsup_{j\to\infty}\max_{x\in\p B_{\beta'}(0)}|f_j(x)|<1.\end{equation}
Indeed, assume that $|f_j(x_j)|\to1$ along a subsequence, for some points $x_j\in\p B_{\beta'}(0)$. Up to a further subsequence, we can assume that $x_j\to x_\infty$
and $f_j\to f_\infty$ locally uniformly, for a holomorphic map $f_\infty:B_1(0)\setminus \bar B_\beta(0)\to\C$.
Since $B_{r_j/4}(0)\subseteq\Delta_j$, we have
$$|\phi_j\circ f_j|\ge\frac{1}{4}.$$
If $f_\infty=0$ then, since $\phi_j\to\phi_\infty$ uniformly near $0$, we would have $\phi_j\circ f_j\to\phi_\infty(0)=0$,
a contradiction.
%
Since $f_j/|f_j|$ has degree $1$ from $\p B_1(0)$ to itself,
it also has degree $1$ from $\p B_{\beta'}(0)$ to $\p B_1(0)$, showing that $f_j$ cannot converge to a nonzero constant on $\p B_{\beta'}(0)$ either.

Thus, $f_\infty$ is not constant, and from Lemma \ref{rouche} we deduce that it is a conformal diffeomorphism.
Since $|f_j|\le1$, we have $|f_\infty|\le1$ as well, and hence $|f_\infty|<1$ by the maximum modulus principle for nonconstant holomorphic maps. Hence, we have
$$|f_j(x_j)|\to|f_\infty(x_\infty)|<1,$$
contradicting the assumption that $|f_j(x_j)|\to1$. Finally, using the fact that $f_j$ is an orientation-preserving diffeomorphism,
it is easy to deduce from \eqref{no.thinner} that
$$\limsup_{j\to\infty}\max_{x\in\p B_{\xi_j}(0)}|f_j(x)|<1,$$
as desired.
\end{proof}

We also used the following well-known fact.

\begin{Lm}\label{rouche}
Assume that we have a sequence of conformal diffeomorphisms $f_j:U_j\to f_j(U_j)$, with $U_j\subseteq\C$ open,
converging to $$f_\infty:U_\infty\to\C$$ locally uniformly, for an open connected $\emptyset\neq U_\infty\subseteq\C$ (meaning that, for any compact $K\subset U$,
we have $K\subseteq U_j$ eventually and $f_j|_K\to f_\infty|_K$ uniformly). If $f_\infty$ is not constant,
then $f_\infty$ is a conformal diffeomorphism. \hfill $\Box$
\end{Lm}

\begin{proof}
Since the maps $f_j$ are holomorphic, the limit $f_\infty$ is also holomorphic.
Assuming that $f_\infty$ is not constant, in order to prove the statement it suffices to show that
$f_\infty$ is injective. Assuming that $f_\infty(x_0)=f_\infty(x_1)$ for two points $x_0\neq x_1$,
up to subtracting a constant we can assume that
$$f_\infty(x_0)=f_\infty(x_1)=0.$$
Since $U_\infty$ is connected and $f_\infty^{-1}(0)$ is discrete, we can enclose $\{x_0,x_1\}$ with a smooth Jordan curve $\gamma$, bounding a subset of $U_\infty$ and whose image avoids the zeros of $f_\infty$.
Using complex notation, the classical Rouch\'e's theorem then gives the contradiction
$$1\ge\frac{1}{2\pi i}\int_\gamma \frac{f_j'(z)}{f_j(z)}\,dz\to\frac{1}{2\pi i}\int_\gamma \frac{f_\infty'(z)}{f_\infty(z)}\,dz\ge 2,$$
since the two integrals count the number of zeros (with multiplicity) enclosed by $\gamma$,
for the two functions $f_j$ and $f_\infty$, respectively.
\end{proof}

\section{Inductive step: preparation}

\subsection{Classification of tangent cones}

Recalling Proposition \ref{nc.is.omega} and \eqref{def.nu},
we now assume that we have
$$\sup_{\Omega}\tilde N\in(\nu-1,\nu]\quad\text{for some }\nu\ge2$$
and we assume inductively that
the regularity result, i.e., Theorem \ref{reg.thm.chart} holds for any PHSLV such that this supremum is at most $\nu-1$.

\begin{Dfi}
From now on, we will use the following notation for the rescalings:
\begin{equation}\label{not.res}
u_{x_0,r}(x):=\delta_{1/s(x_0,r)}\circ\ell_{u(x_0)^{-1}}\circ u(x_0+rx),\quad s(x_0,r)^2:=\int_{B_r(x_0)}|\nabla u|^2\,dx^2.
\end{equation}
Similarly, we let $N_{x_0,r}(x):=N(x_0+rx)$. These two functions $u_{x_0,r}$ and $N_{x_0,r}$ are defined on the dilated set $\Omega_{x_0,r}:=r^{-1}(\Omega-x_0)$.
We will often omit the subscript $x_0$, when the reference point $x_0$ is clear from the context. \hfill $\Box$
\end{Dfi}

In this part we want to show the following structure theorem for tangent cones,
which can be either flat planes or non-flat Schoen--Wolfson cones, described in \cite[Section 7]{SW2}.

\begin{Prop}\label{reg.tg.cone}
Assume that $x_0\in\Omega$ is an admissible point and consider
a blow-up PHSLV $(\C,\hat u_0,\hat N_0)$ arising as the limit of $(\Omega_r,u_r,N_r)$ along a sequence $r\to0$,
as in Proposition \ref{pr-tangent-cone}.
Then its image is either a Lagrangian plane $\mathcal{P}\subset\C^2\times\{0\}$ or a non-flat Schoen--Wolfson cone.
In the first case, $\hat u_0$ is a polynomial of the form
$$\hat u_0(z)=cz^k\quad \text{with }k\in\{1,\dots,\nu\},$$
after we suitably identify $\mathcal P\cong\C$,
while in the second case it is smooth on $\C\setminus\{0\}$ and it is an immersion here,
outside a locally finite subset of $\C\setminus\{0\}$. \hfill $\Box$
\end{Prop}

Let $\mathcal{A}_0$ denote the set of admissible points for $\hat u_0$.
We now make the following observation, which exploits our previous understanding of the flat case.

\begin{Lm}
If $x\in\mathcal{A}_0\setminus\{0\}$, then $\tilde N_0(x)\in\N$ and, for any blow-up at $x$, the parametrization is a homogeneous polynomial. \hfill $\Box$
\end{Lm}

\begin{proof}
Note that \eqref{upper.bd.assumption} is inherited by the blow-up. Since $\hat u_0$ is not constant,
we can apply Proposition \ref{nc.is.omega} and deduce that we can form blow-ups of $(\C,\hat u_0,\hat N_0)$ at all points of $\mathcal{A}_0$.

Let $x_0'\in\mathcal{A}_0\setminus\{0\}$ and consider a blow-up at $x_0'$.
We claim that the limit PHSLV, which we denote by $(\C,\hat u_{0,0},\hat N_{0,0})$, takes values in a Lagrangian plane $\mathcal{P}\subset\C^2\times\{0\}$. Once this is done, we conclude as in the proof of Proposition \ref{rigidity} that this iterated blow-up is a parametrized stationary varifold,
whose corresponding varifold is supported on $\mathcal P$ and has constant integer density $\theta_0$ (recall that $\hat u_{0,0}$ is proper), and hence
$$\tilde N_0(x_0')=\tilde N_{0,0}(0)=\theta_0\in\N.$$
Since proper holomorphic maps $\C\to\C$ are polynomials and $\hat u_0^{-1}(0)=\{0\}$, the statement follows.

To prove the claim, recall that a blow-up takes values into $\{\varphi=0\}$ and the Legendrian condition then says that
$J\hat u_0(x)$ is orthogonal to the image of $\nabla \hat u_0(x)$.
Thus, we see that on any given compact set of $\C$ we have
$$|\Pi_{J\hat u_0(x_0')}\circ\nabla (\hat u_0)_{x_0',r}|\le\ep(r)|\nabla (\hat u_0)_{x_0',r}|,$$
for a vanishing function $\ep(r)\to0$, where $\Pi_{J\hat u_0(x_0')}$ is the orthogonal projection onto $\operatorname{span}\{J\hat u_0(x_0')\}$.

We deduce that the limit PHSLV $(\C,\hat u_{0,0},\hat N_{0,0})$ takes values into $\{\varphi=0\}\cap J\hat u_0(x_0')^\perp$.
Up to a rotation, let us assume that $J\hat u_0(x_0')=e_4$, so that $\hat u_{0,0}$ takes values into $\{z_4=\varphi=0\}$.
Moreover, the vector $J\hat u_{0,0}(x)$ is perpendicular to the image of $\nabla \hat u_{0,0}$, as well.
We now conclude exactly as in the proof of Proposition \ref{rigidity}.
\end{proof}

Crucially, thanks to the previous proposition, if
$\tilde N_0(x)>\nu-1$ at a point $x\in\mathcal{A}_0\setminus\{0\}$
then $\tilde N_0(x)\ge\nu$, and thus $\tilde N_0(x)=\nu$ since $\nu$ is the highest possible value.
Thus, by Proposition \ref{rigidity}, the statement of
Proposition \ref{reg.tg.cone} holds in this case.
Hence, in the sequel we can assume that
\begin{equation}\label{adm.nu.minus.one}
\tilde N_0\le\nu-1\quad\text{on }\mathcal{A}_0\setminus\{0\}.
\end{equation}

\begin{Prop}\label{adm.open}
The set $\mathcal{A}_0\setminus\{0\}$ is open and consists of strongly admissible points. \hfill $\Box$
\end{Prop}

\begin{proof}
%
By a simple compactness argument, outlined below for completeness (cf. also the proof of Lemma \ref{ahlfors}), we can show the following fact: given $\lambda>0$,
there exists $\eta\in(0,1/2)$ (depending only on $\lambda$, $\nu$, and the constant $C_0$ appearing in \eqref{upper.bd.assumption}) such that
if
\begin{equation}\label{many.radii}
\int_{B_{r/2}(x)}|\nabla \hat u_0|^2\,dx^2>\lambda\int_{B_{r}(x)}|\nabla \hat u_0|^2\,dx^2,
\end{equation}
as well as the pinching
\begin{equation}\label{pinching.cond}
\Theta(\mathbf{v}_{0,\om},\hat u_0(x),s(x,r)/\eta)<2\pi\tilde N_0(x)+\eta
\end{equation}
for some open $x\in\om\subset\subset\C$ such that $\bar B_{2s(x,r)/\eta}^\r(\hat u_0(x))\cap \hat u_0(\p\om)=\emptyset$,
where we use the notation \eqref{capital.theta.def}, so that by monotonicity
$$\theta^\chi(\mathbf{v}_{0,\om},\hat u_0(x))\le\Theta(\mathbf{v}_{0,\om},\hat u_0(x),R)<\theta^\chi(\mathbf{v}_{0,\om},\hat u_0(x))+\eta$$
for all $0<R<s(x,r)/\eta$, and
\begin{equation}\label{J.close}|J\hat u_0-J\hat u_0(x)|<\eta|J\hat u_0(x)|\quad\text{on }B_r(x),\end{equation}
then the same bounds hold for a smaller radius $r'\in(\eta r,r/2)$, as long as $\lambda>\lambda_0(\nu,C_0)>0$.

To check this claim we argue by contradiction, rescaling $B_r(x)$ to $B_1(0)$ and replacing $\hat u_0$ with $(\hat u_0)_{x,r}$. In the limit $\eta\to0$ we end up with a PHSLV $(B,\hat w,\hat N_{\hat w})$, where $B=\psi(B_1(0))$ for a suitable quasiconformal homeomorphism.
This PHSLV is nontrivial thanks to \eqref{many.radii} and \eqref{rad.limit}, and the induced varifold $\mathbf{v}_{\hat w,B}$ satisfies
$$\mathbf{v}_{\hat w,B}\le\tilde{\mathbf{v}},$$
where $\tilde{\mathbf{v}}$ is a HSLV which has $\Theta(\tilde{\mathbf{v}},0,R)$ constant in $R\in(0,\infty)$ and equal to the limit (along the sequence $\eta\to0$) of $2\pi\tilde N_0(x)$.
Since $2\pi\tilde N_{\hat w}(0)$ is at least this limit by Corollary \ref{dens.limits},
by arguing as in Proposition \ref{pr-fiber} we see that $\hat w^{-1}(0)=\{0\}$ (recall also Proposition \ref{nc.is.omega}).

Moreover, we have $\nabla^{\mathcal P}\arctan\sigma=0$
on $\operatorname{spt}(\tilde{\mathbf{v}})\setminus\Pi^{-1}(0)$ (see Remark \ref{blow}). Hence, we have
$$\nabla\arctan(\sigma\circ \hat w)=0\quad\text{on }B\setminus\{\hat w=0\}$$
and, as at the end of the proof of Proposition \ref{pr-tangent-cone}, we conclude that $\varphi\circ\hat w=0$.

Note that, since $J\hat u_0$ is perpendicular to the image of $\nabla \hat u_0$ at all points, \eqref{J.close} implies that
$x\neq0$ and
$$|\Pi_{J\hat u_0(x)}\circ\nabla \hat u_0|\le\eta|\nabla \hat u_0|\quad\text{on }B_r(x).$$
Thus, the image of $\nabla \hat w$ is orthogonal to the limit (along the sequence $\eta\to0$) of $\frac{J\hat u_0(x)}{|\hat u_0(x)|}$.
Hence, $\hat w$ is a holomorphic map taking values in a plane, by the same argument used in the previous proof. Since $\mathbf{v}_{\hat w,B}$ has density at most $2\pi\nu$ at the origin,
the degree of $\hat w$ at $0$ is at most $\nu$. We obtain the desired contradiction if $\hat w$ is a strong $W^{1,2}_{loc}$ limit,
since \eqref{many.radii} clearly holds with $\hat u_0$ replaced by $z\mapsto z^k$, for some $k\in\{1,\dots,\nu\}$,
as long as $\lambda\ge\lambda_0(\nu)$.

In general, we can use \eqref{rad.limit} together with the distortion bounds from Lemma \ref{distortion} (recall from \eqref{dist.bd} that  $\psi$ is $\frac{C_0^2}{\pi^2}$-quasiconformal, where $C_0$ is the constant appearing in \eqref{upper.bd.assumption}).

Moreover, with the same proof, we can find another constant $0<\eta'<\eta$ such that, whenever the previous conditions are satisfied with $\eta'$ in place of $\eta$, we have
$$\int_{B_{\eta r}}|\nabla \hat u_0|^2\,dx^2\ge\eta'\int_{B_{r}}|\nabla \hat u_0|^2\,dx^2.$$

By iterating the previous two facts, it follows that any $x\in\C\setminus\{0\}$ satisfying these conditions (with $\eta'$) is strongly admissible.
Clearly, any point in $\mathcal{A}_0\setminus\{0\}$ fits the previous conditions for some $\lambda>0$ and $r>0$.
Hence, $\mathcal{A}_0\setminus\{0\}$ consists of strongly admissible points.

Moreover, the set $\{\tilde N_0<\nu-1\}$ is open (by upper semi-continuity of $\tilde N_0$) and here, by inductive assumption,
the regularity theorem holds; in particular, all points in this set are strongly admissible, i.e.,
$$\{\tilde N_0<\nu-1\}\subseteq\mathcal{A}_0.$$
To conclude, thanks to \eqref{adm.nu.minus.one} and the last inclusion, it suffices to show the following: given $x_0\in\mathcal{A}_0\cap\{\tilde N_0=\nu-1\}\setminus\{0\}$ and a sequence $x_k\to x_0$,
we have $x_k\in\mathcal{A}_0$ for $k$ large enough. If $\tilde N_0(x_k)<\nu-1$, this holds by the previous inclusion.
We can then assume that
$$\tilde N_0(x_k)=\nu-1\quad\text{for all }k.$$
Since $x_0$ is admissible, it satisfies \eqref{many.radii} for some fixed $\lambda>0$, for any $r>0$ small enough (with $\eta'$ in place of $\eta$), as well as \eqref{J.close}. Moreover, we can select $\omega$ such that \eqref{pinching.cond} holds with $\eta'$ in place of
$\eta$, again for $r>0$ small (cf. the proof of Proposition \ref{nc.is.omega}). Since these are open conditions on the set $\{\tilde N_0=\nu-1\}$,
we conclude that $x_k$ satisfies them eventually, and hence is strongly admissible.
\end{proof}

\begin{Lm}\label{distortion}
Given a $K$-quasiconformal homeomorphism $\psi:A\to A'$ between two open sets $A,A'\subseteq\C$, there exists a constant $C_0(K)$ such that the following holds:
for any disk $B_r(x)$ with $B_{C_0(K)r}(x)\subseteq A$ we have
$$B_{r'}(x')\subseteq\psi(B_r(x))\subseteq\psi(B_{2r}(x))\subseteq B_{C_0(K)r'}(x'),$$
where $x':=\psi(x)$ and $r':=\min_{y\in\p B_r(x)}|\psi(y)-x'|$. \hfill $\Box$
\end{Lm}

\begin{proof}
It is a well-known consequence of Teichm\"uller's modulus theorem (see, e.g., the last part of the proof of \cite[Lemma 5.4]{PiRi1})
that the statement holds when $A=\C$.

In the more general form written here, it follows from the entire version by a standard compactness argument, which we now detail.
Assume by contradiction that the statement fails for a sequence of
maps $\psi_k$, with $k=1,2,\dots$ in place of $C_0(K)$. For each of these counterexamples, we can assume
without loss of generality that $x=\psi_k(x)=0$ and $r=1$, as well as $\psi_k(1)=1$.

We now take the Beltrami coefficients $\mu_k:B_k(0)\to\C$ such that
$$\p_{\bar z}\psi_k=\mu_k\p_z\psi_k$$
and extend them by zero on $\C\setminus B_k(0)$, and we consider the normal solution
$F^{\mu_k}:\C\to\C$ to the same equation (see, e.g., \cite[Theorem 4.24]{IT}).
By the chain rule (see, e.g., \cite[Lemma III.6.4]{LV}), we see that
$$f_k:=\psi_k\circ (F^{\mu_k})^{-1}:F^{\mu_k}(B_k(0))\to\psi_k(B_k(0))$$
is a conformal diffeomorphism with $f_k(0)=0$ and $f_k(1)=1$.
By standard compactness properties (see, e.g., \cite[Lemma A.3]{PiRi2}), along a subsequence
we have the local uniform convergence $F^{\mu_k}\to F_\infty$, and the same for the inverses, for a limit quasiconformal homeomorphism $F_\infty:\C\to\C$ (see however \cite[Remark A.5]{PiRi2}).

Likewise, for any fixed $R>1$, eventually we have $(F^{\mu_k})^{-1}(\bar B_R(0))\subset B_k(0)$
(as $(F^{\mu_k})^{-1}\to F_\infty^{-1}$ locally uniformly), and thus eventually $f_k$ is defined on $B_R(0)$.
Moreover, by Koebe's distortion theorem, the functions $\frac{f_k(Rz)}{Rf_k'(0)}$ are a normal family
(i.e., pre-compact in the topology of local uniform convergence on $\mathbb{D}=B_1(0)$), with univalued subsequential limits (by Lemma \ref{rouche}); since $f_k(1)=1$, it follows that
$$0<c\le|f_k'(0)|\le C.$$
Thus, the functions $f_k$ are a normal family on $B_R(0)$, for any $R>1$, and we can then extract a subsequential limit $f_\infty:\C\to\C$.
Since $f_\infty(0)=0$ and $f_\infty(1)=1$, by Lemma \ref{rouche} this map is a conformal diffeomorphism, and hence it is the identity.
As a consequence, we have
$$\psi_k=f_k\circ F^{\mu_k}\to f_\infty\circ F_\infty=F_\infty.$$
Since $F_\infty$ satisfies the statement, we obtain a contradiction.
\end{proof}

As a consequence, we can apply the inductive assumption: in particular, we know that on the open set $\mathcal{A}_0\setminus\{0\}$
the map $\hat u_0$ is a smooth immersion, away from a locally finite set of points.

\begin{Prop}\label{par.cone}
The image of $\hat u_0$ is either a Lagrangian plane $\mathcal{P}\subset\C^2\times\{0\}$ or a non-flat Schoen--Wolfson cone. \hfill $\Box$
\end{Prop}

\begin{proof}
Indeed, thanks to Proposition \ref{zero.dim}, we can select $\tau>0$ such that
$$\tau\not\in\mathfrak{r}\circ \hat u_0(\C\setminus\mathcal{A}_0).$$
We can also assume that $(\mathfrak{r}\circ \hat u_0)^{-1}(\tau)$ consists only of points where $\nabla\hat u_0\neq0$, since the set $$(\mathcal{A}_0\setminus \{0\})\cap\{\nabla \hat u_0=0\}$$ is at most countable,
and that $\tau$ is a regular value for $\mathfrak{r}\circ \hat u_0$, when this function is restricted to the open set $\mathcal{A}_0\setminus\{0\}$.

Then we can let $\omega$ be the connected component of $\{\r\circ\hat u_0<\tau\}$ containing $0$.
Let $\Gamma$ denote its smooth boundary and let $\eta:=\hat u_0|_\Gamma$.
Since $\hat u_0$ is a smooth immersion near $\Gamma$,
by the inverse function theorem we can find a tubular neighborhood of $\Gamma$,
diffeomorphic to $\Gamma\times(-\epsilon,\epsilon)$ through a diffeomorphism $\zeta$, such that
$$\hat u_0\circ\zeta^{-1}:\Gamma\times(-\epsilon,\epsilon)\to\C^2\times\{0\}=\C^2$$
is an immersion and the second component of $\psi$ is precisely $\mathfrak{r}\circ\hat u_0-\tau$.
For $s\in(-\epsilon,\epsilon)$, we let $\eta_s$ be the slice $\hat u_0\circ\psi^{-1}(\cdot,s)$.

Thus, $\eta_s$ takes values in the sphere $\partial B_{\tau+s}(0)\subset\C^2$ (note that in $\C^2$
the Kor\'anyi metric is just the Euclidean one). We consider
$$w:=\Pi\circ \hat u_0\circ\psi^{-1},$$
where $\Pi$ is the nearest point projection onto $\partial B_\tau(0)$.
Since $\hat u_0$ is Lagrangian and $J\hat u_0$ is orthogonal to the image of $\nabla \hat u_0$ at all points, 
the latter image contains the position vector $\hat u_0$.
Thus, the image of $\nabla(\hat u_0\circ\psi^{-1})(\theta,s)$ is the span of $\eta_s'(\theta)$ and the position vector $\hat u_0\circ\psi^{-1}(\theta,s)$, and we deduce that
the partial derivative $\partial_s w(\theta,s)$ is a multiple of $\eta_s'$. We deduce that $\eta_s$ is just a reparametrization of $\eta_0=\eta$,
up to decreasing $\epsilon$.

As a consequence, the image of $\hat u_0\circ\psi^{-1}$ takes the form
$$(\tau-\epsilon,\tau+\epsilon)\cdot \eta(\Gamma)\subset\C^2.$$
Since $w$ is an immersion, the PHSLV condition says that it is critical for the area among Lagrangian maps; hence, $\eta$ is an immersed curve on the three-dimensional sphere, critical for the length
with respect to the sub-Riemannian constraint that $\operatorname{span}\{\eta(\theta),\eta'(\theta)\}$ is a Lagrangian plane.
It follows that the image of each connected component of $\Gamma$ is the cross-section of a plane or a non-flat Schoen--Wolfson cone (cf. \cite[Section 7]{SW2}).

Moreover, $\Gamma$ must actually be connected: since $\omega$ is a smooth connected domain,
$\omega$ equals a disk $\tilde\omega$ with a finite number of inner disks $\omega_j$ removed.
It suffices to show that we cannot have any such domain $\omega_j$.
Since $\hat u_0^{-1}(0)=\{0\}$, we have $0\not\in \hat u_0(\omega_j)$.
In the same spirit of the convex hull property for minimal surfaces, we can consider the Hamiltonian vector field
$$W:=W_{\varphi\chi(\rho^2)},$$
where $\chi\ge0$, $\chi'\ge0$, $\chi=0$ on $[0,\tau^2]$, and $\chi>0$ on $(\tau^2,\infty)$.
Testing it with the PHSLV localized at $\omega_j$,
we obtain that $$\mathfrak{r}\circ\hat u_0\le\tau\quad\text{on }\omega_j.$$ This contradicts the fact that $\tau$ is a regular value for $\mathfrak{r}\circ \hat u_0$, which would imply that $\r\circ\hat u_0$ increases as we enter $\omega_j\subseteq\C\setminus\om$.

We now show that, in fact, there is no other connected component of $\{\mathfrak{r}\circ\hat u_0<\sigma\}$. Given another component $\omega'$, we can repeat the same argument: the image of its boundary is an immersed curve, which has positive $\mathcal{H}^1$-measure,
and hence we can find a unit vector $a\in S^3$ such that, for each $x$ in
$$S_a:=\lf\{x\in\omega'\,:\,\frac{\hat u_0}{|\hat u_0|}(x)=a\rg\},$$
we have $x\in\mathcal{A}_0\setminus\{0\}$ and $\nabla \hat u_0(x)\neq0$ (as well as  $S_a\neq\emptyset$). Since $a$ belongs to the image of $\nabla \hat u_0(x)$ at such points, it is easy to conclude that
$\omega'\cap S_a$ contains a curve whose composition with $\hat u_0$ converges to the origin.
Since $\hat u_0$ is a proper map, we deduce that $\omega'$ intersects $\hat u_0^{-1}(0)=\{0\}$, a contradiction.

In summary, $\{\mathfrak{r}\circ \hat u_0<\sigma\}$ is diffeomorphic to a disk.
Moreover, if we take two different regular values $\tau<\tau'$ as above, yielding two curves $\Gamma,\Gamma'\subset\C$,
the last argument shows that
$$\frac{\hat u_0(\Gamma)}{\tau}=\frac{\hat u_0(\Gamma')}{\tau'}.$$
Thus, the image of $\hat u$ includes a dense subset of $[0,\infty)\cdot \hat u_0(\Gamma)$,
and hence this set itself (as $\hat u_0$ is proper and hence has closed image).

In fact, the previous argument also shows that the image $S$ of $\frac{\hat u_0}{|\hat u_0|}$ (as a map $\C\setminus\{0\}\to S^3$)
satisfies
\begin{equation}\label{neg.remainder}\frac{\hat u_0(\Gamma)}{\tau}\subseteq S,\quad\mathcal{H}^1\lf(S\setminus\frac{\hat u_0(\Gamma)}{\tau}\rg)=0.\end{equation}
Since $S$ is connected, we must have $S=\frac{\hat u_0(\Gamma)}{\tau}$, as otherwise its composition with the distance function from $\frac{\hat u_0(\Gamma)}{\tau}$ would include an interval $(0,\ep')$, a contradiction to \eqref{neg.remainder}.
\end{proof}

\begin{Co}\label{par.cone.imm}
The map $\hat u_0$ is a smooth immersion away from $0$. \hfill $\Box$
\end{Co}

\begin{proof}
Indeed, the punctured Schoen--Wolfson cone, i.e., the image of
$$(0,\infty)\times\Gamma\to\C^2,\quad (s,t)\mapsto s\cdot\eta(t),$$
for the immersed curve $\eta$ found above, is an embedded surface conformally equivalent to $\C\setminus\{0\}$, via a map $h$.
We extend $h$ to the origin by $h(0):=0$.
Since $\hat u_0$ is a conformal immersion on $\C\setminus\{0\}$,
we have $\partial_{\bar z}(h\circ \hat u_0)=0$ away from $0$ and the composition $h\circ\hat u_0$ is continuous at $0$. Hence, $f:=h\circ\hat u_0$ is a proper holomorphic function on $\C$, with $f^{-1}(0)=\{0\}$.

Hence, $f(z)$
is a polynomial vanishing only at the origin, giving $f(z)=cz^k$ for some $c\in\C\setminus\{0\}$ and $k\in\N^*$,
proving the claim since $\hat u_0=h^{-1}\circ f$.
\end{proof}

\begin{Co}\label{tilde.n.cone}
The multiplicity $\tilde N_0$ is a constant integer on $\C\setminus\{0\}$. \hfill $\Box$
\end{Co}

\begin{proof}
Given $x_0'\in\C\setminus\{0\}$, the blow-up at $x_0'$ is unique and is given by the linear map $\nabla\hat u_0(x_0')$,
with constant multiplicity $\tilde N_0(x_0')$, which must then be an integer. Thus, $\tilde N_0$ takes integer values on $\C\setminus\{0\}$.
If it is not constant, then we can find $\mu\in\N$ such that the open set
$$\Omega:=(\C\setminus\{0\})\cap\{\tilde N_0\le\mu\}=(\C\setminus\{0\})\cap\{\tilde N_0<\mu+1\}$$
is not all of $\C\setminus\{0\}$. We can then find a disk $D\subset\subset\C\setminus\{0\}$
such that $D\subset\Omega$ and a boundary point $\bar x\in\de D$ has $\tilde N_0(\bar x)\ge\mu+1$.
Since the blow-up at $\bar x$ obviously arises as a strong $W^{1,2}_{loc}$ limit, \eqref{rad.limit}
implies that the constant multiplicity $\tilde N_{0,0}$ in the blow-up equals $\mu$. Since $\tilde N_{0,0}(0)=\tilde N_0(x_0')\ge\mu+1$,
we arrive at a contradiction.
\end{proof}

Note that, calling $2\pi\kappa>0$ the length of the cross-section of the cone and
$$\tilde\eta:\R/2\pi\kappa\Z\to S^3$$
its parametrization by arclength, we can take
$h$ to be the inverse of
\begin{equation}\label{h.expl}
h^{-1}(re^{i\theta}):=r^\kappa \tilde\eta(\kappa \theta).
\end{equation}
By repeating the proof of Proposition \ref{nc.full.stop} verbatim and using the fact that
$h$ and $h^{-1}$ have polynomial growth at $0$, we obtain the following fact.

\begin{Prop}\label{all.admissible}
In fact, any $x_0\in\Omega$ is strongly admissible. \hfill $\Box$
\end{Prop}


\subsection{\texorpdfstring{Local finiteness of $\bm{\mathcal{S}_{SW}}$ in pinched-density regions}{Local finiteness of $\mathcal{S}_{SW}$ in pinched-density regions}}
In the sequel, we fix an arbitrary point $x_0\in\Omega$. Recall that $x_0$ is strongly admissible, by Proposition \ref{all.admissible}.
We claim that all blow-ups at a given point have isometric images. To show this, we first prove some technical lemmas.

\begin{Lm}
There exists a constant $M>1$ such that, if $x\neq x_0$ is close enough to $x_0$, then
$$\frac{d_K(u(x),u(x_0))}{s(|x-x_0|)}\in[M^{-1},M],$$
where we recall that $s(r)^2=s(x_0,r)^2:=\int_{B_r(x_0)}|\nabla u|^2\,dx^2$. \hfill $\Box$
\end{Lm}

\begin{proof}
If not, then we can find a sequence $x_k\to x_0$ such that
$$\frac{d_K(u(x_k),u(x_0))}{s(|x_k-x_0|)}\to 0\quad\text{or}\quad \frac{d_K(u(x_k),u(x_0))}{s(|x_k-x_0|)}\to\infty.$$
Letting $r_k:=|x_k-x_0|$ and $u_k:=u_{x_0,r_k}$ (as in \eqref{not.res}), writing $x_k=x_0+r_ky_k$,
we have
$$\mathfrak{r}\circ u_k(y_k)\to0 \quad\text{or}\quad \mathfrak{r}\circ u_k(y_k)\to\infty.$$
Since $x_0$ is strongly admissible, by Proposition \ref{pr-tangent-cone}
we can extract a blow-up PHSLV $(\C,\hat u_\infty,\hat N_\infty)$ such that $u_k\to u_\infty=\hat u_\infty\circ\psi$ locally uniformly on $\C$, for a suitable
quasiconformal homeomorphism $\psi:\C\to\C$.
Since $|y_k|=1$, we can also assume that $y_k\to y_\infty\in S^1$
up to a subsequence.
We then have
$$u_k(y_k)\to u_\infty(y_\infty)\in\mathbb{H}^2\setminus\{0\},$$
since $u_\infty^{-1}(0)=\{0\}$. In particular, $\mathfrak{r}\circ u_k(y_k)$ converges to a limit in $(0,\infty)$,
a contradiction.
\end{proof}

\begin{Lm}
There exists a constant $\hat M>1$ such that if
$$\hat M|x-x_0|\le|x'-x_0|$$
then
$$d_K(u(x),u(x_0))\le\frac{1}{2M}d_K(u(x'),u(x_0)),$$
provided that $x,x'\neq x_0$ are close enough to $x_0$. \hfill $\Box$
\end{Lm}

\begin{proof}
If not, then we can find sequences $x_k,x_k'\to x_0$ and $\hat M_k\to\infty$ such that
$$\hat M_k|x_k-x_0|\le|x_k'-x_0|,\quad d_K(u(x_k),u(x_0))>\frac{1}{2M}d_K(u(x_k'),u(x_0)).$$
Letting $r_k:=|x_k'-x_0|$ and defining $u_k$ as above,
we can extract a subsequential limit $u_\infty$. Writing $x_k=x_0+r_ky_k$ and $x_k'=x_0+r_kz_k$,
we have
$$y_k\to0,\quad |z_k|=1.$$
Up to a subsequence, we can assume that $z_k\to z_\infty\in S^1$, so that
$$u_k(y_k)\to u_\infty(0)=0,\quad u_k(z_k)\to u_\infty(z_\infty)\neq0$$
as in the previous proof.
In particular, we see that $\r\circ u_k(y_k)<\frac{1}{2M}\r\circ u_k(z_k)$ eventually, a contradiction.
\end{proof}

In the sequel, with abuse of notation, given $p\in\mathbb{H}^2$ and $\lambda>0$ we will denote
$$\frac{p}{\lambda}:=\delta_{1/\lambda}(p).$$

\begin{Lm}
Letting $\mathcal{C}(r)\subseteq \p B_1^\r(0)$ denote the image of the map
$$\bar B_{\hat M r}(x_0)\setminus B_{r/\hat M}(x_0)\to \p B_1^\r(0),\quad x\mapsto\frac{u(x_0)^{-1}*u(x)}{d_K(u(x_0),u(x))},$$
any subsequential limit $\lim_{r_k\to0}\mathcal{C}(r_k)$ in the Hausdorff topology is the cross-section
of a (possibly flat) Schoen--Wolfson cone. \hfill $\Box$
\end{Lm}

\begin{proof}
Assume that $\mathcal{C}(r_k)\to\mathcal{C}_\infty$ along a subsequence, for some compact set $\mathcal{C}_\infty\subseteq S^3\subset\R^4=\C^2\times\{0\}$.
We consider the rescaled maps $u_k$, defined as above,
and, up to a subsequence, the limit map $u_\infty$.
Letting $\mathcal{C}$ be the image of $\frac{v_\infty}{|v_\infty|}|_{\C\setminus\{0\}}$,
we know that $\mathcal{C}$ is the cross-section of a (possibly flat) Schoen--Wolfson cone.

We claim that $\mathcal{C}_0=\mathcal{C}$.
Given points $y_k\in \bar B_{M}(x_0)\setminus B_{1/M}(x_0)$ converging to $y_\infty$, we know that $u_k(y_k)\to u_\infty(y_\infty)\in(0,\infty)\cdot\mathcal{C}$,
and hence
$$\frac{u_k(y_k)}{\mathfrak{r}\circ u_k(y_k)}
\to\frac{u_\infty(y_\infty)}{\mathfrak{r}\circ u_\infty(y_\infty)}
=\frac{u_\infty(y_\infty)}{|u_\infty(y_\infty)|},$$
showing the inclusion $\mathcal{C}_0\subseteq\mathcal{C}$.
To see that this inclusion is an equality, we fix a reference point $\bar y\in S^1$ and we observe that by the first lemma above we have
$$|u_\infty(\bar y)|\in[M^{-1},M],$$
while from the second lemma we get
$$|u_\infty(y)|\le\frac{1}{2M}|u_\infty(\bar y)|\le\frac12\quad\text{for $|y|\le\hat M^{-1}$}$$
and
$$|u_\infty(y)|\ge 2M|u_\infty(\bar y)|\ge 2\quad\text{for $|y|\ge\hat M$}.$$
Thus, we have
$$\mathcal{C}=u_\infty(\C)\cap S^3=u_\infty(\bar B_{\hat M}(0)\setminus B_{1/\hat M}(0))\cap S^3
=\frac{u_\infty}{|u_\infty|}(\bar B_{\hat M}(0)\setminus B_{1/\hat M}(0)).$$
Using the uniform convergence $\frac{u_k}{\mathfrak{r}\circ u_k}\to \frac{u_\infty}{|u_\infty|}$ on this annulus, we deduce that $\mathcal{C}_0=\mathcal{C}$.
\end{proof}

\begin{Prop}
The images of two blow-ups at $x_0$ are isometric to each other. In particular,
if at least one blow-up at $x_0$ takes values in a plane, then the same holds for all blow-ups. \hfill $\Box$
\end{Prop}

\begin{proof}
Indeed, the set of subsequential limits $\lim_{r\to0}\mathcal{C}(r)$ (endowed with the Hausdorff topology) is connected
and consists of cross-sections of Schoen--Wolfson cones.
By \eqref{upper.bd.assumption}, the parameters defining such cones take value in a finite set (cf. \cite[Section 7]{SW2}). Hence, the possible isometric classes that can be attained as limits have positive distance from each other. Since the set of limits is connected, it must be a subset of just one of them.
\end{proof}

The previous facts were proved at fixed center, but the same ideas can be used to obtain the following result,
whose proof is just outlined.

\begin{Lm}\label{planar.preserved}
Given $\Lambda>0$,
there exist $k_0\in\N^*$ and $\zeta,\eta>0$, depending only on $\Lambda$, $\nu$, and the constant $C_0$ in \eqref{upper.bd.assumption}, such that the following holds. 
If $B_{2^{k_0}r}(x)\subseteq\Omega$ and
\begin{itemize}[leftmargin=8mm]
\item[(i)] the image of the map $\frac{u(x)^{-1}*u}{d_K(u(x),u)}$ on $\bar B_{\hat M r}(x)\setminus B_{r/\hat M}(x)$ is $\zeta$-close to a circle in the Hausdorff metric,
\item[(ii)] $\int_{B_{2^kr}(x)}|\nabla u|^2\,dx^2\le \Lambda^k\int_{B_r(x)}|\nabla u|^2\,dx^2$ for $k=1,\dots,k_0$,
\item[(iii)] $\Theta(\mathbf{v}_{\om},u(x),s(x,r)/\eta)<2\pi\tilde N(x)+\eta$
for some open $x\in\om\subset\subset\Omega$ such that $\bar B_{2s(x,r)/\eta}^\r(u(x))\cap u(\p\om)=\emptyset$,
\end{itemize}
then (i) holds also at a smaller radius $r'\in(\eta r,r/2)$. \hfill $\Box$
\end{Lm}

\begin{proof}
We fix $k_0$ and $\zeta$, to be determined later. Arguing by contradiction as in the proof of Proposition \ref{adm.open},
in the limit $\eta\to0$ we get a nontrivial PHSLV $(\psi(B_{2^{k_0}}(0)),\hat w,\hat N_{\hat w})$
with $\varphi\circ\hat w=0$ and $\hat w^{-1}(0)=\{0\}$, such that the induced varifold satisfies
$$\mathbf{v}_{\hat w,\psi(B_{2^{k_0}}(0))}\le\tilde{\mathbf{v}}$$
for a HLSV $\tilde{\mathbf{v}}$ which has $\Theta(\tilde{\mathbf{v}},0,R)$ constant in $R$ (and bounded by a constant $C(C_0)$),
equal to $2\pi\tilde N_{\hat w}(0)$.

Moreover, letting $w=\hat w\circ\psi$, the image of $\bar B_{\hat M r}(x)\setminus B_{r/\hat M}(x)$ through $\hat w/(\r\circ\hat w)=\hat w/|\hat w|$ is $\zeta$-close to a circle. We now claim that the statement holds for a radius $r'\in(0,1/2)$,
provided that $k_0$ is large enough and $\zeta$ is small enough, giving the desired contradiction.
If not, then letting $k_0\to\infty$ and employing a diagonal argument we obtain a new PHSLV defined on $\C$
with the same properties.

As seen in the proof of Proposition \ref{par.cone}, the image of its parametrization is a Schoen--Wolfson cone
(with parameters bounded in terms of $C_0$). Since its cross-section is $\eta$-close to a circle,
by choosing $\eta$ small enough we see that it is forced to be a circle (and this limit PHSLV is planar),
obtaining a contradiction. 
\end{proof}

Assume now that we have a sequence of points $x_k\to x_0$ with $x_k\in\mathcal{S}_{SW}$ and $x_k\neq x_0$.

\begin{Prop}\label{sw.countable}
For $k$ large enough we have $\tilde N(x_k)<\tilde N(x_0)-\frac{\eta}{4\pi}$. As a consequence, $\mathcal{S}_{SW}$
is at most countable. \hfill $\Box$
\end{Prop}

\begin{proof}
If the claim fails, we can assume without loss of generality that $\tilde N(x_k)\ge\tilde N(x_0)-\frac{\eta}{4\pi}$
for all $k$.
Letting $r_k:=|x_k-x_0|$, we can define $u_k$ and the limit $u_\infty=\hat u_\infty\circ\psi$ as above, where $\psi:\C\to\C$
is a suitable quasiconformal homeomorphism.
Writing $x_k=x_0+r_ky_k$ and assuming $y_k\to y_\infty\in S^1$,
since $\hat u_\infty$ is a smooth embedding near $y_\infty$ by Corollary \ref{par.cone.imm}, the first assumption of the previous lemma holds eventually,
for any fixed $r>0$ small enough.

Also, the second assumption holds for free (see Lemma \ref{ahlfors} below).
Finally, we can select $\omega$ such that the third assumption holds for $x=x_0$ and $\eta/2$ in place of $\eta$, for $r>0$ small enough.
Since $2\pi\tilde N(x_0)\le2\pi\tilde N(x_k)+\frac{\eta}{2}$, we obtain the third assumption eventually also for $x=x_k$.
By iterating the lemma (note that the third assumption automatically holds for smaller radii), we obtain that any blow-up at $x_k$ has cross-section $\zeta$-close to a circle, and hence is planar (assuming that $\zeta$ is small enough), contradicting the fact that $x_k\in\mathcal{S}_{SW}$.

Hence, setting $\gamma:=\frac{\eta}{4\pi}$, the set $\mathcal{S}_{SW}\cap\{\tilde N\in[\nu-\gamma,\nu]\}$
is locally finite. Looking now at $\{\tilde N<\nu-\gamma\}$, which is open by upper semi-continuity of $\tilde N$,
the set $\mathcal{S}_{SW}\cap\{\tilde N\in[\nu-2\gamma,\nu-\gamma)\}$ is again locally finite here, and so on.
This shows that $\mathcal{S}_{SW}$ is at most countable.
\end{proof}

\section{Inductive step: conclusion}
We now want to mimic the reasoning used in Step 3 of the proof of \cite[Theorem 5.7]{PiRi1}.
Differently from that situation, we already reduced ourselves to the case where all points are (strongly) admissible. However, more work is needed in the present situation since,
even on the regular set, the map $u$ does not satisfy a PDE as simple as the Laplace equation,
due to the lack of a priori bounds on the Lagrange multiplier $\beta$ (cf. \eqref{ell.sys})
and on the cardinality of the singular set $\mathcal{S}_{SW}$.

\subsection{\texorpdfstring{The case $\bm{\mathcal{S}_{SW}=\emptyset}$}{The case $\mathcal{S}_{SW}=\emptyset$}}
If $\tilde N>\nu-1$ at all points of $\Omega$, then $N=\nu$ a.e. on $\mathcal{G}_u^f$ (since $N=\tilde N$ a.e. here and $N$ takes integer values),
and thus $\tilde N=\nu$ on $\Omega$ (see the proof of Proposition \ref{co-dens}).
Since the value of $N$ matters only on $\mathcal G_u^f$, we can also assume that $N=\nu$.
Dividing $N$ by $\nu$, we are back to the base case of the induction, where $\tilde N=1$.
In the sequel, we can then assume that
$$\{\tilde N\le\nu-1\}\neq\emptyset.$$
Since $\mathcal{S}_{SW}=\emptyset$, any blow-up at any point $x\in\Omega$ (recall Proposition \ref{all.admissible})
is given by a homogeneous polynomial, showing that $\tilde N(x)\in\N$. Since $\tilde N$ is upper semi-continuous, $\{\tilde N\le\nu-1\}\subset\Omega$ is open.

We let $\Omega_{reg}$ be the largest open set where the regularity theorem holds.
By inductive assumption, we then have
$$\{\tilde N\le\nu-1\}\subseteq\Omega_{reg}.$$
We assume by contradiction that $\Omega_{reg}\subsetneq\Omega$
and we let $\Omega'$ be a connected component of $\Omega_{reg}$ intersecting $\{\tilde N\le\nu-1\}$.
By the regularity theorem, we have $N=\mu$ a.e. on $\Omega'$, for some integer $1\le\mu<\nu$.

We can now find a disk $D\subset\Omega'$ such that $\overline{D}\subset\Omega$ and $\partial D$
contains a point $\bar x\not\in\Omega'$, i.e., $\bar x\not\in\Omega_{reg}$, where necessarily $\tilde N(\bar x)=\nu$.
Recall that, for any blow-up $(\C,\hat u_0,\hat N_0)$ at $\bar x$, the multiplicity $\tilde N_0$ equals $\nu$ at the origin,
while it is a constant $\mu'\ge\mu$ elsewhere (since $\hat u_0$ is a homogeneous polynomial;
the fact that $\mu'\ge\mu$ is a simple consequence of Corollary \ref{dens.limits}).

Let $H$ be the half-plane obtained by blowing up $D$ at $\bar x$. It is enough to show the following fact.

\begin{Prop}\label{blow.up.strong}
For any blow-up $(\C,\hat u_0,\hat N_0)$ at $\bar x$, the map $u_0$ arises as a strong $W^{1,2}_{loc}$ limit on $H$. As a consequence,
$\mu'=\mu$. \hfill $\Box$
\end{Prop}

\begin{Co}
We have $\tilde N\le\mu$ in a punctured neighborhood of $\bar x$. In particular, here the regularity theorem holds
and we have $N=\mu$ a.e. \hfill $\Box$
\end{Co}

\begin{proof}
This corollary follows from upper semi-continuity of $\tilde N$: if we had a sequence of points
$x_j\to\bar x$ with $\tilde N(x_j)\ge\mu+1$ and $x_j\neq\bar x$, then letting $r_j:=|x_j-\bar x|$
we would have $\tilde N_{u_{r_j}}(y_j)\ge\mu+1$ at the point $y_j:=\frac{x_j-\bar x}{|x_j-\bar x|}$.
Thus, we would have $\tilde N_0(\bar y)\ge\mu+1$ at a subsequential limit $\bar y=\lim_{j\to\infty}y_j\in S^1$,
for a certain blow-up $\hat u_0$ at $\bar x$, a contradiction.
\end{proof}

\begin{Co}\label{cor.con}
Any blow-up $(\C,\hat u_0,\hat N_0)$ at $\bar x$ arises as a strong $W^{1,2}_{loc}$ limit on all of $\C$. As a consequence,
$u$ is a smooth immersion in a punctured neighborhood of $\bar x$. \hfill $\Box$
\end{Co}

\begin{proof}
Let $u_0=\hat u_0\circ\psi$ denote the limit of the maps $u_r=u_{\bar x,r}$ along a subsequence.
Recall from \eqref{rad.limit} that we have the convergence of measures
$$\mu\frac{|\nabla u_r|^2}{2}\,dx^2=N_r\frac{|\nabla u_r|^2}{2}\,dx^2\rightharpoonup
N_0|\partial_{x_1} u_0\wedge\partial_{x_2} u_0|\,dx^2\le\mu\frac{|\nabla u_0|^2}{2}\,dx^2,$$
where $N_0:=\hat N_0\circ\psi$ and we used the fact that $\hat N_0=\tilde N_0=\mu'=\mu$ a.e. on $\C$.
Hence, the inequality must be an equality, and this immediately implies the first claim. Thus, we have $\psi=\operatorname{id}$
and $u_0=\hat u_0$.

Moreover, we know that $u_0=\hat u_0$ is a smooth conformal immersion outside the origin, taking values in $\C^2\subset\mathbb{H}^2$.
Taking any a small disk $\mathcal{D}\subset\subset\C\setminus\{0\}$, where $u_0$ is close to an affine function,
we can apply Proposition \ref{ep.reg}, showing that $u_r$ is a smooth embedding on a smaller disk, for $r$ small enough
along our subsequence. Since the subsequence was arbitrary, this proves the second claim.
\end{proof}

Since $\bar x\nin\mathcal{S}_{SW}$, smoothness extends across $\bar x$, as shown at the end of \cite[Section 4]{SW2}, contradicting the fact that $\bar x\not\in\Omega_{reg}$.
It remains to prove the previous proposition.

\begin{proof}[Proof of Proposition \ref{blow.up.strong}]
It suffices to see that, calling $u_r=u_{\bar x,r}$ the usual dilated maps,
we have $u_r\to u_0$ strongly in $W^{1,2}_{loc}(H)$ (along the subsequence defining the blow-up).
Indeed, this will force $N_0=\mu$ a.e. on $H$, where $N_0:=\hat N_0\circ\psi$,
or equivalently $\hat N_0=\mu$ a.e. on $\psi(H)$, and hence $\mu'=\mu$.
The strong $W^{1,2}_{loc}$ convergence on $H$ follows from the following more general result.
%
\end{proof}

We state and prove separately the following more general result, since it will be useful also to deal with the general case $\mathcal{S}_{SW}\neq\emptyset$.

\begin{Prop}\label{strong.conv}
Given $C_0,\Lambda>0$,
assume that $(B_1(0),u_j,N_j)$ are PHSLVs in $\mathbb{H}^2$ with $u_j\to u_\infty$ in $C^0_{loc}$ and weakly in $W^{1,2}_{loc}$, satisfying \eqref{upper.bd.assumption}, as well as the doubling bound
\begin{equation}\label{doubling}
\int_{B_{1}(0)}|\nabla u_j|^2\,dx^2\le\Lambda\int_{B_{1/2}(0)}|\nabla u_j|^2\,dx^2.
\end{equation}
Moreover, assume that the regularity theorem holds for all $j$ on some fixed disk $\mathcal{D}\subseteq B_{1/2}(0)$, with $\mathcal{S}_{SW}(u_j)\cap \mathcal{D}=\emptyset$.
Then we have $u_j\to u_\infty$ strongly in $W^{1,2}_{loc}(\mathcal{D})$. \hfill $\Box$
\end{Prop}

\begin{proof}
Replacing $u_j$ with its projection onto $\C^2$, which we denote by $v_j:=\pi\circ u_j$, recall (see, e.g., \cite[pp. 4--5]{SW2}) that we have the elliptic system
\begin{equation}\label{ell.sys}\begin{cases}
\Delta v_j+i\nabla\beta_j\cdot\nabla v_j=0\\
\Delta\beta_j=0,
\end{cases}\end{equation}
for a suitable harmonic function $\beta_j:\mathcal{D}\to\R$ (indeed, $v_j$ is smooth on $\mathcal{D}$,
with the mean curvature $H_j$ extending smoothly across branch points, as shown in \cite[Section 4]{SW2}).
We would like to show that $\nabla\beta_j$ stays bounded locally on $\mathcal{D}$. After this is done, it is immediate to conclude the strong $W^{1,2}_{loc}$ convergence $v_j\to v_\infty$ on $\mathcal{D}$ (cf. Step 3 of the proof of \cite[Theorem 5.7]{PiRi1}),
and hence the same for $u_j\to u_\infty$.

Taking $\mathcal{D}'\subset\subset \mathcal{D}$ and
writing $\mathcal{D}'=B_\rho(\hat x)$, we define
$$m_j:=\max_{x\in \mathcal{D}'}(\rho-|x-\hat x|)|\nabla\beta_j|(x).$$
Note that the maximum is not attained at the boundary $\partial \mathcal{D}'$, since here $\rho-|x-\hat x|=0$.
To conclude, it is enough to show that $m_j$ stays bounded. Let us assume by contradiction that $m_j\to\infty$ up to a subsequence.
For each $j$, we let $x_j$ be one of the points realizing the maximum and we let
$$\lambda_j:=|\nabla\beta_j|(x_j)^{-1}.$$
We have
$$\lambda_j=\frac{\rho-|x_j-\hat x|}{m_j}=o(\rho-|x_j-\hat x|),$$
and hence we can find another sequence $\tilde\lambda_j=\lambda_j/\delta_j$ such that
$$\delta_j\to0,\quad\tilde\lambda_j=o(\rho-|x_j-\hat x|).$$
Recalling \eqref{not.res}, we consider the rescaled maps
$$w_j:=\pi\circ(u_j)_{x_j,\tilde\lambda_j},$$
as well as
$$\gamma_j:=\beta_j(x_j+\tilde\lambda_jx).$$
Clearly, we still have the equation
$$\Delta w_j+i\nabla\gamma_j\cdot\nabla w_j=0$$
after the rescaling. We observe that $|\nabla\gamma_j|(0)=\delta_j^{-1}$ and that,
given any $x\in\C$, we have $x_j+\tilde\lambda_jx\in\mathcal{D}'$ eventually (as $\tilde\lambda_j=o(\rho-|x_j-\hat x|)$),
as well as
$$\delta_j|\nabla\gamma_j|(x)
=\lambda_j|\nabla\beta_j|(x_j+\tilde\lambda_jx)\le\lambda_j\frac{\rho-|x_j-\hat x|}{\rho-|x_j+\tilde\lambda_jx-\hat x|}|\nabla\beta_j|(x_j)
=\frac{\rho-|x_j-\hat x|}{\rho-|x_j+\tilde\lambda_jx-\hat x|},$$
and this upper bound converges to $1$ uniformly on compact sets,
giving
$$\limsup_{j\to\infty}\delta_j|\nabla\gamma_j|\le1$$
uniformly on compact sets.
Thus, assuming without loss of generality that $\gamma_j(0)=0$, we can extract a limit harmonic map
$$\delta_j\gamma_j\to\gamma_\infty\quad\text{in }C^\infty_{loc}(\C)$$
with $|\nabla\gamma_\infty|(0)=1$ and $|\nabla\gamma_\infty|(x)\le1$ everywhere, and hence $|\nabla\gamma_\infty|\equiv1$.

By the lemma below, for any fixed integer $k\ge1$, for $j$ large enough we have uniform upper bounds of the form
$$\int_{B_{2^k}(0)}|\nabla w_j|^2\,dx^2\le C^k\int_{B_{1}(0)}|\nabla w_j|^2\,dx^2=C^k.$$
By Theorem \ref{th-sequential-weak-closure}, we can then extract a limit PHSLV $(\C,\hat w_\infty,\hat N_\infty)$, and a limit map
$$w_\infty=\lim_{j\to\infty}w_j=\hat w_\infty\circ\psi$$
in $C^0_{loc}$ and weakly in $W^{1,2}_{loc}$, for a suitable quasiconformal homeomorphism $\psi:\C\to\C$.
Since
$$\delta_j\Delta w_j+i(\delta_j\nabla\gamma_j)\cdot\nabla w_j=0,$$
we can pass to the limit and obtain
$$\nabla\gamma_\infty\cdot\nabla w_\infty=0.$$
Since $\nabla\gamma_\infty\neq0$ is constant, this means that the differential $\nabla w_\infty$ is
never invertible. By the chain rule (see, e.g., \cite[Lemma III.6.4]{LV})
and its consequence that quasiconformal homeomorphisms preserve the class of negligible sets,
it follows that the same holds for $\hat w_\infty$. Since the latter is weakly conformal, we deduce that $\nabla\hat w_\infty=0$.
However, this contradicts \eqref{rad.limit} and the fact that
$$\int_{B_1(0)}|\nabla w_j|^2\,dx^2=1,$$
by our definition of a rescaling (recall \eqref{not.res}).
\end{proof}

Lemma \ref{ahlfors} below, used in the previous proof, shows that the measures $|\nabla u_r|^2\,dx^2$ satisfy an Ahlfors-type regularity (in a rather weak sense).

\begin{Lm}\label{ahl.pre}
Given $C_0,\Lambda>0$,
assume that $(B_1(0),u,N)$ is a PHSLV in $\mathbb{H}^2$ satisfying \eqref{upper.bd.assumption}, as well as
\begin{equation}\label{doubling.bis}
\int_{B_{1}(0)}|\nabla u|^2\,dx^2\le\Lambda\int_{B_{1/2}(0)}|\nabla u|^2\,dx^2.
\end{equation}
Then there exist two constants $\eta(C_0,\Lambda)\in(0,1/2)$ and $\Lambda'(C_0)>0$ such that
$$\int_{B_{s'}(0)}|\nabla u|^2\,dx^2\le\Lambda'\int_{B_{s'/2}(0)}|\nabla u|^2\,dx^2$$
for some radius $s'\in(\eta,1/2)$. \hfill $\Box$
\end{Lm}

\begin{proof}
This follows by a straightforward compactness argument: we fix $\Lambda'$ to be specified later and assume that
the claim does not hold for any $\eta>0$. In the limit $\eta\to0$, we obtain a PHSLV $(\psi(B_1(0)),\hat u_\infty,\hat N_\infty)$,
which is nontrivial thanks to \eqref{doubling.bis} and \eqref{rad.limit}.

By our contradiction assumption, we have
$$
\int_{\psi(B_{s'}(0))}\hat N_\infty|\nabla\hat u_\infty|^2\,dx^2
\ge\Lambda'\int_{\psi(B_{s'/2}(0))}\hat N_\infty|\nabla\hat u_\infty|^2\,dx^2
$$
for all radii $s'\in(0,1/2)$.
In turn, recalling \eqref{dist.bd} and using Lemma \ref{distortion}, this implies that
\begin{equation}\label{contr.ahl}
\int_{B_{s'}(0)}\hat N_\infty|\nabla\hat u_\infty|^2\,dx^2
\ge\Lambda'\int_{B_{cs'}(0)}\hat N_\infty|\nabla\hat u_\infty|^2\,dx^2
\end{equation}
for a suitable $c=c(C_0)\in(0,1)$ and all $s'>0$ small enough.

However, by Proposition \ref{all.admissible},
the origin is an admissible point for the limit PHSLV. As in its proof, any blow-up map $\hat w$ takes values in a (possibly flat)
cone, conformally equivalent to the plane via a map $h$ given by \eqref{h.expl}, and the composition $h\circ \hat w(z)=cz^k$. By \eqref{upper.bd.assumption}, there are finitely many possibilities for the map $h$ (up to precomposition with isometries). This gives an upper bound on $k$ and, in turn, a doubling bound of the form
$$\int_{B_{r}(0)}|\nabla\hat w|^2\,dx^2\le C(C_0)\int_{B_{r/2}(0)}|\nabla\hat w|^2\,dx^2$$
for $0<r<1$.
As above, \eqref{contr.ahl} implies
$$\int_{B_{s''}(0)}\hat N_{\hat w}|\nabla\hat w|^2\,dx^2
\ge\Lambda'\int_{B_{c^2s''}(0)}\hat N_{\hat w}|\nabla\hat w|^2\,dx^2$$
for all $s''>0$.
Since $1\le\hat N_{\hat w}\le\frac{C_0}{\pi}$ by \eqref{dist.bd}, this contradicts the previous doubling bound once we take $\Lambda'=\Lambda'(C_0)$ large enough.
\end{proof}

\begin{Lm}
In the situation of the previous statement, we have
$$\int_{B_{1}(0)}|\nabla u|^2\,dx^2\le C\int_{B_{s'}(0)}|\nabla u|^2\,dx^2,$$
for a constant $C$ depending only on $C_0$ and $\Lambda$. \hfill $\Box$
\end{Lm}

\begin{proof}
The statement follows from a direct compactness argument.
\end{proof}

Thanks to the universal $\Lambda'$ appearing in the conclusion of Lemma \ref{ahl.pre},
which depends only on $C_0$, we can iterate its statement infinitely many times, replacing $B_1(0)$ with $B_{s'}(0)$ and so on,
and the constant $\Lambda'$ (and hence also $\eta(C_0,\Lambda')$) stabilizes immediately after the first iteration.
Thus, the two previous lemmas easily imply the following corollary.

\begin{Lm}\label{ahlfors}
Given $C_0,\Lambda>0$,
assume that $(B_1(0),u,N)$ is a PHSLV in $\mathbb{H}^2$ satisfying \eqref{upper.bd.assumption}
and \eqref{doubling.bis}. Then there exists a constant $C(C_0,\Lambda)>0$ such that
$$\int_{B_{s}(x)}|\nabla u|^2\,dx^2\le C\int_{B_{s/2}(x)}|\nabla u|^2\,dx^2$$
for all $x\in B_{1/2}(0)$ and $0<s<1/4$. \hfill $\Box$
\end{Lm}

\begin{proof}
Indeed, another simple compactness argument shows that \eqref{doubling.bis} is also
satisfied replacing $B_1(0)$ with $B_{1/4}(x)$ (and $B_{1/2}(0)$ with $B_{1/8}(x)$), for all $x\in B_{1/2}(0)$, up to replacing $\Lambda$ with another constant $\tilde\Lambda$ depending on $C_0,\Lambda$. The statement now follows by iterating the two previous facts.
\end{proof}

\begin{Rm}
This argument could mislead the reader to think that
one can obtain a controlled decay of Dirichlet energy, and hence admissibility of all points,
in a much easier way compared to the analysis of the previous sections. However,
this is not the case since the proof of Lemma \ref{ahl.pre} did use the fact that any point is admissible. \hfill $\Box$
\end{Rm}

\subsection{General case}
As seen in Proposition \ref{sw.countable}, there exists $\gamma>0$ such that, for any $\ell\in\N^*$, the set of points
$$\mathcal{S}_{SW}\cap\{\tilde N\in[\nu-1+(\ell-1)\gamma,\nu-1+\ell\gamma)\}$$
is locally finite in $\Omega_\ell:=\{\tilde N<\nu-1+\ell\gamma\}$.
In fact, up to shrinking $\gamma$, we can also guarantee the following.

\begin{Prop}\label{gamma.small}
If $\tilde N(x)\le\nu-1+\gamma$ then $\tilde N(x)\le\nu-1$. \hfill $\Box$
\end{Prop}

\begin{proof}
Indeed, any blow-up at $x$ parametrizes a Schoen--Wolfson cone with constant multiplicity,
whose density at the origin is $\theta^\chi(0)=2\pi\tilde N(x)$. Since there is a finite number of such cones with $\theta^\chi(0)\le2\pi\nu$ up to isometries,
the claim follows.
\end{proof}

We prove the regularity theorem on $\Omega_\ell$ by induction over $\ell=0,1,\dots$, up to reaching $\ell=\lfloor\frac{1}{\gamma}\rfloor+1$
(in which case
$\Omega_\ell=\Omega$).

The base case $\ell=0$ holds, since we are assuming the validity of the regularity theorem on $\{\tilde N<\nu-1\}$.
We now assume inductively that regularity holds on $\Omega_\ell$ and show it on $\Omega_{\ell+1}$.

If $\Omega_\ell=\emptyset$ then we have $\tilde N\ge\nu-1+\ell\gamma$ on all of $\Omega_{\ell+1}$.
Hence, in this case $\mathcal{S}_{SW}\cap\Omega_{\ell+1}$ is locally finite. Applying the previous case on
the relative complement $\Omega_{\ell+1}\setminus\mathcal{S}_{SW}$, we obtain that $u$ is smooth here,
and an immersion away from a set $\mathcal{S}_{branch}$ which is locally finite in $\Omega_{\ell+1}\setminus\mathcal{S}_{SW}$.
To conclude, this set is locally finite also in $\Omega_{\ell+1}$: indeed, the proof of
Corollary \ref{cor.con} applies also here, and shows that $u$ must be a smooth immersion on a punctured neighborhood of each $x\in\mathcal{S}_{SW}$,
and thus points in $\mathcal{S}_{brach}$ cannot accumulate towards a point in $\mathcal{S}_{SW}$.

In the sequel, by localizing we can assume that $\Omega=\Omega_{\ell+1}$ is connected and that $$\Omega_\ell\neq\emptyset.$$
Thus, letting $\Omega_{reg}$ be the largest open subset where the regularity theorem holds, we have $\Omega_{reg}\neq\emptyset$. As in the previous case, we let $\Omega'$ be a connected component of it intersecting $\Omega_\ell$.

As before, we can find a disk $D\subset\Omega'$ with $\bar D\subset\Omega=\Omega_{\ell+1}$, such that the boundary $\partial D$
contains a point $\bar x\not\in\Omega_{reg}$. In particular, we must have
$$\tilde N(\bar x)\ge\nu-1+\ell\gamma,$$
while by the regularity theorem we have $N=\mu$ a.e. on $\Omega'$ for some integer $\mu\le\nu-1$.

We want to show that Proposition \ref{blow.up.strong} holds also in this case. Namely,
letting $H$ denote the half-plane obtained by blowing up $D$ at $\bar x$, we want to show that any blow-up $(\C,\hat u_0,\hat N_0)$ at $\bar x$ arises as a strong $W^{1,2}_{loc}$ limit on $H$.

Once this is done, we conclude as in the previous case: using Corollary \ref{tilde.n.cone}, we deduce that any blow-up at $\bar x$ 
has density $\tilde N_0=\mu$ on $\C\setminus\{0\}$.
By upper semi-continuity, $\tilde N<\nu-1+\gamma$ on a punctured neighborhood of $\bar x$.
By Proposition \ref{gamma.small}, this implies that $\tilde N\le\nu-1$ here. Thus, by the outer induction,
the regularity theorem holds on this punctured neighborhood, and hence
we have $N=\mu$ a.e. here (since it intersects $D$). Again, this implies that any blow-up at $\bar x$ arises
as a strong $W^{1,2}_{loc}$ limit on all of $\C$, and we conclude exactly as in the proof of Corollary \ref{cor.con}
(by using Corollary \ref{par.cone.imm}) that $u$ is a smooth immersion in a punctured neighborhood of $\bar x$.
If $\bar x\in\mathcal{S}_{SW}$, then we get $\bar x\in\Omega_{reg}$, a contradiction.
If instead any blow-up at $\bar x$ is planar, then $u$ is smooth across $\bar x$ (see \cite[Section 4]{SW2}), and hence $\bar x\in\Omega_{reg}$, which is again a contradiction.

The rest of the paper is devoted to the proof of Proposition \ref{blow.up.strong} in this more general case.
The main difficulty is that, when looking at the rescalings $u_r=u_{\bar x,r}$ of $u$ around $\bar x$,
we could see more and more points in $\mathcal{S}_{SW}$.
While we have the \emph{qualitative} information that they are locally finite in (the rescalings of) $D$,
they could become denser and denser,
preventing uniform bounds on the one-forms $d\beta_r$ defined on their complement.

Let us consider radii $r_j\to0$ and rescalings $u_j:=u_{\bar x,r_j}$ around $\bar x$, and assume by contradiction that the claim fails.
Then, by Corollary \ref{tilde.n.cone}, the blow-up has $\tilde N_0=\mu'$ on $\C\setminus\{0\}$,
for some integer $\mu'>\mu$, and recalling the proof of \eqref{rad.limit} we have
$$\mu\frac{|\nabla u_j|^2}{2}\,dx^2\rightharpoonup\mu'|\partial_{x_1}u_0\wedge\partial_{x_2}u_0|\,dx^2\quad\text{on }H,$$
where $u_0$ is the limit of $u_j=u_{\bar x,r_j}$ and equals $\hat u_0\circ\tilde\psi$ for a suitable
quasiconformal homeomorphism $\tilde\psi:\C\to\C$. Recall from \eqref{dist.bd} that the latter is
$K_0$-quasiconformal with $K_0:=(\frac{C_0}{\pi})^2$, where $C_0$ is the constant in \eqref{upper.bd.assumption}.


\begin{Dfi}
We let $\mathcal{Q}$ denote the class of $K_0$-quasiconformal homeomorphisms $\xi:\C\to\C$ with $\xi(0)=0$ and $\xi(1)=1$. Recall from \cite[Lemma A.3]{PiRi2}
or the proof of Lemma \ref{distortion}
that $\mathcal{Q}$ is sequentially compact in $C^0_{loc}(\C)$. \hfill $\Box$
\end{Dfi}

Fix any $x_0\in H$ and let $s_k:=\tilde\psi(x_0+2^{-k})-\tilde\psi(x_0)\in\C^*$.
Defining
$$\psi_k(x):=s_k^{-1}[\tilde\psi(x_0+2^{-k}x)-\tilde\psi(x_0)],\quad f_k(x):=|s_k|^{-1}\hat u_0(\tilde\psi(x_0)+s_kx),$$
we see that $|s_k|^{-1}u_0(x_0+2^{-k}x)=f_k\circ\psi_k(x)$ and that $\psi_k\in\mathcal{Q}$, while $f_k$ converges
to a (nontrivial) linear conformal map $\tilde L$ up to a subsequence, since $\hat u_0$ is an immersion near $\tilde\psi(x_0)$.
By sequential compactness of $\mathcal{Q}$, the rescaled maps $f_k\circ\psi_k$
converge in $C^0_{loc}$ to a map of the form $\tilde L\circ\psi_\infty$, for some $\psi_\infty\in\mathcal{Q}$, up to another subsequence.

Further, by the chain rule (see, e.g., \cite[Lemma III.6.4]{LV}), for the Jacobians we have
\begin{equation}\label{jac.conv}
\int_\omega J(f_k\circ\psi_k)\,dx^2=\int_{\psi_k(\omega)}J(f_k)\,dx^2\to\int_{\psi_\infty(\omega)}J(\tilde L)\,dx^2=\int_\omega J(\tilde L\circ\psi_\infty)\,dx^2
\end{equation}
for any smooth domain $\om\subset\subset\C$, showing the convergence
$$J((\hat u_0)_{x_0,2^{-k}})\,dx^2\rightharpoonup c^2J(\tilde L\circ\psi_\infty)\,dx^2$$
for some $c\in(0,\infty)$.

Hence, letting $L:=c\tilde L$, $\psi:=\psi_\infty$, and applying a diagonal argument, we can find suitable rescalings $w_j$ of $u_j$ (around points converging to $\bar x$) such that $w_j\to L\circ\psi$
in $C^0_{loc}$ and weakly in $W^{1,2}_{loc}$,
but with
$$\mu\frac{|\nabla w_j|^2}{2}\,dx^2\rightharpoonup\mu'|\partial_{x_1}(L\circ\psi)\wedge\partial_{x_2}(L\circ\psi)|\,dx^2.$$
In particular, the maps $w_j$ do not converge strongly in $W^{1,2}(\omega)$ to $L\circ\psi$, for any $\omega\subset\subset\C$.

\begin{Lm}
There exists a sequence $\ep_j\to0$ and points $x_j\in B_1(0)$ such that
$$\int_{B_r(x_j)}|\nabla(\pi_{\mathcal P^\perp}\circ w_j)|^2\,dx^2\le\ep_j\int_{B_r(x_j)}|\nabla w_j|^2\,dx^2$$
for any $r\in(0,1]$, where $\mathcal P$ is the Legendrian plane spanned by $L$. \hfill $\Box$
\end{Lm}

In the statement, we let $\pi_{\mathcal{P}}(v):=\ang{v,Z_1}Z_1+\ang{v,Z_2}Z_2$ and $\pi_{\mathcal{P}^\perp}(v):=v-\pi_{\mathcal{P}}(v)$
for all $v\in\R^5$, where $(Z_1,Z_2)$ is an orthonormal basis of $\mathcal{P}$.

\begin{proof}
Since the varifolds induced by $(B_2(0),w_j,\mu)$
converge to the one induced by $(\psi(B_2(0)),L,\mu')$, we have
$$\int_{B_2(0)}|\nabla(\pi_{\mathcal P^\perp}\circ w_j)|^2\,dx^2\to0.$$
Hence, we can find a sequence $\ep_j\to0$ such that the previous integral is $o(\ep_j)$. Now we let $S_j\subseteq B_{1/2}^2$ be the set of points where the claim fails. By the classical Besicovitch covering theorem,
we can cover $S_j$ with a family $\mathcal{F}_j$ of balls such that $\sum_{B\in\mathcal{F}_j}\mathbf{1}_{B}\le C$
and $\int_B|\nabla(\pi_{\mathcal P^\perp}\circ w_j)|^2\,dx^2>\ep_j\int_B|\nabla w_j|^2\,dx^2$.
Hence, we get
$$\int_{S_j}|\nabla w_j|^2\,dx^2\le\ep_j^{-1}\sum_{B\in\mathcal{F}_j}\int_B|\nabla(\pi_{\mathcal P^\perp}\circ w_j)|^2\,dx^2
\le C\ep_j^{-1}\int_{B_2(0)}|\nabla(\pi_{\mathcal P^\perp}\circ w_j)|^2\,dx^2=o(1).$$
On the other hand, we also have
$$\int_{B_1(0)}\mu|\nabla w_j|^2\,dx^2\to\int_{B_1(0)}\mu'|\partial_{x_1}(L\circ\psi)\wedge\partial_{x_2}(L\circ\psi)|\,dx^2>0,$$
showing that $S_j$ is (eventually) a proper subset of $B_1(0)$.
\end{proof}

Since the regularity theorem holds for $(B_2(0),w_j,\mu)$, the set $B_2(0)\cap\mathcal{S}_{SW}(w_j)$ is locally finite.
Thus, considering the points $x_j$ given by the previous lemma, regardless of whether $x_j\in \mathcal{S}_{SW}(w_j)$ or not, we can let
$$\rho_j:=\min\{\operatorname{dist}(x_j,\mathcal{S}_{SW}(w_j)\setminus\{x_j\}),1\}>0.$$

The idea now is that at each scale $\rho_j<r\le1$ the map $w_j$ is almost flat on $B_r(x_j)$,
forcing the corresponding PHSLV to be close to a planar one with constant (a.e.) integer multiplicity. At the largest scale this multiplicity is $\mu'$,
while at the smaller scale $\rho_j$ it will be $\mu$, essentially thanks to an application of Proposition \ref{strong.conv},
yielding a contradiction. We now make this idea precise.

In the sequel, we fix two constants $\hat c,\hat C>0$ such that, recalling \eqref{not.res}, all rescalings
$$w_{j,r}:=(w_j)_{x_j,r}$$
satisfy
$$\|\r\circ w_{j,r}\|_{L^\infty(B_1(0))}\le \hat C,\quad\operatorname{diam}\pi_{\C^2}\circ w_{j,r}(B_{1/2}(0))\ge\hat c,$$
for all $r\in(0,1/4]$. It is clear that such constants exist, thanks to a trivial compactness argument
using the uniform doubling bounds valid for the Dirichlet energy of these maps, established in Lemma \ref{ahlfors}
(note that, although this result was proved in the previous part, its proof did not use the assumption $\mathcal{S}_{SW}=\emptyset$).

\begin{Dfi}
We let $\mathcal{M}$ be the set of maps $f:B_1(0)\to\C^2$ of the form
$f=h\circ\xi$, for a homeomorphism $\xi\in\mathcal{Q}$ and a weakly conformal map $h:\xi(B_1(0))\to\C^2$ such that the bounds
$$\|\nabla f\|_{L^2}\le1,\quad\|f\|_{L^\infty}\le\hat C,\quad\operatorname{diam}f(B_{1/2}(0))\ge\hat c$$
hold true. \hfill $\Box$
\end{Dfi}

\begin{Dfi}
We consider the set $\mathcal{X}$ consisting of pairs $(f,\lambda)$ where $f\in\mathcal{M}$, while $\lambda$ is a (positive) measure on $B_1(0)$ with total mass at most $\nu_0\pi$, where
$$\nu_0:=\lf[\frac{C_0}{\pi}\rg].$$
We endow it
with the product of the $C^0_{loc}$ topology and the weak-$^*$ topology on measures and we let
$d$ be a distance on $\mathcal{X}$ metrizing this product topology. \hfill $\Box$
\end{Dfi}

\begin{Dfi}
For $k=1,\dots,\nu_0$,
we consider the subset ${\mathcal{X}}_k\subset\mathcal{X}$ consisting of those pairs $(f,\lambda)$
with $f\in\mathcal{M}$ smooth and taking values in $\mathcal{P}$, and $\lambda$ given by
$$d\lambda=k|\partial_{x_1}f\wedge\partial_{x_2}f|\,dx^2,$$
whose total mass is at most $\nu_0\pi$ since $|\partial_{x_1}f\wedge\partial_{x_2}f|\le\frac{|\nabla f|^2}{2}$
has integral at most $\pi$. \hfill $\Box$
\end{Dfi}

\begin{Prop}
Each ${\mathcal{X}}_k\subset\mathcal{X}$ is a compact subset of $\mathcal{X}$. \hfill $\Box$
\end{Prop}

\begin{proof}
Given a sequence $(f_j=h_j\circ\xi_j,\lambda_j)\in\mathcal{X}_k$, we can assume that $\xi_j$ converges to a limit $\xi_\infty\in \mathcal{Q}$ in $C^0_{loc}$,
thanks to the compactness of $\mathcal{Q}$ (endowed with the $C^0_{loc}$ topology).
Since the maps $h_j$ are harmonic and uniformly bounded, it follows that a subsequence converges to a limit $h_\infty$ in $C^\infty_{loc}$ on $\xi_\infty(B_1(0))$, and it easily follows that $f_\infty:=h_\infty\circ\xi_\infty\in\mathcal{M}$. To conclude, we just need
to observe that the Jacobians $J(h_j\circ\xi_j)\,dx^2\rightharpoonup J(h_\infty\circ\xi_\infty)\,dx^2$ as measures, which is obtained as in \eqref{jac.conv}.
\end{proof}

Since these sets $\mathcal{X}_k$ are clearly disjoint (the lower bound involving $\hat c$ prevents the problematic pair $(0,0)$, which would otherwise be the common intersection point), we can find $\epsilon>0$ such that
$$d((f,\lambda),(f',\lambda'))>2\epsilon\quad\text{whenever }(f,\lambda)\in{\mathcal{X}}_k,\ (f',\lambda')\in{\mathcal{X}}_{k'},\ k\neq k'.$$

We now define the map
$$F_j:(0,1]\to\mathcal{X},\quad F_j(r):=\lf(\pi_{\C^2}\circ w_{j,r},\mu\frac{|\nabla w_{j,r}|^2}{2}\,dx^2\rg).$$
By continuity of this map, the contradiction will arise once we prove the next two propositions,
which show that $F_j(r)$ is $\epsilon$-close to one of the sets ${\mathcal{X}}_k$ for every $r$
and that, specifically, $k=\mu$ for $r=\rho_j$, while by assumption we know that $k=\mu'>\mu$ for $r=1$
(since $w_j$ is close to $L\circ\psi$ and the measure $\mu\frac{|\nabla w_{j,r}|^2}{2}\,dx^2$
is close to $\mu'|\partial_{x_1}(L\circ\psi)\wedge\partial_{x_2}(L\circ\psi)|\,dx^2$).

\begin{Prop}
For $j$ large enough, for any $0<r\le1$ there exists $k\in\{1,\dots,\nu_0\}$
such that
$$d(F_j(r),(f,\lambda))<\epsilon$$
for some $(f,\lambda)\in{\mathcal{X}}_k$. \hfill $\Box$
\end{Prop}

\begin{proof}
This follows immediately from the fact that, by our choice of $x_j$, we have
$$\int_{B_1(0)}|\nabla(\pi_{\mathcal P^\perp}\circ w_{j,r})|^2\,dx^2\le\ep_j.$$
Indeed, given any sequence of radii $r_j'\in(0,1]$, the PHSLVs $(B_1(0),w_{j,r_j'},\mu)$
will necessarily converge (along a subsequence) to a planar PHSLV of the form $(\xi(B_1(0)),h,k)$ for some integer $k\ge1$,
with $h:\xi(B_1(0))\to\mathcal{P}$ weakly conformal and $w_{j,r_j'}\to h\circ\xi$ in $C^0_{loc}(B_1(0))$, as well as
$$\mu\frac{|\nabla w_{j,r_j'}|^2}{2}\,dx^2\rightharpoonup|\p_{x_1}(h\circ\xi)\wedge\p_{x_2}(h\circ\xi)|\,dx^2$$
on $B_1(0)$. Finally, thanks to \eqref{upper.bd.assumption}, the density of the limit induced varifold
is at most $\frac{C_0}{\pi}$ at each point. By definition of $\nu_0$, we deduce that $k\le\nu_0$.
\end{proof}

\begin{Prop}
For $j$ large enough, for $r=\rho_j$ the previous index $k$ equals $\mu$. \hfill $\Box$
\end{Prop}

\begin{proof}
By our choice of $\rho_j$, the map $w_{j,\rho_j}$ (defined on $B_1(0)$) has $\mathcal{S}_{SW}(w_{j,\rho_j})\subseteq\{0\}$.
The conclusion now follows directly from Lemma \ref{strong.conv}.
\end{proof}

\appendix
\section*{Appendix: a counterexample to the closure of PHSLVs with degenerating conformal class}
\renewcommand{\thesection}{A}
\setcounter{equation}{0}
\setcounter{Th}{0}

In this appendix, inspired by a counterexample by Orriols from \cite{Orr},
we provide an explicit example where a sequence of parametrized Hamiltonian stationary Legendrian varifolds,
given by embedded tori in $S^5$ with degenerating conformal class,
converges to a varifold whose weight is supported on a \emph{one-dimensional} Hopf fiber, and is thus \emph{non-rectifiable}.
This should not be too surprising, since this fiber has nonetheless Hausdorff dimension $2$ for the Carnot--Carath\'eodory distance
associated with the canonical horizontal distribution on $S^5$.

As discussed earlier in the paper, this is a new phenomenon compared to the isotropic setting,
and shows in particular that the control of the conformal class has to be assumed in order to guarantee compactness of PHSLVs,
unless the stronger notion of PHSLV$^*$ is adopted (see Definition \ref{strong.phslv}, Remark \ref{bubbling}, and Remark \ref{neck}).
It also shows that compactness of integral HSLVs fails, while we still know that these are closed among rectifiable varifolds
by Theorem \ref{cpt.int}.

\begin{Th}
\label{th-count-ex}
There exist a sequence of flat rectangular tori  $T_{1,R_k}:={\R}^2/2\pi({\Z}\oplus R_k\Z)$, with $R_k\to\infty$, and a sequence of conformal Hamiltonian stationary Legendrian embeddings
$$u_k:T_{1,R_k}\to S^5\subset\C^3,$$
where $S^5$ is equipped with the canonical contact form, such that the PHSLV
$$(T_{1,R_k},u_k,1)$$
induces a varifold $\mathbf{v}_k$ converging to a \emph{non-rectifiable} Hamiltonian stationary Legendrian varifold.
More precisely, calling $\frak H:S^5\to\mathbb{CP}^2$ the canonical Hopf fibration and
$$\mathcal{C}:=\mathfrak{H}^{-1}([0,0,1])=\{(0,0,e^{i\alpha})\mid \alpha\in\R/2\pi\Z\}$$
one of the Hopf fibers,
the limit of $\mathbf{v}_k(\mathcal P,p)$ disintegrates as
\[
{\bf v}_k \rightharpoonup2\pi\cdot\mu(e^{i\alpha})\otimes\mathcal H^1\res\mathcal{C}(0,0,e^{i\alpha}),
\]
where $\mu(e^{i\alpha})$ is the uniform measure on the set ${\frak P}(e^{i\alpha})$ of Legendrian planes in $T_{(0,0,e^{i\alpha})}S^5$
given by
\[
{\frak P}(e^{i\alpha}):=\{(e^{i\alpha})_*\mathcal{P}_{\tau,\eta} \mid (\tau,\eta)\in(\R/2\pi\Z)^2\},
\]
with $\mathcal{P}_{\tau,\eta}:=\operatorname{span}\{(\cos\tau)e_1+(\sin\tau)e_2,(\cos\eta)e_3+(\sin\eta)e_4\}\subset T_{(0,0,1)}S^5$
and $(e^{i\alpha})_*\mathcal{P}_{\tau,\eta}\subset T_{(0,0,e^{i\alpha})}S^5$ its image through the differential of the map $x\mapsto e^{i\alpha}x$. \hfill $\Box$
\end{Th}

\begin{Rm}
In fact, for a Legendrian immersion $u:\Sigma\to S^5$, the triple $(\Sigma,u,1)$ is a PHSLV if and only if $u$ is Hamiltonian stationary,
which is equivalent to the fact that $\mathfrak{H}\circ u$ is H-minimal (i.e., $\operatorname{div}(JH)=0$; see \cite{Cas}).
Moreover, the stronger notion of PHSLV$^*$ assumed in Remark \ref{neck}
is in fact equivalent to the notion of PHSLV if $\Sigma=S^2$ (since we can write any smooth $\om\subsetneq S^2$ as the difference of two disjoint unions of disks),
but not for higher genus. Thus, this counterexample does not contradict Remark \ref{neck}.
\hfill $\Box$
\end{Rm}

\begin{proof}
We equip $S^5\subset\R^6$ with the canonical contact structure $\al:=\sum_{\ell=1}^3 [x_{2\ell-1}\,dx_{2\ell}-x_{2\ell}\,dx_{2\ell-1}]$ and let $H:=\operatorname{ker}(\al)$. We denote by $\frak H:S^5\,(\subset{\C}^3)\to\mathbb{CP}^2$ the tautological Hopf fibration given by
$\frak H(w_1,w_2,w_3):=[w_1,w_2,w_3]$ and recall that $\nabla{\frak H}(z)$ gives an isometry $H_z\to T_{\pi(z)}\mathbb{CP}^2$. For any $t\in(0,1]$ we consider the following map $\hat{v}_t$  from $T_{1,1}:={\R}^2/2\pi({\Z}\oplus\Z)$ 
to $\mathbb{CP}^2$:
\[
\hat{v}_t(\theta,\phi):=[ te^{i(\theta+\phi)} ,  te^{-i(\theta-\phi)}, 1 ].
\]
In the chart
\[
\Xi:\mathbb{CP}^2\setminus \{[w_1,w_2, 0]\,:\, [w_1,w_2]\in \mathbb{CP}^1\}\to\C^2,\quad [z_1,z_2,1]\mapsto(z_1,z_2),
\]
the K\"ahler form $\om:=i\p\ov{\p}\log(1+|z|^2)$ of $\mathbb{CP}^2$ reads
$$\om=\frac{i}{(1+|z|^2)^2}[(1+|z|^2) (dz_1\wedge d\ov{z}_1+ dz_2\wedge d\ov{z}_2)-( \ov{z}_1\,dz_1+\ov{z}_2\,dz_2)\wedge( {z}_1\,d\ov{z}_1+{z}_2\,d\ov{z}_2) ];$$
in particular, we see that $d\al=\mathfrak{H}^*\om$ (as this holds at $z=0$ in the chart and these objects are $U(3)$-invariant).
Also, we easily compute that
$\hat{v}_t^\ast\om=0$
and thus $\hat{v}_t:T_{1,1}\to\mathbb{CP}^2$ is a Lagrangian embedding for $t\in(0,1]$.
We now look for a Legendrian lift of $\hat{v}_t$, i.e.,  for a real-valued function $\al_t(\theta,\phi)$ such that
\[
u_t(\theta,\phi):= \frac{e^{i\al_t}}{\sqrt{1+2t^2}} ( te^{i(\theta+\phi)} ,  te^{-i(\theta-\phi)}, 1 )\in S^5
\]
is Legendrian. We have respectively
\[
\p_\theta u_t=i(\p_\theta\al_t)u_t+ \frac{e^{i\al_t}}{\sqrt{1+2t^2}} (i te^{i(\theta+\phi)} , -it e^{-i(\theta-\phi)}, 0)
\]
and
\[
\p_\phi u_t=i(\p_\phi\al_t)u_t+\frac{e^{i\al_t}}{\sqrt{1+2t^2}} (ite^{i(\theta+\phi)} , i t e^{-i(\theta-\phi)}, 0).
\]
The Legendrian condition $iu_t\cdot \p_\theta u_t=iu_t\cdot \p_\phi u_t=0$ is then equivalent to
\[
\p_\theta\al_t=0,\quad\p_\phi\al_t+\frac{2t^2}{1+2t^2}=0.
\]
Hence,
\[
u_t(\theta,\phi):= \frac{e^{-i\frac{2t^2}{1+2t^2}\phi }}{\sqrt{1+2t^2}} ( te^{i(\theta+\phi)} , te^{-i(\theta-\phi)}, 1 )
\]
is a Legendrian lift of $\hat{v}_t$. We have 
\[
\p_\theta u_t= \frac{e^{-i\frac{2t^2}{1+2t^2}\phi }}{\sqrt{1+2t^2}} (ite^{i(\theta+\phi)} , -ite^{-i(\theta-\phi)}, 0)
\]
and
$$\p_\phi u_t
=i\lf[ \frac{1}{1+2t^2} u_t - \frac{e^{-i\frac{2t^2}{1+2t^2}\phi }}{\sqrt{1+2t^2}} (0,0,1)\rg],$$
which give
\[
|\p_\theta u_t|^2=\frac{2t^2}{1+2t^2},\quad |\p_\phi u_t|^2=\frac{2t^2}{(1+2t^2)^2},\quad\p_\theta u_t\cdot\p_\phi u_t=0.
\]
Hence, in the new local coordinates $(\theta,{\phi}_t:=\phi/\sqrt{1+2t^2})$, the map $u_t$ is conformal and reads
\[
u_t= \frac{e^{-i\frac{2t^2}{\sqrt{1+2t^2}}\phi_t }}{\sqrt{1+2t^2}} ( te^{i(\theta+\sqrt{1+2t^2}\phi_t)} , te^{-i(\theta-\sqrt{1+2t^2}\phi_t)}, 1 ).
\]
We easily compute
$$\p^2_{\theta^2} u_t
=- u_t+\frac{e^{-i\frac{2t^2}{\sqrt{1+2t^2}}\phi_t }}{\sqrt{1+2t^2}}(0,0,1),$$
as well as
$$\p_{\phi_t} u_t
=i\frac{1}{\sqrt{1+2t^2}} u_t-i{e^{-i\frac{2t^2}{\sqrt{1+2t^2}}\phi_t }} (0,0,1)$$
and
$$\p^2_{\phi^2_t} u_t
=-\frac{1}{{1+2t^2}} u_t+\frac{1-2t^2}{\sqrt{1+2t^2}}{e^{-i\frac{2t^2}{\sqrt{1+2t^2}}\phi_t }} (0,0,1).$$
Combining these identities, we obtain
$$\p^2_{\theta^2} u_t+\p^2_{\phi^2_t} u_t
= -u_t [|\p_\theta u_t|^2+|\p_{\phi_t} u_t|^2]+i \frac{2-2t^2}{\sqrt{1+2t^2}} \p_{\phi_t}u_t.$$
The function $\beta_t:=\frac{2t^2-2}{\sqrt{1+2t^2}}\phi_t$ can be defined only locally, but $d\beta_t$ still gives a globally defined harmonic one-form. Also, in $(\theta,\phi_t)$ coordinates we have
\[
\Delta u_t+u_t|\nabla u_t|^2+i\nabla\beta_t\cdot\nabla u_t=0.
\]
We deduce that $u_t$ is a conformal Hamiltonian stationary immersion. For any $k\in{\N}\setminus\{0,1\}$ we introduce $t_k\in(0,1]$ such that
\[
\frac{2t^2_k}{1+2t_k^2}=\frac{1}{k},\quad\text{i.e.},\quad 1+2t_k^2=\frac{k}{k-1},\quad\text{i.e.},\quad t_k:=\frac{1}{\sqrt{2k-2}}.
\]
Letting $u_k:=u_{t_k}$, $\gamma_k:=\sqrt{\frac{k}{k-1}}$, and writing $\phi$ in place of $\phi_{t_k}$, we have
\[
u_{k}(\theta,\phi)=\gamma_k^{-1} {e^{-i\gamma_k\phi/k }} \lf(\frac{1}{\sqrt{2k-2}} e^{i(\theta+\gamma_k\phi)} , \frac{1}{\sqrt{2k-2}} e^{-i(\theta-\gamma_k\phi)}, 1 \rg).
\]
Observe that $u_k$ is conformal and defines an Hamiltonian stationary Legendrian embedding of
$T_{1, k/\gamma_k}=\R^2/2\pi(\Z\oplus (k/\gamma_k)\Z)$, whose area is
\[
\frac{1}{2}\int_{T_{1, k/\gamma_k}}|\nabla u_k|^2\, dx^2=\int_0^{2\pi}\int_0^{2\pi k/\gamma_k} \frac{2t^2_k}{1+2t_k^2}\,d\phi\,d\theta=4\pi^2\gamma_k^{-1},
\]
which stays bounded.
Hence, calling ${\bf v}_k$ the corresponding rectifiable varifold of multiplicity $1$, which is induced by the PHSLV $(T_{1, k/\gamma_k},u_k,1)$, modulo extraction of a subsequence we have
\[
{\bf v}_{k} \rightharpoonup {\bf v}_\infty.
\]
Since the image of $u_k$ converges to the Hopf fiber $\pi^{-1}([0,0,1])$ in the Hausdorff topology,
we know that $|{\mathbf v}_\infty|$ is supported on this fiber. Let $\Phi:G\to\R$ be a continuous function, where $\Pi:G\to S^5$ is the bundle of Legendrian planes. Clearly,
$
{u}_k=  {e^{-i\gamma_k\phi/k }} (0,0,1)+o(1)
$
and, denoting $\tau:=\theta+\gamma_k\phi$ and $\eta:=-\theta+\gamma_k\phi$, we have
$$\frac{\p_\theta u_k}{|\p_\theta u_k|}
= \frac{1}{\sqrt{2}} {e^{-i\gamma_k\phi/k }}(-\sin\tau,\cos\tau,\sin\eta, -\cos\eta,0,0)+o(1)$$
and 
$$\frac{\p_\phi u_k}{|\p_\phi u_k|}
= \frac{1}{\sqrt{2}}{e^{-i\gamma_k\phi/k }} (-\sin\tau,\cos\tau,-\sin\eta, \cos\eta,0,0)+o(1).$$
We consider the two-dimensional family of Legendrian planes in $T_{(0,0,1)}S^5$ given by
$${\tilde{\mathcal P}}_{\tau,\eta}:=\frac{1}{2} (-\sin\tau,\cos\tau,\sin\eta, -\cos\eta,0,0)\wedge (-\sin\tau,\cos\tau,-\sin\eta, \cos\eta,0,0).$$
The image of the Legendrian plane ${\tilde{\mathcal P}}_{\tau,\eta}$ via the differential of $w\mapsto e^{i\alpha}w$ (as a map from $\C^3$ to $\C^3$) will be simply denoted by $(e^{i\al})_\ast {\tilde{\mathcal P}}_{\tau,\eta}$. The previous computations show that
$$\langle{\bf v}_k,\Phi\rangle
=\frac1k\int_{T_{1,k/\gamma_k}}
\Phi( (e^{-i\gamma_k\phi/k})_\ast {\tilde{\mathcal P}}_{\tau,\eta}, {e^{-i\gamma_k\phi/k}} (0,0,1) ) \, d\theta\,d\phi+o(1).$$
We have $\gamma_k\phi=\frac{1}{2}(\tau+\eta)$ and $\theta:=\frac{1}{2}(\tau-\eta)$. Hence, we get $d\theta\wedge d\phi= \frac{1+o(1)}{2}\, d\tau\wedge d\eta$, and the previous integral becomes
$$\frac1{2k}\int_{\R^2/\Gamma}\Phi( (e^{-i(\tau+\eta)/(2k)})_\ast {\tilde{\mathcal P}}_{\tau,\eta}, {e^{-i(\tau+\eta)/(2k)}} (0,0,1) )\,d\tau\,d\eta,$$
where $\Gamma:=2\pi\Z(1,-1)\oplus2\pi\Z(k,k)$, or equivalently
$$\frac1{2k}\sum_{\ell=1}^k\int_{Q_\ell}\Phi( (e^{-2\pi i\ell/k})_\ast {\tilde{\mathcal P}}_{\tau,\eta}, {e^{-2\pi i\ell/k}} (0,0,1) )\,d\tau\,d\eta+o(1),$$
where $Q_1$ is the square with vertices $(0,0)$, $(2\pi,-2\pi)$, $(4\pi,0)$, and $(2\pi,2\pi)$,
while $Q_\ell$ is its translation by $(\ell-1)(2\pi,2\pi)$ (note that $\frac{\tau+\eta}{2k}=\frac{2\pi\ell}{k}+O(k^{-1})$ on $Q_\ell$).
Thus, we have
\begin{align*}\langle{\bf v}_\infty,\Phi\rangle&=\frac{1}{4\pi}\int_{Q_1}\lf[\int_0^{2\pi}\Phi( (e^{i\al})_\ast {\tilde{\mathcal P}}_{\tau,\eta}, {e^{i\al}} (0,0,1) )\,d\al\rg]\,d\tau\,d\eta\\
&=\frac{1}{4\pi^2}\int_{[0,2\pi]^2}\lf[\int_0^{2\pi}2\pi\Phi( (e^{i\al})_\ast {\tilde{\mathcal P}}_{\tau,\eta}, {e^{i\al}} (0,0,1) )\,d\al\rg]\,d\tau\,d\eta,
\end{align*}
as claimed.
\end{proof}


\begin{thebibliography}{99}
\bibitem{All} W. K. Allard. \textbf{On the first variation of a varifold}. \emph{Ann. of Math. (2)} 95 (1972), 417--491.
\bibitem{BCW} A. Bhattacharya, J. Chen, and M. Warren.
\textbf{Regularity of Hamiltonian stationary equations in symplectic manifolds}.
\emph{Adv. Math.} 424 (2023), art. 109059.
\bibitem{BS} A. Bhattacharya and A. Skorobogatova. \textbf{Variational integrals on Hessian spaces: partial regularity for critical points}.
arXiv preprint 2307.01191.
\bibitem{CDPT} L. Capogna, D. Danielli, S. Pauls, and J. Tyson.
\textbf{An introduction to the Heisenberg group and the sub-Riemannian isoperimetric problem}.
Vol. 259 in \emph{Progress in Mathematics}. Birkh\"auser Verlag, Basel, 2007
\bibitem{Cas} I. Castro, H. LI, and F. Urbano.
\textbf{Hamiltonian-minimal Lagrangian submanifolds in complex space forms}.
\emph{Pacific J. Math.} 227 (2006), no. 1, 43--63. 
\bibitem{DLP} F. Da Lio and A. Pigati.
\textbf{Free boundary minimal surfaces: a nonlocal approach}.
\emph{Ann. Sc. Norm. Super. Pisa, Cl. Sci. (5)} 20 (2020), no. 2, 437--489. 
\bibitem{Daz} P. Dazord. \textbf{Sur la g\'eom\'etrie des sous-fibr\'es et des feuilletages Lagrangiens}. 
\emph{Ann. Sci. \'Ecole Norm. Sup. (4)} 14 (1981), no. 4, 465--480.
\bibitem{DPP} G. De Philippis and A. Pigati.
\textbf{Non-degenerate minimal submanifolds as energy concentration sets: a variational approach}.
\emph{Comm. Pure Appl. Math.} 77 (2024), no. 8, 3581--3627.
\bibitem{EG.reg} L. C. Evans and R. F. Gariepy.
\textbf{On the partial regularity of energy-minimizing, area-preserving maps}.
\emph{Calc. Var. Partial Differential Equations} 9 (1999), no. 4, 357--372.
\bibitem{EG} L. C. Evans and R. F. Gariepy.
\textbf{Measure theory and fine properties of functions}, revised edition. Vol. in \emph{Textbooks in Mathematics}. CRC Press, Boca Raton, FL, 2015.
\bibitem{Gaia} F. Gaia. \textbf{Hamiltonian stationary maps with infinitely many singularities}. arXiv preprint 2406.09344.
\bibitem{GOR}  F. Gaia, G. Orriols, and T. Rivi\`ere.
\textbf{A variational construction of Hamiltonian stationary surfaces with isolated Schoen--Wolfson conical singularities}. To appear  in \emph{Comm. Pure Appl. Math.}
\bibitem{GKN} M. Godlinski, W. Kopczynski, and P. Nurowski.
\textbf{Locally Sasakian manifolds}. \emph{Classical and Quantum Gravity} 17 (2000), no. 18, pp. L105--L115.
\bibitem{Hum} C. Hummel. \textbf{Gromov’s compactness theorem for pseudo-holomorphic curves}. Vol. 151 in \emph{Progress in
 Mathematics}. Birkh\"auser, Basel, 1997.
\bibitem{IT} Y. Imayoshi and M. Taniguchi.
\textbf{An introduction to Teichm\"uller spaces}. Springer--Verlag, Tokyo, 1992.
\bibitem{Kir} B. Kirchheim. \textbf{Rectifiable metric spaces: local structure and regularity of the Hausdorff measure}.
\emph{Proc. Amer. Math. Soc.} 121 (1994), no. 1, 113--123.
\bibitem{LV} O. Lehto and K. I. Virtanen. \textbf{Quasiconformal mappings in the plane}, second edition.
Vol. 126 in \emph{Die Grundlehren der mathematischen Wissenschaften}. Springer--Verlag, Berlin--Heidelberg--New York, 1973.
\bibitem{MiWo} M. Micallef and J. Wolfson.
\textbf{Area minimizers in a K3 surface and holomorphicity}.
\emph{Geom. Funct. Anal.} 16 (2006), no. 2, 437--452.
\bibitem{Min} W. P. Minicozzi.
\textbf{The Willmore functional on Lagrangian tori: its relation to area and existence of smooth minimizers}.
\emph{J. Amer. Math. Soc.} 8 (1995), no. 4, 761--791.
\bibitem{Oh2} Y.-G. Oh. \textbf{Volume minimization of Lagrangian submanifolds under Hamiltonian deformations}.
\emph{Math. Z.} 212 (1993), 175--192.
\bibitem{Orr} G. Orriols. \textbf{Existence and partial regularity of Legendrian area-minimizing currents}.
arXiv preprint 2406.09378.
\bibitem{Pi} A. Pigati. \textbf{The viscosity method for min-max free boundary minimal surfaces}.
\emph{Arch. Ration. Mech. Anal.} 244 (2022), no. 2, 391--441.
\bibitem{PiRi1} A. Pigati and T. Rivi\`ere.
\textbf{The regularity of parametrized integer stationary varifolds in two dimensions}.
\emph{Comm. Pure Appl. Math.} 73 (2020), no. 9, 1981--2042.
\bibitem{PiRi2} A. Pigati and T. Rivi\`ere.
\textbf{A proof of the multiplicity one conjecture for min-max minimal surfaces in arbitrary codimension}.
\emph{Duke Math. J.} 169 (2020), no. 11, 2005--2044.
\bibitem{Riv1} T. Rivi\`ere. \textbf{The regularity of conformal target harmonic maps}.
\emph{Calc. Var. Partial Differential Equations} 56 (2017), no. 4, art. 117.
\bibitem{Riv-IHES} T. Rivi\`ere. \textbf{A viscosity method in the min-max theory of minimal surfaces}.
\emph{Publ. Math. Inst. Hautes \'Etudes Sci.} 126 (2017), 177--246.
\bibitem{Riv-Will} T. Rivi\`ere. \textbf{Min-max hierarchies, minimal fibrations and a PDE based proof of the Willmore conjecture}.
arXiv preprint 2007.05467.
\bibitem{Riv2} T. Rivi\`ere. \textbf{Almost monotonicity formula for H-minimal Legendrian surfaces in the Heisenberg group}.
\emph{Comm. Pure Appl. Math.} 77 (2024), no. 3, 1940--1957.
\bibitem{Riv3} T. Rivi\`ere. \textbf{Area variations under Legendrian constraint}. To appear in \emph{Peking Math. J.}
\bibitem{ScU} R. Schoen and K. Uhlenbeck. \textbf{A regularity theory for harmonic maps}.
\emph{J. Differential Geom.} 17 (1982), no. 2, 307--335.
\bibitem{SW2} R. Schoen and J. Wolfson. \textbf{Minimizing area among Lagrangian surfaces: the mapping problem}.
\emph{J. Differential Geom.} 58 (2001), no. 1, 1--86.
\end{thebibliography}
\end{document}